\documentclass[a4paper, twoside]{report}

\usepackage{geometry}
\usepackage{setspace}
\usepackage[T1]{fontenc}
\usepackage{charter}
\usepackage{float}
\usepackage{todonotes}
\usepackage{mathrsfs}
\usepackage{amsmath}
\usepackage{amssymb}
\usepackage{amsthm}
\usepackage{etoolbox}
\usepackage{graphicx}
\usepackage{enumerate}
\usepackage{tikz-cd}
\usepackage{fancyhdr}
\usepackage{pictex}
\usepackage{multirow}
\usepackage{mathtools}
\usepackage{hyperref}
\usepackage{adjustbox}
\usepackage{footnote}
\usepackage{bbm}
\usepackage{tikz,tikz-3dplot}
\usetikzlibrary{arrows,decorations,patterns,shadings,calc,shapes}
\usepackage{pgfplots}
\pgfplotsset{compat=1.12}

\makesavenoteenv{tabular}
\makesavenoteenv{table}

\newtheorem{theorem}{Theorem}[section]

\newtheorem{cor}[theorem]{Corollary}

\newtheorem{definition}[theorem]{Definition}
\newtheorem{fact}[theorem]{$\textbf{Fact}$}
\newtheorem{factdefinition}[theorem]{Fact/Definition}
\newtheorem{lemma}[theorem]{Lemma}
\newtheorem{example}[theorem]{Example}
\AtEndEnvironment{example}{\null\hfill\qedsymbol}
\newtheorem{notation}[theorem]{Notation}

\newtheorem{theor}{Theorem}
\newtheorem{defi}{Definition}
\newtheorem{lemm}{Lemma}
\newtheorem{factum}{Fact}
\newtheorem{coro}{Corollary}
\newtheorem{claiming}{Claim}

\newtheorem*{th*}{Theorem}
\newtheorem*{not*}{List of symbols and notations}
\newtheorem*{def*}{Definition}
\newtheorem*{fact*}{Fact}
\newtheorem*{example*}{Example}
\newtheorem*{lemma*}{Lemma}
\newtheorem*{cor*}{Corollary}
\newtheorem*{claim*}{Claim}
\theoremstyle{remark}
\newtheorem{remark}[theorem]{\bf Remark}
\newtheorem{remarks}[theorem]{\bf Remarks}

\newtheorem{claim}[theorem]{\bf Claim}
\newtheorem*{ackn*}{{\bf Acknowledgements}}

\newcommand{\bul}{\bullet}

\newcommand{\Q}{\mathbb{Q}}
\newcommand{\R}{\mathbb{R}}

\newcommand{\C}{\mathbb{C}}
\newcommand{\Z}{\mathbb{Z}}
\newcommand{\N}{\mathbb{N}}
\newcommand{\G}{\mathbb{G}}

\newcommand{\PP}{\mathbb{P}^1}
\newcommand\QQ[1]{\{\mathbb{P}^1/{{#1}}\}}
\newcommand\XQ[2]{\{{#1}/{{#2}}\}}
\newcommand\XQQ[2]{\left[{#1}/{{#2}}\right]}

\newcommand\XX[1]{\{X/{{#1}}\}}
\title{Dynamical Topoi}
\author{Jacopo Garofali}

\begin{document}
	\baselineskip=24pt
	\begin{titlepage}
		\begin{center}
			{\Large \textsc{University of Rome}}\\
			\vspace{0.4em}
			{\Large \textsc{ ``Tor Vergata''}}\\
			\vspace{2em}
			\begin{figure}[!h]
				\centering
				  \includegraphics[scale=0.2]{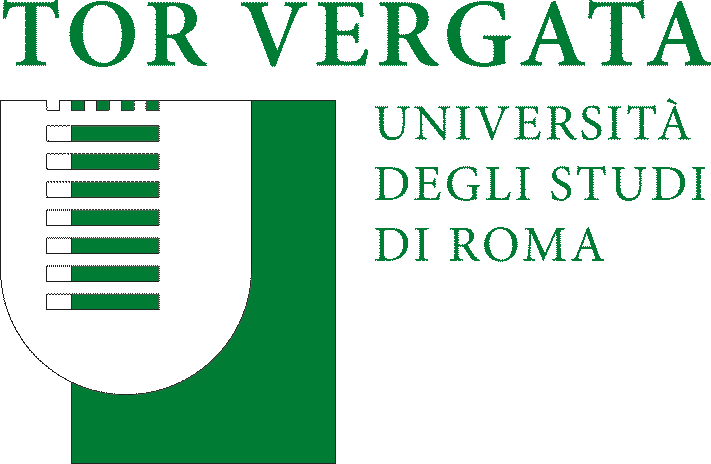}
			\end{figure}
			\vspace{1em}
			{\Large \textsc{Doctoral school in Mathematics}}\\
			\vspace{1em}
			
			{\Large \text{Ph.D. thesis}}\\
			\vspace{0.5 cm}
			\hrulefill \\ 
			\bigbreak
			{\huge \textbf{Dynamical sheaves}\medskip\\}
		\hrulefill
		\end{center}
		\vspace{1 em}
		\begin{center}
			\begin{tabular}{c c c c c c c c}
				Advisor &&&&&&&  Author \\[0.2cm]
				\large{Prof. Michael McQuillan}  &&&&&&& \large{Jacopo Garofali}\\[1.3cm] \\
		&&&	Coordinator \\[0.2cm]
			&&&	\large{Prof. Carlangelo Liverani}  
			\end{tabular}
		\end{center}
		
		\begin{table}[!b]
			\centering
			{\normalsize Cycle XXXIV - Academic Year 2020/2021}
		\end{table}
	\end{titlepage}
\thispagestyle{empty}
\phantom{h}
\newpage
\begin{table}[ht!]
\hrulefill 
\bigbreak
{\centering  {\Large \textsc{ABSTRACT}}

}

\hrulefill \\
\thispagestyle{empty}
\smallbreak
In the present work we define and study the classifying (or ``quotient'') site
$\XQQ{X}{\Sigma}$ for any small site $X$ with (countable) coproducts
endowed with an action of a (countable) semigroup $\Sigma$. 
A simple case (the most relevant to our applications) 
 is the case $\Sigma=\N$, on which, therefore
 we concentrate. 
Our main result consists in establishing an 
equivalence of the corresponding Tòpos 
with the category of sheaves on $X$ with ``$\Sigma-$action''.
We prove also that there is a spectral sequence computing sheaf 
cohomology in $\XQQ{X}{\N}$ and we deduce some topological
properties of this site, such as its fundamental group.
We finally apply the above formalism in Holomorphic Dynamics, 
giving a Tòpos-theoretic interpretation
of Epstein's work on the Fatou-Shishikura Inequality and Infinitesimal Thurston's Rigidity.
\end{table}
\begin{figure}[H]
\centering
\includegraphics[scale=0.4]{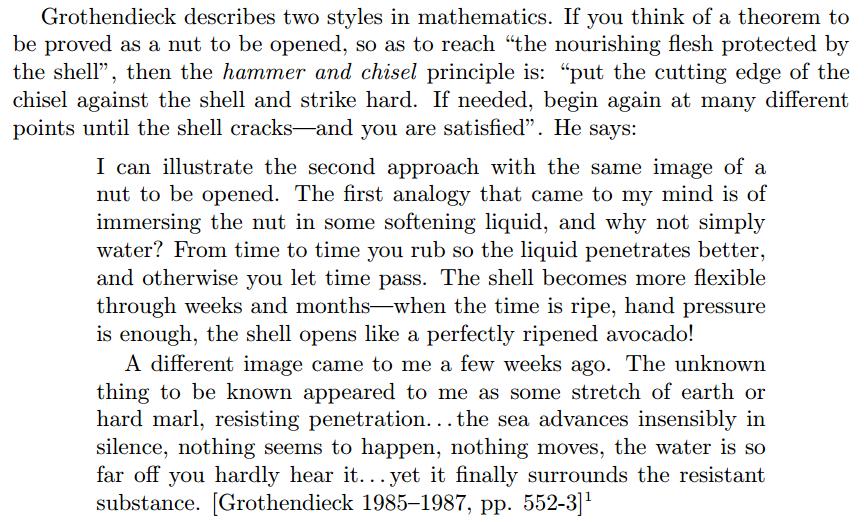}
\end{figure}
\begin{figure}[H]
	\centering
	\includegraphics[scale=1.5]{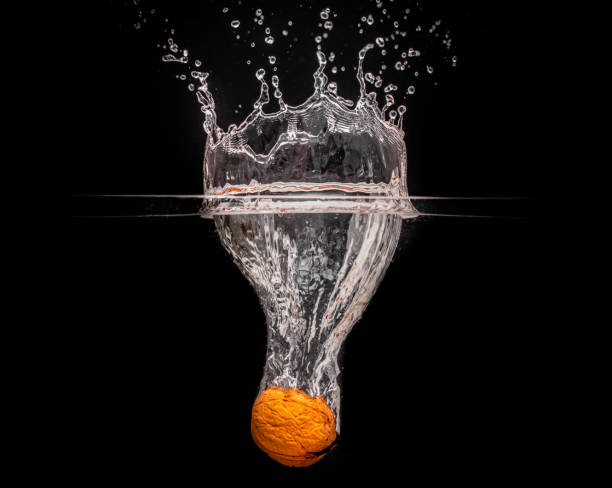}
\end{figure}
\newpage
\phantom{h}
\thispagestyle{empty}
	\tableofcontents
	\newpage
	\begin{not*}\
		
		\noindent
		We set $\N\coloneqq\Z_{\geq 0}$. Let $X$ be a small category and let us denote by $ob(X)$ and $ar(X)$ the set of its objects and arrows, respectively. We sometimes refer improperly to an arrow as a ``map''.
		Throughout this work, we shall employ the following notation.
		\begin{itemize}
		\item $X^{\wedge}\coloneqq [X^{\rm op}, {\rm Set}]$ is the category
		of ${\rm Set}$-valued pre-sheaves on $X$;
		\item Let $U \in ob(X)$, then $\underline{U} \in X^{\wedge}$
		denotes the representable functor associated to $U$, \textit{i.e.} 
		${\rm Hom}_{X}(-,U)$;
		\item If $J_{X}$ is a Grothendieck topology on $X$,
		we denote by $Sh(X, J_{X})$
		the category of $Set$-valued sheaves on $X$. When there is no
		 room for confusion we shall simplify notation and write $Sh(X)$;
		 \item We denote by $s,t: ar(X) \to ob(X)$ 
		 the source and target functors, taking $i: U \to V \in X$ to $U$ and $V$, respectively;
		 \item $a: X^{\wedge} \to Sh(X)$ denotes the associated sheaf functor, cf. \eqref{sheaficationofpreheaf};
		 \item Morphisms between sites may be denoted by $f,g,h$, as well as
		  lower-case Greek letters $\pi, \rho, \varepsilon$. 
		  On the other hand, functors are denoted by lower-case Latin letters $a,b,c, s, t$.
		  \item If $c: Y \to X$ is  a functor, we denote by $c^*$ 
		 the composition functor associated to $c$,
		  \begin{equation}\label{compfunctorcap1}
		  	X^{\wedge} \to Y^{\wedge}, \; \mathcal{F} \mapsto \mathcal{F}\circ c,
		  \end{equation}
	  and with $c_!$ (resp. $c_*$) its left adjoint (resp. its right adjoint).
	  \item If $g: X \to Y$ is a morphism of sites, we denote by 
	  \begin{equation}\label{pushforwardfunctorcap1}
	  g_*:	X^{\wedge} \to Y^{\wedge},
	  \end{equation}
   the functor $(g^{-1})^*$. As usual, we abuse notation by denoting $g^{-1}\coloneqq (g^{-1})_!$ its left adjoint. On the other hand,
   as usual the restriction
   $
   g_*|_{Sh(X)}: Sh(X) \to Sh(Y)
   $
    is still denoted by $g_*$, while its left adjoint is denoted by $g^*$;
    \item Let $X$ be a category with coproducts indexed by a set $A$. 
    Then, for any collection of object $\{ U_{\alpha} \}_{\alpha \in A}$ we denote by
    $\iota_{\alpha} :U_{\alpha} \hookrightarrow \coprod\limits_{\alpha \in A}U_{\alpha}$
    the monomorphisms furnished by the definition of coproduct;
    \item If $c$ is a functor admitting a right adjoint $d$, we write $d=ad(c)$;
    \item If $\mathcal{F}$ is a (pre-)sheaf on a site $X$ and $R$ is a sieve on an object of $X$, we adopt the following notation:
    $
    \mathcal{F}(R)\coloneqq {\rm Hom}_{X^{\wedge}}(R,\mathcal{F}).
    $
		\end{itemize}
	Moreover, as usual, by monoid we mean a semigroup with identity and 
	a monoid morphism $(A,\ast) \to (B, \circ)$ is a multiplicative 
	map preserving the identity element. \\
	\end{not*}
\vfill

	\chapter*{\centering Introduction}
	\addtocontents{toc}{\def\string\@dotsep{100}}
	\addcontentsline{toc}{chapter}{\large{\textbf{Introduction}}}
	The aim of the present work is to show that Grothendieck's
    strategy of studying
	sheaves, and their cohomology, may be applied 
	to the field of Holomorphic Dynamical Systems. 
	The procedure of defining new \textit{Grothendieck topologies} -- with the aim of
	studying the corresponding \textit{T\`opoi} -- has proved an extremely
	successful substitute for classical metric topology
	 in many fields such as Algebraic Geometry 
	in characteristic $p$, so it is reasonable to
	expect analogous results in Dynamics.\\	
	Let $(X, J_X)$ be a site, \textit{i.e.} $X$ is a category, and $J_{X}$
	is a Grothendieck topology on $X$, cf. \cite[II.1]{SGA4}. We write simply $X$ when 
	there is no room for confusion. 
	\begin{defi}\label{intro:def1}
We say that a site $X$ has the property {\rm \bf (D)} if 
	\begin{itemize}	
		\item[\mbox{\boldmath $D_1)$}] $X$
		is a small site that has finite limits and countable coproducts;
		\item[\mbox{\boldmath $D_2)$}]  coproducts  in $X$ commute with finite limits in $X$;
		\item[\mbox{\boldmath $D_3)$}] coproducts in $X$ are disjoint, cf. 
		{\rm\cite{sheavesmaclane}}. In other words, 
		we require that  the defining morphisms 
		$\displaystyle U_{\alpha} \hookrightarrow 
		U=\coprod_{\alpha \in A} U_{\alpha}$
		 are monomorphisms such that $\forall \, \alpha,\beta \in A$
		\begin{equation}
			\emptyset \overset{\sim}{\longrightarrow} U_{\alpha} \times_{U}U_{\beta},
		\end{equation}
	where $\emptyset$ is the initial object of $X$.
	\end{itemize}
	Let us fix a site $X$ that has the property {\rm (D)}, together 
	with an endomorphism of sites $f:X \to X$ 
	that commutes with finite limits and coproducts. 
	We say that 
	\begin{equation}\label{intro:sitewithdyndef}
		\text{$(X,f)$ is a \textbf{site with dynamics}.}
	\end{equation}
	\end{defi}
The objective of Chapter I is to study the
	 \textit{discrete dynamical system} 
	generated by $f$, \textit{i.e.} the monoid morphism 
	\begin{equation}\label{intro:action}
		\begin{tikzcd}[row sep=0.04pc, column sep=1.3pc]
			\Phi: \N\arrow[ r] & {\rm End}(X),
		\end{tikzcd} 
	\end{equation}
satisfying $\Phi(1)=f$. 
	The first section of Chapter I is dedicated to defining
	the ``classifying site'' $\XX{f}$ (or simply ``dynamical site''), 
	cf. \eqref{defdynsite1}, 
	associated to a site with dynamics
	$(X, f)$. 
	We already anticipate that in Chapter II
	we consider any countable
	 monoid $\Sigma$, 
	and an action of $\Sigma$, \textit{i.e.} a
	monoid morphism $\Phi$ as in \eqref{intro:action},
	with $\N$ replaced by $\Sigma$, on a site $X$ having the
	property {\rm (D)}.
	To these data, there is associated a site,
	denoted by $\XQQ{X}{\Sigma}$ ,
	which is called the
	``classifying site for the action of $\Sigma$ on  $X$''.
	Let us observe that:
	\begin{equation}\label{countability}
		\begin{aligned}
			 \text{replacing the assumption "countable coproducts" in $(D_1)$, Definition \ref{intro:def1}, with}\\
			 \text{"coproducts of cardinality $\# \Sigma$", the assumption ``$\Sigma$ countable'' can be dropped.}
		\end{aligned}
	\end{equation}
	When $\Sigma=\N$ and $\Phi$ is generated by $f$,
	the sites $\XX{f}$ and $\XQQ{X}{\N f}$ are equivalent 
	(actually, they are equivalent both as categories 
	and as sites).
	Therefore, the notation $\XX{f}$
	can be considered as a shorthand for 
	$\XQQ{X}{\N f}$. 
	This notation has been set up
	in order to avoid confusion 
	with $\XQQ{X}{f}$,	
	which in turn is employed as a shorthand for the 
	classifying site $\XQQ{X}{\Z f}$ associated to 
	the \textit{group} action on $X$ generated by 
	an automorphism $f$.
 	\begin{defi}\label{intro:def2}
		Let $(X,f)$ be a site with dynamics, \eqref{intro:sitewithdyndef}.
		 Consider the category $\XX{f}$
		whose objects are maps $u: f^{-1}U \to U$, for $U \in ob(X)$, 
		and arrows $u \to v$ are commutative squares
			\begin{equation*}
			\begin{tikzcd}[row sep=2.6pc, column sep=2.6pc]
				f^{-1}U \arrow[r, "u"] 
				\arrow[d," f^{-1}j"'] &
				U \arrow[d, "j"] \\
				f^{-1}V \arrow[r, "v"] &
				V,
			\end{tikzcd} 
		\end{equation*}
	where $j \in ar(X)$. The category $\XX{f}$ endowed 
	with the topology induced, cf. \ref{sitewithdyn}, by the target functor $t: \XX{f} \to X$,
	is called the \textbf{classifying site} for the action of $\N f$ on $X$.
	\end{defi}
	The main achievement of this section is the following description of the 
	 T\`opos $Sh(\XX{f})$, cf. \ref{equivalencelemmaXf}.
	  \begin{theor}\label{intro:theorem1}
	  		Let $(X,f)$ be a site with dynamics,  \eqref{intro:sitewithdyndef}.
Then, a sheaf on the classifying site $\XX{f}$ 
consists of a pair $(\mathcal{F},\varphi)$
		where $\mathcal{F}$ is a sheaf on $X$ and 
		$$
		\varphi: f^*\mathcal{F} \to \mathcal{F}
		$$
		is a sheaf morphism. A morphism of sheaves 
		$(\mathcal{F},\varphi) \to (\mathcal{G},\gamma)$ is a commutative square
			\begin{equation}\label{intro:square1}
			\begin{tikzcd}[row sep=2.6pc, column sep=2.6pc]
				f^{*}\mathcal{F} \arrow[r, "\varphi"] 
				\arrow[d," f^{*}\theta"'] &
				\mathcal{F} \arrow[d, "\theta"] \\
				f^{*}\mathcal{G}\arrow[r, "\gamma"] &
				\mathcal{G}
			\end{tikzcd} 
		\end{equation}
	\end{theor}
Let us now consider a topological space 
$X$ and a continuous self-map $f$. 
Then, the site $Ouv(X)$ of open sets of $X$, cf. \cite[IV.2]{SGA4} 
is not, in general, 
closed for countable, cf.  \eqref{countability}, coproducts, \textit{e.g.} 
the coproduct of two open sets is not necessarily immersed in $X$.
Although the site $\XQ{Ouv(X)}{f}$ (resp. $\XQQ{Ouv(X)}{\Sigma}$) 
can still be defined, it may be, in general, trivial, \textit{i.e.} 
there may be no nontrivial 
backward invariant open sets (\textit{e.g.} $\R/\Q$). 
In this case, \textit{par abus de langage}, we write
$\XX{f}$, or $\XQQ{X}{\N f}$, 
as a shorthand for the classifying site 
obtained by extending $f^{-1}: Ouv(X) \to Ouv(X)$ to the
 category whose 
objects are countable, cf.  \eqref{countability}, disjoint unions of open sets of $X$ and arrows are 
local homeomorphisms, 
(\textit{i.e.} a (very small) 
topological version of the étale site of an algebraic variety $X/k$,
cf.  \cite{milne}, \cite{etalesite}). 
Note that $Ouv(X)$ and the étale 
site of $X$ are equivalent, \textit{i.e.} their respective Tòpoi are equivalent, cf. \ref{cor:topspace}.
Therefore, in Theorem \ref{intro:theorem1} we
can take pairs $(\mathcal{F},\varphi)$ with
 $\mathcal{F}$ a sheaf on the topological space $X$, 
 and $\varphi: f^*\mathcal{F} \to \mathcal{F}$ a morphism of sheaves on $X$.\\
    In order to illustrate the properties of $\XX{f}$, let us 
    consider the simplest, 
    and perhaps the most illuminating, 
    example, namely when $X=pt$  
    is the topological space with one point, and the action of $\N$ is, of course, trivial.
   The resulting dynamical site is called the 
   ``classifying site of $\N$'' and it is denoted by $B_{\N}\coloneqq\XQQ{pt}{\N}$.
    This may be compared with the more 
    familiar notions of classifying space of $\Z$, 
    $B{\Z}$, or classifying champ of $\Z$, $B_{\Z}=[pt/\Z]$.
    Indeed, for $G$ a group, 
   the classifying space $BG$ (a.k.a. the Eilenberg-MacLane 
    $K(G,1)$ for $G$ discrete, cf. \cite{eilenberg-maclane}) 
   involves the construction of a topological space, unique
    up to homotopic equivalence, 
  that classifies isomorphism classes of $G$-principal bundles. 
  In particular, $BG$ carries a natural 
  contractible total space $EG$, and $EG \to BG$ is
   a universal $G$-torsor. Thus, maps from a (paracompact) topological
   space $X$ to $BG$ define, up to homotopic equivalence, a unique 
   isomorphism class of $G$-torsors over $X$. 
  More recently, the classical definition of classifying space 
   has been replaced by a higher 
   categorical construction, \textit{i.e.} the 
 \textit{Deligne-Mumford champ} 
 $B_G=[pt/G]$, cf. \cite[2.4.2]{L-MB}.
The advantages of the latter formalism
are that for any ``reasonable'' category $X$ (e.g. topological spaces)
the maps $X \to B_G$, up to natural transformations rather than homotopy,
classify isomorphism classes of $G$-torsors on $X$.
As an example, in the case $G=\Z$ we have 
$B\Z=S^1$ (up to homotopic equivalence) 
and it is easy to see that any locally constant sheaf
on $S^1$ defines a locally constant sheaf $\mathcal{E}$ on $\R$ invariant for the $\Z$-action 
generated by $x \mapsto x+1$, \textit{i.e.} 
its pullback under this map is isomorphic to 
$\mathcal{E}$ as a $\Z$-torsor. 
Similarly, given a set $F$ with a 
$\Z$-action, the invariants of this 
action define a locally constant sheaf on $S^1$,
but there are many more sheaves on the circle.
However, a sheaf on $B_{\Z}=[pt/\Z]$ is a $\Z$-set. 
Indeed, 
using the theory of Deligne-Mumford champs, 
to any ``reasonable'' category $X$ (e.g. topological spaces)
with $G$ action we can associate a quotient champ $[X/G]$,
\textit{i.e.} the classifier of the action.
The resulting theory of sheaves on $[X/G]$ is characterized by the fact that
the projection map
 $\pi: X \to [X/G]$ has the
 following
 ``descent property'', cf. \cite[12.2.1]{L-MB}: to give a sheaf $\mathcal{F}$ on $[X/G]$
 is equivalent to giving a sheaf on $X$ with a ``descent datum'',
 \textit{i.e.} an isomorphism between the two inverse images of 
 $\pi^*\mathcal{F}$ on $X\times G$, satisfying a cocycle condition.
 The last property is commonly reformulated in terms of $G$-equivariant
 sheaves on $X$, \textit{i.e.} sheaves with an action of $G$:
 \begin{factum}
 	Let $X$ be a topological space.
 	The category of sheaves on the classifying champ 
 	$[X/G]$ is equivalent to the category of $G$-equivariant sheaves on $X$.
 \end{factum}
Actually, if $X=pt$, it is usual to take the latter as 
the definition, cf. \cite[IV.2.4]{SGA4}, \cite[5.1]{giraud}, 
of $\mathcal{B}_G$ as a Tòpos.
Classically, however, 
examples of sites inducing the tòpos 
of $G$-equivariant sheaves, for $G$ a group, on a topological space $X$ already exist. 
For example, if $G$ is a discrete group acting on $X$ through homeomorphisms,
the tòpos 
of $G$-equivariant sheaves on $X$ can be realized as
 the topos of sheaves on the (relative) site whose 
 underlying category is the fibration corresponding 
 to the indexed category assigning the unique object of $G$
 to the category $Ouv(X)$ and the arrow $g \in G$
 to the frame homomorphism $g^{-1}:Ouv(X)  \to Ouv(X)$, 
 cf. \cite[2.1.11(c)]{johnstone2002sketches}.
 In this way, one obtains as site of definition for the topos of $G$-equivariant 
 sheaves on $X$ the site $(O_G(X),E)$, where $O_G(X) $ is the category 
 whose objects are the open sets of $X $ and whose arrows $ U \to V $
 are the one for which there exists an element $g \in G $ such that $g(U) \subset V$ and $E$ is the topology 
 given by the families $\{g_i : U_i \to  V | i \in I\}$ such that $V$ is the union of the $g_i(U_i)$’s. \\
  The novelty of our 
approach, already in the case of groups,
 is to have explicitly provided a description of this Tòpos
in terms of the category of sheaves on the aforesaid
 site $\XQQ{X}{\Sigma}$: in the case $\Sigma=G$ is a group,
the definition of the site 
 $\XQQ{X}{\Sigma}$ is extremely concrete 
 and the corresponding tòpos is equivalent 
 to the category of sheaves on the 
classifying champ $[X/G]$.
Critically, however, our construction, 
cf. \ref{sheavesonSigmasite}, works in the 
more general set up of a 
 monoid $\Sigma$ acting on a site $X$ 
(assuming the property {\rm (D)}, \ref{intro:def1}),
wherein the ``$2$-functor in groupoids'' approach of \cite[VI]{sga1}
fails to yield enough sheaves since not all arrows are invertible.
We refer the reader to the text, cf. \ref{Sigmasite}, for the definition of 
$\XQQ{X}{\Sigma}$. Here, we state the main result of Chapter II.
\begin{theor}\label{intro:theorem2}
	Let $X$ be a site with the property {\rm (D)}, \ref{intro:def1}, and let $\Sigma$
	be a countable, cf.  \eqref{countability}, monoid acting on $X$. Then, the category of 
	sheaves on the site $\XQQ{X}{\Sigma}$ is the following:
		\begin{itemize}
		\item  The objects are pairs $ (\mathcal{F}, \varphi_{\bul}) $ consisting of
		\begin{enumerate}[a)]
			\item A sheaf $ \mathcal{F}\in 
			ob(Sh(X)) $;
			\item A (right) action of $\Sigma$ on $\mathcal{F}$, \textit{i.e.}
			a map of sheaves  
			$$ \varphi_{\bul}\coloneqq 
			\coprod_{\sigma \in \Sigma}\varphi_{\sigma}:
			 \coprod_{\sigma \in \Sigma}\sigma^{*}\mathcal{F} \to \mathcal{F}\in ar(Sh(X)),
			$$
			satisfying $\varphi_{id_{\Sigma}}=id_{\mathcal{F}}$ and the \textit{semigroup property}:
			$$
			\varphi_{\sigma \tau}=\varphi_{\tau}(\tau^*\varphi_{\sigma})\quad \forall \; \sigma,\tau \in \Sigma,
			$$ 
			\textit{i.e.} the following diagram in $Sh(X)$ is commutative:
				\begin{equation*}
				\begin{tikzcd}[row sep=2.6pc, column sep=4.6pc]
					\tau^{*}\mathcal{F} \arrow[r, "\varphi_{\tau}"] & \mathcal{F} \\
					\tau^{*}\left( \sigma^{*}\mathcal{F} \right)=
					(\sigma \tau)^{*}\mathcal{F} \arrow[u, "\tau^{*}\varphi_{\sigma}"] 
					\arrow[ur, "\varphi_{\sigma \tau}"']
				\end{tikzcd} 
			\end{equation*}
		\end{enumerate}
		\item The arrows 
		$  (\mathcal{F}, \varphi_{\bul}) \to (\mathcal{G}, \gamma_{\bul}) $ are natural transformations
		$
		\theta \in {\rm Hom}(\mathcal{F},
		\mathcal{G}),
		$
		such that $\forall\; \sigma \in \Sigma$ the following diagram commutes
		\begin{equation*}
			\begin{tikzcd}[row sep=2.6pc, column sep=2.6pc]
				\sigma^*\mathcal{F} \arrow[r, "\varphi_{\sigma}"] 
				\arrow[d," \sigma^*\theta"'] &
				\mathcal{F} \arrow[d, "\theta"] \\
				\sigma^*\mathcal{G} \arrow[r, "\gamma_{\sigma}"] &
				\mathcal{G}
			\end{tikzcd} 
		\end{equation*}
	\end{itemize}
\end{theor}
Consider the following direct consequence of 
Theorem \ref{intro:theorem1}.
\begin{coro}
	The Tòpos $Sh(B_{\N})$ consists of pairs $(F,\varphi)$, where:
	\begin{itemize}
		\item $F$ is a set;
		\item $\varphi: F \to F$ is a ${\rm Set}$-endomorphism of $F$.
	\end{itemize}
\end{coro}
 Note that $Sh(B_{\N})$ is somehow ``softer'' than $Sh(B_{\Z})$, 
 since the latter consists only of pairs $\left( F, \varphi \right)$,
 where $\varphi$ is an automorphism of $F$.
 A direct generalization of the above discussion is provided by 
 the following example. Let $X$ be a topological space and 
 consider the case in which both $\N$ and $\Z$ act trivially on $X$.
 We compare the new theory resulting from $\XQQ{X}{\N}$ with the
 one provided by $[X/\Z]$,
 \textit{i.e.} $\XX{id_X}$ and $[X/id_X]$ 
 according to the shorthand notation.
 \begin{coro}
 	Let $X$ be a site that has the property {\rm (D)}, \ref{intro:def1}. 
 A sheaf on $\XX{id_X}$ consists of a pair $(\mathcal{F},\varphi)$,
 where $\mathcal{F}$ is a sheaf on $X$ and 
 $\varphi: \mathcal{F} \to \mathcal{F}$ is any map of sheaves.
 \end{coro}
It follows again that considering a monoid action on $X$
results in softening the category $Sh([X/id_X])$, 
since the latter consists only of pairs $(\mathcal{F}, \varphi)$ 
where $\varphi: \mathcal{F} \to \mathcal{F}$ is invertible.
In the second section of Chapter I we 
introduce a different ``dynamical site''
 ${\rm E}_f(X)$,
to which we refer as Epstein's site, 
cf. \ref{Epsdynsitedef}. In order to simplify notation, we
drop the dependence on $X$ and write just ${\rm E}_f$
when there is no room for confusion.
	\begin{defi}
	Let $(X, f)$ be a site with dynamics,  \eqref{intro:sitewithdyndef}.
	Consider the category ${\rm E}_f(X)$
	whose objects are (ordered) pairs 
	$u_{\bul}\coloneqq (u_0,u_1)$, 
	where 
	$$
	u_0: U_0 \to U_1 \in X, \quad 
	u_1: f^{-1}U_0 \to U_1 \in X,\quad  \mbox{for } (U_0,U_1) \in ob(X)\times ob(X),
	$$
	and arrows $u_{\bul} \to v_{\bul}$ are pairs of commutative squares
	\begin{equation*}
	\begin{tabular}{l r }
		\begin{tikzcd}[row sep=2.6pc, column sep=2.6pc]
			U_0 \arrow[r, "u_0"] 
			\arrow[d," j_0"'] &
			U_1 \arrow[d, "j_1"] \\
			V_0 \arrow[r, "v_0"] &
			V_1
		\end{tikzcd} & 	\qquad
		\begin{tikzcd}[row sep=2.6pc, column sep=2.6pc]
			f^{-1}U_0 \arrow[r, "u_1"] 
			\arrow[d,"f^{-1} j_0"'] &
			U_1 \arrow[d, "j_1"] \\
			f^{-1}V_0 \arrow[r, "v_1"] &
			V_1
		\end{tikzcd} 
	\end{tabular}
	\end{equation*}
	where $j_0,j_1 \in ar(X)$. 
\end{defi}
The category ${\rm E}_f(X)$ defined above 
is the underlying category of Epstein's site.
The main result of this section is the following, cf. \ref{sheavesonepsteinsite}.
\begin{theor}\label{intro:theorem3}
Let $(X, f)$ be a site with dynamics,  \eqref{intro:sitewithdyndef}.
A sheaf on ${\rm E}_f(X)$ consists of the following data: 
\begin{itemize}
 \item a pair of sheaves  $\mathcal{F}_{\bul}=(\mathcal{F}_0, \mathcal{F}_1) \in ob(Sh(X)) \times ob(Sh(X))$;
 \item a map 
 $	\varphi_{\bul}\coloneqq\varphi_0\coprod \varphi_1: 
 \mathcal{F}_0\coprod f^*\mathcal{F}_0 \to \mathcal{F}_1 \in ar(Sh(X))$.
\end{itemize}
A morphism of sheaves $(\mathcal{F}_{\bul}, \varphi_{\bul})
\to (\mathcal{G}_{\bul}, \gamma_{\bul})$ consists of two commutative squares
	\begin{equation*}
	\begin{tabular}{l r }
		\begin{tikzcd}[row sep=2.6pc, column sep=2.6pc]
			\mathcal{F}_0 \arrow[r, "\varphi_0"] 
			\arrow[d," \theta_0"'] &
			\mathcal{F}_1 \arrow[d, "\theta_1"] \\
			\mathcal{G}_0 \arrow[r, "\gamma_0"] &
			\mathcal{G}_1
		\end{tikzcd} & 	\qquad
		\begin{tikzcd}[row sep=2.6pc, column sep=2.6pc]
			f^*\mathcal{F}_0 \arrow[r, "\varphi_1"] 
			\arrow[d," f^*\theta_0"'] &
			\mathcal{F}_1 \arrow[d, "\theta_1"] \\
			f^*\mathcal{G}_0 \arrow[r, "\gamma_1"] &
			\mathcal{G}_1
		\end{tikzcd} 
	\end{tabular}
\end{equation*}
\end{theor}
In the same spirit as above, let us
compare the theory resulting from ${\rm E}_f(X)$ 
and the one provided by $\XX{f}$ in the simple case of
the action of $\N$ on $X=pt$.
\begin{factum}
	A sheaf on the site ${\rm E}_{id_{pt}}(pt)$ consists of a pair of sets
	$(F_0,F_1)$, together with a pair of maps 
	$$
	\varphi_i: \mathcal{F}_0 \to F_1, \quad i=0,1.
	$$
\end{factum}
It is evident that this Tòpos is even softer than $\XX{f}$, 
and in fact we
are far from exploiting all of its features in this work. 
 What is needed in our applications 
is a site slightly more 
elaborate than $\XX{f}$ 
but certainly much less general than ${\rm E}_f(X)$.
However, this hypothetical intermediate site may well not exist.
The aim of introducing the site ${\rm E}_f(X)$, 
which in fact ``enlarges'' the site $\XX{f}$, 
is to consider forward orbits that are not immersed in $X$.
As an example, if $X$ is a Hausdorff topological space, one would like to immerse
the discrete topological space $\displaystyle \coprod_{n\geq 0}f^n(x)$ into $X$, but 
this fails as soon as the sequence $\{ f^n(x) \}_{n}$ admits an accumulation point in $X$.
A truncation of the above sequence provides a
sheaf on ${\rm E}_f$.
Namely, if we take as $\mathcal{F}_0$ the set of functions on the first $n$ 
points of the sequence and as $\mathcal{F}_1$ the set of functions on the first
$n-1$ points, with maps $\varphi_0,\varphi_1$ 
given by the natural projection and pullback map, respectively,
we have a sheaf on ${\rm E}_f$ by Theorem \ref{intro:theorem3}.\\
The following result makes evident the fact that $E_f$
``enlarges'' $\XX{f}$, cf. \ref{F0=F1}.
\begin{lemm}\label{intro:lemma0}
	The Tòpos $Sh(\XX{f})$ is equivalent to the subcategory $S_f$ 
	of 
	$Sh({\rm E}_f)$ consisting of diagonal pairs: 
	$$
	ob(S_f)= \{ (\mathcal{F},\mathcal{F},id_{\mathcal{F}}, \varphi) : (\mathcal{F}, \varphi) \in Sh(\XX{f}) \}.
	$$
\end{lemm}

In Chapter III we first establish some notation for the set (or group) of morphisms between two sheaves in $\XX{f}$ (resp. in ${\rm E}_f$).
Then, we consider the category of abelian sheaves on the above-mentioned sites
with the aim of studying ``$\mathrm{Ext}$'' functors. 
We fix a sheaf -- or an abelian sheaf --
$(\mathcal{F}, \varphi)$ on $\XX{f}$ (resp. $(\mathcal{F}_{\bul}, \varphi_{\bul})$ on ${\rm E}_f$) and 
denote by $ \mathbb{H}{\rm om}(\mathcal{F}_{\bul}, \mathcal{G}_{\bul})$ 
the set -- or the group -- of sheaves morphisms from
$(\mathcal{F},\varphi)$ to $(\mathcal{G}, \gamma)$ in
$Sh(\XX{f})$ (resp. in $Sh({\rm E}_f)$).
Note that we drop, for convenience,
the defining morphisms $\varphi$ in the notation
and that we use Lemma \ref{intro:lemma0} to identify 
$(\mathcal{F},\varphi)$ with its diagonal pair in $S_f$.
Moreover, we use the same notation for both 
sheaves in $\XX{f}$ and ${\rm E}_f$.
It follows from Theorem \ref{intro:theorem1}, 
cf. \eqref{HomequalizerX/f}, that for any pair of abelian 
sheaves $(\mathcal{F}, \varphi), (\mathcal{G}, \gamma)$ on $\XX{f}$ we have 
$$
 \mathbb{H}{\rm om}(\mathcal{F}_{\bul}, \mathcal{G}_{\bul})= 
 \ker \left( {\rm Hom}(\mathcal{F}, \mathcal{G})\overset{d^{0,0}}{\longrightarrow}
  {\rm Hom}(f^*\mathcal{F}, \mathcal{G}) \right),
$$
where the map $d^{0,0}$ on the right is given, cf. \eqref{intro:square1},
by 
\begin{equation}\label{differencemap}
	\theta \in  {\rm Hom}(\mathcal{F}, \mathcal{G}) \mapsto 
	\theta\varphi - \gamma (f^*\theta) \in  
	{\rm Hom}(f^*\mathcal{F}, \mathcal{G}).
\end{equation}
In view of the fact that the category
of abelian sheaves on a site has enough injectives, cf. \ref{enoughinj}, 
the derived functors $\mathbb{E}{\rm xt}^i(\mathcal{F}_{\bul},-)
={\rm R}^i \mathbb{H}{\rm om}(\mathcal{F}_{\bul},-)$ are well defined. 
\begin{lemm}\label{intro:lemma1}
Let $\mathcal{F}_{\bul}$ and $\mathcal{G}_{\bul}$ be
sheaves in $Ab(\XX{f})$. There exists a spectral 
	sequence $\{ E_r \}_{r \geq 0}$ degenerating at $r=2$ such that
	$$
	E_r^{p,q} \Rightarrow \mathbb{E}{\rm xt}^{p+q}(\mathcal{F}_{\bul},\mathcal{G}_{\bul}),
	$$
	where
	\begin{equation*}
		\begin{aligned}
		&	E_1^{0,q}={\rm Ext}^q(\mathcal{F},\mathcal{G}) 
		\overset{d^{0,q}}{\longrightarrow} {\rm Ext}^q(f^*\mathcal{F},\mathcal{G})=
		E_1^{1,q}, \quad \forall q \geq 0  \\
		& E_1^{p,q}=0, \quad \forall p >1, q\geq 0,
		\end{aligned}
	\end{equation*}
and the differentials $d^{0,q}$, $q>1$, are the maps derived from $d^{0,0}$, cf. \eqref{ExtisequenceEf}.\\
The homological algebra in this case can be organized in the following long exact sequence
	\begin{equation}\label{intro:longexactseq}
	\begin{tikzcd}[row sep=1.2pc, column sep=0.8pc]
		0 \arrow[r] & {\mathbb{H}{\rm om}(\mathcal{F}_{\bul},\mathcal{G}_{\bul} )}\arrow[r] \arrow[d, phantom, ""{coordinate, name=Z}]
		&  {{\rm Hom}(\mathcal{F},\mathcal{G} )} \arrow[r, "d^{0,0}"] & {{\rm Hom}(f^*\mathcal{F},\mathcal{G})}\arrow[dlll,
		rounded corners,
		to path={ -- ([xshift=3ex]\tikztostart.east)
			|- (Z) [near end]\tikztonodes
			-| ([xshift=-3ex]\tikztotarget.west)
			-- (\tikztotarget)}] \\
	{\mathbb{E}{\rm xt^1}(\mathcal{F}_{\bul}, \mathcal{G}_{\bul})}		\arrow[r]
& {{\rm Ext^1}(\mathcal{F},\mathcal{G})} \arrow[r, "d^{0,1}"] & {{\rm Ext^1}(f^*\mathcal{F},\mathcal{G})} \arrow[r] & \cdots
	\end{tikzcd}
\end{equation}
which clearly splits into short exact sequences 
for each $n \geq 0$:
\begin{equation*}
	\begin{tikzcd}[row sep=2.6pc, column sep=1.6pc]
		0 \arrow[r] & C^{n-1} \arrow[r] & 
		\mathbb{E}{\rm xt}^n(\mathcal{F}_{\bul}, \mathcal{G}_{\bul})
		\arrow[r] & K^n \arrow[r] & 0,
	\end{tikzcd}
\end{equation*}
where $C^{-1}\coloneqq 0$, and for each $n \geq 1$ we have set
	\begin{equation*}
	\begin{tikzcd}[row sep=1.6pc, column sep=1.8pc]
		C^n	\coloneqq	{\rm coker}\, \Big( {\rm Ext}^n(\mathcal{F},\mathcal{G})
		\arrow[r, "d^{0,n}"] &
		{\rm Ext}^n(f^*\mathcal{F},\mathcal{G})	\Big) 
	\end{tikzcd}
\end{equation*}
and
\begin{equation*}
	\begin{tikzcd}[row sep=1.6pc, column sep=1.8pc]
		K^n	\coloneqq	{\rm ker}\, \Big( {\rm Ext}^n(\mathcal{F},\mathcal{G})
		\arrow[r, "d^{0,n}"] &
		{\rm Ext}^n(f^*\mathcal{F},\mathcal{G}) \Big).
	\end{tikzcd}
\end{equation*}
\end{lemm}
The main result proved in this chapter, cf. \ref{exti_cor},
is a refined version 
of Lemma \ref{intro:lemma1}, computing ``$\mathrm{Ext}$''-functors
in $Ab({\rm E}_f)$.
The first section of this chapter is dedicated to prove the existence
of the above-mentioned spectral sequence by taking 
the point of view of \textit{extensions} in $Ab({\rm E}_f)$ (cf. Appendix A).
In the last section of this chapter we prove the last assertion of Lemma 
\ref{intro:lemma1}, cf. \ref{fact:longexact}.

In the same vein: if $f$ is an endomorphism of a ringed space $(X, \mathcal{O}_X)$,
 then the structure sheaf $\mathcal{O}_X$ defines naturally a sheaf 
$\mathcal{O}_{\bul}\coloneqq (\mathcal{O}_X, f^*)$ on $\XX{f}$, 
where $f^*: f^*\mathcal{O}_X \to \mathcal{O}_X$ is the 
defining morphism.
In such a context, the proof of Lemma \ref{intro:lemma1} goes over \textit{verbatim}
on replacing the category $Ab(\XX{f})$ with the subcategory 
of pairs $(\mathcal{F}, \varphi)$, 
where $\mathcal{F}$ is a $\mathcal{O}_X$-modules,
the sheaf of groups $f^*\mathcal{F}$ is required to be an $\mathcal{O}_X$-module, \textit{i.e.}
$\displaystyle f^*\mathcal{F} \otimes _{f^*\mathcal{O}_X}\mathcal{O}_X$,
and the morphism $f^*\mathcal{F} \to \mathcal{F}$
is $\mathcal{O}_X$-linear. 
We refer to this as the category 
of $\mathcal{O}_{\bul}$-modules on $\XX{f}$, and, as usual,
the notation $f^*\mathcal{F}$ is employed irrespectively of whether 
the pullback is to be understood in the category of  $\mathcal{O}_X$-modules, 
or abelian sheaves, since the context is invariably clear.

Chapter IV is an application of the cohomological results of Chapter III.
Here, we describe some topological properties of 
the site $\XX{f}$, for $X$ a Galois category, cf. \cite[V.5]{sga1} for the pro-finite
case, and \cite[III.i]{et} for the pro-discrete case.
Recall that a category $X$ together with a 
functor valued in finite sets,
$F: X \to {\rm F}Set$,
satisfying axioms $G1-G6$ of \cite[V.4]{sga1} is called a Galois category. 
The (pro-finite) group $\pi_1(X,F)={\rm Aut}(F)$ (or simply $\pi_1(X)$) 
is called the fundamental (pro-finite)
group of the Galois category $(X,F)$. In the pro-discrete
case we essentially replace everywhere ``finite'' by ``discrete''.
The main result of this chapter consists in
establishing a relation between the
fundamental (pro-finite) groups of $X$ and $\XX{f}$, cf. \ref{cor:pi1}.
\begin{lemm}
		Let $X$ be a connected, cf. \ref{connectedsite}, Galois site
		satisfying the property ${\rm (D)}$, \ref{intro:def1}.
		Then, the fundamental (pro-finite) group of $\XX{f}$ is 
	$$
	\pi_1(\XX{f})=\Z,
	$$
   if $X$ is simply connected, \textit{i.e.} if $\pi_1(X)=0$,
   and it is an extension of $\Z$ by a (pro-finite) quotient of the 
   (pro-finite) group $\pi_1(X)$ otherwise.
\end{lemm}
This objective is pursued by taking the ``dual'' point of view, \textit{i.e.} by studying
$\underline{\Gamma}$-torsors over $\XX{f}$. Following the usual identifications, 
if $X$ is a category and $\Gamma$ is a finite group (considered as a trivial $\pi_1(X)$-module), 
we have, cf. \cite{torsor}, \cite[XIII.4.5]{sga1} and compare \cite[III.i.5]{et}: 
\begin{equation*}
	\begin{aligned}
&	H^1(X, \underline{\Gamma})=\{ \underline{\Gamma}-\mbox{torsors on } X \}/\cong,\\
& 	H^1(X, \underline{\Gamma})={\rm Hom}_{\bf Grp}(\pi_1(X),\Gamma).
	\end{aligned}
\end{equation*}
The second and third section of this chapter 
exam how classical results about recovering 
analytical information from topology change in the 
presence of an endomorphism. In particular, we consider 
a differentiable manifold $X$ and study the connection
between 
complex line bundles on $\XX{f}$ and the second cohomology group 
of $\Z$. There is, of course, an exponential 
 exact sequence, cf. \eqref{expseq:Xf}, on $\XX{f}$. Specifically, let $\mathcal{A}_{X}$
denote the sheaf of complex-valued differentiable functions on $X$,
and $f^*$ be the natural pullback of such, then
the pair $\mathcal{A}_{\bul}\coloneqq (\mathcal{A}_X,f^*)$
is a sheaf on $\XX{f}$.
Thus, the exponential sequence in $\XX{f}$
 is given by the following commutative diagram
	\begin{equation*}
	\begin{tikzcd}[row sep=1.2pc, column sep=1.8pc]
		0  \arrow[r]& \underline{\Z}(1)  \arrow[r]&  \mathcal{A}_{X}  
		\arrow[r, "{\rm exp}"] & \mathcal{A}_{X}^*   \arrow[r]& 0 \\
		0  \arrow[r]& f^*\underline{\Z}(1) \arrow[r]  \arrow[u, "f^*"]&  
		f^*\mathcal{A}_{X}  \arrow[r, "f^*{\rm exp}"]  \arrow[u, "f^*"] & 
		f^*\mathcal{A}_{X}^*   \arrow[r]  \arrow[u, "f^*"]& 0.
	\end{tikzcd}
\end{equation*}
Consequently, the classical computation can be carried out,
 \textit{mutatis mutandis}, but with the
 difference now that
$\mathcal{A}_{\bul}$-modules 
may not be acyclic, cf. \ref{cor:linebundles} and compare \cite[II.g.1]{et}.
\begin{factum}\label{intro:bundlefact}
	Let $X$ be a complex differentiable manifold.
	Then, the natural map 
	\begin{equation}\label{bundlequation}
		 \mathbb{H}^1(\XX{f}, \mathcal{A}_{\bul}^*) \longrightarrow
		\mathbb{H}^2(\XX{f}, \underline{\Z}(1)_{\bul})
	\end{equation}
	is not, in general, an isomorphism. 
	For example, if $X$ is a simply connected compact K\"aler manifold, 
     its kernel contains a copy of $\C^*$.
\end{factum}
\textit{``The first point to note is that the group 
$\mathbb{H}^1(\XX{f}, \mathcal{A}_{\bul}^*)$ classifies 
(isomorphism classes of)
$\mathbb{G}_m$-torsors on $\XX{f}$. 
Consequently its elements are
not simply line bundles, $\mathcal{E}$, on $X$ with a map of sheaves,
$f^*{\mathcal{E}}\rightarrow \mathcal{E}$, but bundles such that on some
$f$-invariant \'etale cover (in the site sense) $U\rightarrow X$,
the fibre $L_U$ has a nowhere vanishing 
$f$-invariant section.''}\footnote{Note added by the supervisor.}
As such, (isomorphism classes of) line bundles on $\XX{f}$ are classified
by (isomorphism classes of) pairs $(\mathcal{E}, \epsilon)$, where $\mathcal{E}$
is a line bundle on $X$, and
$\epsilon: f^*\mathcal{E} \to \mathcal{E}$ is an 
$\mathcal{A}_X$-linear isomorphism.
An isomorphism between two pairs 
$(\mathcal{E}, \epsilon)$, $(\mathcal{E}', \epsilon')$ is a
commutative diagram (of $\mathcal{A}_{X}$-modules)
\begin{equation*}
		\begin{tikzcd}[row sep=2.6pc, column sep=2.6pc]
		f^*\mathcal{E} \arrow[r, "\epsilon"] 
		\arrow[d," f^*\theta"] &
		\mathcal{E} \arrow[d, "\theta"] \\
		f^*\mathcal{E}' \arrow[r, "\epsilon'"] &
		\mathcal{E}'.
	\end{tikzcd} 
\end{equation*}
If, however, $X$ were simply connected, then 
the map in \eqref{bundlequation} simply sends the pair
$(\mathcal{E}, \epsilon)$ to the differentiable isomorphism
class of $\mathcal{E}$, and, for example:
\begin{factum}
Let $f: \PP \to \PP$ be a rational map of degree $D>1$. Then, we have 
$$
	\mathbb{H}^2(\QQ{f}, \underline{\Z}(1)_{\bul})=0.
$$	
\end{factum}
On the other hand,  an example of a (non-trivial) line bundle on $\XX{f}$ 
may as well arise 
by taking the pair $(\mathcal{A}_{X}, \lambda)$, 
consisting of the trivial line bundle on $X$, and
``multiplication''
$\lambda: f^*\mathcal{A}_{X} \to \mathcal{A}_{X}$ by a non-zero 
(and non-identity) complex number $\lambda \in \C^*$,
\textit{i.e.} the composition of 
$\lambda: \mathcal{A}_{X} \to \mathcal{A}_{X}, \; f \mapsto \lambda f$
with the canonical map $f^{*}\mathcal{A}_{X} \to \mathcal{A}_{X}$.
If we work holomorphically rather that differentiably, this is an exhaustive 
description of line bundles on $X=\PP$, to wit:
\begin{factum}
	Let $f: \PP \to \PP$ be a rational map of degree $D>1$. 
	Then, we have 
	$$
		 \mathbb{H}^1(\QQ{f}, \mathcal{O}_{\bul}^*) \cong \C^*.
	$$
\end{factum}

The next section is dedicated to define the 
De Rham cohomology
on $\XX{f}$, wherein a similar phenomenon is encountered . 
Let us abuse notation and equally denote by $\mathcal{A}_{X}=\mathcal{A}^0_{X}$
the sheaf of real-valued differentiable functions on $X$. Then,
the classical strategy of defining De Rham cohomology
of a manifold $X$ of dimension $n$ 
as the cohomology of the sequence of vector spaces 
\begin{equation*}
	\begin{tikzcd}[row sep=2.6pc, column sep=2.6pc]
		H^0(X, \mathcal{A}_X^{0}) \arrow[r, "d"] & 
		H^0(X, \mathcal{A}_X^{1})  \arrow[r, "d"] &  \cdots \arrow[r, "d"] &
		H^0(X, \mathcal{A}_X^{n}).
	\end{tikzcd} 
\end{equation*}
fails in general, since $\mathcal{A}_X$ modules are not necessarily acyclic,  
and it need to be replaced by its sheaf-theoretic 
definition, \textit{i.e.} the hyper-cohomology of the 
complex of sheaves 
\begin{equation*}
	\begin{tikzcd}[row sep=2.6pc, column sep=2.6pc]
		\mathcal{A}_X^{0} \arrow[r, "d"] & 
		\mathcal{A}_X^{1} \arrow[r, "d"] &  \cdots \arrow[r, "d"] & \mathcal{A}_X^{n}.
	\end{tikzcd} 
\end{equation*}
Subsequently, we compute explicitly the De Rham cohomology
 of $\XX{f}$ in the simplest possible case,
\textit{i.e.} when $X$ is contractible and hence 
has trivial De Rham cohomology groups
in positive degrees, cf. \ref{cor:DeRham}.
\begin{factum}
Let $X$ be a contractible differentiable manifold. 
Then, the De Rham cohomology groups of $\XX{f}$ are 
\begin{equation*}
	\mathbb{H}^{p} (\XX{f}, \underline{\R}_{\bul})= 
	\begin{cases}
		\R \quad \mbox{ if } p=0,1;\\
		0 \; \quad \mbox{ if } p>1.
	\end{cases}
\end{equation*}
\end{factum}
Consequently, we note that, although $B_{\N}$ and $B_{\Z}$
behave differently in terms of sheaf theory, 
they share most of their topological features. 
Indeed, the former does not admit a geometric realization 
analogous to $B_{\Z}$, but as soon as the orbit relations
are inverted in order 
to form an equivalence relation, the two coincide.

    In Chapter V we apply our machinery to
    holomorphic dynamical systems on the Riemann sphere $\PP$.
    Let $f:\PP \to \PP$ be a rational map of degree $D>1$.
    The first section is a revision of known results, namely
    the ``Fatou-Shishikura inequality'' 
    and ``Infinitesimal Thurston's Rigidity''.
    The former is an upper bound for the number of 
    stable regions of $f$ and can be found in its sharp version in 
    \cite{ASENS_1987_4_20_1_1_0}.
    The latter is Epstein's nomenclature  for the key infinitesimal content of 
    Thurston's topological characterization 
    of post-critically finite rational maps, cf. \cite{MR1251582},
    which he adapts in order to give a new
    proof of the former, refining it, cf. \cite{1999math......2158E}.
    In the
    same paper
    he also provides a different
    approach to the Fatou-Shishikura inequality, cf. \eqref{intro:EpsFS},
    revealing for the first time a direct connection between the
    latter and Thurston's Theorem.
     We conclude our revision by describing his original approach,
    cf. \ref{Eps}, \ref{T-E}.
    It was this fundamental paper of Epstein which motivated 
    this thesis since it cries out for a Tòpos theoretic 
    interpretation, and we afford such in the second section of
    Chapter V. 
More concretely, our contribution lies in the following interpretation of
 Epstein's extension of Infinitesimal 
Thurston's rigidity:
this result can be explained as an
	 ``absence of obstruction'' to lifting some 
	 ``local invariant infinitesimal deformations'' of $\PP$ --
	 more precisely ``local infinitesimal deformations'' of $\QQ{f}$ 
	 (resp. of ${\rm E}_f$) -- and 
	 one should probably refer to it as 
	 ``{\bf Thurston-Epstein Vanishing}''.
	 In fact, we first observe
	 that $\QQ{f}$ (resp. ${\rm E}_f$) carries not only a natural structure sheaf $\mathcal{O}_{\bul}$ but 
	also a sheaf of holomorphic differential forms $\Omega_{\bul}$.
	\footnote{From now on we work in the category 
		of sheaves of $\mathcal{O}_{\bul}$-modules on $\XX{f}$ (resp. ${\rm E}_f$), 
		and hence by morphism of sheaves in $\XX{f}$ (resp. ${\rm E}_f$)
		we mean a 
		$\mathcal{O}_X$-linear commutative diagram, cf. \ref{intro:theorem1}.}
	Thus, the tangent space of $\QQ{f}$ (resp. ${\rm E}_f$) is the 
	 $\C$-vector space of
	 sheaf morphisms between $\Omega_{\bul}$ and 
	 $\mathcal{O}_{\bul}$, in our notation
	 $
	 \mathbb{H}{\rm om}(\Omega_{\bul}, \mathcal{O}_{\bul}).
	 $
 The latter is isomorphic, by \ref{intro:lemma1}, to the 
 space of globally invariant vector fields on $\PP$, \textit{i.e.}
 $$
  \mathbb{H}{\rm om}(\Omega_{\bul}, \mathcal{O}_{\bul})
  ={\rm ker}
 \left( H^0(\PP, T_{\PP}) \overset{d^{0,0}}{\longrightarrow} 
 H^0(\PP, f^*T_{\PP}) \right),
 $$
 which actually vanishes, cf. \ref{noinvariantvf}.\\
	Similarly,
	 ``infinitesimal deformations'' of $\QQ{f}$ (resp. ${\rm E}_f$) are, by Lemma \ref{intro:lemma1}, elements of 
	  the vector space
	 $$
	 \mathbb{E}{\rm xt^1}(\Omega_{\bul}, \mathcal{O}_{\bul})=
	 {\rm coker}
	 \left( H^0(\PP, T_{\PP}) \overset{d^{0,0}}{\longrightarrow} 
	 H^0(\PP, f^*T_{\PP})	 \right).
	 $$
	 	 The latter, as it happens, is isomorphic, cf. 
	 	\cite{transversality}, to the 
	 \textit{orbifold} tangent space
	 $T_f{\bf rat_D}$
	 of the moduli space of rational maps on the Riemann sphere
	 up to conjugation, and has dimension $2D-2$.
	 Heuristically speaking, the content of the latter isomorphism 
	 is that infinitesimal deformations of the dynamical systems $(\PP,f)$
	 are simply infinitesimal deformations of the map $f$, 
	 and the dimension is as expected.
	 The space of infinitesimal deformations enjoys a 
	 natural relation with local deformations. Specifically,
	 to any effective divisor $\Delta$ on $\PP$,
	 supported on a cycle of $f$, there is associated
	 a divisor $\Delta_{\bul} $ in $\QQ{f}$, cf. \ref{f.i.divisor}.
	  Thus, there is a
	 canonical restriction map from the space of 
	infinitesimal deformations of $\QQ{f}$
	to the space of local deformations of $\Delta_{\bul}$, \textit{i.e.}
	 \begin{equation}\label{intro:resmap}
		\begin{tikzcd}[row sep=1.2pc, column sep=0.8pc]
		{\rm Res}_{\Delta_{\bul}}:	\mathbb{E}{\rm xt^1}(\Omega_{\bul},\mathcal{O}_{\bul}) \arrow[r]
			& \mathbb{E}{\rm xt^1}(\Omega_{\bul}, \mathcal{O}_{\Delta_{\bul}}).
		\end{tikzcd}
	\end{equation}
	 For example, let
	 $\Delta$ be a divisor on $\PP$ 
	 supported on a fixed point $x$ of $f$ and having multiplicity 
	 $\deg_x(\Delta)=n$. Then, the sheaf
	 $\mathcal{O}_{\Delta}\coloneqq O_{\PP, x}/\mathfrak{m}_x^n$
	 defines, by \ref{intro:theorem1}, a sheaf on $\XX{f}$ 
	 given by the pair
	 $
	 \mathcal{O}_{\Delta_{\bul}}\coloneqq (\mathcal{O}_{\Delta},\varphi),
	 $
	 with $ \varphi: f^*\mathcal{O}_{\Delta} \to \mathcal{O}_{\Delta} $
	 the canonical map.
	 Under certain assumptions on the nature of the 
	 fixed point ($x$ is neither a Cremer point nor a critical point of $f$), 
	 cf.  \cite[8.2]{Milnor},
	 there exists a local analytic conjugation 
	 between $f$ and the map $z \mapsto \lambda z$,
	 for some multiplier $\lambda \in \C^*$. In this case,
	 letting $\Delta=2[x]$, we find in \ref{locallemma1} a one dimensional 
	 space of local deformations of ``the fixed point'', 
	 \begin{equation}\label{intro:easyexample}
	 	 \mathbb{E}{\rm xt^1}(\Omega_{\bul}, \mathcal{O}_{\Delta_{\bul}})\cong \C,
	 \end{equation}
	 arising from deformations of the multiplier $\lambda$.\\
	  More generally, let $\Delta$ be a finite union of nonrepelling cycles of $f$, \textit{i.e.} 
	 those cycles $\{  x, \dots, f^{k-1}x  \}$ for which the multiplier 
	 $\lambda=(f^k)'(x)$ lies inside the closed unit disk. 
	 The restriction map \eqref{intro:resmap} fits into 
	 a natural long exact sequence,
	cf. \eqref{ideallongexact}, \textit{i.e.}
	the long exact sequence in 
	 cohomology associated to $\mathbb{H}{\rm om}(\Omega_{\bul}, - )$ 
	 of the following short exact sequence in $Ab(\XX{f})$,
	cf. \eqref{idealexact}, 
	  $$
	 0 \to  \mathcal{O}(-\Delta_{\bul}) \to \mathcal{O}_{\bul} \to \mathcal{O}_{\Delta_{\bul}} \to 0,
	 $$
	 where $\mathcal{O}(-\Delta_{\bul})$ 
	 is the ``ideal of holomorphic function
	 vanishing with multiplicity on $\Delta_{\bul}$'', \textit{i.e.} the sheaf
	 $(\mathcal{O}_{\PP}(-\Delta), i)$ on $\QQ{f}$, associated to the natural 
	 inclusion
	 $$
	 f^{*}\mathcal{O}_{\PP}(-\Delta)=
	 \mathcal{O}_{\PP}(-f^*\Delta) \hookrightarrow \mathcal{O}_{\PP}(-\Delta).
	 $$
	 The aforesaid long exact sequence in cohomology finishes as follows:
	  \begin{equation}\label{intro:longexact}
	 	\begin{tikzcd}[row sep=1.2pc, column sep=0.8pc]
	 		\mathbb{E}{\rm xt^1}(\Omega_{\bul},\mathcal{O}_{\bul}) \arrow[r]
	 		& \mathbb{E}{\rm xt^1}(\Omega_{\bul}, \mathcal{O}_{\Delta_{\bul}}) \arrow[r] 
	 		& \mathbb{E}{\rm xt^2}(\Omega_{\bul},\mathcal{O}(-\Delta_{\bul})) \arrow[r]& 0.
	 	\end{tikzcd}
	 \end{equation}
 The obstruction, therefore, to lifting local deformations to 
 global ones 
 is the $\mathbb{E}{\rm xt^2}$-group in \ref{intro:longexact}., which is controlled by way of, cf. \ref{Thurstonvanishing}, \ref{Epsteinvanishing}:
	  \begin{claiming}[{\bf Thurston-Epstein vanishing on ${\bf \XX{f}}$}]\label{intro:claim1}
	  	Let $f:\PP \to \PP$ be a rational map of degree $D>1$ and
	 let $\Delta$ be an effective divisor on $\PP$ having 
	 everywhere multiplicity 
	 $\deg_x(\Delta) \leq 1$. 
	 	 Then, if the cohomology group 
	 	$ \mathbb{E}{\rm xt^2}(\Omega_{\bul},\mathcal{O}(-\Delta_{\bul}))$,
	 	computed in $\QQ{f}$,
	 	does not vanish, 
	 	$f$ is a
	   $(2,2,2,2)$ Lattès map, cf. \cite{MR1251582} 
	(and $ \mathbb{E}{\rm xt^2}(\Omega_{\bul},\mathcal{O}(-\Delta_{\bul}))\cong \C$).
	Moreover, if $f$ is not 
	Lattès, $\Delta$ is supported on a union of nonrepelling cycles of $f$ and its multiplicity is $\deg_x(\Delta)\leq 2$ everywhere, 
	then $ \mathbb{E}{\rm xt^2}(\Omega_{\bul},\mathcal{O}(-\Delta_{\bul}))$ still vanishes.
	 \end{claiming}
 As a consequence of Claim \ref{intro:claim1}, 
 and by the fact that
  the contribution of each nonrepelling cycles (with the right multiplicity) 
  to the dimension of  
 $\mathbb{E}{\rm xt^1}(\Omega_{\bul}, \mathcal{O}_{\Delta_{\bul}})$ 
 is at least $1$, cf. \ref{locallemma1}, we get that
 \begin{equation}\label{intro:weakFS}
 	\#\{   \mbox{nonrepelling cycles of } f\} \leq 2D-2,
 \end{equation}
 \textit{i.e.} a weak version of the Fatou-Shishikura Inequality,
 and it can be expressed as
 $$
 \dim_{\C} 
 \mathbb{E}{\rm xt^1}(\Omega_{\bul}, \mathcal{O}_{\Delta_{\bul}}) 
 \leq \dim_{\C} \mathbb{E}{\rm xt^1}(\Omega_{\bul}, \mathcal{O}_{\bul}).
 $$
 As such, the key point in Claim \ref{intro:claim1}, and
 in the best traditions of the subject (e.g. duality implies
 that non compact Riemann surfaces are Stein),
 is that the ``right''
 thing to prove is that the dual group vanishes, which is
 exactly what Epstein did in \cite{1999math......2158E}. 
What we have added to his
 argument is to observe its (dual) functorial content in terms 
 of sheaves on $\QQ{f}$.
 
	 Finally, let us turn to the immersion problem 
	 behind the definition of the site ${\rm E}_f$.
	 In order to obtain the sharp count
	 of nonrepelling cycles of $f$ as in \cite{1999math......2158E},
	 the aforesaid divisor cannot be chosen in $\XX{f}$. In fact,
	 there is a contribution to the dimension of
	 $ \mathbb{E}{\rm xt^1}(\Omega_{\bul}, \mathcal{O}_{\Delta_{\bul}})$ 
	 coming from the critical divisor, together with its entire forward orbit.
	 Unfortunately, that is not, in general, a divisor on $\PP$ and
	 hence it does not define a divisor on $\QQ{f}$.
	 In order to make our argument work, we make a
	 truncation of the forward orbit at a certain time, 
	 which leads up to what we call a ``divisor'' on ${\rm E}_f$, 
	 cf. \ref{E-dynimacaldivisor}, \textit{i.e.} 
	 a pair of divisors $\Delta_{\bul}=(\Delta_{0},\Delta_{1})$ such that 
	 $ \Delta_1 \preceq \Delta_{0} \wedge f^*\Delta_{0} $.\\
	 Associated to the short exact sequence in $Ab({\rm E}_f)$,
	 $$
	0 \to  \mathcal{O}(-\Delta_{\bul}) \to \mathcal{O}_{\bul} \to \mathcal{O}_{\Delta_{\bul}} \to 0,
	 $$
 there is, again, a long exact sequence in cohomology, cf. \eqref{ideallongexact},
	finishing as in \eqref{intro:action}.
 Therefore, the obstruction to lifting local deformations 
 to global deformations of ${\rm E}_f$ still lies 
 is the $\mathbb{E}{\rm xt^2}$-group on the right, and 
 the vanishing of the latter, cf. 
 \ref{Epsteinvanishing}, is still dual to Epstein's extension
 of Thurston's theorem:
 \begin{claiming}[{\bf Thurston-Epstein vanishing on ${\bf E}_f$}]\label{intro:claim2}
 	Let $f:\PP \to \PP$ be a rational map of degree $D>1$
 	which is not a Lattés map.
 	The vanishing of the group $ \mathbb{E}{\rm xt^2}(\Omega_{\bul},\mathcal{O}(-\Delta_{\bul}))$, for some 
 	appropriate choice of the ``divisor'' $\Delta_{\bul}$, in ${\rm E}_f$, 
 	is equivalent, by duality, to (Epstein's extension of)
 	Infinitesimal Thurston's rigidity, cf. \cite{1999math......2158E}.
 \end{claiming}
Let us describe the features of
 Epstein's approach to 
the Fatou-Shishikura inequality. 
Firstly, the count of the number of nonrepelling cycles of 
$f$ in \cite{1999math......2158E} is sharpened
by assigning to each cycle a certain
multiplicity, which may be greater than $1$ if the multiplier 
is a root
of unity, according to a formal invariant of the cycle called the \textit{parabolic multiplicity}, cf. \eqref{parabolicmult}.
Moreover, if we do not count superattracting cycles 
\textit{i.e.} those which contain a critical point, 
the upper bound is also sharpened,
and the degree of $f$ no longer appears.
In fact, Epstein shows that the total count, with multiplicity,
of the nonrepelling and non superattracting 
cycles of $f$, denoted by $\gamma_f$ (which, \textit{a priori}, may be infinite),
is less or equal to the number of \textit{infinite tails} $\delta_f$, \textit{i.e.}
the number of distinct orbit-equivalence classes 
in the non (pre-)periodic part
of the postcritical set of $f$:
\begin{equation}\label{intro:EpsFS}
	\gamma_f \leq \delta_f.
\end{equation}
Note that the statement \eqref{intro:weakFS} is implied by \eqref{intro:EpsFS},
since there at most $2D-2-\delta_f$ superattracting cycles.

The next statement follows from Claim \ref{intro:claim2}
and from \ref{localHomEps}, 
\ref{computationExtcriticaldivisor}:
\begin{claiming}
	For any finite set $A$ of cycles of $f$, 
	there is a divisor $\Delta_{\bul}\coloneqq (\Delta_{0}, \Delta_{1})$
	on ${\rm E}_f$ such that 
	\begin{itemize}
		\item $A$ is contained in the support of $\Delta_0$;
		\item 	the vanishing 
		of $\mathbb{E}{\rm xt^2}(\Omega_{\bul},\mathcal{O}(-\Delta_{\bul}))$
		holds;
		\item if $\gamma_A$ denotes the count (with multiplicity) of the nonrepelling (and non superattracting) cycles in $A$, we have
		\begin{equation*}
			2D-2+ \gamma_{A} - \delta_f	=\dim_{\C} 
			\mathbb{E}{\rm xt^1}(\Omega_{\bul}, \mathcal{O}_{\Delta_{\bul}}) 
			\leq \dim_{\C} \mathbb{E}{\rm xt^1}(\Omega_{\bul}, \mathcal{O}_{\bul})=2D-2.
		\end{equation*}
	\end{itemize}
\end{claiming}
Taking the 
supremum over the family of finite sets of cycles of $f$, 
yields \eqref{intro:EpsFS}.
	 \vskip 4 em
	\begin{ackn*}
     We owe a debt of gratitude 
     to our advisor Prof. Michael McQuillan for having 
     shared his ideas about Epstein's work,
     and also for his constant support. 
     We are also deeply grateful to Adam Epstein for
     his kind hospitality during our visits in Warwick,
     which have been extremely helpful
     to our understanding of his work. 
     Finally, we thank our referees, Olivia Caramello and Carlo Gasbarri,
     who have helped us to adjust some
      mistakes. 
    	\end{ackn*}
    \newpage
    \phantom{h}
    \thispagestyle{empty}
    \newpage
	\chapter{Dynamical sites}
	\section{Dynamical site {\rm \bf  I}}
	Let $(X, J_X)$ be a small site that has finite limits 
	and countable, cf.  \eqref{countability}, coproducts. In particular, $X$ has an 
initial object which we shall denote by $\emptyset$. 
Moreover, let us assume that 
\begin{itemize}
	\item  coproducts commute with finite projective limits;
	\item coproducts in $X$ are disjoint, \textit{i.e.} $\forall \, U,V \in ob(X)$
		\begin{equation}\label{notdisjointcoproducts}
		\emptyset \overset{\sim}{\longrightarrow} U \times_{U \coprod V} V.
	\end{equation}
\end{itemize}
We fix a morphism of sites
$$
f:X \to X,
$$
meaning that there is a (non trivial) functor $f^{-1}: X \to X$ inducing a 
geometric morphism, \textit{i.e.} a pair of adjoint functors
\begin{equation*}
	\begin{tikzcd}[row sep=2.6pc, column sep=1.6pc]
		f^* \dashv f_*:  Sh( X) \arrow[r, shift right] \arrow[ from=r,shift right]&
		Sh(X),
	\end{tikzcd} 
\end{equation*}
with $f^*$ preserving finite limits.
We say that the pair $(X,f)$ satisfying the above hypotheses 
is a \textbf{site with dynamics}, cf. \eqref{intro:sitewithdyndef}.
	We want to define in this generality a category 
	adequate to describe the dynamical system generated by $f$, 
or the action on $X$ of the semigroup $\N[ f]$,
 \textit{i.e.} the image of the homomorphism 
\begin{equation}\label{NactiononX}
	\begin{tikzcd}[row sep=0.04pc, column sep=1.3pc]
		\N \arrow[ r] & {\rm End}(X) \\
		n\arrow[r, mapsto]  &  f^{ n}
	\end{tikzcd} 
\end{equation}

\begin{example}
	A natural example to keep in mind is a topological space $X$ with a  
	continuous map $f: X \to X$. 
	The site to take into consideration is then the \'etale site of the topological space. 
\end{example}

Let $g: X \to Y$ be a morphism
of sites  and recall that we denote by
$$
g_* : Sh(X) \to Sh(Y)
$$
the composition functor associated to $g^{-1}$, 
cf. \eqref{compfunctorcap1}, and by  
$
g^*
$
its left adjoint, cf. \cite[I.3]{SGA4}. Let us recall its construction:
given a sheaf $\mathcal{F}$ on $Y$ we consider 
the pre-sheaf on $X$ given by
\begin{equation}\label{pullbackpresheafgeneral}
 g^{-1} \mathcal{F} 	: U \in ob(X) \mapsto \varinjlim_{\overset{V \in ob(Y)}{U \to g^{-1}V \in X}}  \mathcal{F}(V),
\end{equation}
which is a separated pre-sheaf and $g^*\mathcal{F}=a(g^{-1}\mathcal{F}) \in  Sh( X) $ is the associated 
sheaf, \textit{i.e.}
\begin{equation}\label{sheaficationofpreheaf}
g^*\mathcal{F}(U)= \varinjlim_{R \in J_{{\bf {X}}}(U)}
 (g^{-1} \mathcal{F})(R).
\end{equation}

\begin{definition}
	Let $X,Y$ be sites. We say that a functor $c: Y \to X$ is continuous
	if the composition functor associated to $c$
	sends sheaves on $X$ to sheaves on $Y$.
\end{definition}
Recall the following fact, cf. \cite[III.3]{SGA4}
\begin{factdefinition}[{\bf Induced topology }]\label{sitewithdyn}
Let $X$ be a site and $Y$ any category  
with a functor $c: Y \to X$. Then, there is an induced topology $J_Y$ on $Y$ 
and a continuos functor $c:  (Y, J_Y) \to (X, J_X) $. 
The topology
$J_Y$, obtained by ``pulling back'' the topology 
$J_X$ along the functor $c$, 
is the finest topology on $Y$ among the topologies which makes
$c$ a continuous functor.
\end{factdefinition}
Note that a continuous functor $c:  (Y, J_Y) \to (X, J_X) $ induces 
only a \textit{weak} geometric morphism,
\textit{i.e.} setting $g_*\coloneqq c^*|_{Sh(X)}$, we have a pair of adjoint functors
\begin{equation*}
	\begin{tikzcd}[row sep=2.6pc, column sep=1.6pc]
		g^* \dashv g_*:  Sh( X) \arrow[r, shift right] \arrow[ from=r,shift right]&
		Sh(Y).
	\end{tikzcd} 
\end{equation*}
	We are ready to give our first definition of \textit{classifying site} of $f$.
	\begin{definition}[{\bf Dynamical category of $f$}]\label{defdyncat1}
		Let $\XX{f}$ be the following category: 
		\begin{equation*}
			\begin{aligned}
				\bul \; \;&  ob(\XX{f})=\coprod_{ U \in ob(X)} (f_*\underline{U})(U) =\coprod_{ U \in ob(X)} 
			\{ u:	f^{-1}U \to  U \in X \};\\
			\bul	\; \;& {\rm Hom}_{\XX{f}}(u,v)=\{ i: t(u) \to t(v): iu=(f^{-1}i)v \};
			\end{aligned}
		\end{equation*}
	In other words, the arrows in $\XX{f}$ are commutative diagrams
	\begin{equation*}
		\begin{tikzcd}[row sep=2pc, column sep=1.8pc]
			U	\arrow[r, "i"] &  V \\ 
			f^{-1}U \arrow[u, "u"]  \arrow[ r, "f^{-1}i"'] & f^{-1}V \arrow[u, "v"'] 
		\end{tikzcd} 
	\end{equation*}
		\end{definition}
	\begin{notation}
		The notation is as usual: an object $u \in ob(\XX{f})$ is a pair, 
		$(U,u)\coloneqq (t(u),u)$. We think of an object $(U,u) \in ob(\XX{f})$ as a 
		``backward invariant'' object in $X$ for the action of $f$.
	\end{notation}
Observe that we have a natural target functor $t: u \mapsto t(u)$, whence a functor
\begin{equation}\label{targetfunctorcap1}
	t: \XX{f} \to X, \;(U,u) \to U.
\end{equation}
	\begin{lemma}\label{objectinitialXffact}
		The functor $t: \XX{f} \to X$ admits a left adjoint $b: X \to \XX{f}$, 
		which is defined as follows. Let $\Delta \in ob(X)$, then,
		$$
 b(\Delta)\coloneqq \left( \coprod_{n\geq 0} \,f^{-n}\Delta, \;j_{\Delta} \right)
		$$
		of $\XX{f}$. Here, $j_{\Delta}$ denotes the canonical morphism
		\begin{equation}\label{shiftXf}
			j_{\Delta}: \coprod_{n\geq 0} \,f^{-(n+1)}\Delta \longrightarrow \coprod_{n\geq 0} \,f^{-n}\Delta
		\end{equation}
		defined by the coproduct of the canonical monomorphisms of definition,
		$$\iota_k: f^{-k}\Delta \hookrightarrow \coprod_{n \geq 0}f^{-n}\Delta, \quad k \geq 1.$$
		We set the following notation $U_{\Delta}\coloneqq t(b(\Delta))$.\\
		Moreover, for any $i: \Delta' \to \Delta \in X$, we define 
		$b(i): b(\Delta') \to b(\Delta)$  as 
		the coproduct for $n \geq 0$ of the maps
		\begin{equation*}
			\begin{tikzcd}[row sep=2.6pc, column sep=2.6pc]
				f^{-n}\Delta'  \arrow[r, "f^{-n}i"] 
				& f^{-n}\Delta  \arrow[r, hook] & U_{\Delta}. 
			\end{tikzcd}
		\end{equation*}
	\end{lemma}
\begin{proof}
	The map $j_{\Delta}$ is well defined since $f^{-1}$ 
	commutes with coproducts.
	Moreover, note that $j_{\Delta}$ induces, almost tautologically, a map
	\begin{equation}\label{jbarXf}
			\bar{j}_{\Delta}: b(f^{-1}\Delta) \to b(\Delta).
	\end{equation} 
	Let us fix $i: \Delta' \to \Delta \in X$, and note that the diagram
	\begin{equation}\label{bisafunctorXf}
		\begin{tikzcd}[row sep=2.6pc, column sep=1.8pc]
			\displaystyle	f^{-1} \coprod_{n \geq 0} f^{-n}\Delta'  \arrow[r, "j_{\Delta}"] \arrow[d, "f^{-1}(b(i))"]&\displaystyle
			\coprod_{n \geq 0} f^{-n}\Delta' \arrow[d, "b(i)"]\\
			\displaystyle	f^{-1} \coprod_{n \geq 0} f^{-n}\Delta   \arrow[r, "j_{\Delta'}"] & \displaystyle
			\coprod_{n \geq 0} f^{-n}\Delta
		\end{tikzcd}
	\end{equation}
	commutes, hence $b(i): b(\Delta') \to b(\Delta)$ is well defined.
	The fact that $b$ is a functor follows immediately from the functoriality 
	of the definition.
	 We need to prove that the following bifunctor in the variables $(\Delta,(V, v_{\bul}))$,
	\begin{equation*}
		{\rm Hom}_{\XX{f}}(b(\Delta), (V,v)) \overset{\sim}{\longrightarrow} {\rm Hom}_{X}(\Delta, t(V,v)),
	\end{equation*}
taking $i_{\bul}: b(\Delta) \to (V,v)$ to its target $t(i_{\bul})=i_0: \Delta \to V$,
is an equivalence. First, observe that any $i: \Delta \to V \in X$
where $\Delta \in ob(X)$ and $(V,v) \in ob(\XX{f})$, affords a map
$$
\underline{v}(i)_n: f^{-n}\Delta \to  V \in X,
$$
for each $n \geq 0$. We set $\underline{v}(i)_0=i$.
We can recursively define maps
\begin{equation}\label{actionNobjects}
	v^n: f^{-n}V \to V \in X
\end{equation}
for each $n$, starting with the given maps $v^0=id_V, v^1 \coloneqq v$:
we apply the functor $f^{-1}$ to $v^n$ and compose with $v$ to get $v^{n+1}$,
$$
v^{n+1}\coloneqq v \circ f^{-1}v^n.
$$
Note that each $v^n$ can be viewed as a map 
\begin{equation}\label{inducedmaponX/f}
\bar{v}^n: (f^{-n}V, f^{-n}v) \to (V,v) \in \XX{f},
\end{equation}
essentially by definition. The desired map $\underline{v}(i)$ is obtained 
by applying the functor $f^{-n}$ to $i$ and composing with $v^n$,
$$
\underline{v}(i)_n= v^n \circ f^{-n}i.
$$
Note that we have $\underline{v}(id_V)_n=v^n$ by definition. Consequently, there is a map 
\begin{equation}\label{pieceofcounit}
	\underline{v}(i)_{\bul}\coloneqq\coprod_{n\geq 0}\underline{v}(i)_n : U_{\Delta} \to V,
\end{equation}
which is indeed an arrow in $\XX{f}$, since
\begin{equation}\label{commdiagr2Xf}
	\begin{tikzcd}[row sep=2.6pc, column sep=1.8pc]
		V & \arrow[l, "v"] f^{-1}V \\
		\displaystyle \coprod_{n \geq 0} f^{-n}\Delta \arrow[u, "\underline{v}(i)_{\bul}"]
		 &\displaystyle  \arrow[l, "j_{\Delta}"] \coprod_{n \geq 0} f^{-(n+1)}\Delta \arrow[u, "f^{-1}\underline{v}(i)_{\bul}"']
	\end{tikzcd}
\end{equation}
clearly commutes, having by definition 
$$
v \circ f^{-1}\underline{v}(i)_n=v^{n+1}\circ f^{-(n+1)}i=\underline{v}(i)_{n+1}=\underline{v}(i)_{\bul}\circ  j_{\Delta,n}.
$$
The map $i \mapsto \underline{v}(i)_{\bul}$ is the required inverse, since by 
the commutativity of \eqref{commdiagr2Xf} we deduce a recursive relation 
implying that for any $i_{\bul}: b(\Delta) \to (V,v) \in \XX{f}$ we have 
$i_{\bul}=\underline{v}(i_0)_{\bul}$.
\end{proof}
A useful consequence, cf. \cite[I.5.5]{SGA4}, of the above fact is the following.

\begin{cor}\label{corleftadjointcap1}
The left adjoint of the composition functor associated to $t$, 
cf. \eqref{compfunctorcap1}, is the composition functor associated to $b$.
Equivalently, for any $\Delta \in ob(X)$ the canonical monomorphism $\iota_0: \Delta \hookrightarrow U_{\Delta}$ 
is an initial object in the category $I^{\Delta}$ whose objects are arrows 
$\alpha_v: \Delta \to V \in X$ such that $(V,v) \in ob(\XX{f})$, 
and the arrows $\alpha_v \to \alpha_{v'}$ are maps $i: (V,v) \to (V',v') \in \XX{f}$ 
such that the diagram
\begin{equation}\label{commdiagr1cap1}
	\begin{tikzcd}[row sep=2pc, column sep=1.8pc]
		& 	\Delta \arrow[dl, "\alpha_v"'] \arrow[dr, "\alpha_{v'}"]\\
		V	\arrow[rr, "i"] &		 &  V' 
	\end{tikzcd} 
\end{equation}
commutes, \textit{i.e.}
\begin{equation}\label{initialobXf}
	\forall \,(V,v) \in ob(\XX{f})
	\mbox{ such that  } \exists \, \alpha_v:
	\Delta \to V \in X, 
	\,  \,
	\exists!\, i_v:	(U_{\Delta}, j_{\Delta}) \to (V,v) \in \XX{f} 
\end{equation}
such that	
\begin{equation*}
	\begin{tikzcd}[row sep=2pc, column sep=1.8pc]
		& 	\Delta \arrow[dl, "\iota_0"'] \arrow[dr, "\alpha_v"]\\
		U_{\Delta}	\arrow[rr, "i_v"] &		 &  V 
	\end{tikzcd} 
\end{equation*}

\noindent
commutes. Moreover, the association $(V,v)\in ob(\XX{f}) \mapsto i_v \in ar(I^{\Delta}) $ is 
functorial in the sense that for each diagram as in \eqref{commdiagr1cap1}, we have $i_{v'}=i\circ i_v$.
\end{cor}
\begin{proof}
We set $i_v\coloneqq \underline{v}(\alpha_v)_{\bul}$, cf. \eqref{pieceofcounit}. To prove that 
$\iota_0: \Delta \hookrightarrow U_{\Delta}$ is an initial object in $I^{\Delta}$ it is 
sufficient to observe that the commutativity of \eqref{commdiagr1cap1}
implies that for each $n$
\begin{equation*}\label{commdiagr2cap1}
	\begin{tikzcd}[row sep=2pc, column sep=1.8pc]
		& f^{-n}	\Delta \arrow[dl, "f^{-n}\alpha_v"'] \arrow[dr,"f^{-n}\alpha_{v'}"]\\
		f^{-n}	V	\arrow[rr, "f^{-n}i"] &		 &  f^{-n} V' 
	\end{tikzcd} 
\end{equation*}
commutes. 
\end{proof}
\begin{cor}
	The co-unit morphism of the adjunction \ref{objectinitialXffact} leads to a canonical arrow in $\XX{f}$
	\begin{equation}\label{epimorphismXf}
		b(t(\Delta,\delta))=(U_{\Delta},j_{\Delta}) \overset{\delta_{\bul}}{\longrightarrow} (\Delta,\delta),
	\end{equation}
	for any $(\Delta,\delta) \in ob(\XX{f})$, which is given by $$
	\delta_{\bul}\coloneqq \underline{\delta}(id_{\Delta})_{\bul}= \coprod_{n \geq 0}\delta^n,
	$$ 
	in the notation of \eqref{actionNobjects}, \eqref{pieceofcounit}.
\end{cor}
The role of the co-unit morphism \eqref{epimorphismXf} should be cleared by the following.
\begin{lemma}\label{basisobjectsXf}
	Let $(\Delta,\delta) \in ob(\XX{f})$. Note that the map 
	$$
	b(\delta):	b(f^{-1}\Delta) \to b(\Delta) \in \XX{f}
	$$
	given by the coproduct, for $n \geq 0$ of the compositions
	$$
	f^{-(n+1)}\Delta \overset{f^{-n}\delta}{\longrightarrow} 
	f^{-n}\Delta \hookrightarrow U_{\Delta},
	$$
	fits into a natural coequalizer in $\XX{f}$
	\begin{equation}\label{coequalizerobjectXf}
		\begin{tikzcd}[row sep=2.6pc, column sep=2.6pc]
			b(f^{-1}\Delta)	\arrow[r, shift right, "b(\delta)"'] 
			\arrow[r, shift left, "\bar{j}_{\Delta}"] & b(\Delta)  \arrow[r,"\delta_{\bul}"] & (\Delta,\delta),
		\end{tikzcd} 
	\end{equation}
where $\bar{j}_{\Delta}$ is the map induced by $j_{\Delta}$, cf. \eqref{jbarXf}.
\end{lemma}
\begin{proof}
	Let us fix an arbitrary $(V,v) \in ob(\XX{f})$ and let us prove directly that 
	$(\Delta,\delta)$ satisfies the universal property of coequalizers. 
	Let us fix a map $\alpha_{\bul}: (U_{\Delta},j_{\Delta}) \to (V,v)$ such that
	$\alpha_{\bul}\circ \bar{j}_{\Delta}=\alpha_{\bul}\circ b(\delta)$. 
	This, in particular gives a map $\alpha_0: \Delta \to V$ which satisfies 
	$\alpha_0\circ \delta=v\circ f^{-1}\alpha_0$, \textit{i.e.} we obtain a map 
	$\alpha_0: (\Delta,\delta) \to (V,v) \in \XX{f}$.
	In fact, 
	$$
	(\alpha_{\bul}\circ \bar{j}_{\Delta})_0=\alpha_1,
	$$
	and $\alpha_1=v\circ f^{-1}\alpha_0$, since $\alpha_{\bul}$ is a map in $\XX{f}$, so
	$$
	(\alpha_{\bul}\circ b(\delta))_0=\alpha_0\circ \delta=v\circ f^{-1}\alpha_0.
	$$
	The map $\alpha_0$ is clearly unique.
\end{proof}

	\begin{definition}[{\bf Classifying site of $f$}]\label{defdynsite1}
		\noindent
	Consider the topology $J_{\XX{f}}$ on $\XX{f}$ induced by the functor 
	$t: \XX{f} \to X$, cf. \ref{sitewithdyn}.
	The resulting site
	$\left( \XX{f},J_{\XX{f}}  \right)$ is called the {\bf classifying site} of $f$,  
	and, as usual when there is no room for 
	confusion, we abuse notation by writing just $\XX{f}$.
\end{definition}

Note that, by definition of induced topology, 
the functor $t: \XX{f}\longrightarrow  X$ is a \textit{continuous functor}, 
cf. \cite[III.1.1]{SGA4}.  The following is an immediate consequence of \ref{objectinitialXffact}.

\begin{cor}\label{tgeomorph}
	The functor $t$ induces a morphism of sites
		\begin{equation}\label{projectionmapXf}
		\pi : (X, J_X) \to\left(  \XX{f}, J_{\XX{f}} \right).
	\end{equation}
\end{cor}
 \begin{proof}
 	The functor $t^*|_{Sh(X)}: Sh(X) \to Sh(\XX{f})$ clearly admits a left adjoint 
 	and it coincides with $a\circ b^*: Sh(\XX{f}) \to Sh(X)$. 
 	The fact that it preserves finite limits can be checked by a direct computation.
 	In fact, cf. \cite{finitelimits}, it is enough to check 
 	that it preserves binary products and equalizers, which is straightforward.
 \end{proof}
The following statement is an equivalent formulation of 
\ref{corleftadjointcap1}, written using the geometric map $\pi$.
\begin{cor}\label{pullbackXfonbasefact}
	Let $\mathcal{F}$ be a pre-sheaf on $\XX{f}$. 
	Then, there is a canonical isomorphism
	\begin{equation}\label{pullbackXfonbase}
		b^*\mathcal{F}\overset{\sim}{\longrightarrow}  	\pi^{-1}\mathcal{F}.
	\end{equation}	
\end{cor}
\begin{proof}
	Let $\Delta \in ob(X)$. The object $b(\Delta)=(U_{\Delta}, j_{\Delta})$ 
	is already part of the direct system \eqref{pullbackpresheaf}, so
	we get the map \eqref{pullbackXfonbase}.
	This map is invertible since this object is the final object in the 
	filtered category $I^{\Delta}$, cf. \ref{corleftadjointcap1}.
\end{proof}

It follows from \ref{tgeomorph} that there is a geometric morphism
\begin{equation}\label{geometricmorphism}
	\begin{tikzcd}[row sep=2.6pc, column sep=1.6pc]
		\pi^{*} \dashv \pi_*:   Sh(X)\arrow[r, shift right] \arrow[ from=r,shift right]&
		Sh(\XX{f})
	\end{tikzcd} 
\end{equation}
resulting from a pair of adjoint functors
\begin{equation}\label{quasigeometricmorphism}
	\begin{tikzcd}[row sep=2.6pc, column sep=1.6pc]
		b^* \dashv t^*:   X^{\wedge}\arrow[r, shift right] \arrow[ from=r,shift right]&
		\XX{f}^{\wedge}.
	\end{tikzcd} 
\end{equation}
For any $\mathcal{F} \in \XX{f}^{\wedge}$ we have, cf. \ref{corleftadjointcap1},
\begin{equation}\label{pullbackpresheaf}
	\pi^{-1}\mathcal{F}:	\Delta \in ob(X) \mapsto \varinjlim_{I^{\Delta}} \mathcal{F}(\cdot),
\end{equation}
in the notation of \ref{corleftadjointcap1}.

We think of $\pi$ as the canonical projection of 
${X}$ onto its
``classifying site" $\XX{f}$, when we consider the action on $X$
of the monoid generated by $f$.
The reason for this terminology lies in the fact
 that we are going to prove that $\pi$ is a \textit{descent map}.
		\begin{claim}\label{claimaction}
		Sheaves on $\XX{f}$ 
			corresponds functorially to sheaves $\mathcal{F}$ on $X$ together 
			with a \textit{descent datum}, or an \textit{action} of 
			$f$ on $\mathcal{F}$, \textit{i.e} a map of sheaves
			\begin{equation}\label{actiononsheaf}
				\varphi: f^*\mathcal{F} \to \mathcal{F}.
			\end{equation}
		\end{claim}
	\begin{notation}
		We employ the word \textit{action} to indicate a monoid
		action of $\N [f] \hookrightarrow End(X)$ on a 
		sheaf $\mathcal{F}$ (resp. a pre-sheaf $\mathcal{F}$) on $X$, 
		\textit{i.e.} a map of sheaves
		\begin{equation}\label{Actionmeaning}
			f^*\mathcal{F} \to \mathcal{F},
		\end{equation}
		(resp. a map of pre-sheaves $f^{-1}\mathcal{F}\to \mathcal{F}$).
	\end{notation}

The following is a well-known characterization of the covering sieves in $\XX{f}$, cf. \cite[III.3.2]{SGA4}.
\begin{factdefinition}{\bf [Sieves on \boldmath $\XX{f}$]}\label{defsievesonXf}
	
	\noindent
	A covering sieve $ R_u \hookrightarrow \underline{(U,u)} \in J_{\XX{f}}(U,u) $
	is, by definition, 
	a sieve on $(U,u)$ in $\XX{f}$ such that 
	$\pi^{-1}R_u \hookrightarrow \pi^{-1}\underline{(U,u)}=\underline{U}$
	 is a covering sieve on $U$ in $X$, \textit{i.e.} $\pi^{-1}R_u \in J_{X}(U)$.
	\end{factdefinition}
\begin{proof}
	It follows from the fact that the functor $\pi^*$ is left exact and also 
	that any monomorphism is a covering if and only if it is a bi-covering.
\end{proof}

In the proof of \ref{objectinitialXffact} we have seen that $f$ factorizes 
through $\XX{f}$, \textit{i.e.} there is a natural commutative square
\begin{equation}\label{commutativesquareXf}
	\begin{tikzcd}[row sep=2.6pc, column sep=2.6pc]
		X	\arrow[r, "\pi"] \arrow[d, "f"'] & \XX{f} \arrow[d, "\tilde{f}"] \\ 
		X \arrow[r, "\pi"] & \XX{f},
	\end{tikzcd} 
\end{equation}
where 
\begin{equation}\label{tildefunction}
	\tilde{f}^{-1}(U,u)\coloneqq (f^{-1}U,f^{-1}u).
\end{equation}
Moreover, there are natural maps, cf. \eqref{actionNobjects},
$$
\bar{u}^n: \tilde{f}^{-n}(U,u) \to (U,u) \in \XX{f},
$$
for each $n \in \N$. We set $\bar{u}^0\coloneqq id_{(U,u)}, \bar{u}\coloneqq \bar{u}^1$.
The induced map $\tilde{f}$ should be thought as a ``shift" map, since its action on the 
elements $b(\Delta)$, cf. \ref{objectinitialXffact}, is by shifting.
Note the following property of $\XX{f}^{\wedge}$.
\begin{fact}\label{actiondown}
	For any $\mathcal{F} \in \XX{f}^{\wedge}$ there is a canonical natural transformation of functors 
	\begin{equation}\label{mapfpullbackpreXf}
		\phi(\mathcal{F}):\tilde{f}^{-1}\mathcal{F} \longrightarrow 	\mathcal{F}.
	\end{equation}
 Moreover, any $\theta: \mathcal{F} \to \mathcal{G} \in \XX{f}^{\wedge}$ 
	commutes with \eqref{mapfpullbackpreXf}, \textit{i.e.} there is a commutative diagram in $\XX{f}^{\wedge}$
	\begin{equation*}
		\begin{tikzcd}[row sep=2.6pc, column sep=2.8pc]
			\tilde{f}^{-1}\mathcal{F} \arrow[r, "	\phi(\mathcal{F})"] 
			\arrow[d, "	\tilde{f}^{-1}\theta"'] &
			\mathcal{F} \arrow[d, "\theta"] \\
			\tilde{f}^{-1}\mathcal{G}\arrow[r, "	\phi(\mathcal{G})"] &
			\mathcal{G}
		\end{tikzcd} 
	\end{equation*}
\end{fact}
\begin{proof}
	The adjoint of the natural transformation $\phi(\mathcal{F})$ is defined as follows: 
	$$
	s \in \mathcal{F}(U,u) \mapsto  \mathcal{F}(\bar{u})(s) \in \mathcal{F}(f^{-1}U, f^{-1}u),
	$$
	for any $(U,u) \in ob(\XX{f})$. By functoriality of the construction, 
	this map commutes with any natural transformation $\theta$.
\end{proof}

\begin{definition}\label{pre-toposXf}
	
	\noindent
	Let $\widehat{X}_f$ be the following category
	\begin{itemize}
		\item The objects are pairs
		$(\mathcal{F}, \varphi)$, where $\mathcal{F} $ is a pre-sheaf on $ X$ and 
		$\varphi$ is an action of $f$ on $\mathcal{F}$, \textit{i.e.} a map of pre-sheaves
		$$
		\varphi: f^{-1}\mathcal{F} \to \mathcal{F};
		$$
		\item The arrows from 
		$(\mathcal{F}, \varphi)$ to $(\mathcal{G}, \gamma)$ are maps of pre-sheaves 
		$\theta: \mathcal{F} \to \mathcal{G}$ such that the induced square
		\begin{equation*}
			\begin{tikzcd}[row sep=2.6pc, column sep=2.6pc]
				f^{-1}\mathcal{F} \arrow[r, "\varphi"] 
				\arrow[d," f^{-1}\theta"'] &
				\mathcal{F} \arrow[d, "\theta"] \\
				f^{-1}\mathcal{G}\arrow[r, "\gamma"] &
				\mathcal{G}
			\end{tikzcd} 
		\end{equation*}
		commutes.
	\end{itemize}

\end{definition}
\begin{remark}\label{remarkhomsetXf}
Let $(\mathcal{F}, \varphi),(\mathcal{G}, \gamma)\in ob(\widehat{X}_f)$. 
Note that  the set
${\rm Hom}_{\widehat{X}_f}((\mathcal{F}, \varphi),(\mathcal{G}, \gamma))$
could have been equivalently defined as the equalizer of the following diagram
	\begin{equation*}
	\begin{tikzcd}[row sep=1.8pc, column sep=1.8pc]
		\ker\Big(  	{\rm Hom}(\mathcal{F},\mathcal{G})
	\arrow[r, shift left, "\alpha_0"] \arrow[r, shift right, "\alpha_1"'] & 
	{\rm Hom}(f^{-1}\mathcal{F},\mathcal{G}) \Big),
	\end{tikzcd} 
\end{equation*}
where $\alpha_0(\theta)= \theta\circ \varphi$ and $\alpha_1(\theta)=\gamma\circ f^{-1}\theta$.
\end{remark}

\begin{example}\label{example1cap1}
	Note that by Yoneda embedding any $(U,u) \in ob(\XX{f} )$ is uniquely
	determined by a 
	natural transformation in ${X}^{\wedge}$
	\begin{equation}\label{phi_U}
		h_u: \underline{U} \to f_*\underline{U}.
	\end{equation}
	Therefore, Yoneda embedding allows us to identify the category $\XX{f}$ 
	consisting of pairs $(U,u)$ with the subcategory of $\widehat{X}_f$ 
	consisting of pairs $(\underline{U}, \underline{u})$, 
	with action $\underline{u}$ given by the left adjoint of $h_u$.
\end{example}
Observe that there is an evident forgetful map 
\begin{equation}\label{forgetfulmapXf}
	\widehat{X}_f \to X^{\wedge},
	\;  (\mathcal{F},\varphi) \to  \mathcal{F}.
\end{equation}
The following shows one of the expected properties for $\XX{f}$.
\begin{lemma}\label{boldpullbackfact}
	The map $\pi^{-1}$ in \eqref{quasigeometricmorphism} 
	factorizes through the forgetful map \eqref{forgetfulmapXf}.
	We denote by {\boldmath $\pi^{-1}$} the map so obtained:
	\begin{equation}\label{pullbackXmodf}
		\mbox{\boldmath $\pi^{-1}$} :  \XX{f}^{\wedge} \to\widehat{X}_f , \;
		\mathcal{F} \mapsto (\pi^{-1}\mathcal{F},\varphi).
	\end{equation}
\end{lemma}
\begin{proof}
	The claim follows immediately from \ref{actiondown} and the fact that
	\eqref{commutativesquareXf} is commutative. Indeed, applying the functor 
	$\pi^{-1}$ to the morphism resulting from \ref{actiondown}, we get a map
	\begin{equation}\label{theoricalaction}	
		\pi^{-1}(\tilde{f}^{-1}\mathcal{F} ) \to \pi^{-1}\mathcal{F},
	\end{equation}
	which is the required action $\varphi$ since 
	$	\pi^{-1}\tilde{f}^{-1}=(\tilde{f}\pi)^{-1}=(\pi f)^{-1}=f^{-1}\pi^{-1}$. 
	The map \mbox{\boldmath $\pi^{-1}$} so obtained is 
	easily seen to be a functor, by naturality of the construction.
	The action \eqref{theoricalaction} can be described more explicitly as follows. \\ 
	Given $\mathcal{F} \in\XX{f}^{\wedge}$, we need to define a map
	\begin{equation}\label{actionpresheaf}
		 \varinjlim_{\overset{(V,v) \in \XX{f}}{U \to V \in X}} \mathcal{F}(V,v)\; \longrightarrow
		\varinjlim_{\overset{ (W,w) \in \XX{f}}{f^{-1}U \to W \in X}} \mathcal{F}(W,w),
	\end{equation}
	for each $U \in ob(X)$. Given $(V,v) \in ob(\XX{f})$ such that $\alpha: U \to V \in X$, 
	we consider the map, cf. \ref{inducedmaponX/f}
	$$
	\mathcal{F}(\bar{v}): \mathcal{F}(V,v) \to \mathcal{F}(f^{-1}V, f^{-1}v).
	$$
	By hypothesis, we have $f^{-1}\alpha: f^{-1}U \to f^{-1}V \in X$, so $\mathcal{F}(\bar{v})$ can be 
	composed with the canonical injection of $\mathcal{F}(f^{-1}V, f^{-1}v)$ 
	into the direct limit on the right of \eqref{actionpresheaf}.
	This association is functorial  in $(Vv)$: given $(V',v')\to (V,v) \in \XX{f}$, we have a commutative diagram
	\begin{equation*}
		\begin{tikzcd}[row sep=0.8pc, column sep=1.3pc]
			\mathcal{F}(V,v) \arrow[r] 
			\arrow[d] & \mathcal{F}(f^{-1}V,f^{-1}v)  \arrow[r] \arrow[d] &
			\left( \varinjlim\limits_{\overset{ (W,w) \in \XX{f}}{f^{-1}U \to W \in X}} \mathcal{F}(W,w)    \right) \\
			\mathcal{F}(V',v') \arrow[r] &
			\mathcal{F}(f^{-1}V', f^{-1}v') \arrow[ur]
		\end{tikzcd} 
	\end{equation*}
	which is, by definition, the required map \eqref{actionpresheaf}.
\end{proof}
\begin{remark}
	Note that \ref{boldpullbackfact} would have followed trivially from \ref{pullbackXfonbasefact}. 
	In fact, the map 
	$\pi^{-1}\mathcal{F}(\Delta) \to \pi^{-1}\mathcal{F}(f^{-1}\Delta)$
	is identified to $\mathcal{F}(\bar{j}_{\Delta})$, cf. \ref{jbarXf}. 
	Consequently, the action furnished by \ref{boldpullbackfact} is in some sense ``tautological".
\end{remark}

\begin{definition}[{\bf Dynamical tòpos {\rm \bf I}}]\label{dynamicaltoposdef}
	\noindent
	Let $Sh( X)_f$ be the category whose objects consist of pairs
	$(\mathcal{F}, \varphi)$, where $\mathcal{F}$ is 
	a sheaf on $X$ and 
	$\varphi$ is an action, as in \eqref{actiononsheaf}. 
	The arrows $\theta_{\bul}:  (\mathcal{F}, \varphi) \to (\mathcal{G}, \gamma)$
	are commutative squares
		\begin{equation*}
		\begin{tikzcd}[row sep=2.6pc, column sep=2.6pc]
			f^{*}\mathcal{F} \arrow[r, "\varphi"] 
			\arrow[d," f^{*}\theta"'] &
			\mathcal{F} \arrow[d, "\theta"] \\
			f^{*}\mathcal{G}\arrow[r, "\gamma"] &
			\mathcal{G}
		\end{tikzcd} 
	\end{equation*}
\end{definition}
\begin{remark}\label{remarkhomsetsheafXf}
	Let $(\mathcal{F}, \varphi),(\mathcal{G}, \gamma)\in ob(Sh({X})_f)$. 
	Note that  the set
	$${\rm Hom}_{Sh({X})_f}\Big((\mathcal{F}, \varphi),(\mathcal{G}, \gamma)\Big)$$
	could have been equivalently defined as the equalizer of the following diagram
	\begin{equation*}
		\begin{tikzcd}[row sep=1.8pc, column sep=1.8pc]
			\ker\Big(  	{\rm Hom}(\mathcal{F},\mathcal{G})
			\arrow[r, shift left, "\alpha_0"] \arrow[r, shift right, "\alpha_1"'] & 
			{\rm Hom}(f^{*}\mathcal{F},\mathcal{G}) \Big),
		\end{tikzcd} 
	\end{equation*}
	where $\alpha_0(\theta)= \theta\circ \varphi$ and $\alpha_1(\theta)=\gamma\circ f^{*}\theta$.
\end{remark}
Note that there is an evident forgetful functor $$Sh(X)_f \to Sh(X), \;
(\mathcal{F}, \varphi) \mapsto \mathcal{F}.$$
The following result follows immediately from \ref{boldpullbackfact}.
\begin{cor}
	The inverse image functor $\pi^*$ factorizes through the forgetful functor. We denote by
	\begin{equation}\label{pullbackXmodfsheaves}
		\mbox{\boldmath $\pi^{*}$} :  Sh(\XX{f})\to Sh(X)_f, \;
		\mathcal{F} \mapsto (\pi^{*}\mathcal{F},\varphi),
	\end{equation}
the resulting functor.
\end{cor}
\begin{proof}
	It is sufficient to apply the associated sheaf functor $a$ to \ref{boldpullbackfact}.
\end{proof}
\begin{example}\label{example2cap1}
	Given a sheaf $\mathcal{F}\in Sh({X})$, we have the following isomorphism
	$$
	\mbox{\boldmath $ \pi^{*}$}\pi_*\mathcal{F} \, 
	\overset{\sim}{\longrightarrow} \left(  \prod_{n \geq 0}(f^n)_*\mathcal{F},\; \mbox{$pr_0$} \right),
	$$
	where  \mbox{$pr_0$} is the left adjoint of the canonical projection morphism
	$$
	\prod_{n \geq 0}(f^n)_*\mathcal{F} \longrightarrow 
	f_*\left(  \prod_{n \geq 0}(f^n)_*\mathcal{F}   \right)= \prod_{n \geq 1}(f^n)_*\mathcal{F}.
	$$
	In fact, using \eqref{pullbackXfonbase}, 
	we get an isomorphism at the level of pre-sheaves
	$$
	\pi^{-1}\pi_*\mathcal{F}(\Delta)\overset{\sim}{\longrightarrow}
	\pi_*\mathcal{F}(U_{\Delta}, j_\Delta)=\mathcal{F}(U_{\Delta})=	\prod_{n \geq 0}\mathcal{F}(f^{-n}\Delta).
	$$
	The claim follows by taking the associated sheaf functor ``$a$'', 
	and by noting that pull-back along the ``shift" map \eqref{shiftXf}
	 provides the action  ``\mbox{$pr_0$}''.
\end{example}
Note the following fact.

\begin{cor}\label{basissievesXf}
	The family of sieves on $(U,u)$ in $\XX{f}$ given by
$$
\{ b_!R : R \in J_{X}(U)\}
$$
is a family of covering sieves in $J_{\XX{f}}(U,u)$. Moreover, 
the collection of all such covering sieves forms a basis for the topology $J_{\XX{f}}$.
\end{cor}
\begin{proof}
	Let us fix $(U,u) \in ob(\XX{f})$. The sieves on $(U,u)$ in the statement are covering sieves, since there is a bicovering
	\begin{equation}\label{basisXfup}
		\coprod_{n \geq 0} (f^n)^{-1}R \longrightarrow 	t_!b_!R,
	\end{equation}
	induced by the unit $id \to b^*b_!$ (after the 
	identification $b^*=t_!$), composed with the canonical morphism
	$$
	\coprod_{n \geq 0} (f^n)^{-1}R \to R,
	$$ 
	which is evidently a covering.
	To see that \eqref{basisXfup} is a bicovering, note that for any sheaf $\mathcal{F}$ on $X$ we have 
	$$
	{\rm Hom}(t_!b_!R,\mathcal{F})\overset{\sim}{\longrightarrow} 	
	{\rm Hom}(R,(tb)^*\mathcal{F})=\prod_{n \geq 0}{\rm Hom}(R, (f^n)_*\mathcal{F}).
	$$
	 In order to prove that the above family forms a basis, it is sufficient to observe that it is cofinal 
	 in the topology $J_{\XX{f}}(U,u)$. In fact, or any covering sieve 
	 $R_u \hookrightarrow \underline{(U,u)}$, the co-unit of the adjunction 
	 $b_! \dashv b^*$ computed in $R_u$, gives a map 
	$$
	\epsilon(R_u): b_!b^*R_u \to R_u.
	$$
\end{proof}
\begin{cor}
	The basis of sieves, viewed as elements of $\widehat{X}_f$, is given by 
	$$
	\mbox{\boldmath $\pi^{-1}$}(b_!R)\overset{\sim}{\longrightarrow} 
	\left(  \coprod_{n \geq 0} (f^n)^{-1}R, j_R \right),
	$$
	where $j_R$ is the natural inclusion 
	$$
j_R:	f^{-1}\left(  \coprod_{n \geq 0} (f^n)^{-1}R  \right)= 
	\coprod_{n\geq 0} (f^{n+1})^{-1}R \hookrightarrow \coprod_{n\geq 0} (f^n)^{-1}R.
	$$
	\end{cor}
\begin{proof}
Let $(\mathcal{G},\gamma) \in Sh(X)_f$. We need to show that
$$
{\rm Hom}_{\widehat{X}_f}(\mbox{\boldmath $\pi^{-1}$}(b_!R), 
(\mathcal{G},\gamma))\overset{\sim}{\longrightarrow} \mathcal{K}(R),
$$
where, cf. \ref{example2cap1}
\begin{equation*}
	\begin{tikzcd}[row sep=1.6pc, column sep=1.8pc]
	\mathcal{K} = 
		\ker\Bigg(  \displaystyle	\prod_{n \geq 0}(f^n)_*\mathcal{G}
		\arrow[r, shift left, "ad(\gamma)"] \arrow[r, shift right, "pr^0"'] & 
	\displaystyle	\prod_{n \geq 0}(f^{n+1})_*\mathcal{G} \Bigg).
	\end{tikzcd}
\end{equation*}
We already noted that
$$
{\rm Hom}_{{X}^{\wedge}}(\pi^{-1}(b_!R), \mathcal{G})\cong 
{\rm Hom}_{{X}^{\wedge}}(R,\pi^{-1}\pi_*\mathcal{G})=
\prod\limits_{n \geq 0}(f^n)_*\mathcal{G}(R),
$$
while, from \ref{remarkhomsetXf}, the commutativity condition 
is satisfied if and only if the above morphisms lie in the kernel $\mathcal{K}(R)$.
\end{proof}
As anticipated, the following main result holds. 
\begin{theorem}\label{equivalencelemmaXf}
	The map {\boldmath $\pi^*$} defined in 
	\eqref{pullbackXmodfsheaves} is an equivalence of categories.
\end{theorem}
\begin{proof}
	Let $(\mathcal{G}, \gamma) \in  Sh(X)_f$ and 
	observe that there are two natural transformations
	\begin{equation}\label{inverseequivalenceXf}
		\begin{tikzcd}[row sep=2.6pc, column sep=2.6pc]
			\pi_*\mathcal{G}
			\arrow[r, shift left, " \sigma"] 
			\arrow[r,shift right,"\tau"'] &
			\pi_* f_*	\mathcal{G}
		\end{tikzcd} 
	\end{equation}
	given by
	\begin{equation*}
		\begin{aligned}
			&\sigma=\pi_*(ad(\gamma)); \\
			& \tau_{(V,v)}= \mathcal{G}(v): \mathcal{G}(V) \to \mathcal{G}(f^{-1}V).
		\end{aligned}
	\end{equation*}
	We consider the equalizer
	\begin{equation*}
		\begin{tikzcd}[row sep=2.6pc, column sep=2.6pc]
		(\mathcal{G}, \gamma)^f\coloneqq 	\ker  \Big(	\pi_*\mathcal{G}
			\arrow[r, shift left, " \sigma"] 
			\arrow[r,shift right,"\tau"'] &
			\pi_* f_*\mathcal{G} \Big),
		\end{tikzcd} 
	\end{equation*}
   which is a sheaf on $\XX{f}$. We claim that
	\begin{equation}\label{inversepullbackequivalence}
		(\cdot)^f:Sh(X)_f \to 	 Sh(\XX{f}), \;(\mathcal{G}, \gamma)\mapsto (\mathcal{G}, \gamma)^f
	\end{equation}
	is the essential inverse of {\boldmath $\pi^*$}.\\
	First, we check that for any $(\mathcal{G},\gamma)\in Sh(X)_f$ we have
	$$
	\mbox{\boldmath $\pi^*$}(\mathcal{G},\gamma)^f \overset{\sim}{\longrightarrow} (\mathcal{G},\gamma).
	$$
	The left adjoint of $(\mathcal{G},\gamma)^f \hookrightarrow \pi_*\mathcal{G}$  gives a morphism 
	\begin{equation}\label{co-unitcircinclusion}
		\pi^*(\mathcal{G},\gamma)^f \hookrightarrow \pi^*\pi_*\mathcal{G} \longrightarrow \mathcal{G},
	\end{equation}
	which, as it can be seen directly from the definition of $(\mathcal{G},\gamma)^f$, 
	coincides with the canonical inclusion, eventually composed 
	with the projection on the first factor, cf. \ref{example2cap1}.
	We claim that the above map is invertible, and in order 
	to show this, cf. \ref{pullbackXfonbasefact},
	it is sufficient to prove that it is invertible at the level of pre-sheaves, 
	\textit{i.e.} (cf. \ref{pullbackXfonbasefact})
	$$
b^*(\mathcal{G},\gamma)^f \overset{\sim}{\longrightarrow} 	\mathcal{G}.
	$$
	Let us compute, for any $\Delta \in ob(\XX{f})$,
	\begin{equation*}
		\begin{tikzcd}[row sep=2.6pc, column sep=2.6pc]
			(\mathcal{G},\gamma)^f(U_{\Delta},j_{\Delta})= 
				\ker  \Bigg(\displaystyle	\prod_{n\geq 0} \mathcal{G}_n(\Delta)
			\arrow[r, shift left, " \sigma"] 
			\arrow[r,shift right,"\tau"'] &
	\displaystyle	\prod_{n\geq 0} \mathcal{G}_{n+1}(\Delta) \Bigg),
		\end{tikzcd} 
	\end{equation*}
	where we have set $\mathcal{G}_n\coloneqq (f^{n})_*\mathcal{G}$. \\
	It follows from the definitions that if $x_n \in \mathcal{G}_n(\Delta)\; \forall \,n\in \N $,
	\begin{equation}\label{computationkernelXf}
			\sigma[ (x_n)_{n\geq 0}]=(ad(\gamma)(x_n))_{n\geq 0} \qquad \mbox{ and } \qquad
		\tau[ (x_n)_{n\geq 0}]=(x_{n+1})_{n\geq 0}.
	\end{equation}
	Consequently, the map 
	$$
	x \in \mathcal{G}(\Delta) \mapsto (ad(\gamma)^n(x))_{n\geq 0} 
	\in 	(\mathcal{G},\gamma)^f(U_{\Delta},j_{\Delta})
	$$ 
	is the inverse of 
	\eqref{co-unitcircinclusion}.
	To conclude, recall from \ref{pullbackXfonbasefact} that the action on $\pi^*(\mathcal{G},\gamma)^f$
	is by shifting, and consider the following diagram 
	\begin{equation*}
		\begin{tikzcd}[row sep=2.6pc, column sep=1.6pc]
			\mathcal{G}(\Delta)  \arrow[r] \arrow[d]  & 	b^*(\mathcal{G},\gamma)^f(\Delta) \arrow[d]\\
			\mathcal{G}(f^{-1}\Delta)  \arrow[r] &	b^*(\mathcal{G},\gamma)^f(f^{-1}\Delta)
		\end{tikzcd} 
	\end{equation*}
	where the horizontal arrows are the maps we just described and the vertical
	arrows are the respective actions. It is clearly commutative, 
	with both compositions equal to 
	$\displaystyle x \in \mathcal{G}(\Delta) \mapsto 
	(ad(\gamma)^n(x))_{n \geq 1} \in 	\underset{n\geq 1}{\prod} \mathcal{G}_n(\Delta)$. 
	Finally, we are left to show that
	for any  $\mathcal{F} \in Sh(\XX{f})$ there is an isomorphism
	$$
	\mathcal{F}	\overset{\sim}{\longrightarrow} (\mbox{\boldmath $\pi^*$}\mathcal{F})^f.
	$$
	The unit morphism  $\mathcal{F} \to \pi_*\pi^*\mathcal{F}$ provides the map above, 
	since its image is contained in the kernel of $(\sigma, \tau)$.
	It is sufficient to show that the above morphism is an isomorphism on the basis 
	\ref{basissievesXf}, \textit{i.e.} for any covering sieve $R \hookrightarrow \underline{\Delta}$ in $X$,  
	$$
	\mathcal{F}(b_!R)\overset{\sim}{\longrightarrow}
	\ker( \pi_*\pi^*\mathcal{F}(b_!R) \rightrightarrows \pi_*f_*\pi^*\mathcal{F}(b_!R)).
	$$
The kernel on the right hand side is computed by adjunction as follows
	\begin{equation*}
		\ker\big( b^*\pi_*\pi^*\mathcal{F}(R) \rightrightarrows b^*\pi_*f_*\pi^*\mathcal{F}(R)\big).
	\end{equation*}
	Recall from \ref{example2cap1} that
	$$
b^*\pi_*\pi^*\mathcal{F}(R)=\prod_{n \geq 0} (f^n)_*\pi^*\mathcal{F}(R).
	$$
	Consequently, the computation of the resulting kernel is analogous to
	the previous one, cf. \eqref{computationkernelXf}, giving an isomorphism
	$$
    \pi^*\mathcal{F}(R)\overset{\sim}{\longrightarrow} 	
    \ker\big( \pi_*\pi^*\mathcal{F}(b_!R) \rightrightarrows \pi_*f_*\pi^*\mathcal{F}(b_!R)\big).
	$$
	The isomorphism of pre-sheaves 
	$b^*\mathcal{F}\overset{\sim}{\longrightarrow} \pi^{-1}\mathcal{F}$, cf. \ref{pullbackXfonbasefact},
	yields an isomorphism between the associated sheaves
	$$
	a(b^*\mathcal{F})\overset{\sim}{\longrightarrow} \pi^*\mathcal{F},
	$$
	concluding the proof, since $\mathcal{F}(b_!R)\cong \mathcal{F}(a(b_!R))\cong a(b^*\mathcal{F})(R)$. \\
	Note that what we have provided above is an excessively detailed proof 
	of the isomorphism above, for the sake of explaining its nature. 
	For the sake of brevity, we could have 
	alternatively observed that 
	the sieve $R$ on $(\Delta,\delta)$ generated by \eqref{coequalizerobjectXf} is a covering sieve, 
	so by the sheaf property,
	the resulting equalizer $(\pi^*\mathcal{F})^f(\Delta,\delta)=
	{\rm Hom}_{\XX{f}}(R,\mathcal{F})$
	is isomorphic to $\mathcal{F}(\Delta,\delta)$.
\end{proof}
\begin{cor}\label{coveringsievesXfcor}
	The covering sieves 
	$R_u \hookrightarrow \underline{(U,u)} $ in $ \XX{f}$ shall be identified, 
	by means of \ref{equivalencelemmaXf}, with pairs $(R,\underline{u}')$ 
	where $R$ is a covering sieve $i_R: R\hookrightarrow \underline{U} $ in $X$ and
	$\underline{u}': f^{-1}R \to R$  is an action, such that the following diagram
	\begin{equation*}
		\begin{tikzcd}[row sep=2.6pc, column sep=2.6pc]
			f^{-1}	R	\arrow[r,"\underline{u}'"] \arrow[d, hook, "f^{-1}i_R"'] &  R \arrow[d, hook,  "i_R"] \\ 
			f^{-1}	\underline{U}  \arrow[ r, "\underline{u}"] & \underline{U}
		\end{tikzcd} 
	\end{equation*}
	is commutative, in the notation of \ref{example1cap1}. 
	We shall refer to $(R,\underline{u}')$ as a dynamical sieve on 
	$(\underline{U},\underline{u})$ in $\widehat{X}_f$.
\end{cor}

\begin{proof}
	Note that the inverse of $\mbox{\boldmath $\pi^{*}$}$ 
	provided in \ref{equivalencelemmaXf}, cf. \eqref{inverseequivalenceXf},
	is a well defined functor
	\begin{equation*}
		\begin{tikzcd}[row sep=0.5pc, column sep=1.3pc]
			(R,\underline{u}') \in \widehat{X}_f \arrow[r, mapsto]& (R,\underline{u}')^f \in \XX{f}^{\wedge},\\
			(R,\underline{u}')^f \coloneqq \ker \Big( t^*R 	\arrow[r, shift left, " \sigma"] 
			\arrow[r,shift right,"\tau"'] 
			&  t^*(f_*R)\Big).
		\end{tikzcd} 
	\end{equation*}
	Let $(R, \underline{u}')$  be any dynamical sieve. 
	The fact that $\mbox{\boldmath $\pi^{-1}$}(R, \underline{u}')^f \overset{\sim}{\longrightarrow} (R, \underline{u}')$
	is immediate for $(R,\underline{u}')=(\underline{U},\underline{u})$ the maximal sieve. 
	In fact, it follows by the definition of $(\cdot)^f$ that
	$$
	(\underline{U},\underline{u})^f \overset{\sim}{\longrightarrow} \underline{(U,u)},
	$$
	while from \ref{objectinitialXffact},
	$$
	\mbox{\boldmath $\pi^{-1}$}\underline{(U,u)} \overset{\sim}{\longrightarrow} (\underline{U},\underline{u}).
	$$
	In order to conclude the proof for any dynamical sieve, 
	recall that for any object $\delta: f^{-1}\Delta \to \Delta \in \XX{f}$, the set
	$(R,\underline{u}')(\Delta, \delta)$ consists of those maps $(\Delta, \delta) \to (U,u) \in \XX{f}$,
	which factorizes through $(R,\underline{u}')$. In the notation of \eqref{basisobjectsXf}, 
	this is the same as a map $\delta_{\bul}: b(\Delta) \to (U,u)$ such that 
	$\delta_{\bul} b(\delta)=\delta_{\bul}\bar{j}_\Delta$, 
	which factorizes through $(R,\underline{u}')$. 
	Conversely, let $R_u$ be a covering sieve on $(U,u) \in ob(\XX{f})$. 
	Recall from \ref{defsievesonXf}, \ref{boldpullbackfact} that
	$\mbox{\boldmath $\pi^{-1}$}R_u$ is a dynamical sieve. Applying
	the same argument as before we conclude that 
	$R_u \overset{\sim}{\longrightarrow} (\mbox{\boldmath $\pi^{-1}$}R_u)^f$. 
\end{proof}
\begin{notation}\label{notationcoupleXf}
	We shall refer to $Sh(X)_f$ as the {\bf dynamical tòpos} of $f$.	
	Using \ref{equivalencelemmaXf} we shall identify sheaves on $\XX{f}$
	with pairs $(\mathcal{F}, \varphi) \in Sh(X)_f$.
When there is no room for confusion, we shall omit the ``action"
	$\varphi$ and write just $\mathcal{F}$.
\end{notation}

\begin{remark}
	If $(X, \mathcal{O}_X)$ is a ringed space, and $f$ is an 
	endomorphism of ringed spaces, we have a canonical sheaf of rings sitting in
	$Sh(X)_f$, that is $\mathcal{O}_{\bul}\coloneqq (\mathcal{O}_X, f^*)$, 
	where we abuse notation by writing $f^*$ for the canonical action
	\begin{equation}\label{actionstructuresheaf}
		f^*: f^*\mathcal{O}_X \to \mathcal{O}_X.
	\end{equation}
	Note that in this case we are interested only in a sub-category 
	of $Ab(\XX{f})$. For any sheaf $\mathcal{F}$ 
	of $\mathcal{O}_X$-modules, we replace the abelian sheaf
	$f^*\mathcal{F}$ with the corresponding
	$\mathcal{O}_X$-module 
	$\displaystyle f^*\mathcal{F} \otimes _{f^*\mathcal{O}_X}\mathcal{O}_X$, 
	which abusing notation we still denote by $f^*\mathcal{F}$. Note that 
	the map \eqref{actionstructuresheaf}
	in this notation is an isomorphism induced by the identity. 
	Then, the interesting subcategory of $Ab(\XX{f})$ 
	is the category of \textbf{sheaves of $\mathcal{O}_{\bul}$-modules}, 
	\textit{i.e.} sheaves $\mathcal{F}$ of $\mathcal{O}_X$-modules with a 
	linear action $\varphi: f^*\mathcal{F} \to \mathcal{F}$, in the notation above.
\end{remark}

\section{Dynamical site {\rm \bf  II}}
It seems likely that most of the dynamics of $f$ is captured by the site $\XX{f}$,
but in practice it is not enough: the backward invariant elements $U \in \XX{f}$ 
are not sufficient to describe some aspects of the dynamics, as we shall see in the applications.
The next definitions are motivated 
by our need of enlarging slightly the site in such a way 
that we are allowed more freedom in the action. 
In Lemma \ref{F0=F1} we shall discuss the relations between these sites.
\begin{definition}\label{Epsteintopos}{\bf [Dynamical tòpos {\rm \bf II}]}\
	
	Let  ${\widetilde{\mathit{E}}_f}$ be the following category
		\begin{itemize}
		\item  $ob(\widetilde{\mathit{E}}_f)$ are pairs $ (\mathcal{F}_{\bul}, \varphi_{\bul}) $ consisting of
		\begin{enumerate}[a)]
			\item A pair of sheaves $ \mathcal{F}_{\bul}= (\mathcal{F}_0, \mathcal{F}_1)\in 
			ob(Sh(X)) \times 	ob(Sh(X)) $;
			\item An action map $ \varphi_{\bul}=\varphi_0\coprod \varphi_1: 
			\mathcal{F}_0\coprod f^*\mathcal{F}_0 \to \mathcal{F}_1 \in ar(Sh(X)).
			$
		\end{enumerate}
		\item The arrows from  $  (\mathcal{F}_{\bul}, \varphi_{\bul})$ to $ (\mathcal{G}_{\bul}, \gamma_{\bul}) $  in $
		\widetilde{\mathit{E}}_f $ are pairs of natural transformations
		$$
		\theta_{\bul}=(\theta_0,\theta_1) \in {\rm Hom}(\mathcal{F}_0,
		\mathcal{G}_0)\times {\rm Hom}(\mathcal{F}_1,\mathcal{G}_1)
		$$
		such that the two following diagrams commute
		\begin{equation}\label{arrowsEf}
			\begin{tabular}{l r }
				\begin{tikzcd}[row sep=2.6pc, column sep=2.6pc]
					\mathcal{F}_0 \arrow[r, "\varphi_0"] 
					\arrow[d," \theta_0"'] &
					\mathcal{F}_1 \arrow[d, "\theta_1"] \\
					\mathcal{G}_0 \arrow[r, "\gamma_0"] &
					\mathcal{G}_1
				\end{tikzcd} & 	\qquad
				\begin{tikzcd}[row sep=2.6pc, column sep=2.6pc]
					f^*\mathcal{F}_0 \arrow[r, "\varphi_1"] 
					\arrow[d," f^*\theta_0"'] &
					\mathcal{F}_1 \arrow[d, "\theta_1"] \\
					f^*\mathcal{G}_0 \arrow[r, "\gamma_1"] &
					\mathcal{G}_1
				\end{tikzcd} 
			\end{tabular}
		\end{equation}
	\end{itemize}
\end{definition}

\begin{example}
	If $(X, \mathcal{O}_X)$ is a ringed space, and $f$ is an endomorphism of 
	ringed spaces, we have a canonical sheaf of rings sitting in
	$\widetilde{\mathit{E}}_f$, that is $\mathcal{O}_{\bul}\coloneqq (\mathcal{O}_X,\mathcal{O}_X)$ 
	with action $id_X\coprod f^*$, cf. \eqref{actionstructuresheaf}. 
	In this case we will be mostly interested in the subcategory of $\widetilde{\mathit{E}}_f$ 
	consisting of \textbf{sheaves of $\mathcal{O}_{\bul}$-modules}, 
	\textit{i.e.} pairs of sheaves of $\mathcal{O}_X$-modules with a linear action (here
	the functor $f^*$ has to be intended as the one preserving the module structure).
\end{example}

\begin{claim}
	The category $\widetilde{\mathit{E}}_f$
	coincides with the category of sheaves on a new site, which 
	we define below in \ref{Efcategorydef}, \ref{Epsdynsitedef}.
\end{claim}
\begin{definition}[{\bf Coproduct site}]
	Let $(X_0, J_{X_0}),(X_1, J_{X_1})$
	 be two small sites with countable, cf.  \eqref{countability}, coproducts. 
	Let us denote by $\emptyset_0,\emptyset_1$ their initial object, respectively.\\
	The site $X_0 \coprod X_1$ is defined as follows:
	\begin{itemize}
		\item The underlying category is the coproduct 
		category $X_0 \coprod X_1$, whose objects are 
		$(\{ 0 \}\times ob(X_0) )\cup( \{ 1 \}\times ob(X_1))$;
		\item The arrows are, for $i,j=0,1$,
		$$
		{\rm Hom}_{X_0 \coprod X_1}((i,U), (j,V))= 
		\begin{cases}
				{\rm Hom}_{X_i}(U, V), \mbox{ if } i=j;\\
			 \emptyset, \mbox{ otherwise. }
		\end{cases}
		$$
	\end{itemize}
A covering sieve $R_i \hookrightarrow (i,U)$ is a covering sieve
$R \in J_{X_i}(U)$. The axioms of a Grothendieck topology are 
trivially satisfied in $X_0 \coprod X_1$, since they are satisfied in each $X_i$. 
\end{definition}
The problem arising from the above definition is that it
provides a site that is not closed for countable, cf.  \eqref{countability}, coproducts. 
\begin{factdefinition}[{\bf \'Etale coproduct site}]\label{fact:prodvscoprod}
	Let $(X_0, J_{X_0}),(X_1, J_{X_1})$
be two small sites with countable, cf.  \eqref{countability}, coproducts. 
Let us denote by $\emptyset_0,\emptyset_1$ their initial object, respectively.
The category $X_0 \vee X_1$ is defined as follows:
\begin{itemize}
\item The objects are $ob\left( X_0 \coprod X_1 \right)$;
\item The arrows are the union of the arrows in 
$X_0 \coprod X_1$  and the set of arrows generated
 by $\{ i_0,i_1 \}$, where 
 $$
 i_0: \emptyset_0 \to \emptyset_1; \quad i_1: \emptyset_1 \to \emptyset_0.
 $$
\end{itemize}
The arrows $i_0,i_1$ are isomorphisms and hence they identify 
the two initial objects (since both $\emptyset_0,\emptyset_1$ are
 defined up to a unique isomorphism), thus defining an initial object 
 $\emptyset \in ob(X_0 \vee X_1)$. Note that, although 
 $X_0 \vee X_1$ and $X_0 \coprod X_1$ are different, 
 they define equivalent sites. Therefore,
 $$ 
 Sh(X_0 \vee X_1)= Sh\left(X _0 \coprod X_1 \right)= Sh(X_0) \prod Sh(X_1).
 $$\\
Let us now consider the 
closure of the site $X_0 \vee X_1$ by countable, cf.  \eqref{countability}, coproducts, defined 
as follows. The objects of the underlying category, 
denoted by $(\mbox{\'Et}(X_0 \vee X_1))$, are
 countable, cf.  \eqref{countability}, disjoint
coproducts $\coprod_{\alpha \in A}U_{\alpha}$, 
where $A \subseteq \N$, and 
$U_{\alpha} \in ob(X_0 \vee X_1)$. 
Using some version of the axiom of choice if necessary, 
one can arrange things such that each object of this new category is written 
in a unique way as $U_0 \coprod U_1$, where $U_0 \in ob(X_0)$, and 
$U_1\in ob(X_1)$ (essentially because both $X_0$ and $X_1$ are closed 
for coproducts). 
Then, each arrow $U_0 \coprod U_1 \to V_0 \coprod V_1$ in 
$\mbox{\'Et}(X_0 \vee X_1) $  arises as a pair $(j_0,j_1)$, where
$j_0:U_0  \to V_0 \coprod V_1$ and $j_1: U_1  \to V_0 \coprod V_1$
are in $\mbox{\'Et}(X_0 \vee X_1)$, meaning that for $i=0,1$, the arrow 
$j_i: U_i  \to V_0 \coprod V_1$ factorizes as 
$j_i: U_i  \to V_i \hookrightarrow V_0 \coprod V_1$, 
where the first arrow is in $X_i$.\\
Let $X_0 \times X_1$ be the product category. 
Its objects are pairs $(U_0,U_1) \in ob(X)\times ob(X)$
and its arrows are defined componentwise. 
If $X_0, X_1$ are closed for
countable, cf.  \eqref{countability}, coproducts, so it is $X_0 \times X_1$ and we have
$(U_0,U_1) \coprod (V_0,V_1)= (U_0 \coprod U_1, V_0 \coprod V_1) $.
It follows that there is an equivalence of categories
\begin{equation}\label{coprodvsprod}
 \Psi:	\mbox{\'Et}(X_0 \vee X_1) \overset{\sim}{\longrightarrow} 
 X_0 \times X_1, \; U_0 \coprod U_1 \mapsto (U_0,U_1),
\end{equation}
and hence we can turn the category $X_0 \times X_1$ 
in a site whose category of sheaves is the product $Sh(X_0) \prod Sh(X_1)$. 
With an abuse of language, we still denote this site by 
$X_0 \times X_1$.
\end{factdefinition}
\begin{proof}
	The sites $X_0 \vee X_1$ and $X_0 \coprod X_1$ 
	 are equivalent since their respective categories of pre-sheaves are equivalent. In fact, they differ on objects by only one object, which is the initial object of
	 $X_0 \vee X_1$. 
	 On arrows instead, their difference consists of
	 the presence in $X_0 \vee X_1$ of universal maps from its initial object  to all of its other objects, 
	  with pullback along this maps of any pre-sheaf being the trivial map from a set to a pointed set. 
	  Hence, we can transport the topology of $X_0 \coprod X_1$
	  to a topology on $X_0 \vee X_1$, and the sheaves are the same.
	The functor 
	$\Psi$ is well defined, thanks to the description of the 
	arrows in $\mbox{\'Et}(X_0 \vee X_1)$ provided above, and 
	have has obvious essential inverse obtained by sending
	$(U_0,U_1)$ to $\left( U_0 \coprod \emptyset_1 \right) \coprod 
	\left(  \emptyset_0  \coprod U_1 \right)$ (we choice the order once for all in the case $X_0=X_1$). 
	Note that the initial object 
	$\emptyset =[\emptyset_0]=[\emptyset_1] \in 	\mbox{\'Et}(X_0 \vee X_1)$
	is isomorphic to $\emptyset_0 \coprod \emptyset_1$. 
	\end{proof}
\begin{definition}
Let $X$ be a small site with countable, cf.  \eqref{countability}, coproducts and $f:X \to X$ a 
morphism of sites. Let us choice $X_0\coloneqq X_1=X$ in \ref{fact:prodvscoprod}. 
There is an obvious 
morphism of sites that we denote by 
$\left( \mathbbm{1} + f \right)$, which on objects is
\begin{equation}\label{generatorepsteinsite}
	\begin{tikzcd}[row sep=0.6pc, column sep=2.6pc]
		&	\left( \mathbbm{1} + f \right):X \times X \to X \times X, \\
		& 	\left( \mathbbm{1} + f \right)^{-1}(U_0,U_1)\coloneqq (\emptyset, U_0 \coprod f^{-1}U_0).
	\end{tikzcd} 
\end{equation}
\end{definition}
In the following $X$ is a a small site with countable, cf.  \eqref{countability}, coproducts and $X\times X$ is the site
 defined above.
The morphism of sites $\left( \mathbbm{1} + f \right)$ yields to the definition of
a new site.
	\begin{definition}[{\bf $f$-compatibility}]\label{f-compatibility}
	A ``$f$-compatibility'' on a pair of objects $U_{\bul}\coloneqq (U_0,U_1)$, with $U_0,U_1 \in  ob(X)$ 
	is the datum of a map 
	$$
 u_{\bul}: U_{0}\coprod  f^{-1} U_{0} \rightarrow	U_{1}  \; \in X.
	$$
\end{definition}
One should think of a $f$-compatibility as an extension of the notion of 
``backward invariant'' objects in $X$, cf. \ref{notationcoupleXf}. 
In fact, to each $(U,u)\in ob(\XX{f}) $  
there is associated a 
$f$-compatibility on $(U,U)$, namely $u_{\bul}\coloneqq id_U\coprod u$.
Observe that for any $f$-compatibility on $(U_0,U_1)$, the canonical maps 
\begin{equation}\label{differentmapsEf}
	u_0: U_0 \rightarrow U_1, \quad 
	u_1:  f^{-1}U_0 \rightarrow U_1  \quad \in X
\end{equation}
furnished by \ref{f-compatibility} are considered two independent data, 
even if they coincide as maps in $X$, \textit{e.g.} whenever $U_0=f^{-1}U_0$, $u_0=u_1$.
\begin{definition}[{\bf Epstein's dynamical category}]\label{Efcategorydef}\
		\noindent
	The category ${{\rm E}_f}$ is defined as follows:
	\begin{itemize}
		\item	$\displaystyle ob({{\rm E}_f})=\{ (U_{\bul},u_{\bul}): U_{\bul} \in ob(X)^2, \;u_{\bul} \; f\mbox{-compatibility on } U_{\bul}  \}$;
		\item   The arrows $j_{\bul}: u_{\bul} \to v_{\bul}$ 
		are ordered pairs $(j_0,j_1)$ of arrows $j_i \in {\rm Hom}_{X}(U_i,V_i)$, $i=0,1$, such that
		the following diagrams commute
			\begin{equation}\label{conditionmapEf}
			\begin{tabular}{l r }
				\begin{tikzcd}[row sep=2.6pc, column sep=2.6pc]
				U_0 \arrow[r, "u_0"] 
					\arrow[d," j_0"'] &
					U_1 \arrow[d, "j_1"] \\
					V_0 \arrow[r, "v_0"] &
					V_1
				\end{tikzcd} & 	\qquad
				\begin{tikzcd}[row sep=2.6pc, column sep=2.6pc]
					f^{-1}U_0 \arrow[r, "u_1"] 
				\arrow[d,"f^{-1} j_0"'] &
				U_1 \arrow[d, "j_1"] \\
				f^{-1}V_0 \arrow[r, "v_1"] &
				V_1
				\end{tikzcd} 
			\end{tabular}
		\end{equation}
	\end{itemize}
\end{definition}
	
	\begin{definition}[{\bf Epstein's dynamical site}]\label{Epsdynsitedef}\
		\noindent
		The category ${\rm E}_f$ is a site when 
		we consider the topology induced,
		 cf. \ref{sitewithdyn}, from the topology of ${X}\times {X}$, cf.
		 \ref{fact:prodvscoprod}, through the functor which on objects is
		 the following
		 $$
		 T: {{\rm E}_f} \rightarrow {X}\times {X}, \; (U_{\bul}, u_{\bul})\mapsto U_{\bul}.
		 $$ 
			\end{definition}	We shall refer to
$({{\rm E}_f},J_{{\rm E}_f} )$ as the Epstein's dynamical site of $f$,  
and as usual when there is no room for 
confusion we abuse notation by writing just ${{\rm E}_f}$.
\begin{definition}\label{pre-topoEf}
	Let $\widehat{\mathit{E}}_f$ denote the category of pre-sheaves 
	corresponding to $\widetilde{\mathit{E}}_f$, \ref{Epsteintopos}.
	 Its objects (resp. its arrows) are pairs
	$(\mathcal{F}_{\bul}, \varphi_{\bul})$ as in \ref{Epsteintopos}, 
	with the only difference that all the objects 
	involved are
	pre-sheaves, in analogy with \ref{pre-toposXf}. Note that Yoneda embedding induces a fully faithful
	functor ${{\rm E}_f} \to \widehat{\mathit{E}}_f$, cf. \ref{example1cap1}.
\end{definition}
A series of results for Epstein's site are analogous to those of $\XX{f}$, to wit
\begin{lemma}\label{resultsEpsteinsite}
	The following results hold:
	\begin{enumerate}
		\item $T$ admits a left adjoint $B: X \times X \to {{\rm E}_f}$
		 given on objects by
	\begin{equation}\label{leftadjointtargetEf}
			B: \; (U_0,U_1) \mapsto B^0(U_0)\coprod B^1(U_1),
		\end{equation}
	where we have set $B=B^0pr^0 \amalg B^1pr^1$,
	$$
	pr^i: X \times X \to X, \quad i=0,1,
	$$
	are the canonical projections on the first and second factor, 
	respectively, while the functors $B^i$, $i=0,1$ are defined on objects as
	\begin{equation*}
		\begin{aligned}	
			&B^0(U)\coloneqq ( (U,U\coprod f^{-1}U), (\iota_0,\iota_1) ),\\
			&B^1(U)\coloneqq ( (\emptyset,U), (\ast_U,\ast_U) ),
		\end{aligned}
		\end{equation*}
 for any $U \in ob(X)$, where $\iota_{\alpha}$ are the canonical monomorphisms 
	$\displaystyle \iota_{\alpha}: U_{\alpha}\hookrightarrow \coprod_{\alpha}U_{\alpha}$ and 
	$\ast_U$ is the canonical initial morphism $\ast_U:\emptyset \hookrightarrow U$;
	\item The co-unit of the above adjunction yields, for any 
	$(U_{\bul},u_{\bul})\in ob(	{{\rm E}_f})$, a map 
	$$
\epsilon_{\bul}=\epsilon_{\bul}(U_{\bul},u_{\bul}):	B(T(U_{\bul},u_{\bul})) \to (U_{\bul},u_{\bul}),
	$$
	that fits into a natural coequalizer
		\begin{equation}\label{coeqobjectsEf}
			\begin{tikzcd}[row sep=2.6pc, column sep=2.6pc]
				B^1(U_0\coprod f^{-1}U_0)
				\arrow[r, shift left, "{j_{U_{\bul}}}"] 
				\arrow[r,shift right, "B^1(u_{\bul})"'] &
				B(U_0, U_1) \arrow[r, "\epsilon_{\bul}"] 
				& (U_{\bul}, u_{\bul}),
			\end{tikzcd} 
		\end{equation}
	where the map ${j_{U_{\bul}}}$ on the left is defined as follows:
	\begin{equation*}
		{j_{U_{\bul}}}\coloneqq \ast_{U_0} \times id_{U_0 \coprod f^{-1}U_0};
	\end{equation*}
	\item 	The family of sieves 
	$$
	\{ B_!(R_{\bul}): R_{\bul}  \in J_{X\times X}(U_{\bul})\}
	$$
	is a family of covering sieves on $(U_{\bul},u_{\bul})$. This collection forms a basis for ${{\rm E}_f}$;
	\item The functor 
	$$
	\Pi^{-1}:{{\rm E}_f}^{\wedge} \to (X \times X)^{\wedge}, 
	$$
	coincides by \ref{resultsEpsteinsite} with the functor $B^*$, taking a pre-sheaf $\mathcal{F}\in {{\rm E}_f}^{\wedge}$ to 
	$$
	\Pi^{-1}\mathcal{F}:\; U_{\bul}\mapsto \mathcal{F}(B(U_{\bul}))
	$$

	\item There is a functor
	\begin{equation}\label{fuctoreqEf}
		\mbox{\boldmath	$\Pi^{*}$}: Sh({{\rm E}_f}) \to \widetilde{\mathit{E}}_f,\; 
		\mathcal{F} \mapsto(\mathcal{F}_{\bul}, \varphi_{\bul}),
	\end{equation}	
obtained applying the associated sheaf
functor ``$a$'' to the construction above . 
We have set,
for any $U \in ob({X})$, 
\begin{equation}
\begin{aligned}
	&	\mathcal{F}_0 (U)\coloneqq\mathcal{F}(B^0(U));\\
	&	\mathcal{F}_1(U)\coloneqq \mathcal{F}(B^1(U)),
\end{aligned}
\end{equation}
and the map $\varphi_{\bul}=\varphi_0 \coprod \varphi_1: 	
\mathcal{F}_0 \coprod  f^{*}	\mathcal{F}_0 \to 	\mathcal{F}_1$ is defined as follows:
\begin{equation}
	\varphi_0(U)\coloneqq\mathcal{F}(\bar{\iota}_0: B^1(U) \hookrightarrow B^0(U))	;
\end{equation}
while the 	adjoint of  $\varphi_1$, 	$ad(\varphi_1): \mathcal{F}_0 \to 	f_*\mathcal{F}_1$	 
\begin{equation}
	ad(\varphi_1)(U)\coloneqq\mathcal{F}(\bar{\iota}_1: B^1(f^{-1}U)  \hookrightarrow B^0(U)).
\end{equation}
Note that $\varphi_0$ and $\varphi_1$ are independent, since by hypothesis 
we have 
$$
(\emptyset, U) \times_{ (U,U \coprod f^{-1}U)} (\emptyset, f^{-1}U) \overset{\sim}{\longrightarrow} \emptyset_{X\times X}
$$
hence
$$
B^0(U) \times_{ B^0(U)} B^1(f^{-1}U) \overset{\sim}{\longrightarrow} \emptyset_{{\rm E}_f}
$$
	\end{enumerate}
\end{lemma}
\begin{proof}\
	\begin{enumerate}
		\item There is a natural transformation of bifunctors which on objects is
		$$
		{\rm Hom}_{\rm E_f}(B(U_0,U_1), (V_{\bul}, v_{\bul})) \to 
		{\rm Hom}_{X\times X}((U_0,U_1), T(V_{\bul}, v_{\bul})),
		$$
		obtained applying the functor $T$. In order to see that it is invertible, note that
		the commutativity condition on the diagrams involved are trivially satisfied,
	   as a consequence of the fact that $f^{-1}\emptyset \cong \emptyset$;
		\item We can check easily that the universal property of coequalizers is satisfied. 
		In fact, giving a map $j_{\bul}:(U_{\bul},u_{\bul}) \to (V_{\bul},v_{\bul}) \in {\rm E}_f$
		is, by definition, equivalent to giving the map 
		$
j_{\bul}\epsilon_{\bul}	: B(U_{\bul}) \to (V_{\bul}, v_{\bul}),
		$
		\textit{i.e.} a pair of maps $j_i: U_{i} \to T_{i} \in X$, $i=0,1$, such that 
		the commutativity condition \eqref{conditionmapEf} is satisfied. 
		This, in turn, is equivalent to asking that 
		$
	j_{\bul}\epsilon_{\bul}	j_{U_{\bul}}	=j_{\bul}\epsilon_{\bul}B^1(u_{\bul}),
		$
		or, more explicitly, that $j_1u_0=v_1j_0$ and $j_1u_1=v_1f^{-1}j_0$, \textit{i.e.} \eqref{conditionmapEf};
		\item We claim that the sieve $B^*B_!R_{\bul}$ is a covering sieve on $U_{\bul}$, since
		 the unit morphism $R_{\bul} \to B^*B_!R_{\bul}$ is a bi-covering.
		 First, note that they are both products, since
		 $B^*B_!R_{\bul}= (B^0pr^0)^*B_!R_{\bul} \times (B^1pr^1)^*B_!R_{\bul}$, as
		  its value on an object $\Delta_{\bul}$ 
		  consists by definition of those maps 
		 $$B(\Delta_{\bul})= B^0(\Delta_0)\coprod B^1(\Delta_1) \to (U_{\bul}, u_{\bul})$$ factorizing through $B_!R_{\bul}$.
		 Therefore, it is sufficient to show that 
		 $$
		 R_i \to (B^ipr^i)^*B_!R_{\bul}  , \quad i=0,1,
		 $$
		 is a bi-covering. Let $\mathcal{F}_{\bul}=(\mathcal{F}_0,\mathcal{F}_1)$ be a pair of sheaves on $X$, then
		 by adjunction, we need to show that
		 $$
		{\rm Hom}_{X}(R_{\bul}, \Pi^{-1}\Pi_*\mathcal{F}_{\bul}) \overset{\sim}{\longrightarrow} 
		  {\rm Hom}_{X}(R_0, \mathcal{F}_0) \times 
		 {\rm Hom}_{X}(R_1, \mathcal{F}_1),
		 $$
		which can be checked explicitly. In fact, we have functorial isomorphisms
		\begin{equation}\label{equationsheavesonbasisEf}
		\begin{aligned}
		&(\Pi^{-1}\Pi_*\mathcal{F}_{\bul})(\Delta_{\bul})=
		(\Pi_*\mathcal{F}_{\bul})(B(\Delta_{\bul}))= \\
		& (\Pi_*\mathcal{F}_{\bul})(B^0(\Delta_{0}))\times (\Pi_*\mathcal{F}_{\bul})(B^1(\Delta_{1})) \cong \mathcal{F}_{\bul}(\Delta_{\bul}),
		\end{aligned}	
		\end{equation}
	where in the second equality we used that $\mathcal{F}_0,\mathcal{F}_1$ are sheaves. 
	The last isomorphism is obtained by projecting on the first and second factor, 
	respectively. Its inverse is given by 
$$
\mathcal{F}_{\bul} \mapsto (pr^0)_*\left( \mathcal{F}_0 \times \mathcal{F}_1 
\times (f_*\mathcal{F}_1) \right) \times (pr^1)_*\left( pt \times \mathcal{F}_1 \right);
$$
		\item	It follows already by 1) and \cite[I.5.5]{SGA4};
		\item  It follows from the computations in 3).
	\end{enumerate}
\end{proof}

\begin{cor}
	The functor $T$ induces a morphism of sites
	\begin{equation}\label{PimapEf}
		\Pi: \left( X \times X , J_{X \times X}\right)
		\rightarrow \left( {{\rm E}_f} , J_{{{\rm E}_f}}\right).
	\end{equation}
\end{cor}
\begin{proof}
	The proof is a consequence of \ref{resultsEpsteinsite} and it analogous to \ref{tgeomorph}.
\end{proof}
\begin{theorem}\label{sheavesonepsteinsite}
	The functor $	\mbox{\boldmath	$\Pi^{*}$}: Sh({{\rm E}_f}) \to \widetilde{\mathit{E}}_f$ is an equivalence of categories.
\end{theorem}
\begin{proof}
In order to provide the essential 
 inverse of $	\mbox{\boldmath	$\Pi^{*}$}$,
	note that
	for any
	$(\mathcal{F}_{\bul}, \varphi_{\bul})\in ob(\widetilde{\mathit{E}}_f)$, and any 
	$(U_{\bul},u_{\bul})\in ob({{\rm E}_f})$,  there are canonical maps 
	\begin{equation}\label{inverseequivalence}
		\begin{tikzcd}[row sep=2.6pc, column sep=2.6pc]
			\mathcal{F}_0(U_0)\times	\mathcal{F}_1(U_1) 
			\arrow[r, shift left, "\alpha_0(U_{\bul})"] 
			\arrow[r,shift right," \alpha_1(U_{\bul})"'] &
			\mathcal{F}_1(U_0) \times 	\mathcal{F}_1(f^{-1}U_0) 
		\end{tikzcd} 
	\end{equation}
	where, in the notation of \eqref{differentmapsEf},
	\begin{equation*}
		\begin{aligned}
			& \alpha_0(U_{\bul})(s_0,s_1)= (\varphi_0(s_0),ad(\varphi_1)(s_0)); \\
			& \alpha_1(U_{\bul})(s_0,s_1)= ({u}_0^*(s_1), {u}_1^*(s_1)).
		\end{aligned}
	\end{equation*}
Note that the diagram above is functorial in $(U_{\bul},u_{\bul})$.
In fact, it may be written as the equalizer of a diagram in $Sh({\rm E}_f)$:
	\begin{equation*}
	\begin{tikzcd}[row sep=2.6pc, column sep=2.6pc]
	\Pi_*\left( 	\mathcal{F}_0\times	\mathcal{F}_1 \right)
		\arrow[r, shift left, "\alpha_0"] 
		\arrow[r,shift right," \alpha_1"'] &
	\Pi_*\left( \mathbbm{1} + f \right)_*\left( 	\mathcal{F}_0\times	\mathcal{F}_1 \right),
	\end{tikzcd} 
\end{equation*}
where we have used the notation \eqref{generatorepsteinsite}.
We claim that the equalizer of \eqref{inverseequivalence} is the required inverse, \textit{i.e.} we define
	\begin{equation*}
		(\mathcal{F}_{\bul}, \varphi_{\bul})^f\coloneqq\ker (\alpha_0, \alpha_1).
	\end{equation*}
	Note that projections on the two factors of the equalizer of \eqref{inverseequivalence}
	give a functorial isomorphism of pre-sheaves
	$$
	\mbox{\boldmath	$\Pi^{-1}$}\left[ (\mathcal{F}_{\bul}, \varphi_{\bul}) 
	^f \right] \overset{\sim}{\longrightarrow} (\mathcal{F}_{\bul}, \varphi_{\bul}),
	$$
	for any $(\mathcal{F}_{\bul}, \varphi_{\bul}) \in \widetilde{\mathit{E}}_f.$
	In fact, recalling that $\Pi^{-1}=B^*$, cf. \ref{resultsEpsteinsite}, its inverse is obtained as follows:
	\begin{equation*}
		\begin{aligned}
			& s \in \mathcal{F}_{0}(U) \mapsto 
			 (s,\varphi_0(s)\times ad(\varphi_1)(s) )\in 
	 ( \mathcal{F}_{0}(U) \times \mathcal{F}_{1}(U \coprod f^{-1}U)  )\\
			& s \in \mathcal{F}_{1}(U) \mapsto 	  (pt, s) \in ( pt\times \mathcal{F}_{1}(U) ).
		\end{aligned}
	\end{equation*} 
 The map above is well defined, \textit{i.e.} one can check directly that its 
 compositions with $\alpha_0,\alpha_1$ coincide. 
 Moreover, another straightforward computation shows that they 
 are inverse to each other, and hence the above isomorphism 
 yields an isomorphism between the associated sheaves, \textit{i.e.}
	\begin{equation*}
			\mbox{\boldmath	$\Pi^{*}$}\left[ (\mathcal{F}_{\bul}, \varphi_{\bul})^f  \right] 
			\overset{\sim}{\longrightarrow} (\mathcal{F}_{\bul}, \varphi_{\bul}).
	\end{equation*}
Finally, there is a functorial isomorphism
$$
\mathcal{F}  \overset{\sim}{\longrightarrow}   (	\mbox{\boldmath	$\Pi^{*}$}\mathcal{F})^f,
$$
for any $\mathcal{F}\in Sh({{\rm E}_f})$, which at the level of pre-sheaves is 
given by the pull-back along the co-unit 
$B(T(U_{\bul},u_{\bul})) \to (U_{\bul},u_{\bul})$. 
In fact, a direct computation shows that
$$
\left( \mbox{\boldmath	$\Pi^{*}$}\mathcal{F} \right)^f(U_{\bul}, u_{\bul}) 
=\ker\left( \mathcal{F}(B(U_{0}, U_1)) \rightrightarrows \mathcal{F}(B^1(U_{0}\coprod f^{-1}U_0))\right),
$$
where the parallel morphisms are given by the pull-back along the 
maps in \eqref{coeqobjectsEf}.
Since  the sieve generated by the family
\eqref{coeqobjectsEf} is a covering sieve, 
the resulting equalizer is isomorphic 
to $\mathcal{F}(U_{\bul},u_{\bul})$, by the sheaf property.
\end{proof}

\begin{cor}[{\bf E-dynamical sieves}]\label{coveringsievesEfcor} 
	
		\noindent
	Let us consider, for any $(U_{\bul},u_{\bul}) \in ob({{\rm E}_f})$, the following collection:
	$$
	\{ R_{\bul}: R_i \in J_{X}(U_i) ,\,  i=0,1, \mbox{ such that } u_{\bul}  \mbox{ factorizes as } R_{0}\coprod f^{-1} R_{0}
	\rightarrow	R_{1}\}
	$$
	to which we refer as the family of $E$-dynamical sieves.
	Then, by \ref{sheavesonepsteinsite}, we shall identify the covering sieves 
	$R_{u_{\bul}} \hookrightarrow (U_{\bul},u_{\bul})$ with the $E$-dynamical sieve given by 
	$\Pi^{-1}R_{u_{\bul}} $. In other words, giving a covering sieve $i_{R_{\bul}}: R_{u_{\bul}} \hookrightarrow (U_{\bul},u_{\bul})$ in
	${{\rm E}_f}$
	is equivalent to giving
	a pair of covering sieves
	$i_0: R_0 \hookrightarrow \underline{U}_0$, $i_1: R_1 \hookrightarrow \underline{U}_1$ in $X$, 
	with a morphism 
	$$
	\underline{u}'_{\bul}: R_{0}\coprod f^{-1} R_{0}
	\rightarrow	R_{1},
	$$
	such that the the following diagrams commute
	\begin{equation*}
		\begin{tabular}{l r }
			\begin{tikzcd}[row sep=2.3pc, column sep=2.3pc]
				R_0	\arrow[r, "\underline{u}'_{0}"] \arrow[d, hook ,  "i_0"'] &  R_1 \arrow[d, hook,  "i_1"] \\ 
				\underline{U}_0  \arrow[ r, "u_0"] & \underline{U}_1
			\end{tikzcd} 
			& 	\qquad
			\begin{tikzcd}[row sep=2.3pc, column sep=2.3pc]
				f^{-1}R_0	\arrow[r, "\underline{u}'_{1}"] \arrow[d, hook , "f^{-1}i_0"'] &  R_1 \arrow[d, hook,  "i_1"] \\ 
				f^{-1}	\underline{U}_0  \arrow[ r, "u_1"] & \underline{U}_1
			\end{tikzcd} 
		\end{tabular}
	\end{equation*}
\end{cor}
\begin{proof}
	In the notation of \ref{sheavesonepsteinsite} we need to show that
	for any covering sieve $R_{u_{\bul}} \hookrightarrow (U_{\bul},u_{\bul})$ in ${{\rm E}_f}$ and any pair of covering sieves
	$R_{\bul}$ on $U_{\bul}$ in $X \times X$, with action $\underline{u}'_{\bul}: R_{0}\coprod f^{-1} R_{0}
	\rightarrow	R_{1}$ we have
	$$
	R_{u_{\bul}}  \overset{\sim}{\longrightarrow} (\Pi^{-1}R_{u_{\bul}})^f,
	$$
	and 
	$$
	\Pi^{-1}\left[(R_{\bul},\underline{u}'_{\bul})^f\right] \overset{\sim}{\longrightarrow} (R_{\bul},\underline{u}'_{\bul}).
	$$
	The first isomorphism, following the same argument as in \ref{coveringsievesXfcor}, is deduced from \ref{resultsEpsteinsite}, 2).
	On the other hand, the second isomorphism is deduced by the computations in \ref{resultsEpsteinsite}, 3).
	\end{proof}
\begin{notation}
	In view of \ref{sheavesonepsteinsite} we shall employ the convention that a sheaf on 
	${{\rm E}_f}$ is a pair $(\mathcal{F}_{\bul}, \varphi_{\bul}) \in ob(\widetilde{\mathit{E}}_f)$.
	When there is no room for confusion 
	we shall omit the the map $ \varphi_{\bul}$ and write just $\mathcal{F}_{\bul}$.
	\end{notation}
Intuition suggests that the category 
$Sh(\XX{f})$ may be identified with a sub-category of $Sh({{\rm E}_f})$.
Consider the projection map given by 
$$
p: {{\rm E}_f} \to \XX{f},\;  p^{-1}(U,u)=((U,U), (id_U \coprod u)).
$$
\begin{definition}
Let	$S^f \hookrightarrow Sh({{\rm E}_f})$ be the full sub-category consisting
of pairs $(\mathcal{F}_{\bul}, \varphi_{\bul})$ for which 
$\mathcal{F}_{0}=\mathcal{F}_{1}$ and $\varphi_0=id_{\mathcal{F}_{0}}$.
\end{definition}

\begin{lemma}\label{F0=F1}
	$S^f$ is equivalent to $Sh(\XX{f})$.
\end{lemma}
\begin{proof}
	Note that $p^*$ is a fully faithful functor. Hence, the claim will follow once we 
	prove that $S^f$ is its essential image.
	Let $U \in ob({X})$ and observe that is sufficient to show the following
	$$
	(p^*\mathcal{F})(B^0(U))=(p^*\mathcal{F})(B^1(U))
	$$
	at the level of pre-sheaves, \textit{i.e.}
	$$
	\varinjlim_{ B^0(U) \to p^{-1}(V, v)} \mathcal{F}(V,v)= 
	\varinjlim_{B^1(U) \to p^{-1}(V, v)} \mathcal{F}(V,v).
	$$
	The above equality is a consequence of the fact that the objects $\mathcal{F}(V,v)$ 
	in the two direct systems above are the same. In fact, the existence of a map
	$$
(U, U \coprod f^{-1}U) \to 	(V, V)
	$$
	clearly implies the existence of a map
	$$
(\emptyset, U)\to	(V, V).
	$$
	For the converse implication, we observe that from any map $j: U \to V$ we obtain a map
	$f^{-1}j: f^{-1}U \to f^{-1}V $ and consequently, 
	since $v: f^{-1}V \to V $, we obtain by composition a map $f^{-1}U \to V$.
\end{proof}

In view of the above Lemma, we set the following
\begin{notation}\label{notationXfbul}
In view of \ref{F0=F1}, any sheaf $(\mathcal{F}, \varphi)$ on $\XX{f}$
 shall be identified with the ``diagonal'' pair $(\mathcal{F}_{\bul}, \varphi_{\bul})$, \textit{i.e.}
 $\mathcal{F}_{\bul}=(\mathcal{F},\mathcal{F})$ and $\varphi_{\bul}=id_{\mathcal{F}} \coprod \varphi $.
  When there is no room for confusion we shall write just $\mathcal{F}_{\bul}$.
\end{notation}

Note that there exist two morphisms of sites that 
fit into a commutative diagram
	\begin{equation*}
	\begin{tikzcd}[row sep=2.6pc, column sep=1.6pc]
	 X  \arrow[dr, "\pi"'] \arrow[rr, dashed]& & {{\rm E}_f} \arrow[dl, "p"] \\
	 &\XX{f}  & 
	\end{tikzcd} 
\end{equation*}
whose defining functors are the projections on the two factors of the underlying object $U_{\bul} \in X \times X$.
Although neither of these maps are descent maps,
there is a preferred choice between them given by the morphism
induced by the second projection,
$$
\varepsilon:  {X}   \to {{\rm E}_f}, \; \varepsilon^{-1}(U_{\bul}, u_{\bul})= U_1.
$$
Our reason for the aforementioned preference lies on the fact that $\varepsilon$
has the following property: for any sheaf $\mathcal{F} $ on ${\bf	X}$ we have, cf. \ref{sheavesonepsteinsite}, that
$
\varepsilon_*\mathcal{F}
$
corresponds to the pair of sheaves
$$
U \mapsto (\mathcal{F}(U \coprod f^{-1}U), \mathcal{F}(U)).
$$
On the contrary, the choice of the first projection would have lead to 
a pair whose second entry is trivial.
In this case, the resulting map would not satisfy the following nice property,
which instead holds true for $\varepsilon$:
$$
\varepsilon_*\pi^* \overset{\sim}{\longrightarrow} p^*.
$$
\section{Dynamical considerations}
  Let us consider the action of the group $Aut({X})$
on $End({ X})$ given by conjugation.
The first question that arises naturally is whether there is 
a relation between the Tòpoi $Sh(\XX{f})$ and $Sh(\XX{g})$, 
when $f$ and $g$ are conjugated by an automorphism of $X$.
 \begin{fact}
	The site $\XX{f}$
	depends only on the conjugation class of $f$, 
	\textit{i.e.} if $g=\phi^{-1}f\phi$, for some $\phi \in Aut(X)$, 
	then $\XX{f}$ and $\XX{g}$ are equivalent sites.
\end{fact}
\begin{proof}
	Let $\phi \in Aut({ X})$ be as in the statement. 
	We claim that the map
	$$
	(U,u) \in \XX{f} \mapsto (\phi(U), \phi(u)) \in \XX{g}
	$$
	is an equivalence of categories.
	Setting $V=\phi(U)$, the map $u$
	can be written as 
	$f^{-1}(\phi^{-1} V) \to \phi^{-1}(V)$, so applying $\phi$ we obtain a map $g^{-1}V \to V$. It is easy to see that the above functor
	induces an equivalence between their respective Tòpoi.
\end{proof}
Let us describe the relation between $\XX{f}$
and $\XX{f^n}$.
\begin{fact}
	Let $k>0$ be an integer that divides $n$. Then, we have a pair of adjoint functors
	\begin{equation*}
		\begin{tikzcd}[row sep=2.6pc, column sep=1.6pc]
			p_{n,k} \dashv s_{n,k}:  \XX{f^k} \arrow[r, shift right] \arrow[ from=r,shift right]&
			\XX{f^n}.
		\end{tikzcd} 
	\end{equation*}
\end{fact}
\begin{proof}
Let $n>1$ be an integer. Note that the map
	\begin{equation*}
	\begin{tikzcd}[row sep=2.6pc, column sep=1.8pc]
			X  \arrow[r, bend right=45, "\pi_n"']&  \XX{f^n}  \arrow[l, bend right=45, " \sigma_n"']
	\end{tikzcd} 
\end{equation*}
defined in \eqref{projectionmapXf} and \ref{pullbackXfonbasefact} is a particular case of the following. This generalizes to maps
	\begin{equation*}
	\begin{tikzcd}[row sep=2.6pc, column sep=1.6pc]
	\XX{f^k}  \arrow[r, bend right=45, "p_{n,k}"']& \XX{f^n}.  \arrow[l, bend right=45, " s_{n,k}"']
	\end{tikzcd} 
\end{equation*}
for every  $k>0$ that divides $n$. The map $p_{n,k}$ commutes with the projection maps $\pi_k,\pi_n$
and $s_{n,k}$ commutes with the sections $\sigma_k, \sigma_n$.
Specifically, setting $n=km$,
$p_{n,k}$ is induced by the inclusion functor,
$$
p_{n,k}^{-1}(U, u)=(U, u^m),
$$
while $s_{n,k}$ is induced by the functor
$$
s_{n,k}^{-1}(U,u)= (U \coprod f^{-k}U \coprod \dots \coprod f^{-(m-1)k}U, \mbox{shift}),
$$
where $\mbox{shift}$ denotes the natural permutation for the first $(m-1)$ factors, and the map $u$ for the last factor.
These are well defined morphisms of sites. Indeed,
if we assume there is a map $u:f^{-k}U \to U$, then for any $j >0$, we obtain a map
$u^j: f^{-jk}U \to U$ by applying inductively the functor 
$f^{-k}$ at each step, eventually composed with the map found at the 
previous step.
 On the other hand, if we are given a map $u: f^{-n}U \to U$, then there 
is a map 
$$
f^{-k}(t(s_{n,k}^{-1}(U,u)))= f^{-k}U \coprod  f^{-(k+1)}U \coprod \dots 
\coprod f^{-n}U \rightarrow t(s_{n,k}^{-1}(U,u)),
$$
given by shifting, except for the last factor 
where we apply the given map $u$ valued in the first factor.
\end{proof}
From the above facts it is not difficult to see the following.
\begin{cor}
	The association $n \mapsto \XX{f^n}$ is functorial in $n$, \textit{i.e.}
	it defines a $2$-functor
	$$
	(\N, |) \to \mathscr{C}/{ X},
	$$
	where $(\N, |)$ is the $1$-category where $i \to j \iff i|j$, and
	 $ \mathscr{C}/{X}$ denotes the $2$-category of
	 categories over ${ X}$.
\end{cor}
\begin{fact}
The map $p_{n,k}$ defined above is not, in general, a $\Z/m\Z$-torsor. 
\end{fact}
\begin{proof}
If $p_{n,k}$ were a $\Z/m\Z$-torsor, from the universal property of $\Z/m\Z$-torsors, cf. \cite{et},
would follow that for any sheaf $\mathcal{F}_{\bul}$  
on $\XX{f^k}$, the sheaf $p_{n,k}^*\mathcal{F}_{\bul}$
is a sheaf on $\XX{f^n}$ with a group action of $\Z/m\Z$. 
However, in general, $p_{n,k}^*\mathcal{F}_{\bul}$ carries only 
an action of the
\textit{monoid} $(\Z/m\Z)$
since the structure of $\mathcal{F}_{\bul}$
needs not to be invertible. The semigroup action is described as follows:
for each $g \in [0,m-1] \cong \Z/m\Z$ we denote by $g$ the map $f^{-gk}$.
Let $\mathcal{F}_{\bul}$ be given, cf. \ref{notationXfbul}, by 
$(\mathcal{F},\varphi)$, where $\varphi:(f^k)^*\mathcal{F} \to \mathcal{F}$ 
is the action of $\mathcal{F}_{\bul}$. Observe that, cf. \ref{pullbackXfonbasefact},
$$
p_{n,k}^*\mathcal{F}= \prod_{g \in  \Z/m\Z} g_*\mathcal{F}
$$
carries maps
$$
\psi_h: h^*(p_{n,k}^*\mathcal{F}_{\bul}) \to p_{n,k}^*\mathcal{F}_{\bul},
$$
for each $h \in [1,m-1]$ given by shifting in the first $m-1$ positions,
and the appropriate power of $\varphi$ in position $m$. 
These maps evidently satisfy the semigroup property. 
Note that existence of a group action of $\Z/m\Z$ implies that $\varphi$
is invertible. 
In that case in fact, for every $g \in \Z/m\Z$
the map $g^*\psi_{g^{-1}}$ is the inverse of $\psi_g$.	
\end{proof}
 \begin{fact}[ \textbf{(Geometric points of ${\bf \XX{f}}$)}]. 
 	Let $X$ be a topological space and 
 $x \in X$ a fixed point of $f$. 
 Given a set $F$, let us denote by $\mathcal{F}$ the \textit{skyscraper sheaf} of $X$ supported on $x$. 
 Then, there is a \textit{skyscraper sheaf}
 	 $\mathcal{F}_{\bul}$ on $\XX{f}$ supported on $x$, defined by
 	  the identity map:
 	\begin{equation*}
 		\begin{tikzcd}[row sep=2.6pc, column sep=2.6pc]
 			f^*(\mathcal{F}_{ x} )=\mathcal{F}_{f(x)} \arrow[r, equal ]
 			&	 \mathcal{F}_{x}	
 		\end{tikzcd} 
 	\end{equation*}
 This association is functorial in $F$, and hence defines a geometric
 morphism 
 \begin{equation*}
 	\begin{tikzcd}[row sep=2.6pc, column sep=1.6pc]
 		x^* \dashv x_*:  {\rm Set} \arrow[r, shift right] \arrow[ from=r,shift right]&
 		Sh(\XX{f}),
 	\end{tikzcd} 
 \end{equation*}
\textit{i.e.} a geometric point of $Sh(\XX{f})$. 
 \end{fact}
\newpage
\phantom{h}
\thispagestyle{empty}
\newpage
\chapter{Generalizations}
In this chapter we generalize the previous construction
and consider a site of $X$ satisfying the property (D), cf. Definition \ref{intro:def1}.
Then, we define the quotient of $X$ by the action of a  countable, cf.  \eqref{countability}, monoid $\Sigma$.
Let us fix a countable, cf.  \eqref{countability}, semigroup with identity $\Sigma$ and an action of
$\Sigma$ on the site $X$, \textit{i.e.} a semigroup homomorphism

$$\Phi: \Sigma \to {\rm End}(X).$$ 
We abuse notation and denote by 
$$
\sigma: X \to X,\; U \mapsto \sigma^{-1}U,
$$
the morphism of sites $A(\sigma)$ corresponding to $\sigma \in \Sigma$.
With this notation we see that the morphism of sites $``\sigma\circ \tau''$
corresponds to the functor $U \mapsto (\sigma\tau)^{-1}U$.

\section{The classifying site $\XQQ{X}{\Sigma}$}
\begin{definition}
	Let us denote by $A: X \to X$ the morphism of sites
	defined by 
	\begin{equation}\label{def:ActionSigma}
		A^{-1}(U)\coloneqq \coprod_{\sigma \in \Sigma}\sigma^{-1}U.	
	\end{equation}
	We say that a {\bf right action} of $\Sigma$ on $U \in ob(X)$ is a map 
	$$ u_{\bul}\coloneqq\coprod_{\sigma \in \Sigma}u_{\sigma}: 	A^{-1}(U) \longrightarrow U$$ 
	such that $u_{id_{\Sigma}}=id_{U}$ and $u_{\tau}(\tau^{-1}u_{\sigma})=u_{\sigma\tau} \quad \forall \sigma,\tau \in \Sigma$.
		\begin{equation*}
		\begin{tikzcd}[row sep=2.6pc, column sep=4.6pc]
			\tau^{-1}U \arrow[r, "u_{\tau}"] & U \\
			\tau^{-1}\left( \sigma^{-1}U \right)=(\sigma \tau)^{-1}U \arrow[u, "\tau^{-1}u_{\sigma}"] 
			\arrow[ur, "u_{\sigma \tau}"']
		\end{tikzcd} 
	\end{equation*}
\end{definition}

\begin{definition}\label{Sigmasite}
	Let $\left( \XQQ{X}{\Sigma}, J_{\Sigma} \right)$ be the following
	site. The category $\XQQ{X}{\Sigma}$ is defined as having
	\begin{itemize}
		\item $ob(\XQQ{X}{\Sigma})=
		\left\{(U,u_{\bul}): U \in ob(X), u_{\bul} \mbox{ is a right action of } \Sigma \mbox{ on } U\right\}$;
		\item ${\rm Hom}_{\XQQ{X}{\Sigma}}((U,u_{\bul}), (V,v_{\bul}))= \{ h:U \to V: 
		h u_{\sigma}=v_{\sigma}(\sigma^{-1}h ) \quad \forall \sigma \in \Sigma\}$;
		\end{itemize}	
	We say that a map $h:U \to V$ satisfying the condition: $	hu_{\sigma}=v_{\sigma}(\sigma^{-1}h ) \quad \forall \sigma \in \Sigma$,
	descends to a map $\bar{h}: (U,u_{\bul})\to (V,v_{\bul})$.
	On the other hand, we say that $h$ is the defining map of $\bar{h}$.\\
	The topology $J_{\Sigma}$ is the one induced by the target functor $$t_{\Sigma}:\XQQ{X}{\Sigma} \to X, \; (U,u_{\bul}) \mapsto U.$$
\end{definition}
\begin{definition}\label{def:Sigmatopos}
	
	Let  $\widetilde{X}_{\Sigma}$ be the following category:
	\begin{itemize}
		\item  $ob(\widetilde{X}_{\Sigma})$ are pairs $ (\mathcal{F}, \varphi_{\bul}) $ consisting of
		\begin{enumerate}[a)]
			\item A sheaf $ \mathcal{F}\in 
			ob(Sh(X)) $;
			\item A (right) action of $\Sigma$ on $\mathcal{F}$, \textit{i.e.}
			 a map of sheaves  
			 $$ \varphi_{\bul}\coloneqq 
		\coprod_{\sigma \in \Sigma}\varphi_{\sigma}: \coprod_{\sigma \in \Sigma}\sigma^{*}\mathcal{F} \to \mathcal{F}\in ar(Sh(X)),
			$$
			satisfying $\varphi_{id_{\Sigma}}=id_{\mathcal{F}}$ and
			 $\varphi_{\sigma \tau}=\varphi_{\tau}(\tau^*\varphi_{\sigma})$ $\forall \; \sigma,\tau \in \Sigma$.
		\end{enumerate}
		\item The arrows from  $  (\mathcal{F}, \varphi_{\bul})$ to $ (\mathcal{G}, \gamma_{\bul}) $  in $
		\widetilde{X}_{\Sigma}$ are natural transformations
		$$
		\theta \in {\rm Hom}(\mathcal{F},
		\mathcal{G})
		$$
		such that $\forall\; \sigma \in \Sigma$ the following diagram commutes
		\begin{equation}\label{arrows:Sigmatopos}
				\begin{tikzcd}[row sep=2.6pc, column sep=2.6pc]
					\mathcal{F} \arrow[r, "\varphi_{\sigma}"] 
					\arrow[d," \sigma^*\theta"'] &
					\mathcal{F} \arrow[d, "\theta"] \\
					\mathcal{G} \arrow[r, "\gamma_{\sigma}"] &
					\mathcal{G}
				\end{tikzcd} 
		\end{equation}
	\end{itemize}
\end{definition}

The main result of this section can be formulated as follows.
\begin{claim*}
	The category $Sh(\XQQ{X}{\Sigma})$ is equivalent to $\widetilde{X}_{\Sigma}$.
\end{claim*}

\begin{definition}
	Let $\widehat{X}_{\Sigma}$ denote the category of pre-sheaves 
	corresponding to $\widetilde{X}_{\Sigma}$, \ref{def:Sigmatopos}.
	Its objects (resp. its arrows) are pairs
	$(\mathcal{F}, \varphi_{\bul})$ as in \ref{def:Sigmatopos}, 
	with the only difference that now all the objects and arrows involved are taken in the category of
	pre-sheaves, cf. \ref{pre-toposXf}.
\end{definition}
The following results are analogous to those of $\XX{f}$, to wit:
\begin{lemma}\label{results:Sigmasite}
	The following results hold:
	\begin{enumerate}
		\item
		The morphism $A$ factorizes 
		through $\XQQ{X}{\Sigma}$, \textit{i.e.} there is a commutative diagram
		\begin{equation}\label{commutativesquareSigmaf}
			\begin{tikzcd}[row sep=2.6pc, column sep=2.6pc]
				X	\arrow[r, "\pi_{\Sigma}"] \arrow[d, "A"'] & \XQQ{X}{\Sigma} \arrow[d, "\widetilde{A}"] \\ 
				X \arrow[r, "\pi_{\Sigma}"] & \XQQ{X}{\Sigma}.
			\end{tikzcd} 
		\end{equation}
	We have  
	$$
	\widetilde{A}^{-1}(U,u_{\bul})\coloneqq\left( A^{-1}(U),\; \tilde{u}_{\bul} \right),
	$$
	where $\forall \;\sigma \in \Sigma$, $\tilde{u}_{\sigma}$ is the composition
	\begin{equation*}
			\begin{tikzcd}[row sep=2.6pc, column sep=2.6pc]
		\sigma^{-1}	A^{-1}(U) \arrow[r, "\sigma^{-1}u_{\bul}"] & 
		\sigma^{-1}U \arrow[r, hook] &	A^{-1}(U).
		\end{tikzcd} 
	\end{equation*}
		\item The functor $t_{\Sigma}$ admits a left adjoint $b_{\Sigma}: X  \to \XQQ{X}{\Sigma}$ given by
		\begin{equation}\label{leftadjointtargetSigmaf}
			b_{\Sigma}(U)\coloneqq \left(A^{-1}(U),\; j_{U, \bul} \right), \quad \forall U \in ob(X),
		\end{equation}
	where $\displaystyle j_{U, \bul}= \coprod_{\sigma \in \Sigma}j_{U,\sigma}$,
	and 
	$$
	\displaystyle j_{U,\sigma}: \sigma^{-1}A^{-1}U=
	\coprod_{\tau \in \Sigma}(\tau\sigma)^{-1}U \to \coprod_{\tau \in \Sigma}\tau^{-1}U 
	$$ 
	is induced by the identity via shifting, \textit{i.e.} sends $(\tau \sigma)^{-1}U$ in position 
	$\tau$ on the left to its copy on position $\tau \sigma$ on the right.
	For any $h: U' \to U \in X$, the map 
	$b_{\Sigma}(h): b_{\Sigma}(U') \to b_{\Sigma}(U)$  is
	the coproduct for $\sigma \in {\Sigma}$ of the maps
	\begin{equation*}
		\begin{tikzcd}[row sep=2.6pc, column sep=2.6pc]
			\sigma^{-1}U'  \arrow[r, "\sigma^{-1}h"] 
			& \sigma^{-1} U \arrow[r, hook] & A^{-1}(U).
		\end{tikzcd}
	\end{equation*}
Moreover, for any $U \in ob(X)$, the arrow $j_{U,{\bul}} \in X$ 
descends to an arrow
\begin{equation}\label{jbarSigma}
	\bar{j}_{U}: b_{\Sigma} (A^{-1}(U)) \to 	b_{\Sigma}(U) \in \XQQ{X}{\Sigma}. 
\end{equation}
		\item The co-unit of the above adjunction yields, for any 
		$(U,u_{\bul})\in ob(\XQQ{X{\Sigma}})$, a map 
		$$
		\epsilon_{\bul}=\epsilon_{\bul}(U,u_{\bul}):	b_{\Sigma}(t_{\Sigma}(U,u_{\bul})) \to (U,u_{\bul}),
		$$
		that fits into a natural coequalizer
		\begin{equation}\label{coeqobjectsSigmaf}
			\begin{tikzcd}[row sep=2.6pc, column sep=2.6pc]
				b_{\Sigma}(A^{-1}(U))
				\arrow[r, shift left, "{\bar{j}_{U}}"] 
				\arrow[r,shift right, "b_{\Sigma}(u_{\bul})"'] &
				b_{\Sigma}(U) \arrow[r, "\epsilon_{\bul}"] 
				& (U_{\bul}, u_{\bul}).
			\end{tikzcd} 
		\end{equation}
		\item 	The family of sieves 
		$$
		\{ (b_{\Sigma})_!(R): R \in J_{X}(U)\}
		$$
		is a family of covering sieves on $(U,u_{\bul})$. This collection forms a basis for $\XQQ{X}{\Sigma}$;
		\item The functor 
		$$
		\pi_{\Sigma}^{-1}: \XQQ{X}{\Sigma}^{\wedge} \to  X^{\wedge}, 
		$$
		coincides by adjunction with the functor $b_{\Sigma}^*$, which assigns to a 
		pre-sheaf $\mathcal{F}\in \XQQ{X}{\Sigma}^{\wedge}$  the pre-sheaf 
		$$
		\pi_{\Sigma}^{-1}\mathcal{F}:\; U \mapsto \mathcal{F}(b_{\Sigma}(U)).
		$$
		Moreover, $\pi_{\Sigma}^{-1}:\XQQ{X}{\Sigma}^{\wedge} \to X^{\wedge} $ factorizes through the forgetful functor
		$\widetilde{X}_{\Sigma} \to X^{\wedge}$. We denote by
			\begin{equation}\label{pre-actionSigmaf}
			\mbox{\boldmath	$\pi_{\Sigma}^{-1}$}: \XQQ{X}{\Sigma}^{\wedge} \to 
			\widetilde{X}_{\Sigma},\; 
			\mathcal{F} \mapsto(\pi_{\Sigma}^{-1}\mathcal{F}, \varphi_{\bul})
		\end{equation}	
	the resulting map. The right adjoint of $\varphi_{\bul}$
	is induced, for any $U \in ob(X)$,
	by pullback along the morphism $\bar{j}_{U}$, cf. \eqref{jbarSigma}.
		\item Applying the associated sheaf functor ``$a$'' to the construction above yields a functor
		\begin{equation}\label{fuctoreqSigmaf}
			\mbox{\boldmath	$\pi_{\Sigma}^{*}$}: Sh(\XQQ{X}{\Sigma}) \to 
			\widetilde{X}_{\Sigma},\; 
			\mathcal{F} \mapsto(\pi_{\Sigma}^*\mathcal{F}, \varphi_{\bul}).
		\end{equation}	
	\end{enumerate}
\end{lemma}
\begin{proof}\
	\begin{enumerate}
		\item The proof of the commutativity of 
		\eqref{commutativesquareSigmaf} is analogous to \eqref{commutativesquareXf};
		\item
		There is a natural functorial map in the variables $U,V$
		$$
		{\rm Hom}_{\XQQ{X}{\Sigma}}(b_{\Sigma}(U), (V,v_{\bul})) \longrightarrow {\rm Hom}_X(U, t_{\Sigma}(V,v_{\bul})),
		$$
		obtained by applying the target functor $t=t_{\Sigma}$.
		Let us fix now an arrow $j_0: U \to t_{\Sigma}(V,v_{\bul}) \in X$.
		The commutativity condition on a map 
		$ b_{\Sigma}(U)\to (V,v_{\bul}) \in \XQQ{X}{\Sigma}$
		is equivalent to asking that $\sigma^{-1}U \to V$
		is the arrow obtained by the following composition
		\begin{equation*}
		\begin{tikzcd}[row sep=2.6pc, column sep=2.6pc]
			\sigma^{-1}U  \arrow[r, "\sigma^{-1}j_0"] 
			& \sigma^{-1} V \arrow[r, "v_{\sigma}"] & V.
		\end{tikzcd}
	\end{equation*}
		Therefore, we obtain an inverse of $t$.
	The arrow $\bar{j}_U$ is well defined, 
		since the following diagram is commutative,
		\begin{equation*}
			\begin{tikzcd}[row sep=2.6pc, column sep=2.6pc]
				\displaystyle	\coprod_{\tau  \in \Sigma}\tau^{-1}
				\left( \coprod_{\sigma \in \Sigma}\sigma^{-1}(U) \right)
				\arrow[r, "j_{U,\bul}"] 
				& \displaystyle \coprod_{\sigma \in \Sigma}\sigma^{-1}(U)
					 	\\
				\displaystyle	\coprod_{\sigma,\tau,\upsilon \in \Sigma}\tau^{-1}
				\left(
				\coprod_{\upsilon \in \Sigma}\upsilon^{-1}
				\left( \coprod_{\sigma \in \Sigma}\sigma^{-1}(U)
				\right)
				\right)	 \arrow[r, "j_{A^{-1}(U),\bul}"] 	\arrow[u, "\coprod_{\tau}\tau^{-1} j_{U,\bul}"]&
				\displaystyle	\coprod_{\upsilon  \in \Sigma}\upsilon^{-1}
				\left( \coprod_{\sigma \in \Sigma}\sigma^{-1}(U) \right) 
				\arrow[u, "j_{U,\bul}"]
			\end{tikzcd}
		\end{equation*}
		
		\item We can check easily that the universal property of 
		coequalizers is satisfied, cf. \ref{resultsEpsteinsite}, \ref{basisobjectsXf};
	
		\item The morphism
		$$
		\coprod_{\sigma \in \Sigma} \sigma^{-1}R \to b_{\Sigma}^*(b_{\Sigma})_!R
		$$
		obtained by composing the covering 
		$	\coprod_{\sigma \in \Sigma} \sigma^{-1}R \to R$ with the unit morphism
		$R \to b_{\Sigma}^*(b_{\Sigma})_!R$ is a bicovering. 
		The proof is analogous to \ref{basissievesXf} and \ref{resultsEpsteinsite};
		\item	It follows already by 2) and \cite[I.5.5]{SGA4};
		\item  It is immediate from 5).
	\end{enumerate}
\end{proof}
\begin{cor}
	There is a  projection morphism, \textit{i.e.} a morphism of sites 
	\begin{equation}\label{PimapSigma}
		\pi_{\Sigma}: {X} 
		\rightarrow \XQQ{X}{\Sigma}.
	\end{equation}
induced by the functor $t_{\Sigma}$.
\end{cor}
\begin{proof}
The proof is a consequence of \ref{results:Sigmasite} and it is 
analogous to \ref{tgeomorph}.
\end{proof}
\begin{theorem}\label{sheavesonSigmasite}
	The functor $	\mbox{\boldmath	$\pi_{\Sigma}^{*}$}: Sh(\XQQ{X}{\Sigma}) \to 
	\widetilde{X}_{\Sigma}$ is an equivalence of categories.
\end{theorem}
\begin{proof}
 Let  $(\mathcal{G}, \gamma_{\bul}) \in \widetilde{X}_{\Sigma}$, 
 and consider the following diagram in $Sh(\XQQ{X}{\Sigma})$:
	\begin{equation}\label{inverseequivalenceSigma}
			\begin{tikzcd}[row sep=2.6pc, column sep=2.6pc]
		(	\pi_{\Sigma})_*\mathcal{G}	
			\arrow[r, shift left, "\alpha_0"] 
			\arrow[r,shift right," \alpha_1"'] &
			(	\pi_{\Sigma})_*\left( A_* \mathcal{G}\right),
		\end{tikzcd} 
	\end{equation}
where $\alpha_0$ is obtained by applying the functor
$(\pi_{\Sigma})_*$ to right adjoint of the defining morphism 
$\gamma_{\bul}: A^*\mathcal{G} \to \mathcal{G}$, while $\alpha_1$ is given by
$$
(\alpha_1)_{(U,u_{\bul})}=\mathcal{G}(u_{\bul}):
\mathcal{G}(U) \to \prod_{\sigma \in \Sigma} \sigma_*(\mathcal{G})(U).
$$

	The equalizer of \eqref{inverseequivalence} provides the required inverse, \textit{i.e.} we define
	\begin{equation*}
		(\mathcal{G}, \gamma_{\bul})^f\coloneqq\ker (\alpha_0, \alpha_1).
	\end{equation*}
	Note that the following is a functorial isomorphism of pre-sheaves
	$$
	\mbox{\boldmath	$	\pi_{\Sigma}^{-1}$}\left[ (\mathcal{G}, \gamma_{\bul}) 
	^f \right] \overset{\sim}{\longrightarrow} (\mathcal{G}, \gamma_{\bul}),
	$$
	which gives an isomorphism of the corresponding sheaves.
	In fact, a direct computation, cf. \ref{results:Sigmasite}, shows:
	\begin{equation*}
	\begin{tikzcd}[row sep=2.6pc, column sep=2.6pc]
		(\mathcal{G},\gamma_{\bul})^f(U,u_{\bul})= 	
		\ker  \Bigg(\displaystyle	\prod_{\sigma \in \Sigma} \mathcal{G}_{\sigma}(U)
		\arrow[r, shift left, " \alpha_0"] 
		\arrow[r,shift right,"\alpha_1"'] &
		\displaystyle	\prod_{\sigma, \tau \in \Sigma} \mathcal{G}_{\sigma\tau}(U) \Bigg),
	\end{tikzcd} 
\end{equation*}
where we have set $\mathcal{G}_{\sigma}\coloneqq
 \sigma_*\mathcal{G}, \; \forall\;\sigma \in \Sigma$ and we have used \ref{results:Sigmasite}, 5).
We have
$$
\left( \alpha_0 (x_{\sigma})_{\sigma}  \right)_{(\sigma,\tau)}=\gamma_{\tau}x_{\sigma}, \quad \quad 
\left( \alpha_1 (x_{\sigma})_{\sigma}  \right)_{(\sigma,\tau)}=x_{\sigma\tau}.
$$
Therefore, the equalizer is isomorphic to $\mathcal{G}$ via 
$
x \mapsto (\gamma_{\sigma}x)_{\sigma}.
$

	Finally, there is a functorial isomorphism
	$$
	\mathcal{F}  \overset{\sim}{\longrightarrow}   (	\mbox{\boldmath	$	\pi_{\Sigma}^{*}$}\mathcal{F})^f,
	$$
	for any $\mathcal{F}\in Sh(\XQQ{X}{\Sigma})$.
	In fact, a direct computation shows that
	$$
	\left( \mbox{\boldmath	$\pi_{\Sigma}^{*}$}\mathcal{F} \right)^f(U, u_{\bul}) 
	=\ker\left( \mathcal{F}(b_{\Sigma}(U)) \rightrightarrows \mathcal{F}(b_{\Sigma}(A^{-1}(U)))\right),
	$$
	where the parallel morphisms are given by the pull-back along the 
	maps in \eqref{results:Sigmasite}, 3).
	Since  the sieve generated by that family
 is a covering sieve, 
	the resulting equalizer is isomorphic 
	to $\mathcal{F}(U,u_{\bul})$, by the sheaf property.
\end{proof}

\begin{cor}\label{coveringsievesSigmafcor} 
	
	\noindent
	Let us consider, for any $(U,u_{\bul}) \in ob(\XQQ{X}{\Sigma})$, the following collection:
	$$
	\{ R: R \in J_{X}(U), \mbox{ such that } u_{\bul}  \mbox{ factorizes as }  A^{-1} R
	\rightarrow	R\}
	$$
	to which we refer as the family of $\Sigma$-dynamical sieves.
	Then, by \ref{sheavesonSigmasite}, we shall identify the covering sieves 
	$R_{u_{\bul}} \hookrightarrow (U,u_{\bul})$ with the $\Sigma$-dynamical sieve given by 
	$\pi_{\Sigma}^{-1}R_{u_{\bul}} $. In other words, 
	giving a covering sieve $i_{R_{\bul}}: R_{u_{\bul}} \hookrightarrow (U,u_{\bul})$ in
	$\XQQ{X}{\Sigma}$
	is equivalent to giving
	a covering sieve
	$i: R \hookrightarrow \underline{U}$, 
	with a morphism 
	$$
 u'_{\bul}:A^{-1} R
	\rightarrow	R,
	$$
	such that the the following diagrams commute
	\begin{equation*}
			\begin{tikzcd}[row sep=2.6pc, column sep=2.6pc]
				A^{-1}R	\arrow[r, "u'_{\bul}"] \arrow[d, hook ,  "A^{-1}i"'] &  R \arrow[d, hook,  "i"] \\ 
				A^{-1}\underline{U}  \arrow[ r, "u_{\bul}"] & \underline{U}
			\end{tikzcd} 
	\end{equation*}
\end{cor}
\begin{cor}
 Let $\XX{f}$ and ${\rm E}_f$ be the sites defined in \ref{defdynsite1} 
 and in  \ref{Epsdynsitedef}, respectively. Then, there are natural 
 morphisms of sites
 $$
 \XX{f} \overset{\sim}{\longrightarrow} \XQQ{X}{\N\,f}, 
 \quad \quad 
 {\rm E}_f \overset{\sim}{\longrightarrow} \XQQ{(X\times X)}{\N\,(\mathbbm{1}+f)},
 $$
 cf. \eqref{generatorepsteinsite} for the definition of $\mathbbm{1}+f$,
that induce an equivalence of sites.
\end{cor}
\begin{proof}
	In the proof of \ref{objectinitialXffact} we have seen that any $u: f^{-1} \to U$
	defines a unique $u^n: f^{-n} \to U$ such that $u^m(f^{-m}u_n)=u_{n+m}$.
	This implies easily that the categories $\XX{f}$ and $\XQQ{X}{\N\,f}$ are equivalent.
	Moreover, using \ref{equivalencelemmaXf} and \ref{sheavesonSigmasite}, 
	and applying the same argument to sheaves, we see that the functor induces an equivalence of sites.
	Finally, giving a pair of arrows $u_0:U_0 \to U_1$, $u_1:f^{-1}U_0 \to U_1$
	is equivalent to giving a unique arrow $u_{\bul}: (\mathbbm{1}+f)^{-1}U \to U$. 
	It follows easily that the categories ${\rm E}_f$ and 
	$\XQQ{(X\times X)}{\N\,(\mathbbm{1}+f)}$ are equivalent. Finally,
	applying \ref{sheavesonepsteinsite} and \ref{sheavesonSigmasite}
	we see that there is an induced equivalence between their respective sites.
\end{proof}
\begin{cor}\label{cor:topspace}
		Let $X_0$ be a topological space with a continuous
		 action of a monoid $\Sigma$.  
		 Let $\mathcal{X}$ denote the closure of the site $Ouv(X_0)$
		  of open sets of $X_0$
	 with respect to countable, cf.  \eqref{countability}, coproducts. 
		Then, Theorem \ref{intro:theorem2} holds for when we replace 
		$X$ by $X_0$. Therefore, giving a sheaf on
		 $\XQQ{\mathcal{X}}{\Sigma}$, is equivalent to giving a pair
		 $(\mathcal{F},\varphi_{\bul})$ such that $\mathcal{F}$
		 is a sheaf on $X_0$ and $\varphi_{\bul}$ is a morphism of sheaves 
		 on $X_0$.
\end{cor}
\begin{proof}
The category of sheaves on $X_0$ is equivalent to the category
of sheaves on $X$ since for any sheaf $\mathcal{F} \in Sh(X_0)$, we have
$$
\mathcal{F}\left( \coprod_{\alpha \in A}U_{\alpha} \right)= \prod_{\alpha \in A}\mathcal{F}(U_{\alpha}).
$$
\end{proof}
The above result justifies the notation $\XQQ{X}{\Sigma}$ for $X$
a topological space.
\begin{cor}
	Let $X$ be a topological space and let $\Sigma=G$ be a group.
	Then, the category of sheaves on the  
	site $\XQQ{X}{\Sigma}=\XQQ{X}{G}$, defined in \ref{Sigmasite}, is equivalent to the category of $G$-equivariant 
	sheaves on $X$.
\end{cor}
\begin{proof}
	If every $\sigma \in \Sigma$ is invertible, so is its image through $\Phi$,
	\textit{i.e.} $\sigma \in Aut(X)$ in our notation. 
	Therefore, to the group action on 
	the site $X$ corresponds a group action on its Tòpos: let 
	$(\mathcal{F}, \varphi_{\bul} )$ be as in \ref{def:Sigmatopos},
	and let us show that the right action $\varphi_{\bul}$ is an action.
	In fact, it is evident from the definition that $$\psi_{\sigma}\coloneqq\sigma^*(\varphi_{\sigma^{-1}}),$$ 
	is a right inverse of $\varphi_{\sigma}$. Applying 
	$(\sigma^{-1})^{*}$ to the resulting diagram, cf. Theorem \ref{intro:theorem2}, 
	and using that $(\sigma^{-1})^{*}\sigma^*=(\sigma \sigma^{-1})^*=id$, 
	we get that $\psi_{\sigma^{-1}}=(\sigma^{-1})^*\varphi_{\sigma}$ is a left inverse of 
	$\varphi_{\sigma^{-1}}$.
	To conclude the proof, we use the usual strategy, cf. \cite[ 6.1.2.b)]{MR0498552}: 
	we can choose for each $\sigma \in \Sigma$ an isomorphism
	 $\eta_{\sigma}: \mathcal{F} \overset{\sim}{\longrightarrow} \sigma^*\mathcal{F}$ such that 
	 the resulting morphisms 
	 $$
	 \mathcal{A}_{\sigma}\coloneqq\varphi_{\sigma}\,\eta_{\sigma}:  \mathcal{F} \longrightarrow \mathcal{F},
	 $$
	 satisfy the (right) group action axiom, \textit{i.e.} $\mathcal{A}_{\tau}\mathcal{A}_{\sigma}=\mathcal{A}_{\sigma \tau}$.
\end{proof}
\chapter{Dynamical Ext functors}
\section{Dynamical Hom functors} 
Recall that in view of Lemma \ref{sheavesonepsteinsite} we can identify a sheaf on 
${{\rm E}_f}$ with a pair $(\mathcal{F}_{\bul}, \varphi_{\bul})$. 
Let us simplify notation by writing $\mathcal{F}_{\bul}$ instead of
 $(\mathcal{F}_{\bul}, \varphi_{\bul})$:
the action $\varphi_{\bul}$ is still to be considered part of the data but it
will be implicit when there is no room for confusion. For example, we write 
$\gamma_{\bul}$ for the action on $\mathcal{G}_{\bul}$, $\varepsilon_{\bul}$ 
for the action on $\mathcal{E}_{\bul}$, etc.\\
Let $\mathcal{F}_{\bul},\mathcal{G}_{\bul}$ be two sheaves on ${{\rm E}_f}$ 
and note that, by definition, the set of morphisms in $Sh({{\rm E}_f})$
$$
{\rm\mathbb{H}om}(\mathcal{F}_{\bul}, \mathcal{G}_{\bul})=
\{ \theta_{\bul}: \mathcal{F}_{\bul} \to \mathcal{G}_{\bul} \in Sh({{\rm E}_f}) \}
$$
can be described as the equalizer of the following diagram
\begin{equation}\label{HomequalizerEf}
	\begin{tikzcd}[row sep=1.6pc, column sep=1.8pc]
		{\rm Hom}(\mathcal{F}_0,\mathcal{G}_0)\times {\rm Hom}(\mathcal{F}_1,\mathcal{G}_1) 
		\arrow[r, shift left, "s"] \arrow[r, shift right, "t"'] & 
		{\rm Hom}(\mathcal{F}_0,\mathcal{G}_1)\times {\rm Hom}(f^*\mathcal{F}_0,\mathcal{G}_1).
	\end{tikzcd}
\end{equation}
where the two arrows assign to $(\theta_0, \theta_1)$
the two maps obtained in diagram \eqref{arrowsEf}, \textit{i.e.}
\begin{equation}\label{HomequalizerEfmaps}
\begin{aligned}
&s(\theta_0, \theta_1)= (\theta_1\circ \varphi_0,\theta_1\circ \varphi_1 ) \\
&t(\theta_0, \theta_1)=(\gamma_0\circ \theta_0, \gamma_1\circ f^*\theta_0)
\end{aligned}	
\end{equation}
It is evident from the definition that the functor 
$$
{\rm\mathbb{H}om}(\mathcal{F}_{\bul}, -) : Sh({{\rm E}_f}) \rightarrow {\rm Set}
$$
is a left exact functor.\\
Recall from Corollary \ref{F0=F1} that if $\mathcal{F},\mathcal{G}$ are sheaves on 
$\XX{f}$ with actions given by $\varphi,\gamma$, then we have 
\begin{equation}\label{HomequalizerX/f}
	\begin{tikzcd}[row sep=1.6pc, column sep=1.8pc]
	{\rm\mathbb{H}om}(p^*\mathcal{F},p^* \mathcal{G})= 
	\ker\Big(  	{\rm Hom}(\mathcal{F},\mathcal{G})
		\arrow[r, shift left, "s"] \arrow[r, shift right, "t"'] & 
	{\rm Hom}(f^*\mathcal{F},\mathcal{G}) \Big).
	\end{tikzcd}
\end{equation}
where the maps here are simplified: 
$s(\theta)=\theta\circ \varphi$, $t(\theta)=\gamma\circ f^*\theta$.\\
It is natural to ask whether it is possible to find a sub-functor of 
$	{\rm\mathbb{H}om}(\mathcal{F}_{\bul}, -)$  which 
can be ``controlled'' by the $0$-th piece of the sheaves involved, as in \eqref{HomequalizerX/f},
and keeping an adequate level of generality as in \eqref{HomequalizerEf}.
The good choice in most of our applications will be the following: we fix $\mathcal{F} \in Sh(\XX{f})$ and consider
$$
{\rm\mathbb{H}om}(p^*\mathcal{F}, -): Sh({{\rm E}_f}) \to {\bf Set}.
$$
 Note that the above functor maps a sheaf $\mathcal{G}_{\bul}$ on ${{\rm E}_f}$ 
 to the equalizer of the following diagram
   \begin{equation}\label{Homequalizerfinal}
 	\begin{tikzcd}[row sep=1.6pc, column sep=1.9pc]
 	 {\rm Hom}(\mathcal{F},\mathcal{G}_0) \arrow[r, shift left, "s"] \arrow[r, shift right, "t"']
 		&  {\rm Hom}(f^*\mathcal{F},\mathcal{G}_1).
 	\end{tikzcd}
 \end{equation}
where the maps $s,t$ behave exactly as in \eqref{HomequalizerEfmaps}, 
keeping in mind that in this case $\theta_1$ is determined by $\theta_0$:
\begin{equation}\label{mapsintheapplication}
	s(\theta)=\gamma_0 \theta \varphi; \quad t(\theta)= \gamma\circ f^*\theta.
\end{equation}
In fact, we have $\theta_1=\gamma_0 \theta_0$, as explained by the following diagram
 \begin{equation}\label{Homdiagramfinal}
	\begin{tikzcd}[row sep=2.6pc, column sep=2.6pc]
		f^*(p^*\mathcal{F})_0 \arrow[r, "\varphi"]
		\arrow[d," f^*\theta_0"'] &
	(p^*\mathcal{F})_1 \arrow[d, "\theta_1"] \arrow[r,equal] & (p^*\mathcal{F})_0 \arrow[d,  "\theta_0"]\\
	f^*\mathcal{G}_0 \arrow[r, "\gamma_1"] &
		\mathcal{G}_1 & \mathcal{G}_0 \arrow[l, "\gamma_0"']
	\end{tikzcd}
\end{equation}
\section{Dynamical extensions}
We restrict our attention to the category of sheaves on ${{\rm E}_f}$ with values in a fixed abelian category 
$\mathfrak{C}$ (= Groups, R-modules, etc.), denoted by $Ab({{\rm E}_f})$. 
The aim of this section is to study the right derived
functors $R^i{\rm\mathbb{H}om}(\mathcal{F}_{\bul},-)$ which are denoted as usual 
by ${\rm\mathbb{E}xt^i}(\mathcal{F}_{\bul},-)$, $i\geq 1$. 
Their existence
is guaranteed from the existence of injective resolutions,
which generalizes the Godement resolution for sheaves on topological spaces, cf. \cite{enoughinjective}:
\begin{fact}\label{enoughinj}
	 The category of 
	abelian sheaves on a site has enough injectives. 
\end{fact}
First, we establish some useful notation.
\begin{notation}
	We use the index $n=0,1$ to enumerate the $\mathcal{F}_{n}$-piece of a sheaf
	$\mathcal{F}_{\bul} \in Sh({{\rm E}_f})$.
	On the other hand, we use the index $g=0,1$ to enumerate its actions
	$\varphi_g$. Moreover, we abuse notation by writing $``g"$ also for the corresponding
	maps $f^g$ appearing in the domain of $\varphi_g$, \textit{i.e.} $(f^g)^*\mathcal{F}_{0}$. 
	In this way, the index $g=0$ corresponds to map $0\coloneqq id_X$, while the index $g=1$ corresponds to the map
	$1\coloneqq f$. 
\end{notation}

Recall that in an abelian category it is always possible to take 
the point of view of extensions to study the derived functors ${\rm\mathbb{E}xt^i}(\mathcal{F}_{\bul}, -)$, cf. Appendix A.
\begin{factdefinition}
	The set ${\rm\mathbb{E}xt^i}(\mathcal{F}_{\bul}, \mathcal{G}_{\bul})$
	is the set of isomorphism classes
	of $i$-extensions in $Ab({{\rm E}_f})$, 
	\textit{i.e} exact sequences in $Ab({{\rm E}_f})$ of the form
	\begin{equation}\label{exti}
		\begin{tikzcd}[row sep=1.6pc, column sep=1.4pc]
			\xi_{\bul}: & 0 \arrow[r] &  \mathcal{G}_{\bul} \arrow[r, "e_{\bul}^i"] &
			\mathcal{E}^i_{\bul} \arrow[r,  "e_{\bul}^{i-1}"] & \dots \arrow[r, "e_{\bul}^1" ]&
			\mathcal{E}^1_{\bul} \arrow[r,  "e_{\bul}^0"]& \mathcal{F}_{\bul} \arrow[r] &  0
		\end{tikzcd}
	\end{equation}
	Observe that an $i$-extension in $Ab({{\rm E}_f})$ consists of the following data:
	\begin{enumerate}
		\item For $n =0,1$, an $i$-extension
		$\xi_{n}$ of $\mathcal{F}_n$ by $\mathcal{G}_n$,
		\textit{i.e} an exact sequence in $Ab({X})$ of the form
		\begin{equation}\label{locexti}
			\begin{tikzcd}[row sep=1.6pc, column sep=1.4pc]
				\xi_{n}: & 0 \arrow[r] &  \mathcal{G}_{n} \arrow[r, "e_{n}^i"] &
				\mathcal{E}^i_{n} \arrow[r,  "e_{n}^{i-1}"] & \dots \arrow[r, "e_{n}^1" ]&
				\mathcal{E}^1_{n} \arrow[r,  "e_{n}^0"]& \mathcal{F}_{n} \arrow[r] &  0,
			\end{tikzcd}
		\end{equation}
		\item For $g=0,1$, a chain map
		\begin{equation}\label{chainmap}
			\epsilon_g^{\bul}: g^*\xi_0 \to \xi_1,
		\end{equation}
		which agrees with the actions of $\mathcal{F}_{\bul}, \mathcal{G}_{\bul}$
		on the sides, such that each resulting square commutes.
	\end{enumerate}
	Two $i$-extensions $\xi_{\bul}, \xi'_{\bul}$
	are said to be equivalent if there are chain maps
	$\xi_n \to \xi'_n$, for $n=0,1$, which consist of the identity at
	$\mathcal{F}_n, \mathcal{G}_n$, such that the resulting cubes obtained 
	applying the chain maps \eqref{chainmap} commute.\\
	Moreover, there exists a structure of abelian group on the set of equivalence
	classes of $i$-extensions on $Ab({{\rm E}_f})$ given by their Baer sum.
\end{factdefinition}
\begin{remark}
	Recall that $f^*$ is an exact functor, hence \eqref{chainmap}
	makes sense as a map of extensions, cf. \cite{inverseimage}.
\end{remark}
\begin{example}\label{1-ext}
	{\bf [The group $\mathbb{E}{\rm xt}^1(\mathcal{F}_{\bul}, \mathcal{G} _{\bul})$]}\\
	Given two $1$-extensions
	\begin{equation*}
		\begin{tikzcd}[row sep=1.6pc, column sep=1.4pc]
			\xi_{\bul}: & 0 \arrow[r] &  \mathcal{G}_{\bul} \arrow[r, "e_{\bul}^1"] &
			\mathcal{E}_{\bul}  \arrow[r,  "e_{\bul}^0"]& \mathcal{F}_{\bul} \arrow[r]  &  0\\
			\xi'_{\bul}: & 0 \arrow[r] &  \mathcal{G}_{\bul} \arrow[r, "(e'_{\bul})^1"] &
			\mathcal{E}'_{\bul}  \arrow[r,  "(e'_{\bul})^0"]& \mathcal{F}_{\bul} \arrow[r] &  0
		\end{tikzcd}
	\end{equation*}
	their Baer sum is
	\begin{equation*}
		\begin{tikzcd}[row sep=1.6pc, column sep=1.4pc]
		\xi''_{\bul}\coloneqq	\xi_{\bul}+\xi'_{\bul}: & 0 \arrow[r] &  \mathcal{G}_{\bul} \arrow[r, "(e''_{\bul})^1"] &
			\mathcal{S}_{\bul}  \arrow[r,  "(e''_{\bul})^0"]& \mathcal{F}_{\bul} \arrow[r]  &  0
		\end{tikzcd}
	\end{equation*}
	where $\mathcal{S}_{\bul}$ is the sheaf
	locally given by the usual Baer sum, \textit{i.e.}
	$$
	\mathcal{S}_n=(\mathcal{E}_n\times_{\mathcal{F}_n} \mathcal{E}'_n)/im(e_{n}^1\times (-(e'_{n})^1))
	$$
	for $n=0,1$, and having as actions the ones constructed
	from those of $\mathcal{E}_{\bul}, \mathcal{E}'_{\bul}$,
	in the natural way by
	noting that the (local) Baer
	sum commutes with the  functor $f^*$.\\
	It can also be proved that the Baer sum
	 is commutative and associative. It gives
	$\mathbb{E}{\rm xt}^1(\mathcal{F}_{\bul}, \mathcal{G} _{\bul})$
	a group structure with trivial element given by
	the class of extensions which are ``globally split".
	Recall that a ``globally split" extension is an extension that
	is ``locally split" \textit{i.e} for $n=0,1$, there exist maps
	$s_n: \mathcal{F}_n \to \mathcal{E}_n$ such that
	$e_n^0\circ s_n=id_{\mathcal{F}_n}$, with the property that
	the local sections $s_n$ glue together into a map $s_{\bul}: \mathcal{F}_{\bul} \to \mathcal{E}_{\bul}$ of sheaves in $Ab({{\rm E}_f})$.
	Note that we obtain the inverse element of $\xi_{\bul}$
	by replacing $e_{\bul}^0$ with $-e_{\bul}^0$.
\end{example}

Observe that the maps $s,t$ appearing in \eqref{HomequalizerEfmaps} are 
functorial, hence they have can be derived to obtain maps $s^i,t^i$ for $i\geq 1$:
\begin{equation}\label{ExtisequenceEf}
	\begin{tikzcd}[row sep=1.6pc, column sep=2.8pc]
\displaystyle	\prod\limits_{n =0,1}	{\rm  Ext^{i}}(\mathcal{F}_{n}, \mathcal{G}_{n})
		\arrow[r, shift left, "s^i"] \arrow[r, shift right, "t^i"'] &
\displaystyle	\prod\limits_{g =0,1}	{\rm  Ext^i}(g^*\mathcal{F}_0, \mathcal{G}_1).
	\end{tikzcd}
\end{equation}
In the following we give a detailed description of these maps for $i=1$.
\begin{example}\label{ex:dermapsExt1}
	For $g=0,1$, one map takes the local extensions
	\begin{equation}\label{ext1_n}
		\begin{tikzcd}[row sep=1.6pc, column sep=2.6pc]
			0 \arrow[r] & \mathcal{G}_n \arrow[r, "e_n^1"] & \mathcal{E}_n \arrow[r, "e_n^0"] & \mathcal{F}_n \arrow[r] & 0
		\end{tikzcd}
	\end{equation}
	for $n=0,1$, and maps it to the short exact sequence
	constructed by pulling back the maps $e_1^0, \varphi_g$,
	\textit{i.e.} we set
	$\mathcal{X}_g\coloneqq\mathcal{E}_1 \times_{\mathcal{F}_1} g^*\mathcal{F}_0 $, thus
	obtaining the following diagram with exact rows:
	\begin{equation}\label{pullback}
		\begin{tikzcd}[row sep={4em,between origins}, column sep={5.3em,between origins}]
			0 \arrow[r] & \mathcal{G}_1 \arrow[r, "e_1^1"] & 
			\mathcal{E}_1 \arrow[r, "e_1^0"] \arrow[dr, phantom, ""{name=U, below , draw=black}]{} & 
			\mathcal{F}_1 \arrow[r] & 0 \\
			0 \arrow[r] & \mathcal{G}_1 \arrow[u, dash, shift left=.1em] \arrow[u,  dash, shift right=.1em]  \arrow[r] & 
			\mathcal{X}_g \arrow[u] \arrow[r] & g^*\mathcal{F}_0 \arrow[u, "\varphi_g"'] \arrow[r] & 0
		\end{tikzcd}
	\end{equation}
	On the other hand, the second map assigns to the  \eqref{ext1_n}
	the short exact sequence constructed by pushing out the maps
	$g^*e_0^1, \gamma_g$, \textit{i.e.} we set
	$\mathcal{Y}_g\coloneqq\mathcal{G}_1\oplus_{g^*\mathcal{G}_0}  g^*\mathcal{E}_0$,
	thus obtaining the following diagram with exact rows: 
	\begin{equation}\label{pushout}
		\begin{tikzcd}[row sep={4em,between origins}, column sep={5.3em,between origins}]
			0 \arrow[r] & \mathcal{G}_1 \arrow[r] \arrow[dr, phantom, ""{name=U, below , draw=black}]{} & 
			\mathcal{Y}_g \arrow[r] & g^*\mathcal{F}_0 \arrow[r] & 0\\
			0 \arrow[r] & g^*\mathcal{G}_0 \arrow[u, "\gamma_g"] \arrow[r, "g^*e_0^1"'] &  
			g^*\mathcal{E}_0  \arrow[u] \arrow[r, "g^*e_0^0"'] & 
			g^*\mathcal{F}_0 \arrow[u, dash, shift left=.1em] \arrow[u,  dash, shift right=.1em] \arrow[r] & 0
		\end{tikzcd}
	\end{equation}
\end{example}

\begin{notation}
Let us denote by $K^1$ the equalizer of \eqref{ExtisequenceEf} in the case $i=1$
and by $C^0$ the coequalizer of \eqref{HomequalizerEf}.
\end{notation}
We can now state the main Lemma of this section:
\begin{lemma}\label{ext1_lemma}
	There exists a canonical short exact sequence
	\begin{equation}\label{ext1sequence}
		\begin{tikzcd}[row sep=1.8pc, column sep=2.6pc]
			0 \arrow[r] & C^0 \arrow[r] & \mathbb{E}{\rm xt}^1(\mathcal{F}_{\bul}, \mathcal{G}_{\bul})
			\arrow[r] & K^1 \arrow[r] & 0
		\end{tikzcd}
	\end{equation}
\end{lemma}
\begin{proof}
	Let us consider the forgetful map
	$\displaystyle \mathbb{E}{\rm xt}^1(\mathcal{F}_{\bul}, \mathcal{G}_{\bul})
	\to \prod_{n=0,1}{\rm Ext}^1(\mathcal{F}_{n}, \mathcal{G}_{n})$,
	taking \eqref{exti} to the product of \eqref{locexti}.
	We are going to show that its image is exactly $K^1$.
	Note that, putting together \eqref{chainmap}, \eqref{pullback} and \eqref{pushout},
	we obtain the following diagram
	\begin{equation}\label{parallelepiped}
		\begin{tikzcd}[row sep={2.5em,between origins}, column sep={2.8em,between origins}]
			& &  & & & & & & 0 \arrow[dll] & & & 0 \arrow[dll] \\
			& & & & & & \mathcal{G}_1 \arrow[dd, dash, shift left=.1em] 
			\arrow[dd,  dash, shift right=.1em]  \arrow[dll] & & & 
			g^*\mathcal{G}_0 \arrow[lll, "\gamma_g"] \arrow[dll] \arrow[dd, "\gamma_g"]  \\
			& & & & \mathcal{E}_1  \arrow[dll] & & & 
			g^*\mathcal{E}_0 \arrow[lll, "\epsilon_g", crossing over] \arrow[dll, crossing over] & 0 \arrow[dll]& & & 
			0 \arrow[dll] \\
			& & \mathcal{F}_1 \arrow[dll]& & & 
			g^*\mathcal{F}_0 \arrow[lll, "\varphi_g"] &  
			\mathcal{G}_1 \arrow[dll] \arrow[from=rrr, dash, shift left=.1em] \arrow[from=rrr, dash, shift right=.1em] & & & 
			\mathcal{G}_1  \arrow[dll] \\
			0 &  & & & \mathcal{X}_g \arrow[uu] \arrow[dll] & & & 
			\mathcal{Y}_g \arrow[from=uu, crossing over]\arrow[lll, dashrightarrow] \arrow[dll]\\
			& & g^*\mathcal{F}_0 \arrow[uu, "\varphi_g"] 
			\arrow[dll] \arrow[from=rrr, dash, shift left=.1em] \arrow[from=rrr, dash, shift right=.1em] & & & 
			g^*\mathcal{F}_0  \arrow[uu, equal, crossing over]   \arrow[dll] \\
			0 &  & &  0
		\end{tikzcd}
	\end{equation}
	for $g=0,1$. In order to prove the claim,
	it is sufficient (\textit{e.g.} by the Five Lemma)
	to show that the map $\epsilon_g$ we start with
	(\textit{i.e.} a map such that the top face of the diagram commutes)
	induces an arrow $\mathcal{Y}_g\dashrightarrow \mathcal{X}_g$
	that makes the bottom face commutative.\\
	Given $\epsilon_g$ there is defined, 
	by the universal property of the pull-back,
	a canonical map
	$g^*\mathcal{E}_0 \to \mathcal{X}_g$,
	while there is a well defined map $\mathcal{G}_1 \to \mathcal{X}_g$
	given by $e_1^1\times 0$.
	Hence we have a canonical map
	$\mathcal{G}_n \oplus g^*\mathcal{E}_m \to \mathcal{X}_g$,
	which factorizes through $\mathcal{Y}_g$ by a classic diagram chase.\\
	Conversely, given the dashed arrow of diagram \eqref{parallelepiped},
	with the property that the bottom face commutes,
	we obtain by composition a map $\epsilon_g$
	which makes the top face commutative.
	Hence, the map
	\begin{equation}\label{rightside}
		{\rm \mathbb{E}xt^1}(\mathcal{F}_{\bul}, \mathcal{G}_{\bul}) 
		\longrightarrow K^1
	\end{equation}
	is surjective.\\
	We claim finally that $C^0$ is exactly the kernel of \eqref{rightside}.
	Observe that the last one consists of isomorphism classes of
	 extensions having the property of being ``locally split"
	(cf. Example \ref{1-ext}) and hence it is characterized as follows:
	starting with local sections
	$s_n: \mathcal{F}_n \to \mathcal{E}_n$ for $n=0,1$,
	we are going to compute,
	for $g=0,1$
	the isomorphism classes of action
	maps $\epsilon_g: g^*\mathcal{E}_0 \to \mathcal{E}_1$ that fit
	into a map of extensions \eqref{chainmap},
	modulo the subgroup given by those maps for which
	there exists a global section
	$s'_{\bul}: \mathcal{F}_{\bul} \to \mathcal{E}_{\bul}$.\\
	The situation is perhaps clarified by the following diagram:
	\begin{equation}\label{locallysplit}
		\begin{tikzcd}[row sep={4em,between origins}, column sep={5.3em,between origins}]
			0 \arrow[r] & \mathcal{G}_1 \arrow[r,"e_1^1"] & \mathcal{E}_1 \arrow[r, "e_1^0"] & 
			\mathcal{F}_1 \arrow[r]  \arrow[l, bend right=50, "s_1"'] & 0 \\
			0 \arrow[r] & g^*\mathcal{G}_0  \arrow[r, "g^*e_0^1"]  \arrow[u, "\gamma_g"] & 
			g^*\mathcal{E}_0  \arrow[r, "g^*e_0^0"]  \arrow[u, dashrightarrow, "\epsilon_g"] & 
			g^*\mathcal{F}_0  \arrow[r] \arrow[l, bend right=50, "g^*s_0"'] \arrow[u, "\varphi_g"'] & 0
		\end{tikzcd}
	\end{equation}
	First we note that giving $\epsilon_g$
	is equivalent to giving $\epsilon_g (g^*s_0)$,
	since the map $\epsilon_g  (g^*e_0^1)$ is already determined
	by the commutativity of the left square of \eqref{locallysplit}
	(recall that the local section $s_0$, together with $e_0^1$,
	gives an isomorphism of $\mathcal{E}_0$ with
	$\mathcal{G}_0\oplus\mathcal{F}_0$). \\
	Moreover, since the map $s_1\varphi_g$ is already part of our data,
	we have that, modulo equivalence,
	the kernel of \eqref{rightside} is classified by the cocycle
	\begin{equation}\label{cocycle}
		h_g\coloneqq\epsilon_g (g^*s_0)- s_1\varphi_g.
	\end{equation}
	Note that $h_g$ may be viewed as an element of
	$Hom(g^*\mathcal{F}_0, \mathcal{G}_1)$, since
	\begin{equation}\label{comp_1}
		\begin{aligned}
			e_1^0h_g\,
			&= (e_1^0\epsilon_g)(g^*s_0)- (e_1^0s_1) \varphi_g= \\
			& = (\varphi_g (g^*e_0^0))(g^*s_0)- \varphi_g = \\
			& =                    \varphi_g-\varphi_g=0.
		\end{aligned}
	\end{equation}
	In order to conclude, we need to prove
	that the two equivalence
	relations coincide, \textit{i.e.} the
	 extension defined by $(h_g)_{g=0,1}$
	is trivial
	if and only if 
	each $h_g$ is of the form
	\begin{equation}\label{coboundary}
		h_g= e_n^1(\gamma_g\theta_m-\theta_n\varphi_g)
	\end{equation}
	for some maps
	$
	\theta_n: \mathcal{F}_{n}\to \mathcal{G}_{n},\, \theta_m: \mathcal{F}_{m}\to \mathcal{G}_{m}
	$,
	where in \eqref{coboundary} we are simplifying the notation
	by confusing $\theta_m$ with its image through $g^*$.\\
	Suppose first that there exists a global section $s'_{\bul}$ of $e_{\bul}^0$.
	Observe that, since $e_1^0s_1=e_1^0s'_1=id_{\mathcal{F}_1}$,
	the map $s_1-s'_1$ defines an element
	$\theta_1 \in Hom(\mathcal{F}_1,\mathcal{G}_1)$ and
	let $\theta_0$ be analogously defined.
	Consequently, noting that by hypothesis
	$h'_g\coloneqq\epsilon_g s'_0-s'_1 \varphi_g=0$, we can write
	\begin{equation}\label{comp_2}
		\begin{aligned}
			e_1^1(\gamma_g\theta_0-\theta_1\varphi_g)&
			=\epsilon_g (e_0^1\theta_0)- (e_1^1\theta_1)\varphi_g=\\
			&               =\epsilon_g (s_0-s'_0)- (s_1-s'_1)\varphi_g=h_g-h'_g=h_g
		\end{aligned}
	\end{equation}
\end{proof}
\begin{cor}\label{exti_cor}
	Let $\mathcal{F}_{\bul}, \mathcal{G}_{\bul}$
	be abelian sheaves on ${{\rm E}_f}$. There exists a spectral sequence converging to
	$\mathbb{E}{\rm xt}^{\bul}(\mathcal{F}_{\bul}, \mathcal{G}_{\bul})$
	which degenerates at the second sheet.
	In particular, for any $i\geq 0$ there is a short exact sequence:
	\begin{equation}\label{extisequence}
		\begin{tikzcd}[row sep=1.8pc, column sep=2.6pc]
			0 \arrow[r] & C^i \arrow[r] &
			\mathbb{E}{\rm xt}^{i+1}(\mathcal{F}_{\bul}, \mathcal{G}_{\bul})
			\arrow[r] & K^{i+1} \arrow[r] & 0
		\end{tikzcd}
	\end{equation}
	where
	\begin{equation*}
	\begin{tikzcd}[row sep=1.6pc, column sep=1.8pc]
			C^i\coloneqq \displaystyle	{\rm coker}\, \Bigg(	{\rm \prod\limits_{n =0,1} Ext^{i}}(\mathcal{F}_{n}, \mathcal{G}_{n})
			\arrow[r, shift left, "s^i"] \arrow[r, shift right, "t^i"'] &
		\displaystyle	{\rm \prod\limits_{g =0,1} Ext^i}(g^*\mathcal{F}_0, \mathcal{G}_1) \Bigg) 
		\end{tikzcd}
	\end{equation*}
	and
	\begin{equation*}
		\begin{tikzcd}[row sep=1.6pc, column sep=1.6pc]
		K^{i}\coloneqq	\displaystyle\ker\Bigg(	{\rm \prod\limits_{n =0,1} Ext^{i}}(\mathcal{F}_{n}, \mathcal{G}_{n})
			\arrow[r, shift left, "s^i"] \arrow[r, shift right, "t^i"'] &
		\displaystyle	{\rm \prod\limits_{g =0,1} Ext^{i}}(g^*\mathcal{F}_0, \mathcal{G}_1) \Bigg) 
		\end{tikzcd}
	\end{equation*}
\end{cor}
\begin{proof}
	Using Lemma \ref{ext1_lemma} as the base of an induction,
	we suppose \eqref{extisequence} be true for $i>0$. \\
	Since the category $Ab({{\rm E}_f})$ has enough injectives,
	we have a short exact sequence as follows
	\begin{equation}\label{shift}
		\begin{tikzcd}[row sep=1.8pc, column sep=0.8pc]
			0 \arrow[r] & \mathcal{G}_{\bul} \arrow[r] & \mathcal{I}_{\bul}
			\arrow[r] & \mathcal{Q}_{\bul}\arrow[r] & 0
		\end{tikzcd}
	\end{equation}
	where $\mathcal{I}_{\bul}$ is an injective object in $Ab({{\rm E}_f})$ 
	and $\mathcal{Q}_{\bul}$ its quotient.\\
	Note that the long exact sequence associated to \eqref{shift}
	gives an isomorphism
	$\mathbb{E}{\rm xt}^j(\mathcal{F}_{\bul}, \mathcal{Q}_{\bul})
	\overset{\sim}{\longrightarrow}
	\mathbb{E}{\rm xt}^{j+1}(\mathcal{F}_{\bul}, \mathcal{G}_{\bul})$,
	since injectives are acyclic. Since each $\mathcal{I}_n$ is also injective
	we have the analogous isomorphism at each level. By induction,
	we obtain the sequence \eqref{extisequence} with $\mathcal{G}_{\bul}$
	replaced by $\mathcal{Q}_{\bul}$.
	Finally, the naturality of the maps \eqref{ExtisequenceEf}
	guarantees that they transform in the expected way, \textit{i.e.} 
 for $j=i+1$ the following commutes 
		\begin{equation*}
		\begin{tikzcd}[row sep=1.6pc, column sep=1.8pc]
	\displaystyle		{\rm \prod\limits_{n =0,1} Ext^{j}}(\mathcal{F}_{n}, \mathcal{Q}_{n})
			\arrow[r, shift left, "s^j"] \arrow[r, shift right, "t^j"'] 
			\arrow[d, "\sim" {anchor=south, rotate=90, inner sep=.5mm}, "\delta"]& \displaystyle
			{\rm \prod\limits_{g =0,1} Ext^j}(g^*\mathcal{F}_0, \mathcal{Q}_1)  
			\arrow[d, "\sim" {anchor=south, rotate=90, inner sep=.5mm}, "\delta"] \\
			\displaystyle	{\rm \prod\limits_{n =0,1} Ext^{j+1}}(\mathcal{F}_{n}, \mathcal{G}_{n})
			\arrow[r, shift left, "s^{j+1}"] \arrow[r, shift right, "t^{j+1}"'] & 
		\displaystyle	{\rm \prod\limits_{g =0,1} Ext^{j+1}}(g^*\mathcal{F}_0, \mathcal{G}_1) 
		\end{tikzcd}
	\end{equation*}

	so that we obtain \eqref{extisequence} with $i$ replaced by $i+1$.
\end{proof}

\begin{cor}\label{ext1_cor}
	Let $\mathcal{F}, \mathcal{G}$
	be abelian sheaves of $\XX{f}$. We define
	\begin{equation}\label{coker_S}
			\begin{tikzcd}[row sep=1.6pc, column sep=1.8pc]
		E^{1,0}	\coloneqq	{\rm coker}\, \Big( {\rm Hom}(\mathcal{F},\mathcal{G})
			\arrow[r, shift left, "s"] \arrow[r, shift right, "t"'] &
		 {\rm Hom}(f^*\mathcal{F},\mathcal{G})	\Big) 
		\end{tikzcd}
	\end{equation}
	and
	\begin{equation}\label{ker_S}
			\begin{tikzcd}[row sep=1.6pc, column sep=1.8pc]
			E^{0,1}	\coloneqq	{\rm ker}\, \Big( {\rm Ext}^1(\mathcal{F},\mathcal{G})
			\arrow[r, shift left, "s^1"] \arrow[r, shift right, "t^1"'] &
		{\rm Ext}^1(f^*\mathcal{F},\mathcal{G}) \Big).
		\end{tikzcd}
		\end{equation}
	Then, there is a short exact sequence 
	\begin{equation}\label{ext1}
		\begin{tikzcd}[row sep=2.6pc, column sep=1.6pc]
			0 \arrow[r] & E^{1,0} \arrow[r] & \mathbb{E}{\rm xt}^1(\mathcal{F}, \mathcal{G})
			\arrow[r] & E^{0,1} \arrow[r] & 0,
		\end{tikzcd}
	\end{equation}
   which can be viewed as the reduction of \eqref{ext1sequence} in the case 
   $\mathcal{F}_{0}=\mathcal{F}_{1}$ and $\mathcal{G}_{0}=\mathcal{G}_{1}$.\\
	Moreover, analogous reductions of \eqref{extisequence} hold
	with the appropriate changes for all $i\geq 0$.
\end{cor}
\begin{proof}
	Let us consider the sheaves $\mathcal{F}_{\bul}\coloneqq\rho^*\mathcal{F},\; 
	\mathcal{G}_{\bul}\coloneqq\rho^* \mathcal{G} \in Ab({{\rm E}_f})$.
	It is clear that any extension $\mathcal{E}_{\bul} \in
	\mathbb{E}{\rm xt}^1(\mathcal{F}_{\bul}, \mathcal{G}_{\bul})$
	is determined in degree 0, hence
	the natural projection $K^1 \to E^{0,1}$ is surjective.
	Composing this map with the canonical map \eqref{rightside}
	gives the surjective map on the right. With the same arguments
	as in the proof of Lemma \ref{ext1_lemma} we find that its kernel
	is exactly $E^{1,0}$.
\end{proof}
In the applications we shall need a mixed version of the above results, to wit:
\begin{cor}[{\bf Mixed version}]\label{ext1_corXf}
	Let $(\mathcal{F}, \varphi) \in Ab(\XX{f})$ and $(\mathcal{G}_{\bul}, \gamma_{\bul}) \in Ab({\rm E}_f)$. 
	Let us identify $(\mathcal{F}, \varphi)$ 
	with a sheaf $\mathcal{F}_{\bul}$ 
	on ${\rm E}_f$, cf. \ref{notationXfbul}.
	Then, there exists
	a spectral sequence $\{ E_r \}_r$ 
	converging to
	${\rm \mathbb{E}xt^{\bul}}(\mathcal{F}_{\bul}, \mathcal{G}_{\bul})$
	which degenerates at the second sheet.
	The latter, in turn, splits into short exact sequences 
	for each $n \geq 0$:
	\begin{equation*}
		\begin{tikzcd}[row sep=2.6pc, column sep=1.6pc]
			0 \arrow[r] & C^{n-1} \arrow[r] & 
			\mathbb{E}{\rm xt}^n(\mathcal{F}_{\bul}, \mathcal{G}_{\bul})
			\arrow[r] & K^n \arrow[r] & 0,
		\end{tikzcd}
	\end{equation*}
	where $C^{-1}\coloneqq 0$, and for each $n \geq 1$ we have set
	\begin{equation*}
		\begin{tikzcd}[row sep=1.6pc, column sep=1.8pc]
			C^n	\coloneqq	{\rm coker}\, \Big( {\rm Ext}^n(\mathcal{F},\mathcal{G}_0)
			\arrow[r, "d^{0,n}"] &
			{\rm Ext}^n(f^*\mathcal{F},\mathcal{G}_1)	\Big) 
		\end{tikzcd}
	\end{equation*}
	and
	\begin{equation*}
		\begin{tikzcd}[row sep=1.6pc, column sep=1.8pc]
			K^n	\coloneqq	{\rm ker}\, \Big( {\rm Ext}^n(\mathcal{F},\mathcal{G}_0)
			\arrow[r, "d^{0,n}"] &
			{\rm Ext}^n(f^*\mathcal{F},\mathcal{G}_1) \Big),
		\end{tikzcd}
	\end{equation*} 
	where the maps $d_1^{0,q}$ are the maps derived from \eqref{Homequalizerfinal}. 
\end{cor}
The figure below explains better the structure of the spectral sequence:
\begin{center}
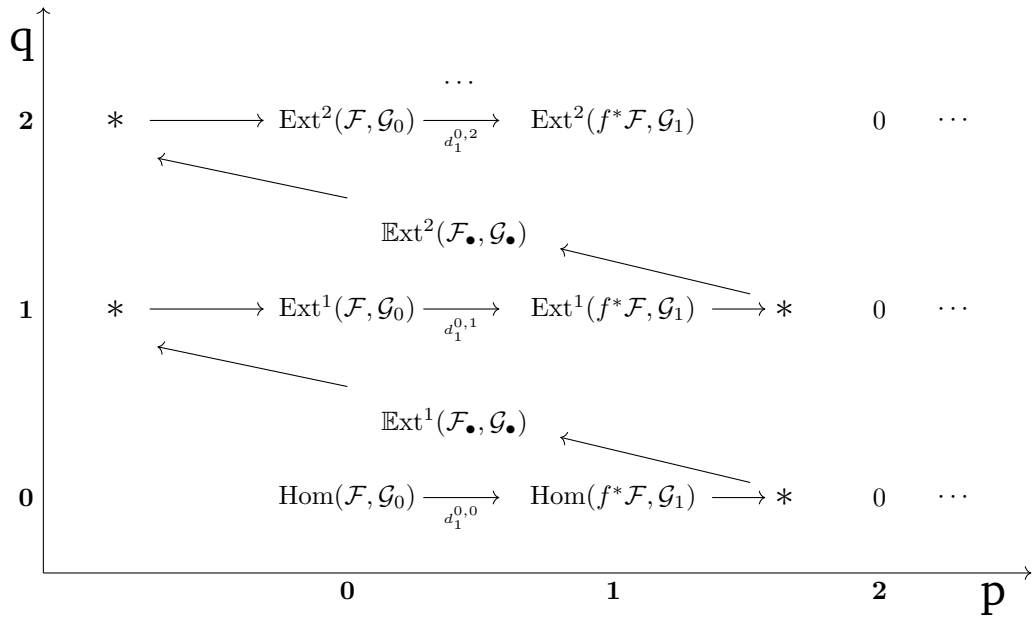
\begin{figure}[H]
	\begin{tikzpicture}
		\pgfplotsset{ticks=none}
		\draw[->] (-4,-1) -- (8.5,-1) node[below] {\LARGE{p}} -- (9,-1);
		\draw (0.0,-1) node[below] {${\bf 0}$};	
		\draw (3.5,-1) node[below] {${\bf 1}$};
		\draw (7,-1) node[below] {${\bf 2}$};
		\draw (-4,0) node[left] {${\bf 0}$};
		\draw (-4,2.5) node[left] {${\bf 1}$};
		\draw (-4,5) node[left] {${\bf 2}$};
		\draw (3.5,0) node[] {${\rm Hom}(f^*\mathcal{F}, \mathcal{G}_1)$};
		\draw[->] (-4,-1) -- (-4,6) node[left] {\LARGE{q}} -- (-4,6.5);
		\draw (0.0,0.0) node[] {${\rm Hom}(\mathcal{F}, \mathcal{G}_0)$};
		\draw (0,2.5) node[] {${\rm Ext^1}(\mathcal{F}, \mathcal{G}_0)$};
		\draw (3.5,2.5) node[] {${\rm Ext^1}(f^*\mathcal{F}, \mathcal{G}_1)$};
		\draw (0,5) node[] {${\rm Ext^2}(\mathcal{F}, \mathcal{G}_0)$};
		\draw (3.5,5) node[] {${\rm Ext^2}(f^*\mathcal{F},\mathcal{G}_1)$};	
		\draw (7,5) node[] {$0$};
		\draw (7,0) node[] {$0$};
		\draw (7,2.5) node[] {$0$};
		\draw[->] (1,0) -- (2,0);
		\draw[->] (1,2.5) -- (2,2.5);
		\draw[->] (1,5) -- (2,5);
		\draw[->] (4.8,0) -- (5.5,0);
		\draw[->] (4.8,2.5) -- (5.5,2.5);
		\draw (1.5,0) node[below] {\tiny$d_1^{0,0}$};
		\draw (1.5,2.5) node[below] {\tiny$d_1^{0,1}$};	
		\draw (1.5,5) node[below] {\tiny$d_1^{0,2}$};
		\draw (1.5,5.5) node[] {$\cdots$};
		\draw (8,5) node[] {$\cdots$};	
		\draw (8,2.5) node[] {$\cdots$};
		\draw (8,0) node[] {$\cdots$};				
		\draw (6,0) node[left] {\Large{${\bf \ast}$}};
		\draw (6,2.5) node[left] {\Large{${\bf \ast}$}};
		\draw[->] (5.3,0.2) -- (2.8, 0.8);
		\draw[->] (5.3,2.7) -- (2.8, 3.3);
		\draw (2.5,1) node[left] {${\mathbb{E}{\rm xt}^1}(\mathcal{F}_{\bul},\mathcal{G}_{\bul})$};
		\draw (2.5,3.5) node[left] {${\mathbb{E}{\rm xt}^2}(\mathcal{F}_{\bul},\mathcal{G}_{\bul})$};
		\draw[->] (0,1.472) -- (-2.5, 2);
		\draw[->] (0,3.972) -- (-2.5, 4.5);
		\draw (-2.8,5) node[left] {\Large{${\bf \ast}$}};
		\draw (-2.8,2.5) node[left] {\Large{${\bf \ast}$}};
		\draw[->] (-2.6,2.5) -- (-1.1, 2.5);
		\draw[->] (-2.6,5) -- (-1.1, 5);
	\end{tikzpicture}	
	\caption{$E_1$ of the spectral sequence in the ``mixed case''with inline computation of $E_2$.}\label{spectralE1}
\end{figure}
\end{center}
\subsection{More Homological Algebra}
Let $\mathcal{F}_{\bul}$, $\mathcal{G}_{\bul}$ be two abelian sheaves 
 ${\bf E}_f$ and consider an injective resolution 
$$
0 \to \mathcal{G}_{\bul} \to \mathcal{I}_{\bul}^{\bul}
$$
in the category $Ab({\bf E}_f)$, \ref{enoughinj}.
Consider the induced commutative diagram 

		\begin{equation}\label{specseq2}
			\begin{tikzcd}[row sep=1.6pc, column sep=0.8pc]
				0 \arrow[r] & {\rm\mathbb{H}om}(\mathcal{F}_{\bul}, \mathcal{I}_{\bul}^0) \arrow[d] \arrow[r] &
				{\rm Hom}(\mathcal{F}_0,\mathcal{I}_0^0)\times {\rm Hom}(\mathcal{F}_1,\mathcal{I}_1^0) \arrow[r, shift left] \arrow[d] \arrow[r, shift right]
				& {\rm Hom}(\mathcal{F}_0,\mathcal{I}_1^0)\times {\rm Hom}(f^*\mathcal{F}_0,\mathcal{I}_1^0) \arrow[d]  \\
				0 \arrow[r] & {\rm\mathbb{H}om}(\mathcal{F}_{\bul}, \mathcal{I}_{\bul}^1) \arrow[d] \arrow[r] &
				{\rm Hom}(\mathcal{F}_0,\mathcal{I}_0^1)\times {\rm Hom}(\mathcal{F}_1,\mathcal{I}_1^1) \arrow[r, shift left] \arrow[d] \arrow[r, shift right]
				& {\rm Hom}(\mathcal{F}_0,\mathcal{I}_1^1)\times {\rm Hom}(f^*\mathcal{F}_0,\mathcal{I}_1^1) \arrow[d]  \\
				0 \arrow[r] & {\rm\mathbb{H}om}(\mathcal{F}_{\bul}, \mathcal{I}_{\bul}^2) \arrow[d] \arrow[r] &
				{\rm Hom}(\mathcal{F}_0,\mathcal{I}_0^2)\times {\rm Hom}(\mathcal{F}_1,\mathcal{I}_1^2) \arrow[r, shift left] \arrow[d] \arrow[r, shift right]
				& {\rm Hom}(\mathcal{F}_0,\mathcal{I}_1^2)\times {\rm Hom}(f^*\mathcal{F}_0,\mathcal{I}_1^2) \arrow[d]\\
				& \dots & \dots & \dots
			\end{tikzcd}
		\end{equation}

\begin{fact}\label{fact:longexact}
	The rows in \eqref{specseq2} are exact. 
	Moreover, the spectral sequence in \ref{exti_cor} is just the long exact 
	sequence in cohomology associated to the short exact sequence of complex \eqref{specseq2}.
\end{fact}
\begin{proof}
	By definition, for any injective sheaf $\mathcal{I}_{\bul}$ on ${{\rm E}_f}$ 
	the groups ${\rm \mathbb{E}xt^i}(\mathcal{F}_{\bul}, \mathcal{I}_{\bul})$ vanish for $i \geq 1$. 
	Moreover, by \ref{exti_cor}, we see that in each row in \eqref{specseq2} 
	the co-kernels of the 
	maps on the right vanish, since they compute the ``subgroup"
	piece of ${\rm \mathbb{E}xt^1}(\mathcal{F}_{\bul}, \mathcal{I}^k_{\bul}), \; k \geq 0$. 
	The associated long exact sequence recovers the spectral sequence \ref{exti_cor}
	by a simple diagram chase.
\end{proof}

\begin{cor}
	Fix $\mathcal{F}_{\bul} \in Ab({\bf E}_f)$. 
	Then, for any short exact sequence
	$$
0 \longrightarrow \mathcal{A}_{\bul} \longrightarrow
  \mathcal{B}_{\bul} \longrightarrow \mathcal{C}_{\bul} \longrightarrow 0
	$$
	in $Ab({\bf E}_f)$, there is associated a long exact sequence
		\begin{equation}\label{longexactcohomologyseqEf}
		\begin{tikzcd}[row sep=1.2pc, column sep=1.8pc]
		0 \arrow[r]& {\rm\mathbb{H}om}(\mathcal{F}_{\bul}, \mathcal{A}_{\bul} )
		  \arrow[r] &  {\rm\mathbb{H}om}(\mathcal{F}_{\bul},  \mathcal{B}_{\bul} )
		  \arrow[r] \arrow[d, phantom, ""{coordinate, name=Z}]
			& {\rm\mathbb{H}om}(\mathcal{F}_{\bul}, \mathcal{C}_{\bul}  ) 
			\arrow[dll,
			rounded corners,
			to path={ -- ([xshift=2ex]\tikztostart.east)
				|- (Z) [near end]\tikztonodes
				-| ([xshift=-2ex]\tikztotarget.west)
				-- (\tikztotarget)}, "\delta"] \\
	&	{\rm \mathbb{E}xt^1}(\mathcal{F}_{\bul}, \mathcal{A}_{\bul}) 
		\arrow[r]
			& 	{\rm \mathbb{E}xt^1}(\mathcal{F}_{\bul}, \mathcal{B}_{\bul}) 
			\arrow[r]  \arrow[d, phantom, ""{coordinate, name=Z}]
			& 	{\rm \mathbb{E}xt^1}(\mathcal{F}_{\bul}, \mathcal{C}_{\bul})
				\arrow[dll,
			rounded corners,
			to path={ -- ([xshift=2ex]\tikztostart.east)
				|- (Z) [near end]\tikztonodes
				-| ([xshift=-2ex]\tikztotarget.west)
				-- (\tikztotarget)}, "\delta"] \\
			&	{\rm \mathbb{E}xt^2}(\mathcal{F}_{\bul}, \mathcal{A}_{\bul}) 
			\arrow[r]
			& 	{\rm \mathbb{E}xt^2}(\mathcal{F}_{\bul}, \mathcal{B}_{\bul}) \arrow[r]  
			& \cdots
		\end{tikzcd}
	\end{equation}
\end{cor}
\begin{proof}
 The proof is a straightforward exercise in Homological Algebra, when
 considering the 3-dimensional diagram whose slices are \eqref{specseq2}, 
 for some injective resolution of the sheaves 
 $\mathcal{A}_{\bul}$, $\mathcal{B}_{\bul}$, $\mathcal{C}_{\bul}$.
\end{proof}
\chapter{Topology of \XX{f}}
Let $X$ be a Galois site, cf. \cite[V.5]{sga1}. 
Let us denote by $F$ the fibre functor on $X$, and
consider the fundamental (pro-finite) group of 
$X$ is $\pi_1(X)={\rm Aut}(F)$.
Then, $F$ affords an equivalence of categories
\begin{equation}\label{fibrefunctor}
	F: X \overset{\sim}{\longrightarrow} {\rm F}Set(\pi_1(X)),
\end{equation}
where ${\rm F}Set(\pi_1(X))$ is the category of finite sets with (right)
action of $\pi_1(X)$. 
For any group $\Gamma$ let us denote by $\underline{\Gamma}$ the
constant sheaf on $X$ with values in $\Gamma$.
Recall the following definitions.
\begin{definition}\label{connectedsite}
$	X$ is a connected site if and only if 
$$
\mathbb{H}^0(X, \underline{\Z})=\Z.
$$
\end{definition}
We assume $X$ is a connected site with a fundamental (pro-finite) group in the sense of \cite[V.5]{sga1}.
The following characterization is well known, cf. \cite[III]{et}.
\begin{factdefinition}
	$	X$ is a simply connected site if and only if for any 
	finite group $\Gamma$ we have 
	$$
	\mathbb{H}^1(X, \underline{\Gamma})=0.
	$$
	Analogously, if $X$ is a topological champ that is locally connected and locally simply connected, cf. \cite[III.i]{et}, $X$ is simply connected if and only if for any 
	discrete group $\Gamma$ we have 
	$$
	\mathbb{H}^1(X, \underline{\Gamma})=0.
	$$
\end{factdefinition}

It follows from the isomorphism \ref{fibrefunctor}
that there is an equivalence between 
the respective group objects, 
\textit{i.e.} a functorial isomorphism, cf. \cite[III]{et}
$$
H^1(X, \underline{\Gamma})=H^1(\pi_1(X),\Gamma).
$$
When $\Gamma$ is the trivial $\pi_1(X)$- module, 
we have the following characterization
of the fundamental group:
\begin{equation}\label{torsors=maps}
	H^1(X, \underline{\Gamma})={\rm Hom}_{\bf Grp}(\pi_1(X),\Gamma).
\end{equation}

\section{${\bf \Gamma}$-torsors on \XX{f}}
Let us denote, by an abuse of notation, $\mathcal{F}$ a sheaf of groups on $X$.
\begin{definition}
	A (right) $\mathcal{F}$-torsor $\mathcal{E}$ on 
	$X $ is by definition any sheaf of sets on 
	$X$ with (right) $\mathcal{F}$-action 
	which is locally constant, \textit{i.e.} 
	locally isomorphic to the constant sheaf 
	$\mathcal{F}$ with right action given by translation. 
	Equivalently, it can be defined as a locally constant sheaf 
	$\mathcal{E}$ with a transitive action of $\mathcal{F}$ such that for any $U \in ob(X)$
	there exists a covering family $\{ U_i \to U \}$ such that $\mathcal{E}(U_i) \neq \emptyset$.
\end{definition}

In geometric terms, a $\underline{\Gamma}$-torsors arising from a group $\Gamma$ 
may be viewed
as a site $E$ with an action $T$ of $\Gamma$, together with a map
$$
\mathcal{E}:  E \longrightarrow X
$$
which is trivial on a basis of $X$:
\begin{equation*}
	\begin{tikzcd}[row sep=1.8pc, column sep=0.8pc]
		E \arrow[d] &	E_{U} \arrow[d]&\arrow[l, hook'] 
		 \arrow[d] \arrow[dl, phantom, "\square"]S \times \Gamma \\
		X & 	U & 	\arrow[l, hook'] S 
	\end{tikzcd}
\end{equation*}
If, as usual, the topology on $X$ were given by coverings, 
we could use \v{C}ech 1 co-cycles with values in $\Gamma$ to classify 
$\underline{\Gamma}$-torsors.
Recall the following.
\begin{fact}\
	\begin{itemize}
		\item The category whose object consists of 
		$\mathcal{F}$- torsors, and arrows are morphisms of $\mathcal{F}$-torsor (\textit{i.e.} $\mathcal{F}$-
		equivariant map) is a groupoid;
		\item Let $(X, \mathcal{O})$ be a ringed site, and 
		$\mathcal{F}$ a $\mathcal{O}$ module. Then there is an equivalence of the above groupoid 
		with the groupoid whose object are extensions 
		$$
		\{ 0 \to  \mathcal{F} \to E \to \mathcal{O} \to 0\},
		$$
		and arrows are morphisms of extensions. In particular, we have
		$$
		{\rm Ext}^1(\mathcal{O}, \mathcal{F})= \{ \mathcal{F} \mbox{-torsors} \}/\cong.
		$$
	\end{itemize}
\end{fact}
	\begin{proof}
	The first assertion follows from the fact that any morphism of $\mathcal{F}$- torsors is an isomorphism. 
	The second assertion is well known and we will provide a sketch
	of the proof, which will be found in \cite{torsor}. 
	Let $\mathcal{F},\mathcal{G}$ be two locally isomorphic 
	sheaves on $X$. Then, let
$\mathcal{I}\coloneqq\mathcal{I}so(\mathcal{F},\mathcal{G})$ be 
the canonical sheaf on $X$ given by
	$U \mapsto {\rm Iso}(\mathcal{F}|_U,\mathcal{G}|_U)$.
	Recall that $\mathcal{I}$ is a torsor for the right action of $\Gamma\coloneqq Aut(\mathcal{F})$
	given by $A_{\gamma}: \varphi \mapsto \varphi \gamma$. 
	Note that the quotient of $\mathcal{I}\times \mathcal{F}$
	by the natural right diagonal action, is in fact isomorphic to $Y$.
	In order to conclude, we apply the above
	 fact to the following situation: given an extension of 
	 $\mathcal{O}$-modules
	$$
	\{ 0 \to  \mathcal{F} \to E \to \mathcal{O} \to 0\},
	$$
	 we observe that locally there is no obstruction in lifting the 
	 map on the right, whence we have local section 
	 $s_i: \mathcal{O}|_{U_i} \to E|_{U_i}$, yielding a local splitting of $E$. Therefore, $E$
	 is locally isomorphic to $\mathcal{F}\otimes \mathcal{O}$.
	 Applying the above fact we get a $\mathcal{F}$-torsor 
	 $\mathcal{I}$,
, where $s \in \mathcal{F}$ acts on $\mathcal{F}\otimes \mathcal{O}$ by 
	  the diagonal matrix $\begin{pmatrix}
	  	s & 0 \\
	  	0  &  1
	  \end{pmatrix}$.
\end{proof}
As a consequence, we have
\begin{factdefinition}
	$$
	\mathbb{H}^1(\XX{f}, \underline{\Gamma}_{\bul})=
	\{ \mbox{isomorphism classes of  } \,\underline{\Gamma}_{\bul}-\mbox{torsors on } \XX{f}\}
	$$
\end{factdefinition}
By means of \ref{equivalencelemmaXf} we have the following description
of $\mathbb{H}^1(\XX{f}, \underline{\Gamma}_{\bul})$.
\begin{lemma}
	The set of $\underline{\Gamma}_{\bul}$-torsors on $\XX{f}$ is in bijective correspondence with the set of
	$\underline{\Gamma}$-torsors $\mathcal{E}$ on $	X$ with an action of $f$
	$$
	f^*\mathcal{E} \to \mathcal{E},
	$$
	which is locally a self-bijection of $\underline{\Gamma}$.
\end{lemma}
\begin{proof}
The sheaf $f^*\underline{\Gamma}$ is a constant sheaf with values in 
$\Gamma$ and hence coincides with $\underline{\Gamma}$.
We deduce that, \ref{equivalencelemmaXf}, each self map
$$
\phi: \Gamma \to \Gamma,
$$
defines a sheaf $\underline{\Gamma}_{\bul}^{\phi}$ on $\XX{f}$: it is
the sheaf $\underline{\Gamma}$ on $X$ with action given by $\phi$.

\end{proof}

\begin{example}
Translation by an element $x \in \Gamma$ of the group defines a 
$\underline{\Gamma}_{\bul}$- torsor on $\XX{f}$ which, as a locally constant sheaf,
is the sheaf $\underline{\Gamma}_{\bul}^{x}$ defined as follows.
	$$
	\{\underline{\Gamma}_{\bul}^{x}\coloneqq (\underline{\Gamma}, T_x) : x\in \Gamma\}
	$$
	where $T_x: \underline{\Gamma} \to \underline{\Gamma}$ denotes the translation by $x$ on $\Gamma$.\\
	This are, by construction,
	non-trivial $\underline{\Gamma}_{\bul}$-torsor on $\XX{f}$ 
	which pull back to the trivial $\underline{\Gamma}$-torsor on ${\bf 	X}$. 
We abuse notation, when there is no room
for confusion, by writing $\underline{\Gamma}_{\bul}$ for the sheaf on $\XX{f}$
given by the trivial action, \textit{i.e.} $\phi=id_{\Gamma}$, 
which we call the \textit{constant sheaf} on $\XX{f}$ with values in $\Gamma$.
\end{example}

	Let us consider the following group morphism
\begin{equation*}
\begin{tikzcd}[row sep=1.8pc, column sep=1.8pc]
	{\rm H}^1(X, \underline{\Gamma})	\arrow[r, shift left, "f^*"] \arrow[r, shift right, "id"'] & 
	{\rm H}^1(X, \underline{\Gamma}),
\end{tikzcd}
\end{equation*}
where $f^*$ is the morphism that assigns to any 
(isomorphism class of) $\Gamma$-torsor $\mathcal{E}$
the (isomorphism class of) $\Gamma$-torsor $f^*\mathcal{E}$.\\
The following is a consequence of Lemma \ref{intro:lemma1}.
\begin{lemma}\label{top:Gammatorsorlemma}
	\begin{enumerate}
		\item The map 
		$$T: \Gamma \longrightarrow 
		\mathbb{H}^1(\XX{f}, \underline{\Gamma}_{\bul}),\;
		x \in \Gamma \mapsto \underline{\Gamma}_{\bul}^{x}$$ 
		is injective;
		\item The image of $T$ is the kernel of the 
		restriction
			\begin{equation}\label{top:restrictiontorsor}
			\begin{tikzcd}[row sep=1.8pc, column sep=1.8pc]
				\mathbb{H}^1(\XX{f}, \underline{\Gamma}_{\bul}) 
				\arrow[r]  & 
				\ker(f^* -id) \subseteq  {\rm H}^1(X, \underline{\Gamma}).
			\end{tikzcd}
		\end{equation}
	    Therefore, $T$ is
		an isomorphism iff 
		there are no non-trivial $\underline{\Gamma}$-torsors $\mathcal{E}$
		on $X$ such that 
		$$
		f^*\mathcal{E} \cong \mathcal{E}.
		$$		
	\end{enumerate}
\end{lemma}
\begin{proof}
	The spectral sequence computing 
	$\mathbb{H}^{\ast}(\XX{f}, \underline{\Gamma}_{\bul})$ is the following:
\begin{center}
	\begin{figure}[H]
		\begin{tikzpicture}
			\pgfplotsset{ticks=none}
			\draw[->] (-4,-1) -- (8.5,-1) node[below] {\LARGE{p}} -- (9,-1);
			\draw (0.0,-1) node[below] {${\bf 0}$};	
			\draw (3.5,-1) node[below] {${\bf 1}$};
			\draw (-4,0) node[left] {${\bf 0}$};
			\draw (-4,2.5) node[left] {${\bf 1}$};
			\draw (3.5,0) node[] {${\rm H}^0(X, \underline{\Gamma})$};
			\draw[->] (-4,-1) -- (-4,6) node[left] {\LARGE{q}} -- (-4,6.5);
			\draw (0.0,0.0) node[] {${\rm H}^0(X, \underline{\Gamma})$};
			\draw (0,2.5) node[] {${\rm H}^1(X, \underline{\Gamma})$};
			\draw (3.5,2.5) node[] {${\rm H}^1(X, \underline{\Gamma})$};
			\draw (7,0) node[] {$0$};
			\draw (7,2.5) node[] {$0$};
			\draw[->] (1.25,-0.1) -- (2.25,-0.1);
			\draw[->] (1.25,0.1) -- (2.25,0.1);
	    	\draw[->] (1.25,2.4) -- (2.25,2.4);
	    	\draw[->] (1.25,2.6) -- (2.25,2.6);
			\draw[->] (4.8,0) -- (5.5,0);
			\draw (1.75,-0.1) node[below] {\tiny$id$};
				\draw (1.75,0) node[rotate=90]  {\small $=$};
			\draw (1.75,0.1) node[above] {\tiny$f^*$};
			\draw (1.75,2.4) node[below] {\tiny$id$};
			\draw (1.75,2.6) node[above] {\tiny$f^*$};	
			\draw (8,2.5) node[] {$\cdots$};
			\draw (8,0) node[] {$\cdots$};				
			\draw (6,0) node[left] {\Large{${\bf \ast}$}};
			\draw[->] (5.3,0.2) -- (2.8, 0.8);
			\draw (2.5,1) node[left] {$\mathbb{H}^1(\XX{f},\underline{\Gamma}_{\bul})$};
			\draw[->] (0,1.472) -- (-2.5, 2);
			\draw (-2.8,2.5) node[left] {\Large{${\bf \ast}$}};
			\draw[->] (-2.6,2.5) -- (-1.1, 2.5);
		\end{tikzpicture}	
	\caption{$E_1$ of the spectral sequence} \label{Gammatorsorfigure}
	\end{figure}
\end{center}
The coequalizer on the bottom row is exactly 
${\rm H}^0(X, \underline{\Gamma})=\Gamma$ 
which is thus injective in $\mathbb{H}^1(\XX{f},\underline{\Gamma}_{\bul})$. 
The co-kernel of this injection consists on 
$\underline{\Gamma}$-torsors $\mathcal{E}$ on $X$
which are in the kernel of the top parallel arrows, \textit{i.e} satisfy 
the condition $f^*\mathcal{E} \cong \mathcal{E}$. 
In particular, it is trivial when there are no non-trivial 
$\underline{\Gamma}$-torsors on $X$, \textit{e.g} when $X$ is simply connected.
\end{proof}
\begin{remark}
The proof of \ref{top:Gammatorsorlemma} can be simplified
and we do not really need \ref{ext1_lemma}. In fact, 
in any abelian category
the automorphism of extensions $Aut(0 \to B \overset{h}{\rightarrow} E \to A \overset{k}{\rightarrow} 0 )$
are isomorphic to ${\rm Hom}(A,B)$ via $f \mapsto f-id_E$. Note in fact that
$k(f-id_E)=0$ and also that $(f-id_E)h=0$.
Therefore, (isomorphism classes of) 
$\underline{\Gamma}$-torsors $\mathcal{E}$ 
that are isomorphic to $f^*\mathcal{E}$,
\textit{i.e.} the class of $\mathcal{E}$ lies in the kernel of \eqref{top:restrictiontorsor}, are
in bijective correspondence with the automorphism 
group of extension, cf. Appendix A, 
in 
${\rm H}^1(X, \underline{\Gamma})\cong {\rm Ext}^1(\mathcal{O}_X, \underline{\Gamma})$.
\end{remark}

\begin{cor}\label{cor:pi1}
	The fundamental group of $\XX{f}$ is 
	$$
	\pi_1(\XX{f})=\Z
	$$
	if $X$ 
	is simply connected, and it is an extension of $\Z$ by a quotient of 
	$\pi_1(X)$ otherwise.
\end{cor}
\begin{proof}
It follows from the characterization of fundamental group $\pi_1(\XX{f})$ 
\begin{equation}\label{pi0prodef}
	\mathbb{H}^1(\XX{f},\underline{\Gamma}_{\bul})\overset{\sim}{\rightarrow}
	{\rm Hom}_{ \Z}(\pi_1(\XX{f}), \Gamma).
\end{equation}
We have seen that if $X$ 
is simply connected, there is a bijection 
$$
\mathbb{H}^1(\XX{f},\underline{\Gamma}_{\bul})\cong {\rm Hom}_{\Z}(\Z, \Gamma)\cong \Gamma.
$$
There exists a canonical injection 
	\begin{equation*}
Q \hookrightarrow  \pi_1(\XX{f})
\end{equation*}
where $Q$ is a quotient of $\pi_1(X)$,
\textit{i.e.} the image of $ \pi_1(X)$ under 
the canonical map, cf. \eqref{projectionmapXf},
$$
\pi_*: \pi_1(X) \to  \pi_1(\XX{f}).
$$
The exact sequence 
	\begin{equation*}
	\begin{tikzcd}[row sep=1.8pc, column sep=0.8pc]
		0 \arrow[r] & \Gamma \arrow[r] & \mathbb{H}^1(\XX{f},\underline{\Gamma}_{\bul})
		\arrow[r] & K \arrow[r] & 0
	\end{tikzcd}
\end{equation*}
discussed above is obtained by applying ${\rm Hom}(-,\Gamma)$
to an exact sequence of the form
	\begin{equation*}
	\begin{tikzcd}[row sep=1.8pc, column sep=0.8pc]
		0 \arrow[r] & Q \arrow[r] & \pi_1(\XX{f})
		\arrow[r] &Q'\arrow[r] & 0
	\end{tikzcd}
\end{equation*}
where $Q'$ is the quotient of the above inclusion. We deduce that,
up to isomorphism, $Q'=\Z$.
\end{proof}
In geometric terms the isomorphism \eqref{pi0prodef} can be explained as
follows: up to
 isomorphism, a $\underline{\Gamma}_{\bul}$-torsor $\mathcal{E}$ is determined by
  $\pi^*\mathcal{E}= X\times_{\XX{f}}{ E}$, which is 
   a trivial $\underline{\Gamma}$-torsor on $X$ by assumption and thus it is isomorphic to $\underline{\Gamma}$. 
   The action
   on $\mathcal{E}$ pulls back to an action on $\pi^*\mathcal{E}$, and thus gives a map
   $$
   \phi: \underline{\Gamma} \to \underline{\Gamma},
   $$
   \textit{i.e.} an element $x$ of $\Gamma$ such that $\phi=T_x$.
    On the other hand, by functoriality to each homomorphism 
    $x: \Z \to \Gamma$ (\textit{i.e} an element $x \in \Gamma$), there is associated a homomorphism
    $$
   x_*: \mathbb{H}^1(\XX{f},\underline{\Z}_{\bul}) \to \mathbb{H}^1(\XX{f},\underline{\Gamma}_{\bul}).
    $$
     This map is determined by the image of the generator $\underline{\Z}_1$
     which is the sheaf $\underline{\Z}$ with action given by translation by $1$. The commutativity of
    	\begin{equation*}
    	\begin{tikzcd}[row sep=1.8pc, column sep=1.8pc]
    		\underline{\Z} \arrow[d, "T_1"] \arrow[r, "x"]& \arrow[d, "A"] \underline{\Gamma}  \\
    		\underline{\Z} \arrow[r, "x"]&  \underline{\Gamma}
    	\end{tikzcd}
    \end{equation*}
gives $A=T_x$, and therefore we have
 $x_*(\underline{\Z}_{\bul}^1)=\underline{\Gamma}_{\bul}^x$.\\
\begin{fact}
	If ${X}$ is connected, so is ${{\rm E}_f}$.
	Moreover, if ${X}$ is simply connected, then
$$
	\pi_1({{\rm E}_f})\cong \Z.
$$
\end{fact}
\begin{proof}
	We can define $\underline{\Z}_{\bul}$ as the constant sheaf on ${{\rm E}_f}$
	given by $(\underline{\Z},\underline{\Z})$ with trivial action map, \ref{sheavesonepsteinsite}, 
	which coincides clearly with the pull back of the constant sheaf $\underline{\Z}_{\bul}$ on 
	$\XX{f}$ discussed before. Global sections of $\underline{\Z}_{\bul} $, \eqref{HomequalizerEf},
	are isomorphic to $	\underline{\Z}$ through the diagonal map
		\begin{equation*}
		\begin{tikzcd}[row sep=1.8pc, column sep=1.8pc]
	0 \arrow[r] & 	\underline{\Z} \arrow[r, "\Delta"]	&
	\underline{\Z}\times \underline{\Z} \arrow[r, shift right, "id"']	\arrow[r, shift left, "id"]&
		 	\underline{\Z}\times \underline{\Z}.
		\end{tikzcd}
	\end{equation*}
Hence, the co-kernel above, which is by definition 
$E_2^{1,0}=E_{\infty}^{1,0}$ of the spectral sequence computing
$\mathbb{H}^p({{\rm E}_f}, \underline{\Z}_{\bul})$, 
is isomorphic to $\underline{\Z}$. We conclude as before 
by using the hypothesis that ${\bf 	X}$ is simply connected.
\end{proof}
\section{Complex line bundles over \XX{f}}
Let $X$ be a complex manifold (resp. a differentiable manifold) and let 
$\mathcal{O}_X$ denote its structure sheaf 
(resp. $\mathcal{A}_X$ the sheaf of infinitely differentiable complex valued functions). 
Let us fix a holomorphic 
(resp. infinitely differentiable) self map $f$ of $X$. 
Since sheaf cohomology on $X$ coincides with hypercohomology on $X$,
we will write simply $H^p(X, \mathcal{F})$ for the cohomology groups on $X$.
In this setting, we are able to give a description of the cohomology group
$\mathbb{H}^2(\XX{f}, \underline{\Z})$.
Recall that we can define the exponential map 
$$
\mbox{exp}: \mathcal{A}_X \to  \mathcal{A}_X^*,
$$
where $ \mathcal{A}_X^* \hookrightarrow   \mathcal{A}_X$ is the 
sub-sheaf of $\C^*$-valued functions. 
There is a canonical short exact sequence
	\begin{equation*}
	\begin{tikzcd}[row sep=1.8pc, column sep=1.8pc]
	0 \arrow[r]& \underline{\Z}(1) \arrow[r] &  \mathcal{A}_X \arrow[r, "exp" ]&  \mathcal{A}_X^* \arrow[r] & 0
	\end{tikzcd}
\end{equation*}
The associated long exact sequence reads
$$
\dots \to \mathbb\mathcal{{}H}^1(X, \mathcal{A}_X) 
\to \mathbb{H}^1(X, \mathcal{A}_X^*) \overset{\delta}{\rightarrow}
\mathbb{H}^2(X, \underline{\Z}(1)) \to \mathbb{H}^2(X, \mathcal{A}_X) \to \dots
$$
whence,
\begin{cor}\label{expexactseqX}
  Under our hypothesis $\mathcal{A}_X$-modules are acyclic,
 there exist partitions of unity, cf. \cite{BottTu}. Therefore, the group
    $\mathbb{H}^2(X, \underline{\Z}(1))$
    parametrizes complex line bundles on $X$ up 
    to isomorphism, \textit{i.e} the connecting homomorphism $\delta$
    is an isomorphism.
\end{cor}
A natural exponential exact sequence appears also in $Ab(\XX{f})$, since the natural map
$$
f^*: f^*\mathcal{A}_X \to \mathcal{A}_X,
$$
which defines the sheaf $\mathcal{A}_{\bul}$ on $\XX{f}$, restricts naturally to a map
$$
f^*: f^*\mathcal{A}_X^* \to \mathcal{A}_X^*,
$$
which makes the following diagram commutative
	\begin{equation*}
	\begin{tikzcd}[row sep=1.2pc, column sep=1.8pc]
		0  \arrow[r]& \underline{\Z}(1)  \arrow[r]&  \mathcal{A}_X  
		\arrow[r, "{\rm exp}"] & \mathcal{A}_X^*   \arrow[r]& 0 \\
			0  \arrow[r]& f^* \underline{\Z}(1) \arrow[r]  \arrow[u, equal]&  
			f^*\mathcal{A}_X  \arrow[r, "f^*{\rm exp}"]  \arrow[u, "f^*"] & 
			f^*\mathcal{A}_X^*   \arrow[r]  \arrow[u, "f^*"]& 0.
	\end{tikzcd}
\end{equation*}
Therefore, we get the induced short exact sequence in $Ab(\XX{f})$
	\begin{equation}\label{expseq:Xf}
	\begin{tikzcd}[row sep=1.8pc, column sep=1.8pc]
		0 \arrow[r]& \underline{\Z}(1)_{\bul} \arrow[r] &  
		\mathcal{A}_{\bul} \arrow[r, "{\rm exp}" ]&  \mathcal{A}_{\bul}^* \arrow[r] & 0,
	\end{tikzcd}
\end{equation}
whence the long exact sequence in cohomology 
	\begin{equation}\label{expexactseqXf}
	\begin{tikzcd}[row sep=1.2pc, column sep=1.8pc]
			\mathbb{H}^1(\XX{f}, \underline{\Z}(1)_{\bul})  \arrow[r] &
			 \mathbb{H}^1(\XX{f}, \mathcal{A}_{\bul}) \arrow[r] \arrow[d, phantom, ""{coordinate, name=Z}]
		& \mathbb{H}^1(\XX{f}, \mathcal{A}_{\bul}^*) \arrow[dll,
		rounded corners,
		to path={ -- ([xshift=2ex]\tikztostart.east)
			|- (Z) [near end]\tikztonodes
			-| ([xshift=-2ex]\tikztotarget.west)
			-- (\tikztotarget)}, "\delta"] \\
	\mathbb{H}^2(\XX{f}, \underline{\Z}(1)_{\bul}) \arrow[r]
		&\mathbb{H}^2(\XX{f}, \mathcal{A}_{\bul}).
	\end{tikzcd}
\end{equation}
By means of \ref{ext1_corXf}, we see that the group $\mathbb{H}^2(\XX{f}, \mathcal{A}_{\bul}) $ 
is trivial since $\mathcal{A}_X$ is acyclic, so the connecting homomorphism 
$\delta$ is still onto, but fails in general to be an isomorphism.
In fact, the obstruction given by $\mathbb{H}^1(\XX{f}, \mathcal{A}_{\bul}) $ 
is evident already for $X$ a simply connected compact K\"aler manifold, in which case
\begin{fact}\label{fact:cohomologylinebundles}
	If $X$ is a simply connected compact K\"aler manifold, then	
	\begin{equation*}
		\begin{tikzcd}[row sep=1.2pc, column sep=1.8pc]
			\mathbb{H}^1(\XX{f}, \mathcal{O}_{\bul}) 
			\cong E^{1,0}	=\textrm{coker}\big( {\rm H}^0(X, \mathcal{O}_X) 
			\arrow[r,shift left, "f^*"] \arrow[r,shift right, "id"']  & {\rm H}^0(X, f^*\mathcal{O}_X) \big) \cong \C.
		\end{tikzcd}
\end{equation*}
\end{fact}
\begin{proof}
	Under the above hypotheses, cf. \cite[7]{griffithsharris}, we have ${\rm H}^1(X, \mathcal{O}_X)=0$ and hence $E^{0,1}=0$.
\end{proof}
In any case, the group $	\mathbb{H}^1(\XX{f}, \mathcal{O}_{\bul}) $
has non trivial image into $\mathbb{H}^1(\XX{f}, \mathcal{O}_{\bul}^*) $
which, by the same reasoning, contains a copy of $\C^*$, whence we have
\begin{fact}\label{imageofconnectingXf}
There is a subgroup of 
$\mathbb{H}^1(\XX{f}, \mathcal{O}_{\bul}^*) $ 
contained in the image of
$\mathbb{H}^1(\XX{f}, \mathcal{O}_{\bul}) $, consisting of the
 group of holomorphic line bundles on $\XX{f}$ given by the
pair $(\mathcal{O}_X, \lambda)$, where the action is given by multiplication 
by a non-zero complex number $\lambda \in \C^*$,
$$
\lambda: f^*\mathcal{O}_X \to \mathcal{O}_X, \; f^*g \mapsto \lambda g.
$$
\end{fact}
\begin{remark}
	Note that the result of \ref{imageofconnectingXf} remains true if 
	we replace everywhere
	the sheaf $\mathcal{O}_{\bul}$ by $\mathcal{A}_{\bul}$.
\end{remark}
In the setting of \ref{fact:cohomologylinebundles}, we have a complete description of 
the exact sequence \eqref{expexactseqXf}
\begin{fact}\label{cor:linebundles}
	If $X$ is a simply connected compact K\"aler manifold, the sequence \eqref{expexactseqXf}
	reduces to
		\begin{equation*}
		\begin{tikzcd}[row sep=1.2pc, column sep=1.8pc]
			0 \arrow[r]& \Z(1) \arrow[r] &  \C \arrow[r ]& 
			\mathbb{H}^1(\XX{f}, \mathcal{O}_{\bul}^*)\arrow[r, "\delta"] &
				\mathbb{H}^2(\XX{f}, \underline{\Z}(1)_{\bul}) \arrow[r] & 0
		\end{tikzcd}
	\end{equation*}
where
\begin{equation*}
		\mathbb{H}^2(\XX{f}, \underline{\Z}(1)_{\bul})\cong E^{0,2}=
		\ker ( H^2(X,  \underline{\Z}(1)) \rightrightarrows H^2(X,  f^*\underline{\Z(1)})   )
	\end{equation*}
Moreover, through the isomorphism in \ref{expexactseqX}, we find
\begin{equation*}
	\begin{aligned}
		&	\mathbb{H}^2(\XX{f}, \underline{\Z}(1)_{\bul})	
			= \ker(  	\mathbb{H}^1(\XX{f}, \mathcal{O}_X^*) 
			\rightrightarrows 	\mathbb{H}^1(\XX{f}, f^*\mathcal{O}_X^*)) = \\
			& =\{ \mbox{holomorphic line bundles } E \mbox{ on }
			X  : f^*E \overset{\sim}{\rightarrow} E \}/\sim
	\end{aligned}
\end{equation*}
	where the equivalence relation is given, as usual, by isomorphisms that respect
	the action.
\end{fact}
\begin{example}
	For $X=\PP$, and $f$ a rational map of degree greater than $1$, we have 
$$
	\mathbb{H}^2(\QQ{f}, \underline{\Z}(1)_{\bul})=0,
$$
while, being compact cf. \ref{imageofconnectingXf}, 
 $$
 	\mathbb{H}^1(\QQ{f},\mathcal{O}_{\bul}^*)\cong \C^*.
 $$
\end{example}
\begin{proof}
	For a holomorphic line bundle $E$ on $\PP$ we have 
	$$
	\deg(f^*E)=\deg(f)\deg(E),
	$$
	and since $deg(f)>1$ we conclude that the map $f^*-id$ 
	on $	\mathbb{H}^1(\QQ{f},\mathcal{O}_{\bul}^*)$ is injective.
\end{proof}

\section{De Rham cohomology of \XX{f}}
Let $X$ be a differentiable manifold. Let us denote, abusing notation, 
by $\mathcal{A}_X=\mathcal{A}_X^0$ the sheaf of real valued infinitely differentiable functions on $X$.
 Recall that we can define a sheaf $\mathcal{A}_X^p$ on $X$
given by differential $p$-forms and if $X$ has dimension $n$,
we can consider the De-Rham complex of $X$,
$$
\mathcal{A}_X^0 \overset{d}{\longrightarrow} \mathcal{A}_X^1
 \overset{d}{\longrightarrow} \mathcal{A}_X^2 \overset{d}{\longrightarrow} 
\dots \overset{d}{\longrightarrow}   \mathcal{A}_X^{n-1}
 \overset{d}{\longrightarrow}  \mathcal{A}_X^{n},
$$
whose hypercohomology gives the De-Rham cohomology 
$H_{DR}^{\ast}(X)=H^{\ast}(X, \underline{\R})$.
Consider now an infinitely differentiable map $f: X \to X$ and recall that
we have a functorial action of $f$ on the sheaf $\mathcal{A}_X^p$ defined
as follows. If $U$ and $V$ are two open sets in $\R^n$ immersed in $X$ 
with coordinate functions $x_i$ and $y_i$, respectively, such that 
$U \subset f^{-1}V$ in $X$, then 
or $i=1, \dots, n$, the functorial maps
$$
f^*dy_i= (\partial_{x_j}f_i )\,dx_j
$$
define a morphism of sheaves
$$
f^*\mathcal{A}_X^1 \to \mathcal{A}_X^1.
$$
The above map is to be intended as a morphism of $\mathcal{A}^0$-modules,
and the sheaf $f^*\mathcal{A}^1$ has to be understood as
$\mathcal{A}^0$-module given by
$$
f^*\mathcal{A}^1 \otimes_{f^*\mathcal{A}^0} \mathcal{A}^0,
$$
in the notation of \eqref{sheaficationofpreheaf}. 
Therefore, by means of \ref{equivalencelemmaXf},
there is defined a sheaf of differential forms $\mathcal{A}_{\bul}^1$ on $\XX{f}$.
As usual, taking exterior power affords an action of $f$ on $\mathcal{A}_X^p$, 
and hence a sheaf $\mathcal{A}_{\bul}^p$ on $\XX{f}$. Moreover, since
this action commutes
with exterior differentiation, we have a De-Rham complex on $\XX{f}$
$$
\mathcal{A}_{\bul}^0 \overset{d_{\bul}}{\longrightarrow} \mathcal{A}_{\bul}^1
 \overset{d_{\bul}}{\longrightarrow} \mathcal{A}_{\bul}^{2}
  \overset{d_{\bul}}{\longrightarrow} 
\dots \overset{d_{\bul}}{\longrightarrow}  \mathcal{A}_{\bul}^{n-1}
 \overset{d}{\longrightarrow}  \mathcal{A}_{\bul}^{n}.
$$
\begin{definition}
	The De-Rham cohomology ${\mathbb{H}}_{DR}^{\ast}(\XX{f})$ is the
	 hypercohomology of the above complex, \text{a.k.a} the cohomology
	 $$
	 \mathbb{H}^{\ast} (\XX{f}, \underline{\R}_{\bul}).
	 $$
\end{definition}
\begin{remark}
	Observe that this definition of De-Rham cohomology does not coincide
	with the usual definition for manifolds, \textit{i.e.} the cohomology of the complex
	$$
	H^0(\XX{f}, \mathcal{A}_{\bul}^{0}) \overset{d_{\bul}}{\longrightarrow}
	H^0(\XX{f}, \mathcal{A}_{\bul}^{1})  \overset{d_{\bul}}{\longrightarrow}\dots
	\overset{d_{\bul}}{\longrightarrow} H^0(\XX{f}, \mathcal{A}_{\bul}^{n}).
	$$
	The reason is that $\mathcal{A}_{\bul}^{0}$-modules do not need to be acyclic in general, 
	and they are not as soon as $E_2^{1,0}$ of the spectral sequence \ref{ext1_corXf} is non trivial. 
\end{remark}
From the discussion above, we deduce the following result.
\begin{cor}\label{cor:DeRham}
If $X$ is simply connected, the De Rham cohomology groups of $\XX{f}$ are
\begin{equation*}
	\mathbb{H}^{p} (\XX{f}, \underline{\R}_{\bul})= 
	\begin{cases}
		 \R \quad \mbox{ if } p=0,1;\\
		  0 \; \quad \mbox{ if } p>1.
	\end{cases}
\end{equation*}
 \end{cor}
\newpage
\phantom{h}
\thispagestyle{empty}
\newpage
\chapter{Applications in holomorphic dynamics on $\PP$}

\section{Revisions on Holomorphic dynamics on $\PP$}
Throughout this section, $f: \PP \to \PP$ is a rational map of degree $D>1$.
We denote by $C_f$ the set of critical points of $f$ and with
$\displaystyle \Gamma_f=\sum_{x\in C_f}(deg_x(f)-1)[x]$ the correspondent
ramification divisor. As usual, $S_f$ and $\mathcal{P}_f$ denote, respectively,
the set of critical values and the postcritical set of $f$,
$$
S_f=f(C_f), \quad \mathcal{P}_f=\bigcup_{n \geq 1} f^n(C_f).
$$
\noindent
The following two sections, in which we set up notation,
are a revision of the results we find in \cite{1999math......2158E},
while the last one
contains our original proof of the Fatou-Shishikura Inequality.

\subsection{The Fatou-Shishikura Inequality}\label{Fatou-Shishikura}
The bound on the number of invariant Fatou components of $f$,
depending on its degree, is a famous result
in holomorphic dynamics. The sharp bound of $2D-2$
is due to Shishikura, \cite{ASENS_1987_4_20_1_1_0}, who uses perturbative methods
to count the number of nonrepelling periodic cycles.
Observe that in Epstein's formulation of this result,
\cite{1999math......2158E}, the degree of $f$ no longer appears.
  In his paper, Epstein manages to develop an accurate machinery
that produces an algebraic proof of the
Fatou-Shishikura Inequality, 
relying only on (his extension of) Thurston's
fundamental result
(whose only known proof is transcendental).
The aim of the second section of
this chapter is to show that, once we organize
the homological algebra involved in a proper way,
these results find a nice geometric explanation
within the language of Tòpoi.
In particular,
we are not providing another proof of the 
fundamental result in \cite{1999math......2158E},
\textit{i.e.} a refined version of 
``Infinitesimal Thurston's Rigidity''. 	 \\
Given a cycle $\langle x\rangle=\{x,\dots,f^{k-1}(x)\}$ of $f$, with multiplier
$\rho=(f^k)'(x)$, we say that
$$
\langle x \rangle \quad \mbox{ is } \quad \begin{cases}
	\mbox{superattracting}, & \mbox{if } \rho=0 ;\\
	\mbox{attracting}, & \mbox{if } 0<|\rho| <1; \\
	\mbox{indifferent}, & \mbox{if } |\rho|=1; \\
	\mbox{repelling}, & \mbox{if } |\rho|>1.
\end{cases}
$$
An indifferent cycle may be \textit{rationally} indifferent
if $\rho$ is a root of unity, or \textit{irrationally}
indifferent otherwise.
Note that the count of
parabolic, \textit{i.e.} rationally indifferent,
since we are taking $D >1$, cycles in \cite{1999math......2158E}
has been refined by taking account of the
\textit{parabolic multiplicity}. We recall briefly its
definition, see e.g. \cite{article}
for a more detailed presentation.
Let $\langle x \rangle$ be a parabolic cycle of period $k$ and
suppose its multiplier $\rho$ is a primitive element of
$\mu_q$. It is known that
the multiplicity of $x$ as a fixed point of $f^{kq}$ is
congruent to 1 modulo $q$. If $N+1$ is this multiplicity,
let $N=\nu q$: 
\begin{equation}\label{parabolicmult}
	\text{the \textit{parabolic multiplicity}
		of the cycle $\langle x \rangle$ is defined
		to be the integer $\nu$.}
\end{equation}
Moreover, there
exists a \textit{preferred local coordinate} $\zeta$
around $x$ such that,
in this coordinate, $f^k$ is expressible as
\begin{equation}\label{preferred}
	\zeta \mapsto \rho\zeta( 1+ \zeta^{N} + \alpha\zeta^{2N}+ \mathcal{O}(\zeta^{2N+1})).
\end{equation}
The complex number $\alpha$, called
the \textit{formal invariant} of the cycle, is
related to \'Ecalle's ``r\'esidu it\'eratif" by
$$
r\acute{e}sit(f, \langle x \rangle)=\beta\coloneqq\frac{N+1}{2}-\alpha.
$$
Following Epstein,
to each cycle $\langle x \rangle= \{x, \dots, f^{k-1}(x)\}$,
we assign the multiplicity $\gamma_{\langle x\rangle}$, where
\begin{equation*}\label{gamma}
	\gamma_{\langle x\rangle}=\begin{cases}
		0, & \mbox{if } \langle x\rangle  \mbox{ is superattracting or repelling};\\
		1, & \mbox{if } \langle x\rangle \mbox{ is attracting or irrationally indifferent};\\
		\nu, & \mbox{if } \langle x\rangle
		\mbox{ is parabolic-repelling, \textit{i.e.} } \Re(\beta)>0;\\
		\nu +1, & \mbox{if } \langle x\rangle \mbox{ is parabolic-nonrepelling, \textit{i.e.} } \Re(\beta)\leq 0.
	\end{cases}
\end{equation*}
For $A \subset \PP$ a finite set, we write
$$
\gamma_A\coloneqq\sum_{\langle x\rangle\subset A} \gamma_{\langle x\rangle},
$$
where the sum is taken over all cycles of $f$ inside $A$.\\
The total count of nonrepelling cycles, with multiplicity, is given by
$$
\gamma_f\coloneqq \sup_A \gamma_A=\sum_{\langle x\rangle\subset \PP} \gamma_{\langle x\rangle},
$$
which \textit{a priori} might be infinite.
Finally, let $\mathcal{P}_f^{\infty}\subset \mathcal{P}_f$ denote
the set of non-preperiodic points lying in the postcritical set.
We write $\delta_f$ for the number of infinite tails, which is by definition
the number of orbit-equivalence classes
in $\mathcal{P}_f^{\infty}$. It is useful to note that,
if $S_n=S_{f^n}$ is the set of critical
values of $f^n$, then
\begin{equation}\label{delta}
	\delta_f= \# S_{N+1}-\#S_N
\end{equation}
for $N\in \N$ large enough.\\
In fact, if we set $A_n=f^n(S_f)$, for $n\geq 0$, from the chain rule
we get $S_{n+1}=S_n \cup A_{n}$,
so the sequence $\#S_n$ is nondecreasing and
thus $a_n\coloneqq\#S_{n+1}- \#S_n$ is a positive integer valued sequence.
Moreover, $a_n$ is nonincreasing and hence eventually constant: note that we have
$a_n=\#( A_n \setminus S_n)$ and also,
by construction, $f(A_{n}\cap S_{n}) \subset A_{n+1}\cap S_{n+1}$.
So,
$$
\begin{aligned}
	\#( A_{n+1} \setminus S_{n+1})=&  \# A_{n+1} - \#(A_{n+1} \cap S_{n+1})   \\
	&\leq \#f(A_{n}\setminus S_{n}) + \#f(A_{n}\cap  S_{n}) -\#(A_{n+1} \cap S_{n+1})  \\
	&\leq \#f(A_{n}\setminus S_{n}) \leq  \#( A_{n}\setminus S_{n})
\end{aligned}
$$
The following result, as mentioned before, is obtained in \cite{1999math......2158E} .
\begin{theorem}\label{F-S}
	$$
	\gamma_f\leq \delta_f.
	$$
\end{theorem}
The well-known formulation of the Fatou-Shishikura Inequality is recovered
as a corollary of the above theorem
once noticed that the number of superattracting cycles is at most $2D-2-\delta_f$,
since the chain rule implies they contain at least one critical point.
\subsection{The pushforward operator}
The key point in proving Theorem \ref{F-S} is the assertion
$f_*q \neq q$ for any $q$
belonging to some subspace of $\mathcal{M}_2(\PP)$.
Here $\mathcal{M}_k(\PP), \;k\in \N$ denotes the space of meromorphic
$k$-differentials on $\PP$, \textit{i.e.}
meromorphic sections of the sheaf $\Omega_{\PP}^{\otimes k}$,
and $f_*$ is the \textit{pushforward} operator, whose definiton,
which we give in a much wider generality than we actually need, is as follows.\\
Let $f: X \to Y$ be an analytic map between Riemann surfaces.\\
Recall that for any open set $V\subseteq Y $ and $U \subseteq f^{-1}(V)$.
we have a pullback operator
\begin{equation}\label{pull-back}
	f^*: H^0(V, \Omega_Y^{\otimes k}) \to H^0(U, \Omega_X^{\otimes k})
\end{equation}
obtained by composing the canonical map
\begin{equation}\label{canonicalpullback}
	f^{\star}: H^0(V, \Omega_Y^{\otimes k}) \to H^0(U, f^*\Omega_Y^{\otimes k})
\end{equation}
with the functorial map induced by
$
df^{\top}: f^*\Omega_Y \to \Omega_X,
$
that is, the transpose of the differential of $f$,
$$
df: T_X  \to f^*T_Y.
$$
The map $f^*$ extends to a map $\mathcal{M}_k(Y) \to \mathcal{M}_k(X)$,
which we still denote by $f^*$.
\begin{factdefinition}
	There is a well-defined linear map
	$$
	f_*: \mathcal{M}_k(X) \to \mathcal{M}_k(Y),
	$$
	given by
	\begin{equation}\label{push-forward}
		f_*q=\sum_g g^*q,
	\end{equation}
	where the sum ranges over the inverse branches of $f$. \\
	
\end{factdefinition}
Observe that, if $V\subset Y\setminus S_f$ is small enough
so that $U\coloneqq f^{-1}(V)= \coprod_{i}U_i$, we have $D$ inverse branches
of $f$, $g_i: V \to U_i$, and the map $g_i^*$,
cf. \eqref{pull-back}, is well-defined. On a punctured neighborhood
of $y \in S_f$ the local inverses glue two by two and create a pole
singularity in $y$, since in local coordinates $g_i^*$ consists in dividing
by an appropriate power of the derivative of $f$.\\
We set now $X=Y$ and observe that
$$
f_*f^*q=Dq
$$
We are interested in studying the fixed points of
$f_*$, that is, the kernel of the linear endomorphism $\nabla_f\coloneqq Id-f_*$.
If $A\subset X$, we write $\mathcal{M}_k(X, A)$
for the subspace of meromorphic $k$-differentials whose poles
are contained in $A$. Then, for $B \subset Y$ such that
$B \supseteq f(A) \cup S_f$, we still denote by $f_*$
the restriction
$$
f_*|_{\mathcal{M}_k(X, A)}: \mathcal{M}_k(X, A) \to \mathcal{M}_k(Y, B)
$$
Set now $X=Y$ and assume also that $A\subseteq B$, so that
there is an inclusion
$i: \mathcal{M}_k(X, A) \hookrightarrow \mathcal{M}_k(X, B)$.
We still write $\nabla_f$ for the map
$$
i-f_*:  \mathcal{M}_k(X, A) \to \mathcal{M}_k(X, B).
$$
Considering that our aim is to provide a systematic
co-homological interpretation of the mentioned results,
we prefer to set up notation in a different way.
For $E \subset X$ a finite set, we denote by $[E]$ the divisor
\begin{equation}\label{divisor1}
	[E]=\sum_{x \in E}[x].
\end{equation}
Recall that for any analytic map $g: U\subset X \to Y$, and any
meromorphic $k$-differential $q$, we have
$$
\forall x \in X \quad \quad ord_x g^*q=deg_x(g)(ord_{g(x)}q +k)-k,
$$
so that
\begin{equation}\label{order}
	\forall y \in Y \quad \quad ord_y f_*q \geq \min_{x \in f^{-1}y} ( \dfrac{ord_x q +k}{deg_x(f)}-k).
\end{equation}
Hence, for $q\in H^0(X, \Omega_X^{\otimes k}(+N[E]))$, where $N \geq k$ is an integer,
then
$$
\forall y \in Y \quad ord_y f_*q \geq  \frac{-N+k}{D}-k \geq -N.
$$
Consequently, we have an induced map
\begin{equation}\label{f_*0}
	f_*: H^0(X, \Omega_X^{\otimes k}(+N[A])) \to   
	H^0(Y,\Omega_Y^{\otimes k}(+N[B])),
\end{equation}
where $B \supseteq S_f$ is a finite set and $A \subseteq f^{-1}(B)$.\\
Let now choose a divisor $\Delta_0$ on $Y$, without
assuming that its support contains the critical values.
We can still write a pushforward map with target space
$H^0(Y, \Omega_Y^{\otimes k}(+\Delta_0))$
if we restrict the domain to the subspace of differentials of $X$ having
zeroes of the appropriate multiplicity along the critical points.
In fact, equation \eqref{order} implies that if we have
$ord_x(q)\geq (k-1)(\deg_x(f)-1)$ along all critical points $x$, then
for $y \in S_f$,
$$
ord_y f_*q\geq \min_{x \in f^{-1}y} \frac{(k-1)(\deg_x(f)-1)+k}{\deg_x(f)}-k=
 \min_{x \in f^{-1}y} \frac{1}{\deg_x(f)}-1 >-1.
$$
Thus, for any divisor
$\Delta_1\preceq f^{*}\Delta_0$
the discussion above shows that we have an induced map
\begin{equation}\label{f_*}
	f_*: H^0(X, \Omega_X^{\otimes k}(+ \Delta_1- (k-1)\Gamma_f)) \to   
	H^0(Y,\Omega_Y^{\otimes k}(+ \Delta_0))
\end{equation}
Set now $X=Y$ and let $\Delta_0$ be any divisor on $X$.
Then, for any divisor
\begin{equation}\label{simplicialdivisor}
	\Delta_1 \preceq \Delta_0 \wedge f^*\Delta_0,
\end{equation}
we have a well-defined map
\begin{equation}\label{nabla_f}
	\nabla_f: H^0(X, \Omega_X^{\otimes k}(+ \Delta_1- (k-1)\Gamma_f)) \to  
	 H^0(X,\Omega_X^{\otimes k}(+ \Delta_0)).
\end{equation}
We will provide in \ref{ThurstonvanishingLemma} a co-homological description of the transpose of the maps above.
\subsection{Infinitesimal Thurston's Rigidity}
In the following we give a description of ``Infinitesimal Thurston Rigidity''.
We first attempt to keep a certain level of generality, which is largely 
more than what is needed in practice. 
The reader who is only interested in the applications may skip to \ref{T-E}.
Let us fix a completely invariant compact subset $Z \subset X$, \textit{i.e.} $f^{-1}(Z)=Z$ .
There is a well defined positive measure associated to
every $q \in \mathcal{M}_k(Z)$. The integral of such measure defines
a norm, $\|q\|_Z$, which might be infinite.
Let us fix a (continuous) $k$-differential $q$ with support on $Z$.
We consider the pseudo-metric associated with $q$,
given by $w=w(q)=(\overline{q}\otimes q)^{1/k}$ and
the Hausdorff measure on $Z$ of dimension $2r>0$,
which we denote as $H_w^r$, associated to this pseudo-metric.
Recall that $H_w^r$ is defined as the result of Carath\'eodory construction with respect to
$
\zeta_w(S)=const\cdot \mbox{diam}_w^{2r}(S)
$,
for $S$ a Borel set.\\
We define the Hausdorff dimension of $Z$ with respect to $w$ as the number
\begin{equation}\label{eq0}
	m=m(w)=\inf\{ r>0: H_w^r(Z)=0 \}=
	\sup\{ r>0: H_w^r(Z)=+\infty \}
\end{equation}
It is rather easy to see that, in fact, $m$ does not depend on $w$,
\textit{i.e.} the condition $H_w^r(Z)=0$ (and also $H_w^r=+\infty$)
are independent of the pseudo-metric $w$. If $w_1$ and $w_2$ were real metrics
(\textit{i.e.} $w_i(x)$ a positive definite bilinear form on the tangent space
$T_xX$ for each $x\in Z$, $i=1,2$) this would be an 
immediate consequence of the compactness of $Z$,
which implies $0<c\leq g\leq C$, for $g$ the transition function $w_1=gw_2$, and hence
\begin{equation}\label{eq1}
	c\cdot\mbox{diam}_{w_2}(S)\leq \mbox{diam}_{w_1}(S)\leq C\cdot\mbox{diam}_{w_2}(S)
\end{equation}
In general we compare $w_1=w(q_1)$ and $w_2=w(q_2)$ with a metric $v$,
\textit{i.e.} we find nonnegative functions $g_i\geq 0,\;i=1,2$ such that $w_1=g_1v$ and $w_2=g_2v$.
Denote $\Omega$ the finite set given by the union of the zeroes of $g_i$.\\
For all $\varepsilon >0$, $p\in \Omega$ we find Borel sets $p\in S_{\varepsilon}(p)$ 
such that $\mbox{diam}_{w_i}(S_{\varepsilon}(p))<\varepsilon$, $i=1,2$.
Choose now a covering $\{ S_{\alpha}\}_{\alpha \in A}$ of $Z$ 
such that $H_{w_i}^r(Z)=const\cdot\sum_{\alpha}\mbox{diam}_{w_i}^r(S_{\alpha})+\varepsilon$.
We can always find a refinement of such a covering for which the $S_{\varepsilon}(p)$
are members of the covering and no other member intersects them, while the approximating
result for the measure is still valid. Denote $A'$ the set of indices which do not include the $S_{\varepsilon}(p)$.\\
It is clear that for $\alpha\in A'$ we still have the inequality \eqref{eq1} with $S=S_{\alpha}$, 
while the contribution from $A\setminus A'$ is less than 
$\varepsilon'\coloneqq const\cdot|\Omega|\varepsilon^r$, so that finally we have
\begin{equation}\label{eq2}
	c^r\cdot\sum_{\alpha}\mbox{diam}_{w_2}^r(S_{\alpha})-\varepsilon'\leq \sum_{\alpha}
	\mbox{diam}_{w_1}^r(S_{\alpha})\leq C^r\cdot\sum_{\alpha}\mbox{diam}_{w_2}^r(S_{\alpha})+\varepsilon'
\end{equation}
and finally, as $\varepsilon$ is arbitrary small, we deduce that
\begin{equation}\label{eq3}
	H_{w_1}^r(Z)=0 \quad (\mbox{resp. } +\infty) \iff H_{w_2}^r(Z)=0 \quad(\mbox{resp. } +\infty)
\end{equation}
This argument is clearly independent of the choice of the metric $v$.
\begin{definition}
	\rm{For a }$k$-\rm{differential }$q$\rm{ supported on }$Z$ \rm{we define its norm as}
	$$
	\|q\|_Z=H_{w(q)}^m(Z)
	$$
	\rm{where m is as in }\eqref{eq0}. \\
	We will say that $q$ is (nontrivially) integrable if $\|q\|_Z\neq 0, +\infty$.
\end{definition}

\begin{remarks}\
	\begin{enumerate}
		\item  Observe that, if we have a global coordinate $z$ on $Z$,
		we can write $q$ as $q(z)dz^{\otimes k}$ for some function $q(z)$
		defined on a neighborhood of $Z$ and supported on $Z$.
		If we denote by $v(z)$ the standard Hermitian metric $d\overline{z}\otimes dz$,
		then $w(q)(z)=|q(z)|^{2/k}v(z)$ and by a change of variable we have
		\begin{equation}\label{norm}
			\|q\|_Z=\int_Z dH_{w(q)}^m(z)=\int_Z |q(z)|^{2m/k}dH_v^m(z)
		\end{equation}
		We will use from now on the following notation, which is justified by the last equality:
		\begin{equation}\label{int}
			\|q\|_Z=\int_Z|q|^{2m/k}.
		\end{equation}
		\item If we assume $q$ to be meromorphic,
		then formally, if $\Delta$ is the polar divisor of $q$,
		we have
		\begin{equation}\label{dl}
			q\in  \varinjlim\limits_{Z\subset U} H^0(U, \Omega^{\otimes k}(+\Delta))
		\end{equation}
		Note that we can define $f_*q$ as in \eqref{push-forward}, since $Z$ is assumed to
		be completely invariant. Around any interior point of $Z$ it is clear.
		If instead $p\in Z$ is on the boundary, we proceed as follows:
		we take $q$ defined on a neighborhood $U$ of $Z$ and argue that around $p$
		we can take a small neighborhood $W_p$ with the property that
		$f^{-1}W_p \subset U$, so that  $f_*q$ is well defined on $W_p$,
		and hence it is well defined as an element of \eqref{dl}.
	\end{enumerate}
\end{remarks}

Let $\alpha =\frac{2m}{k}$. Since we will always assume $k\geq2$, we have $\alpha\leq 1$, 
and $\alpha=1$ if and only if $k=2$ and $m=1$, that is the case when $Z$ 
has positive Lesbegue measure and $q$ is a quadratic differential. 
Recall that for every $v\in \C^n$, $\|v\|_1\leq\|v\|_{\alpha}$.\\
We still have the following fundamental fact.
\begin{lemma}[\bf{Contraction Principle}]
	Let $q$ be a meromorphic $k$-differential with support on $Z$. Then
	\begin{equation}\label{contr}
		\|f_*q\|_Z \leq \|q\|_Z.
	\end{equation}
	In particular, if $q$ is integrable, so is $f_*q$.
\end{lemma}
\begin{proof}
	This is an immediate consequence of the triangle inequality:
	\begin{equation}
		\begin{aligned}
			&\|f_*q\|_Z=\int_Z \Big| \sum_g g^*q\Big|^{\alpha}\leq  \int_Z \Big( \sum_g |g^*q|\Big)^{\alpha}\leq  \\
			&  \leq \int_Z \sum_g| g^*q|^{\alpha}=\int_Z f_*|q|^{\alpha}=\|q\|_Z
		\end{aligned}
	\end{equation}
\end{proof}

\begin{lemma} {\bf{[Rigidity principle]}}\\
	Let $q$ be integrable. If $f_*q=q$, then $f^*q=Dq$.
\end{lemma}

\begin{proof}
	Since $\|f_*q\|_Z=\|q\|_Z$, the inequalities in \eqref{contr} are all equalities, 
	and this means that $\Big| \sum_g g^*q\Big|= \sum_g |g^*q|$. 
	This in turn means, (c.f. \cite{1999math......2158E}), that the quotient
	$g^*q/\phi^*q$ is a real valued meromorphic function on an open subset $U\subset X \setminus S_f$,
	hence it is globally constant,
	where $\phi:U \to \PP$ is any preferred inverse branch of $f$.
	We deduce that $f_*q=\lambda \phi^*q$, for some constant $\lambda\in \C$.\\
	Consequently, $f^*f_*q=\lambda f^*\phi^*q=\lambda (\phi f)^*q=\lambda q$ on $\phi(U)$.
	It follows easily, since we have $f_*f^*q=Dq$ for any differential $q$
	and in particular for $f_*q$, that $\lambda=D$.
\end{proof}
Let now consider the case $Z=X=\PP$. The class of maps $f:\PP \to \PP$
for which there exists a $k$-differential $q$ such that $f^*q$ is a constant multiple of $q$
have been classified, e.g. in \cite{MR1251582}, and they are all quotients of an
endomorphism of a torus. We refer to it as the class of \textit{Latt\`es maps}.\\
Let us denote $\mathcal{Q}_k(\PP)$ the set of integrable $k$-differentials on $\PP$.
Recall that, by \eqref{contr}, the map $\nabla_f$ restricts to an endomorphism of $\mathcal{Q}_k(\PP)$.
An easy computation shows that $q \in \mathcal{Q}_k(\PP)$ if and only if there is an effective divisor
$\Delta$ with $deg_x(\Delta)<k\quad \forall x \in \PP$, such that 
$q\in H^0(\PP, \Omega_{\PP}^{\otimes k}(+\Delta))$. The above discussion may be summarized in the following.
\begin{theorem}\label{T-E}{\bf{[Infinitesimal Thurston Rigidity]}}\\
	Assume $f$ is not a Latt\`es map. Then,
	$$
	\nabla_f: \mathcal{Q}_k(\PP) \to \mathcal{Q}_k(\PP)
	$$
	is injective.
\end{theorem}
The interesting case for applications is, of course,
the case of quadratic differentials.
In order to state the extension of Theorem \ref{T-E} due to Epstein,
we need the notion of \textit{invariant divergences}.
The space of \textit{algebraic divergences} is, by definition,
the quotient space $\mathcal{D}(\PP)\coloneqq\mathcal{M}(\PP)/\mathcal{Q}(\PP)$,
where we have suppressed the subscript $k=2$.
They consist in polar parts, $[q]=([q]_x)_{x \in X}$, of order $\geq -2$ of
meromorphic quadratic differentials.
There is an induced endomorphism $\nabla_f: \mathcal{D}(\PP) \to \mathcal{D}(\PP)$,
whose kernel is defined to be the space $\mathcal{D}(f)$ of \textit{invariant divergences}.
Observe that \eqref{order} implies, for $[q]\in \mathcal{D}(f)\setminus \{0\}$, since we have $[q]_x=0$
for all but finite $x\in \PP$, that
$[q]_x\neq 0$ if and only if $x$ is periodic.
Moreover, for $A \subseteq \PP$ we denote $\mathcal{D}(\PP, A)$ the corresponding
quotient of $\mathcal{M}(\PP, A)$. We have, again by \eqref{order},
a map
$$
\nabla_f|_{\mathcal{D}(\PP, A)}: \mathcal{D}(\PP, A) \to \mathcal{D}(\PP, f(A)).
$$
We set $\mathcal{D}_A(f)= \ker \nabla_f|_{\mathcal{D}(\PP,A)}$,
and observe that
$$
\mathcal{D}_A(f)=
\bigoplus_{\langle x \rangle \subset A}\mathcal{D}_{\langle x \rangle}(f).
$$
The space $\mathcal{D}_{\langle x \rangle}(f)$ is computed in \cite{1999math......2158E}
in terms of the local invariants of $\langle x \rangle$.
Moreover, it is sufficient to compute it for fixed points,
since the morphism
$\mathcal{D}_{\langle x \rangle}(f) \overset{\pi}{\longrightarrow} \mathcal{D}_x(f^{p})$,
induced by the projection $\mathcal{D}_x(\PP, \langle x \rangle ) \to \mathcal{D}(\PP, \{x\})$,
is an isomorphism. Indeed, its inverse is
\begin{equation}\label{cyclereduction}
	\pi^{-1}( [q]_x)=\bigoplus\limits_{k=0}^{p-1}(f^k)_*( [q]_x).
\end{equation}
Let us denote with $\mathcal{D}^{k}_x $
the subspace of $\mathcal{D}_x(\PP)$ given by
$$
\mathcal{D}^{k}_x =\C\left[ \frac{(d\zeta)^2}{\zeta^{k}}\right].
$$
The following is the result obtained in \cite{1999math......2158E}.
\begin{lemma}\label{divergence1}
	Let $x$ be a fixed point of $f$ and $\zeta$ a local coordinate centered at $x$.
	Assume $\zeta$ is a preferred local coordinate in the parabolic case, \eqref{preferred}.
	Then, in the notation of \ref{Fatou-Shishikura},
	$$
	\mathcal{D}_x(f)= \begin{cases}
		0, & \mbox{if } x \mbox{ is superattracting or repelling};\\
		\mathcal{D}^{2}_x , & \mbox{if } x \mbox{ is attracting or irrationally indifferent};\\
		\mathcal{D}_x^{\circ}(f) \oplus \C[q_f], & \mbox{if } x \mbox{ is parabolic},
	\end{cases}
	$$
	where  $\mathcal{D}_x^{\circ}(f)\coloneqq\bigoplus\limits_{k=0}^{\nu-1}\mathcal{D}^{kn+2}_x $, and
	$\displaystyle [q_f]=\left[\frac{(d\zeta)^2}{(\zeta^{N+1}-\beta \zeta^{N+2})^2}\right]$.
\end{lemma}
Following Epstein, along any cycle
$\langle x\rangle$ of $f$ there is a dynamically defined residue
associated to an invariant divergence $[q]\in \mathcal{D}(f)$,
which is denoted $Res_{\langle x\rangle}(f, [q])$. We refer to
\cite{1999math......2158E} for its definition, which is not relevant
for our discussion.
The space $\mathcal{D}^{\flat}(f)=
\bigoplus_{\langle x\rangle \subset \PP}\mathcal{D}_{\langle x\rangle}^{\flat}(f)$ is, by definition,
the subspace of $\mathcal{D}(f)$ having nonpositive residue along any cycle.
\begin{lemma}\label{divergence2}
	In the hypothesis of Lemma \ref{divergence1}, we have
	$$
	\mathcal{D}^{\flat}_x(f)= \begin{cases}
		0, & \mbox{if } x \mbox{ is superattracting or repelling};\\
		\mathcal{D}^{2}_x , & \mbox{if } x \mbox{ is attracting or irrationally indifferent};\\
		\mathcal{D}_x^{\circ}(f) , & \mbox{if } x \mbox{ is parabolic-repelling};\\
		\vspace{0.3cm}
		\mathcal{D}_x^{\circ}(f)  \oplus \C[q_f], & \mbox{if } x \mbox{ is   parabolic-nonrepelling}.
	\end{cases}
	$$
	If $\langle x\rangle$ is a cycle of $f$ of period $p$, then
	$\mathcal{D}^{\flat}_{\langle x\rangle}(f)=\pi^{-1}(\mathcal{D}^{\flat}_x(f^p))$.
\end{lemma}
Let us consider $A\subset\PP$ a finite set and $B\subset\PP$ such that $B\supseteq A \cup f(A)\cup S_f$.
We have a commutative diagram with exact rows:
\begin{equation}\label{diagramEps}
	\begin{tikzcd}
		0 \arrow[r] & \mathcal{Q}(\PP,A)\arrow[r] \arrow[d, "\nabla_f"] & 
		\mathcal{M}(\PP,A) \arrow[r] \arrow[d, "\nabla_f"] & \mathcal{D}(\PP,A)\arrow[r] \arrow[d, "\nabla_f"] & 0\\
		0 \arrow[r] & \mathcal{Q}(\PP,B)\arrow[r] & \mathcal{M}(\PP,B) \arrow[r] & \mathcal{D}(\PP,B)\arrow[r] & 0
	\end{tikzcd}
\end{equation}
By the Snake Lemma, there is a connecting homomorphism
\begin{equation}\label{connecting}
	\blacktriangledown: \mathcal{D}(f, A) \to \mathcal{Q}(\PP,B)/\nabla_f \mathcal{Q}(\PP,A).
\end{equation}
The extension of Theorem \ref{T-E} proved by Epstein can be formulated as follows.
\begin{theorem} \label{Eps} {\bf{[Epstein]}}\\
	Assume $f$ is not a Latt\`es map. Then $\blacktriangledown|_{\mathcal{D}^{\flat} (f, A)}$ is injective.
\end{theorem}

The Theorem above is, after all, the Fatou-Shishikura Inequality: in fact, by
Lemma \ref{divergence2}, if $A$ is finite, we have
$$
\dim_{\C}\mathcal{D}^{\flat}(f, A)= \gamma_A.
$$
On the other hand, by classical results in deformation theory of Riemann surfaces,
and by Theorem \ref{T-E}, we have $\dim_{\C} \mathcal{Q}(\PP,B)/\nabla_f \mathcal{Q}(\PP,A)=\#B- \#A$,
whenever $\#A \geq 3$. Hence, if we take $A$ any finite set of the form $A=S_{f^n}\cup C$,
where $C$ is a collection of nonrepelling cycles with $\#C \geq 3$,
$n\gg 0$ is as in \eqref{delta}, and $B=A\cup f(A)$, we get
$$
\dim_{\C}\mathcal{D}^{\flat}(f, A) \leq \delta_f.
$$
Taking the sup over $A$ yelds Theorem \ref{F-S}.
\begin{remark}\label{hatrmk}
	Looking at the definition of the connecting
	homomorphism $\blacktriangledown$, we see immediately that Theorem \ref{Eps} is
	equivalent to the injectivity of $\nabla_f$ on the subspace
	of
	$$
	\hat{\mathcal{Q}}(f, A)\coloneqq\{q \in \mathcal{M}(\PP,A): [q] \in \mathcal{D}(f, A)\}=
	\nabla_f^{-1}\mathcal{Q}(\PP, B)
	$$
	given by
	$\hat{\mathcal{Q}}^{\flat}(f, A)
	\coloneqq\{q \in \mathcal{M}(\PP,A): [q] \in \mathcal{D}^{\flat}(f, A)\}$.
\end{remark}

\section{The Snake argument reloaded}
As said before, we shall present the previews results
in a different way. The aim is to show that they have a manifestation
 in the Tòpoi of $\QQ{f}$ and ${\rm E}_f$.\\
Let us first give some general definitions.
Let $f: X \to X$ be an analytic endomorphism of a Riemann surface,
and let $\XX{f}$ denote the associated classifying site, cf \ref{defdynsite1}.
Recall that we denote with $\mathcal{O}_{\bul}$ the structure sheaf of $\XX{f}$
and with $\Omega_{\bul}$ the sheaf of holomorphic differential forms on $\XX{f}$.
\begin{definition}[{\bf Divisor on \XX{f}}]\label{f.i.divisor} 
	Let us consider an effective divisor $\Delta$ on $X$ and let us suppose 
	that $\Delta$ is forward invariant, \textit{i.e.} 
	$f(\Delta) \subseteq \Delta$ or, equivalently,
	\begin{equation}\label{f.i.divisorcondition}
		\Delta\preceq f^*\Delta.
	\end{equation}
	Let us denote by $\Delta_{\bul}$ a divisor $\Delta$ on $X$ 
	satisfying condition \eqref{f.i.divisorcondition}. 
	With an abuse of language 
	we refer to $\Delta_{\bul}$ as a ``divisor'' on $\XX{f}$.
	The reason is that $\Delta_{\bul}$ defines a sheaf 
	$\mathcal{O}(-\Delta_{\bul})$ on $\XX{f}$ given 
	by the pair $(\mathcal{O}_{X}(-\Delta), \iota)$, where 
	\begin{equation}\label{inclusiondivisorsheaf}
		\iota: f^*\mathcal{O}_{X}(-\Delta) \cong \mathcal{O}_{X}(-f^*\Delta) \hookrightarrow \mathcal{O}_{X}(-\Delta)
	\end{equation}
	is induced by the natural inclusion. 
\end{definition}

If $\Delta_{\bul}$ is a divisor on $\XX{f}$,
we have a canonical injection in $Ab(\XX{f})$ given by
$$
i: \mathcal{O}(-\Delta_{\bul}) \hookrightarrow \mathcal{O}_{\bul},
$$
whose co-kernel will be denoted with $\mathcal{O}_{\Delta_{\bul}}$.\\
Applying the functor $\mathbb{H}{\rm om}(\Omega_{\bul}, -)$ to the following exact
sequence in $Ab(\XX{f})$
\begin{equation}\label{idealexact}
	\begin{tikzcd}[row sep=2.6pc, column sep=1.6pc]
		0 \arrow[r] & \mathcal{O}(-\Delta_{\bul})\arrow[r]
		&  \mathcal{O}_{\bul}\arrow[r]  & \mathcal{O}_{\Delta_{\bul}} \arrow[r]  & 0
	\end{tikzcd}
\end{equation}
we obtain a long exact sequence of $\C$-vector spaces
\begin{equation}\label{ideallongexact}
	\begin{tikzcd}[row sep=1.6pc, column sep=0.8pc]
		0 \arrow[r] & \mathbb{H}{\rm om}(\Omega_{\bul},\mathcal{O}(-\Delta_{\bul})) \arrow[r]
		& \mathbb{H}{\rm om}(\Omega_{\bul},\mathcal{O}_{\bul} ) \arrow[r]
		\arrow[d, phantom, ""{coordinate, name=Z}]
		&  \mathbb{H}{\rm om}(\Omega_{\bul}, \mathcal{O}_{\Delta_{\bul}}) \arrow[dll,
		"\delta^1",
		rounded corners,
		to path={ -- ([xshift=2ex]\tikztostart.east)
			|- (Z) [near end]\tikztonodes
			-| ([xshift=-2ex]\tikztotarget.west)
			-- (\tikztotarget)}] \\
		& \mathbb{E}{\rm xt^1}(\Omega_{\bul}, \mathcal{O}(-\Delta_{\bul})) \arrow[r]
		& \mathbb{E}{\rm xt^1}(\Omega_{\bul},\mathcal{O}_{\bul} ) \arrow[r]
		\arrow[d, phantom, ""{coordinate, name=Z}]
		& \mathbb{E}{\rm xt^1}(\Omega_{\bul}, \mathcal{O}_{\Delta_{\bul}}) \arrow[dll,
		"\delta^2",
		rounded corners,
		to path={ -- ([xshift=2ex]\tikztostart.east)
			|- (Z) [near end]\tikztonodes
			-| ([xshift=-2ex]\tikztotarget.west)
			-- (\tikztotarget)}]\\
		& \mathbb{E}{\rm xt^2}(\Omega_{\bul}, \mathcal{O}(-\Delta_{\bul})) \arrow[r]
		& \mathbb{E}{\rm xt^2}(\Omega_{\bul},\mathcal{O}_{\bul} ) \arrow[r]
		& 0
	\end{tikzcd}
\end{equation}
Observe that for any pair of coherent sheaves $\mathcal{F},\mathcal{G}$ 
on a $d$-dimensional complex manifold $X$, the groups
${\rm Ext}^i(\mathcal{F},\mathcal{G})$ vanish for all $i>d$, hence by 
\ref{exti_cor} we obtain the vanishing of the terms
$\mathbb{E}{\rm xt^2}(\Omega_{\bul},\mathcal{O}_{\Delta_{\bul}})$,
$\mathbb{E}{\rm xt^3}(\Omega_{\bul},\mathcal{O}_{\bul})$, etc. 

Let us consider now the analogous sheaves $\mathcal{O}_{\bul}$ and $\Omega_{\bul}$ on Epstein's site ${\rm E}_f$.
\begin{definition}[{\bf E-dynamical divisor}]\label{E-dynimacaldivisor}
	Let $\Delta_{0}, \Delta_{1}$ be effective divisors on $X$ such that
	\begin{equation}\label{divisorcondition}
		\Delta_{1}\preceq \Delta_0 \wedge f^*\Delta_0.
	\end{equation}
	Let us denote by $\Delta_{\bul}\coloneqq (\Delta_{0},\Delta_{1})$ a pair of divisors on $X$ 
	satisfying condition \eqref{divisorcondition}. With an abuse of language 
	we refer to $\Delta_{\bul}$ as a ``divisor'' on ${\rm E}_f$, or simply ``{\rm E}-dynamical'' divisor.
	The reason is that $\Delta_{\bul}$ defines a sheaf 
	on ${\rm E}_f$ given by 
	$$
	\mathcal{O}(-\Delta_{\bul})\coloneqq\big((\mathcal{O}_{X}(-\Delta_0), \mathcal{O}_{X}(-\Delta_1)), \iota_0 \coprod \iota_1\big),
	$$ 
	where 
	\begin{equation}\label{inclusiondivisorEf}
		\iota_0: \mathcal{O}_{X}(- \Delta_0) \hookrightarrow \mathcal{O}_{X}(-\Delta_1), \quad 
		\iota_1:	\mathcal{O}_{X}(- f^*\Delta_0) \hookrightarrow \mathcal{O}_{X}(-\Delta_1)
	\end{equation}
\end{definition}
Note that, \textit{mutatis mutandis}, we have analogous exact sequences
\eqref{idealexact}, \eqref{ideallongexact} in $Ab({\rm E}_f)$, where 
in this case we have denoted by $\Delta_{\bul}$ a divisor on ${\rm E}_f$.
 Let 
$(\mathcal{G}_{\bul}, \gamma_{\bul}) \in Sh({\rm E}_f)$ be
one of the sheaves appearing in the sequence \eqref{idealexact} 
(viewed as a sequence in $Sh({\rm E}_f)$)
and let $(\mathcal{F}, \varphi)\in Sh(\XX{f})$ be a
sheaf of $\mathcal{O}_{\bul}$-modules
such that $\mathcal{F}$ is coherent. Recall that we identify 
$(\mathcal{F}, \varphi)$ with $(\mathcal{F}_{\bul},\varphi_{\bul}) \in Sh({\rm E}_f)$, cf. \ref{notationXfbul}.
By Corollary \eqref{ext1_corXf} there is a converging spectral sequence
$$
E_r^{p,q} \Rightarrow {\rm \mathbb{E}xt^{p+q}}(\mathcal{F}_{\bul}, \mathcal{G}_{\bul}),
$$
which degenerates at $E_2$.

\begin{center}
	\begin{figure}[H]
		\begin{tikzpicture}
			\pgfplotsset{ticks=none}
			\draw[->] (-1.5,-1) -- (8.8,-1) node[below] {${\bf{p}}$} -- (9,-1);
			\draw (0.0,-1) node[below] {$0$};
			\draw[->] (-1.5,-1) -- (-1.5,3.8) node[left] {${\bf{q}}$} -- (-1.5,4);
			\draw (0.0,0.0) node[] {${\rm Hom}(\mathcal{F}, \mathcal{G}_0)$};
			\draw (-1.5,0) node[left] {$0$};
			\draw (3.5,0) node[] {${\rm Hom}(f^*\mathcal{F}, \mathcal{G}_1)$};
			\draw (3.5,-1) node[below] {$1$};
			\draw (7,-1) node[below] {$2$};
			\draw (0,1.25) node[] {${\rm Ext^1}(\mathcal{F}, \mathcal{G}_0)$};
			\draw (3.5,1.25) node[] {${\rm Ext^1}(f^*\mathcal{F}, \mathcal{G}_1)$};
			\draw (0,2.5) node[] {$0$};
			\draw (7,0) node[] {$0$};
			\draw (7,1.25) node[] {$0$};
			\draw (7,2.5) node[] {$0$};
			\draw (-1.5,1.25) node[left] {$1$};
			\draw (-1.5,2.5) node[left] {$2$};
			\draw (3.5,2.5) node[] {$0$};
			\draw[->] (1,0) -- (2,0);
			\draw[->] (1,1.25) -- (2,1.25);
			\draw (1.5,0) node[below] {\tiny$d_1^{0,0}$};
			\draw (1.5,1.25) node[below] {\tiny$d_1^{0,1}$};
		\end{tikzpicture}	
		\caption{$E_1$ of the spectral sequence}\label{spectralE1P1}
	\end{figure}
\end{center}

The maps $d_1^{0,0}, d_1^{0,1}$ in Figure \ref{spectralE1P1} are, respectively, given by
the difference of the maps appearing on the right side of 
\eqref{Homequalizerfinal},
and its derived one.
Recall, cf. \cite{BottTu}, that there is a canonical isomorphism
\begin{equation}\label{tensor}
	{\rm Ext}^i(\mathcal{F},\mathcal{G})
	\overset{\sim}{\longrightarrow}
	{\rm H}^i(\mathcal{F}^{\vee} \otimes \mathcal{G}),
\end{equation}
for all $i \geq 0$.\\
We fix from now on $\mathcal{F}_{\bul}=\Omega_{\bul}$.
We can identify
the differentials of the spectral sequence by means of \eqref{Homequalizerfinal}, \eqref{tensor} as follows:
\begin{equation}\label{alpha}
	\begin{tikzcd}[row sep=1.2pc, column sep=0.8pc]
		H^0(T_X\otimes \mathcal{G}_0 ) \arrow[r, "d_1^{0,0}"]
		& H^0(f^*T_X\otimes \mathcal{G}_1) \\
		x \arrow[r, mapsto] & (df\otimes \gamma_0)(x) - (id \otimes \gamma_1)(f^{\star}x)
	\end{tikzcd}
\end{equation}
where
$f^{\star}:  H^0(T_X\otimes \mathcal{G}_0 )
\to  H^0(f^*T_X\otimes f^*\mathcal{G}_0) $ denotes the canonical pullback map.
Observe that, for our choice of $\mathcal{G}_{\bul}$, the map $\gamma_0$ is given,
respectively, by an inclusion, an identity or a projection
$\mathcal{O}_{\Delta_0} \to \mathcal{O}_{\Delta_1}$,
and the same is true for $\gamma_1$.
Thus, we are justified to abuse notation and set,
for all of them,
\begin{equation}\label{differencemapnablaf}
d_1^{0,0}=df-f^{\star}.	
\end{equation}
We fix from now on $X=\PP$.
Let $f: \PP \to \PP$ be a rational map and let $\QQ{f}$ be
the associated site, cf \ref{defdynsite1}.

\begin{definition}
	Let $C= \mathop{\bigcup}\limits_iC_i$ be a finite (disjoint) union of
	cycles of $f$ and let $\Delta$ be an effective divisor on $\PP$
  shaving support $|\Delta|=C$. 
  If for for each $x$ in the support of $\Delta$ we have
$$
deg_x(\Delta)=
\begin{cases}
	2, & \mbox{if } x\in C_i \mbox{ for some nonrepelling cycle } C_i;\\
	1, & \mbox{otherwise},
\end{cases}
$$
we say that the divisor $\Delta_{\bul}$ on $\QQ{f}$
defined by $\Delta$ is a collection of
``rigid cycles'' on $\QQ{f}$.
\end{definition}
\begin{claim}[{\bf Fatou-Shishikura (weak version)}]\label{claimF-SXf}
	Assume that $f$ is not a Latt\`es map. As a consequence of
	of Theorem \ref{F-S} we deduce the following Vanishing Theorem:
	let $\Delta_{\bul}$ consists of rigid cycles on $\XX{f}$,
	then
	\begin{equation}\label{vanishingXf}
		\mathbb{E}{\rm xt^2}(\Omega_{\bul}, \mathcal{O}(-\Delta_{\bul}))=0.
	\end{equation}
	Moreover, part of the exact sequence \eqref{ideallongexact}
	reads as the following exact sequence of $\C$-vector spaces
	(which have exactly the dimension appearing below)
	\begin{equation}\label{simplicialF-SXf}
		\begin{tikzcd}[row sep=1.2pc, column sep=0.8pc]
			\underset{ 2D-2 }{\mathbb{E}{\rm xt^1}(\Omega_{\bul},\mathcal{O}_{\bul})} \arrow[r]
			& \underset{N}{\mathbb{E}{\rm xt^1}(\Omega_{\bul}, \mathcal{O}_{\Delta_{\bul}})} \arrow[r] & 0
		\end{tikzcd}
	\end{equation}
	where $N$ denotes the total number of nonrepelling cycles contained in the support of $\Delta$.
\end{claim}
The above Claim shall follow from a more general statement, in
which it is needed the use 
 Epstein's site ${\rm E}_f$.
Using the identification \ref{notationXfbul} we can deduce the first one from the general statement, to wit:
\begin{claim}\label{claim1}
	Assume $f$ is not a Latt\`es map. An equivalent formulation
	of Theorem \ref{F-S} is the following Vanishing Theorem:
	for an appropriate choice of a ``{\rm E}-dynamical'' divisor $\Delta_{\bul}$,
	that will be specified later, cf. \ref{rigiddivisordef} and  \eqref{sharprigiddivisorEf}, we have
	\begin{equation}\label{vanishing}
		\mathbb{E}{\rm xt^2}(\Omega_{\bul}, \mathcal{O}(-\Delta_{\bul}))=0.
	\end{equation}
	Moreover, the exact sequence \eqref{ideallongexact}
	reduces to the following exact sequence of $\C$-vector spaces
	(which have exactly the dimension appearing below)
	\begin{equation}\label{simplicialF-S}
		\begin{tikzcd}[row sep=1.2pc, column sep=0.8pc]
			0 \arrow[r] &  \underset{2D-2+ \gamma_A}{\mathbb{H}{\rm om}(\Omega_{\bul}, \mathcal{O}_{\Delta_{\bul}})}\arrow[r,
			"\delta^1"] \arrow[d, phantom, ""{coordinate, name=Z}]
			& \underset{2D-2+\delta_f}{\mathbb{E}{\rm xt^1}(\Omega_{\bul}, \mathcal{O}(-\Delta_{\bul}))}  \arrow[dll,
			rounded corners,
			to path={ -- ([xshift=2ex]\tikztostart.east)
				|- (Z) [near end]\tikztonodes
				-| ([xshift=-2ex]\tikztotarget.west)
				-- (\tikztotarget)}] \\
			\underset{ 2D-2 }{\mathbb{E}{\rm xt^1}(\Omega_{\bul},\mathcal{O}_{\bul})} \arrow[r]
			& \underset{2D-2+ \gamma_A -\delta_f}{\mathbb{E}{\rm xt^1}(\Omega_{\bul}, \mathcal{O}_{\Delta_{\bul}})} \arrow[r] & 0
		\end{tikzcd}
	\end{equation}
	where $A= |\Delta_0|$ denotes the support of $\Delta_0$.
\end{claim}
\begin{remark}
	Both rows of diagram \eqref{simplicialF-S} are manifestations
	of the Fatou-Shishikura Inequality \ref{F-S}: taking the sup over $A$
	yields $\gamma_f\leq \delta_f$. Moreover, the second row has a
	nice geometric interpretation: the space
	$\mathbb{E}{\rm xt^1}(\Omega_{\bul},\mathcal{O}_{\bul})$
	is canonically isomorphic to the (orbifold) tangent space at $[f]$
	of the moduli space $\textbf{rat}_D$.
	In our formalism, it is evident that 
	$T_f\textbf{rat}_D$ coincides with
	 the space of globally invariant
	(infinitesimal) deformations and carries
	a natural restriction map 
	$$
	{\rm Res}_{\Delta_{\bul}} :
	\mathbb{E}{\rm xt^1}(\Omega_{\bul},\mathcal{O}_{\bul}) \longrightarrow
	\mathbb{E}{\rm xt^1}(\Omega_{\bul}, \mathcal{O}_{\Delta_{\bul}})
	$$
	to the space of local invariant deformations around the points of $|\Delta_{\bul}|$.
	Thurston's theorem states that this map is surjective, or equivalently that there is no obstruction
	in extending such a local deformation into a global one.
\end{remark}

Let us give a description of the maps involved in the
spectral sequence \eqref{spectralE1P1}, which have been studied earlier,
although in a different way,
in \cite{1999math......2158E}, \cite{transversality}. 
Let us explain the differences between them.
 A substantial difference is that we are not identifying
the sheaf $f^*T_X$ with $T_X(+\Gamma_f)$.\\
Moreover,
another novelty of this new approach 
is that our formalism forces us to compare
two independent phenomena:
the dynamics of differentials forms 
and the dynamics of divisors (supported both on
cycles and the critical locus of $f$) under the action of $f^*$. 
A comparison of the two leads up to 
the definition of the operator, cf. \eqref{differencemapnablaf},
$d_1^{0,0}$.
On the contrary, in Epstein's work these two aspects are merged together in a 
single operator $f_*$, which is eventually compared with the identity. \\
For example, in the case
$\mathcal{G}_{\bul}= \mathcal{O}_{\bul}$, the map $d_1^{0,0}$ is canonically identified with
the derivative at the identity of the conjugation by $f$:
\begin{equation}\label{mobiusaction}
	Aut(\PP) \to {\bf Rat_D}, \; \varphi \mapsto \varphi^{-1}f \varphi,
\end{equation}
where ${\bf Rat_D}$ denotes the parameter space of rational maps of degree $D$, which
is a smooth projective variety of dimension $2D+1$.
In fact, recall that the tangent space at $f$ of ${\bf Rat_D}$
is canonically isomorphic to $H^0(f^*T_{\PP})$: if $f_t$ is an analytic path in ${\bf Rat_D}$
with $f_0=f$, then $\displaystyle \frac{d}{dt}f_t |_{t=0}\in H^0(f^*T_{\PP})$.
On the other hand, we have an isomorphism of $\mathcal{O}_{\PP}$-modules
\begin{equation}\label{dfinverse}
	Df^{-1}: f^*T_{\PP} \overset{\sim}{\longrightarrow} T_{\PP}(+ \Gamma_f),
\end{equation}
which is expressible on stalks by
$$
f^*\frac{\partial}{\partial w} \mapsto \frac{1}{Df(z)}\frac{\partial}{\partial z},
$$
where $(U,z)$ and $(V,w)$ are local coordinate around, respectively, $x,f(x)$,
such that $w=f(z)$ and $Df(z)$ denotes the derivative of $f$ in $z\in U$.
We
denote again with $Df^{-1}$ the isomorphism induced on the
respective $H^0$'s.
Note that the operation given by the composition
$Df^{-1} \circ f^{\star}$ coincides with the usual pull-back
of vector fields
$$
f^{\ast}: H^0(V, T_{\PP}) \to H^0(U,  T_{\PP}(+ \Gamma_f)),
\quad h(w)\frac{\partial}{\partial w} \mapsto \frac{h(f(z))}{Df(z)}\frac{\partial}{\partial z}.
$$
In this way, the map $Df^{-1} \circ d_1^{0,0}$ becomes the
linear map $\Delta_f\coloneqq id- f^{\ast}$,
studied by Epstein.

The fact that that $\Delta_f$ has vanishing kernel, 
cf. \cite{transversality},
is an assertion on the tangent space of
$\QQ{f}$, \textit{i.e.} ${\rm \mathbb{H}om}(\Omega_{\bul}, \mathcal{O}_{\bul}).$ 
The equivalent statement is the following.
\begin{lemma}\label{noinvariantvf}
	Let $f$ be a rational map of degree $D>1$.
	Then, there is no invariant vector field, \textit{i.e.}
	$$
	{\rm \mathbb{H}om}(\Omega_{\bul}, \mathcal{O}_{\bul})=0.
	$$
\end{lemma}
\begin{proof}
	Let us suppose that there exists a nonzero global vector field $X$
	on $\PP$ such that $f^*v=(df)v$. 
	The equation implies easily that the divisor $\mathcal{Z}$ of zeroes of $v$ 
	is backward invariant and consists of critical points of $f$. Then,
	since $\mathcal{Z}$ consists of exceptional points of $f$, cf. \cite{Milnor},
	his cardinality is less than $2$. Therefore, there is a coordinate $z$
	in which $v=z^{\pm 1}\partial_{z}$ and $f(z)=z^{\pm D}$. This yields 
	to an absurd, since we have $f^*v \neq (df)v$.
	 \end{proof}
\begin{cor}\label{noinvariantvfzero}
	Let $\Delta_{\bul}$ be a ``dynamical'' divisor. \\
	Then,
	$$
	{\rm \mathbb{H}om}(\Omega_{\bul}, \mathcal{O}(-\Delta_{\bul}))=0.
	$$
\end{cor}
The orbifold ${\bf rat_D}$ is by definition the quotient
of ${\bf Rat_D}$ by the action \eqref{mobiusaction}.
Thus, its orbifold tangent space
at $[f]$, which we denote with $T_f{\bf rat_D}$, is canonically isomorphic to
${\rm coker}(d_1^{0,0})$ and its dimension is
$$
\dim T_f{\bf rat_D}=\dim H^0(f^*T_{\PP})- \dim H^0(T_{\PP})= 2D-2.
$$
The following is an immediate consequence of Lemma
\ref{ext1_lemma} and the fact that $H^1(T_{\PP})=0$.
\begin{lemma}\label{globaldeformations}
	We have a canonical isomorphism
	$$
	{\rm \mathbb{E}xt^1}(\Omega_{\bul}, \mathcal{O}_{\bul}) \overset{\sim}{\longrightarrow} T_f{\bf rat_D}.
	$$
\end{lemma}
\begin{cor}\label{nohigherdeformations}
	For all $i \geq 2$, we have
	$$
	{\rm \mathbb{E}xt}^i(\Omega_{\bul}, \mathcal{O}_{\bul})=0.
	$$
\end{cor}
We recall the notion of \textit{Euler characteristic}, cf. \cite{BottTu}, of a spectral sequence $\{E_r\}_r$. 
Setting $e_r^{p,q}=\dim E_r^{p,q}$, the quantity
$$
\chi(E_r)=\sum_{p,q} (-1)^{p+q}\,e_r^{p,q}
$$
is constant in $r$ and it is denoted with $\chi(E_{\bul})$.
Recall that when $\{E_r\}_r$ is converging to
${\rm \mathbb{E}xt}^{\ast}(\mathcal{F}_{\bul}, \mathcal{G}_{\bul})$,
for $r \gg 0$ we have
$$
E_r^{p,q}= {\rm \mathbb{E}xt}^{p+q}(\mathcal{F}_{\bul}, \mathcal{G}_{\bul}).
$$
Hence,
$$
\chi(E_{\bul})=
\chi(\mathcal{F}_{\bul}, \mathcal{G}_{\bul})\coloneqq
\sum_i (-1)^i \, {\rm \mathbbm{e}xt}^i(\mathcal{F}_{\bul},\mathcal{G}_{\bul}),
$$
where ${\rm \mathbbm{e}xt}^i(\mathcal{F}_{\bul},\mathcal{G}_{\bul})\coloneqq
\dim_{\C} {\rm \mathbb{E}xt}^i(\mathcal{F}_{\bul}, \mathcal{G}_{\bul})$.

\begin{proof}[Proof of Corollary \ref{nohigherdeformations}]
	In the case $i \geq 3$, this is an
	immediate consequence of Corollary \ref{exti_cor}:
	by Grothendieck's vanishing theorem, for any abelian sheaf $\mathcal{F}$,
	$H^i(\PP, \mathcal{F})=0$
	for all $i \geq 2$.\\
	Let $\{E_r\}_r$ be the spectral sequence computing
	${\rm \mathbb{E}xt}^{\ast}(\Omega_{\bul},\mathcal{O}_{\bul})$.
	The value of the \textit{Euler characteristic} is given by
	(cf. Figure \ref{spectralE1P1})
	$$
	\chi(\Omega_X, \mathcal{O}_X)- \chi(f^*\Omega_X, \mathcal{O}_X),
	$$
	which for $X=\PP$ is equal to $2-2D$ by the \textit{Riemann-Roch} Theorem.\\
	Hence, setting
	$\mathbbm{e}^i\coloneqq {\rm \mathbbm{e}xt}^i(\Omega_{\bul}, \mathcal{O}_{\bul})$,
	we have
	$$
	\mathbbm{e}^0-\mathbbm{e}^1+\mathbbm{e}^2=2-2D.
	$$
	The vanishing of $\mathbbm{e}^2$ now follows from Lemma \ref{noinvariantvf}
	and Lemma \ref{globaldeformations}.
\end{proof}

We now study the group ${\rm \mathbb{E}xt}^{\ast}(\Omega_{\bul},\mathcal{O}(-\Delta_{\bul}))$,
where $\Delta_{\bul}$ is any ``dynamical'' divisor.
Observe that, with the same argument as in the proof above,
we have
$${\rm \mathbbm{e}xt}^i(\Omega_{\bul}, \mathcal{O}(-\Delta_{\bul}))=0
\quad \quad  i \geq 3.$$ 
\begin{notation}
	Let us denote with $\delta$ the nonnegative integer
	$$
	\delta\coloneqq\deg(\Delta_0)-\deg(\Delta_1)
	$$
\end{notation}

\begin{lemma}\label{lemmaRR}
	Let $\Delta_{\bul}$ be a ``dynamical'' divisor. Then,
	$$
	{\rm \mathbbm{e}xt}^1(\Omega_{\bul}, 
	\mathcal{O}(-\Delta_{\bul}))-{\rm \mathbbm{e}xt}^2(\Omega_{\bul}, \mathcal{O}(-\Delta_{\bul}))=2D-2+\delta
	$$
\end{lemma}
\begin{proof}
	The opposite of the quantity on the left is the
	\textit{Euler characteristic} $\chi(\Omega_{\bul}, \mathcal{O}(-\Delta_{\bul}))$.
	Thus, it coincides with
	$$
	\chi(f^*\Omega_{\PP}, \mathcal{O}_{\PP}(-\Delta_1))-
	\chi(\Omega_{\PP}, \mathcal{O}_{\PP}(-\Delta_0)).
	$$
	Finally, the \textit{Riemann-Roch}
	formula implies that the latter is
	$$
	\deg(f^*T_{\PP} (-\Delta_1))-\deg(T_{\PP} (-\Delta_0))=
	2D-2+\delta.
	$$
\end{proof}

The following Lemma establishes the relation between
Thurston's Theorem and our Vanishing Theorem.
\begin{lemma}\label{ThurstonvanishingLemma}
	Let $\Delta_{\bul}$ be a ``dynamical'' divisor on ${\rm E}_f$. \\
	We have a natural isomorphism
	$$
	{\rm\mathbb{E}xt}^{2}(\Omega_{\bul},\mathcal{O}(-\Delta_{\bul}))^{\vee}
	\overset{\sim}{\longrightarrow}
	\ker\left( H^0(\Omega_{\PP}^{\otimes 2}(+\Delta_1-\Gamma_f))
	\overset{\nabla_f}{\longrightarrow}
	H^0(\Omega_{\PP}^{\otimes 2}(+\Delta_0))\right)
	$$
\end{lemma}
\begin{proof}
	We know by Corollary \ref{exti_cor} 
	that ${\rm\mathbb{E}xt}^{2}(\Omega_{\bul},\mathcal{O}(-\Delta_{\bul}))$
	is naturally isomorphic to the co-kernel of
	$$
	H^1(T_{\PP} (-\Delta_0)) \overset{d_1^{1,0}}{\longrightarrow}
	H^1(f^*T_{\PP} (-\Delta_1)).
	$$
	Since all the spaces involved are finite dimensional
	vector spaces, its dual is the kernel of the transpose map $d'\coloneqq d_1^{{1,0}^{\top}}$.\\
	Via \textit{Serre duality}, the latter is canonically identified with a map
	$$
	H^0(\Omega_{\PP}\otimes f^*\Omega_{\PP}(+\Delta_1))
	\overset{d'}{\longrightarrow}
	H^0(\Omega_{\PP}^{\otimes 2}(+\Delta_0)).
	$$
	On the other hand, we have an isomorphism
	$$
	H^0(\Omega_{\PP}\otimes f^*\Omega_{\PP}(+\Delta_1))
	\overset{\eta}{\longleftarrow}
	H^0(\Omega_{\PP}^{\otimes 2}(+\Delta_1-\Gamma_f)),
	$$
	obtained by transposing the map
	$$
	Df^{-1}\otimes id :H^1(f^*T_{\PP} (-\Delta_1))
	\overset{\sim}{\longrightarrow}
	H^1(T_{\PP} (+\Gamma_f-\Delta_1)).
	$$
	In a local coordinate $(U,z)$, the map $\eta$ is expressible as
	follows: for $\omega$ a local section of
	$\Omega_{\PP}(+\Delta_1-\Gamma_f)$, we have
	$\omega(z) \otimes dz \mapsto \frac{\omega(z)}{Df(z)}\otimes f^*dz$.
	The conclusion follows immediately
	once we show that the diagram
	\begin{equation}\label{diagramext2vanishing}
		\begin{tikzcd}
			H^0(\Omega_{\PP}\otimes f^*\Omega_{\PP}(+\Delta_1))
			\arrow[r,"d'"] \arrow[from=d, "\eta"] &
			H^0(\Omega_{\PP}^{\otimes 2}(+\Delta_0)) \\
			H^0(\Omega_{\PP}^{\otimes 2}(+\Delta_1-\Gamma_f)) \arrow[ur, "\nabla_f"']
		\end{tikzcd}
	\end{equation}
	commutes.\\
	In order to compute the map $d_1^{1,0}=df-f^{\star}$ we use the equivalence between Grothendieck and
	Dolbeault cohomology.
	We take a sufficiently fine open covering $U=\coprod_{i} U_i$ of $\PP \setminus S(f)$.
	An element $\xi \in H^1(T_{\PP} (-\Delta_0))$ is represented by a cocycle
	$\{\xi_i\}_i$ of the form $\xi_i=\overline{\tau}_i\otimes v_i$, where $\tau_i,v_i$
	are respectively a nonvanishing holomorphic form and a vector field on $U_i$,
	$v_i$ vanishing on the divisor $\Delta_0$ with multiplicity.
	Now, $d_1^{1,0}(\xi)=\{d_1^{1,0}(\xi_i)\}$ and we have by functoriality
	$df(\xi_i)=\overline{\tau}_i\otimes df(v_i)$. It follows that the Serre transpose of $df$
	is the $\mathcal{O}_{\PP}$-linear map $df^{\top}$, which is locally of the form
	$\omega \otimes f^*dz \mapsto \omega\otimes df^{\top}(f^*dz)= Df(z)\omega\otimes dz$,
	and thus it is immediate to check that $df^{\top}\eta= i$, where $i$ is the natural inclusion.
	On the other hand,
	$f^{\star}(\xi_i)=f^*\overline{\tau}_i \otimes f^{\star}(v_i)$,
	where $f^*\overline{\tau}_i$ denotes the usual pullback
	of differential forms.
	The map $(f^{\star})^{\top}$ is computed as follows:
	by the change of variables $z=f(w)$, we have
	\begin{equation}\label{serreduality}
		\begin{aligned}
			& \int_{f^{-1}U_i} \big((f^*\overline{\tau}_i \otimes f^{\star}(v_i))
			\wedge (\omega \otimes f^*dz)\big)(w)= &\\
			& \int_{f^{-1}U_i} \big((f^*\overline{\tau}_i \wedge \omega) \langle f^{\star}(v_i), f^*dz\rangle\big)(w)=
			& \\
			& \int_{U_i} \big( (\overline{\tau}_i \wedge f_*\omega) \langle v_i, dz\rangle \big)(z)=\\
			&  \int_{U_i} \big( (\overline{\tau}_i \otimes v_i) \wedge (f_*\omega \otimes dz) \big)(z).
		\end{aligned}
	\end{equation}
	Hence $(f^{\star})^{\top}(\omega \otimes f^*dz)=f_*\omega \otimes dz$.\\
	Now, if $\phi: U_i \to \PP$ is a local inverse branch of $f$, we have
	$$(Df(\phi(z)))^{-1}=D\phi(z),$$ and hence locally
	$$
	(f^{\star})^{\top}\eta(\omega \otimes dz)=
	f_*\left( \frac{\omega}{Df(z)} \right)\otimes dz=f_*(\omega \otimes dz).
	$$
\end{proof}
Thanks to the above Lemma we can formulate Theorems \ref{T-E} and \ref{Eps}. Recall, cf. \cite{MR1251582},  \cite{1999math......2158E}, that
 the condition $f_*q=q$ for an integrable $q$
implies $f^*q=\lambda q$
for a real constant $\lambda$.
Hence $f$ is a Latt\`es map and lifts to an affine map
$z \mapsto \alpha z + \beta$. Thus, $\alpha$ is real and the only possibility is
that $\alpha$ is an integer satisfying $|\alpha|^2=D$, since the integrability
condition on $q$ excludes the case $f(z)=z^{\pm D}$.
The only possibility remaining is the orbifold $(2,2,2,2)$ Latt\`es map.
In this case the holomorphic
quadratic differential $dz^{\otimes 2}$ on $\C$ descends to a quadratic differential
with $4$ simple poles that is invariant under the pushforward operator.

Consequently, we deduce the following equivalent statement
of Thurston's original theorem (which, in Epstein's formalism,
 is equivalent to the injectivity of $\nabla_f$ on $\mathcal{Q}(\PP)$):
\begin{theorem}\label{Thurstonvanishing}{\bf{[Thurston vanishing]}}\\
	Let $f: \PP \to \PP$ a rational map of degree $D$.
	If $\Delta_{\bul}$ is a ``dynamical'' (effective)
	divisor on $ \QQ{f}$ with $deg_x(\Delta_0)\leq 1 \; \forall x$. Then,
	$$
	{\rm\mathbbm{e}xt}^{2}(\Omega_{\bul},\mathcal{O}(-\Delta_{\bul})
	=\begin{cases}
		1, & \mbox{if } f \mbox{ is a $(2,2,2,2)$ Latt\`es map} \\
		0, & \mbox{otherwise}.
	\end{cases}
	$$
\end{theorem}
The equivalent formulation of Epstein's result requires some definitions.
\begin{definition}\label{rigiddivisordef}
	Let $\Delta$ be an effective divisor on $\PP$ whose support 
	contains a  (disjoint) union of
	cycles $C=\cup_iC_i$ of $f$. 
	Let us assume 
	that there are no superattracting cycles in $|\Delta|$.
	Recall that if $C_i$ is parabolic of multiplier $\rho \in \mu_{q_i}$,
	we denote with $\nu_i$ its parabolic multiplicity,
	and with $N_i=\nu_i q_i$.
	We say that $\Delta$ is ``rigid'' for $f$ if
	for each $x$ in the support of $\Delta$ we have
	$$
	deg_x(\Delta)\leq
	\begin{cases}
		2, & \mbox{if } x\in C_i \mbox{ and } C_i \mbox{ is attracting/irrationally indifferent;}\\
		2N_i+1, & \mbox{if } x\in C_i \mbox{ and } C_i \mbox{ is parabolic-repelling;} \\
		2N_i+2, & \mbox{if } x\in C_i \mbox{ and } C_i \mbox{ is parabolic-nonrepelling;}\\
		1, & \mbox{if } x\in C_i \mbox{ and } C_i \mbox{ is repelling;} 
	\end{cases}
	$$
\end{definition}
A divisor satisfying the equality is said to be ``sharp rigid''.
Let $\Delta$ be a divisor supported on a cycle $C$ of $f$.
Following the notation of Remark \ref{hatrmk},
the kernel
$$
K\coloneqq\ker( \nabla_f: \mathcal{M}(\PP, C)\to \mathcal{M}(\PP, C \cup S_f))
$$
is contained in $\hat{\mathcal{Q}}(f,C)$.
Now, if $\Delta$ is rigid for $f$, from Lemma \ref{divergence1} and \ref{divergence2} we get
$$
H^0(\Omega^{\otimes 2}(+\Delta))\cap \hat{\mathcal{Q}}(f,C) \subset \hat{\mathcal{Q}}^{\flat}(f,C),
$$
which implies, in view of Theorem \ref{Eps},
$$
\ker(\nabla_f|_{H^0(\Omega^{\otimes 2}(+\Delta))})\subset K \cap \hat{\mathcal{Q}}^{\flat}(f,C)= 0.
$$
The discussion above justifies the following definition.
\begin{definition}
	We say that a ``dynamical'' divisor $\Delta_{\bul}$ on
	$\QQ{f}$ is \textit{rigid}
	if the divisor $\Delta\coloneqq\Delta_1-\Gamma_f$ is rigid for $f$.
	In particular, if $\Delta_{\bul}$ is rigid, then
	$$
	\Gamma_f \preceq \Delta_1 \mbox{  and  } \Gamma_f + {\bf S_f} \preceq \Delta_0.
	$$
\end{definition}
Finally, Epstein's extension of Infinitesimal
Thurston Rigidity takes the following form.
\begin{theorem}\label{Epsteinvanishing}{\bf{[Epstein vanishing]}}\\
	Let $\Delta_{\bul}$ be a rigid ``dynamical'' divisor on $\QQ{f}$. If $f$
	is not a Latt\`es map, then
	$$
	{\rm\mathbb{E}xt}^{2}(\Omega_{\bul},\mathcal{O}(-\Delta_{\bul})=0.
	$$
\end{theorem}
\subsection{Local computations}\label{section:localcomputations}
To complete the proof of Claim \ref{claim1} it is left to show
that there is a choice of $\Delta_{\bul}$ for which the dimension of
${\rm\mathbb{H}om}(\Omega_{\bul},\mathcal{O}_{\Delta_{\bul}})$,
which we abbreviate with $\mathbbm{e}^0$
is exactly
$$
\mathbbm{e}^0=2D-2+\gamma_A.
$$
In fact, the exactness of \eqref{simplicialF-S}
allows us to calculate indirectly the dimension of
${\rm\mathbb{E}xt}^{1}(\Omega_{\bul},\mathcal{O}_{\Delta_{\bul}})$.
The nature of this computation is local
since the sheaves involved are skyscraper sheaves.
We are in a situation analogue to what we find in \cite{1999math......2158E}
(cf. \eqref{cyclereduction}). This should not be a surprise since the
two computations are dual: the vector space $\mathcal{D}_x^{k+1} $
of Lemma \ref{divergence1}
is canonically dual to
$T_x^k\coloneqq T_{X,x}\otimes (\mathfrak{m}_x^k/\mathfrak{m}_x^{k+1})$, via the pairing
\begin{equation*}
	\begin{tikzcd}
		\mathcal{D}_x^{k+1}  \times T_x^k
		\arrow[r] &
		\C\\
		([q]_x, [v]_x) \arrow[r,mapsto] & {\rm Res}_x(q \otimes v).
	\end{tikzcd}
\end{equation*}
Moreover, we have the following adjoint relation: for a fixed point $x$ of $f$
$$
{\rm Res}_x(q\otimes (f^*v-v))={\rm Res}_x((f_*q-q)\otimes v).
$$
Observe that the contribution to $\mathbbm{e}^0$ coming from the cycles
of $f$ is separated and we can compute it cycle by cycle:
if $\Delta,\Delta'$ are forward invariant
divisors with disjoint support, we have canonically
\begin{equation}\label{splitproperty}
	{\rm Hom}(\Omega_X, \mathcal{O}_{\Delta + \Delta'})=
	{\rm Hom}(\Omega_X, \mathcal{O}_{\Delta })\prod {\rm Hom}(\Omega_X, \mathcal{O}_{\Delta'}).
\end{equation}
Thus, the map $d_1^{0,0}$ splits into a product, together with its kernel.
Moreover, it is possible to show that the computation,
in the case in which the cycle
is not superattracting, reduces
to the study of fixed points. We use \textit{mutatis mutandis}
the same argument as in $\eqref{cyclereduction}$.\\
Let $\langle x\rangle=\{x=x_0,\dots,x_{p-1}\}$ be a cycle of $f$. For $N$ a positive integer, $0\leq k\leq p-1$, we write
$$
V_{k}=H^0(T_X \otimes \mathcal{O}_{N[x_k]}), \quad V'_{k}=H^0(f^*T_X \otimes \mathcal{O}_{N[x_k]}).
$$
The map
$$
d_1^{0,0}= df-f^{\star}: H^0(T_X\otimes \mathcal{O}_{N[ \langle x \rangle]}) \to
H^0(f^*T_X\otimes \mathcal{O}_{N[ \langle x \rangle]})
$$
may be schematized as follows.
\begin{equation*}
	\begin{tikzcd}
		\dots &  {V_{p-1}} \arrow[d, "df"] &
		{V_1} \arrow[d, "df"] \arrow[dl, "f^{\star}"] &
		{V_2} \arrow[d, "df"] \arrow[dl, "f^{\star}"] & \dots \\
		\dots &  {V'_{p-1}} &
		{V'_1}  &
		{V'_2}  & \dots
	\end{tikzcd}
\end{equation*}
If $\langle x\rangle$ is not
superattracting, the map $(df)_{x_k}: V_k \to V'_k$ is an isomorphism.
Thus, if $[v]_k$ is a germ of holomorphic vector field at $x_k$,
an element $v=\displaystyle \oplus_{k=0}^{p-1}[v]_k \in \ker(d_1^{0,0})$
satisfies $[v]_k=(df)^{-1}f^{\star} [v]_{k+1}= f^*[v]_{k+1}$. This implies
$[v]_0=(f^p)^*[v]_0$ or, equivalently, $(df^p)[v]_0=(f^{p})^{\star}[v]_0$.
In other words, the dashed arrow in the following
diagram with exact rows,
\begin{equation*}
	\begin{tikzcd}
		0 \arrow[r] & K(f) \arrow[d, "\pi", dashrightarrow] \arrow[r] &
		Hom(\Omega_X, \mathcal{O}_{N[ \langle x \rangle]}) \arrow[d, "\pi"]
		\arrow[r, "d_1^{0,0}"] & Hom(f^*\Omega_X, \mathcal{O}_{N[ \langle x \rangle]})
		\arrow[d, "\pi (df^p)^{\top}"]\\
		0 \arrow[r] & K(f^p) \arrow[r] & Hom(\Omega_X, \mathcal{O}_{N[x]})
		\arrow[r, "d_1^{0,0}"] & Hom((f^p)^*\Omega_X, \mathcal{O}_{N[x]})
	\end{tikzcd}
\end{equation*}
is well-defined. It is indeed an isomorphism, with inverse given by
$$
\pi^{-1}( [v]_x)=\bigoplus\limits_{k=0}^{p-1}(\phi_k)^* [v]_x,
$$
where
$(\phi_k)^*:  T_{X,x} \to T_{X,x_k}$ is the tangent map of $f^k$ in $x$.\\
Recall that if $E\subset \PP$ is a finite forward invariant set,
we can define a \textit{forward invariant }
``dynamical'' divisor $E_{\bul}$ with $\Delta_0=\Delta_1=[E]$.\\
The following computation is
our version of Epstein's computation of invariant divergences.
\begin{lemma}\label{locallemma1}
	Let $f$ be an analytic germ fixing $x$.
	The dimension $\mathbbm{e}^0(N)$ of the vector space
	$\mathbb{E}^0(N)\coloneqq {\rm \mathbb{H}om}(\Omega_{\bul}, \mathcal{O}_{N[x]_{\bul}})$
	is computed in terms of the associated formal invariants:
	\begin{enumerate}[a)]
		\item If $x$ is superattracting,
		then,
		$$
		\mathbbm{e}^0(N)=\min\{ N-1, deg_x(f)-1\};
		$$
		\item If $x$ is an attracting, repelling, or
		 linearizable irrationally indifferent fixed point of $f$,
		then for all $N \geq 2$
		$$
		\mathbbm{e}^0(N)=1.
		$$
		On the other hand, if $x$ is a Cremer point, we get $\mathbbm{e}^0(2)=1$.
		\item If $x$ is parabolic, with multiplier $\rho \in \mu_{n}$ and
		parabolic multiplicity $\nu_x$, then setting $N_x=\nu_x n$, we 
		distinguish the following relevant cases
		$$
		\mathbbm{e}^0(N)=
		\begin{cases}
			1, & \mbox{if } N=2;\\
			\nu, & \mbox{if } N_x\leq N\leq 2N_x+1; \\
			\nu +1, & \mbox{if } N = 2N_x +2;
		\end{cases}
		$$
		\item Otherwise, $f$ has finite order and $\mathbbm{e}^0(N)$ is infinite.
	\end{enumerate}
\end{lemma}
\begin{proof}
	Let $N \geq 1$ be an integer. Recall that the space
	$\mathbb{E}^0(N)$
	is the equalizer of
	\begin{equation}\label{localcomp1}
		\begin{tikzcd}[row sep=1pc, column sep=1.8pc]
			{\rm Hom}(\Omega_X, \mathcal{O}_{N[x]})
			\arrow[r,shift left, "f^{\star}"]
			\arrow[r,shift right, "df"'] &
			{\rm Hom}(\Omega_X, \mathcal{O}_{N[x]})
		\end{tikzcd}
	\end{equation}
	In a local coordinate $\zeta$ vanishing at $x$, we have
	\begin{equation}\label{localcomp2}
		\begin{tikzcd}[row sep=0.4em, column sep=1em]
			& t(f(\zeta))[f^*\frac{\partial}{\partial \zeta}] \mbox{ mod } \zeta^N \\
			t= t(\zeta)[\frac{\partial}{\partial \zeta}] \mbox{ mod } \zeta^N
			\arrow[ru, mapsto, bend left, shift right={0.1em}, end anchor={[xshift=0.1ex]west}, "f^{\star}"]
			\arrow[rd, mapsto, bend right, shift left={0.1em}, end anchor={[xshift=0.1ex]west}, "df"] \\
			&t(\zeta) Df(\zeta)[f^*\frac{\partial}{\partial \zeta}] \mbox{ mod } \zeta^N
		\end{tikzcd}
	\end{equation}
	We set $t(\zeta)=t_i\zeta^{i}$ mod $\zeta^N$, for some constants $t_i\in \C$.
	\begin{enumerate}[a)]
		\item For a B\"{o}ttcher's coordinate $\zeta$, we can write $f: \zeta \mapsto \zeta^e$.
		We find the following equations:
		$$
		t_i\zeta^{ei}=et_i\zeta^{i + e-1} \quad \mbox{mod } \zeta^N,
		$$
		from which we deduce immediately $t_0=0$. If $N=e$ this is all.
		As $N$ grows, we find new independent equations $t_1=0, t_2=0, \dots$, hence the
		dimension of the equalizer is always $e-1$. When $N=e$ it coincides
		with the subspace $T_x \otimes \left( \mathfrak{m}_x/ \mathfrak{m}^{N}_x \right)$.
		\item if there is a linearizing coordinate $\zeta$, \textit{i.e.}
		such that $f: \zeta \mapsto \rho \zeta$.
		In this coordinate, we find
		\begin{equation}\label{linearcomp}
			\rho^it_i\zeta^i=\rho t_i \zeta^i \quad \mbox{mod } \zeta^N.
		\end{equation}
		If $\rho$ is not a root a unity, the equalizer is the subspace
		generated by $\zeta[\frac{\partial}{\partial \zeta}]$. 
		If $x$ is a Cremer point, we have only $f: \zeta \mapsto \rho \zeta + \mathcal{O}(\zeta^2)$, in which case equation \eqref{linearcomp} holds only for $N=2$.
		\item Let us fix a preferred coordinate $\zeta$ (cf. \eqref{preferred}),
		for which we have
		$$
		f: \zeta \mapsto \rho\zeta(1+\zeta^{N_x} + \alpha \zeta^{2N_x}) + \mathcal{O}(\zeta^N).
		$$
		In this coordinate, we find
		\begin{equation*}
			\begin{aligned}
				f^{\star}t=   &\rho^it_i\zeta^i(1+ \zeta^{N_x}+ \alpha\zeta^{2N_x} + \mathcal{O}(\zeta^{2N_x+1}))^i=\\
				= & \rho^it_i\zeta^i(1+i\zeta^{N_x}+ i( \alpha+ \frac{i-1}{2})\zeta^{2N_x}
				) \quad \mbox{mod } \zeta^N
			\end{aligned}
		\end{equation*}
		and
		\begin{equation*}
			(df)t= \rho t_i \zeta^i(1+(N_x+1)\zeta^{N_x}+ \alpha(2N_x+1)\zeta^{2N_x}) \quad \mbox{mod } \zeta^N
		\end{equation*}
		We distinguish the following cases.
		\begin{enumerate}[$\bul$]
			\item For $N=2$, the dimension of the kernel is clearly 1; 
			\item For $N=N_x$, we find \eqref{linearcomp} and hence
			$$
			\mathbb{E}^0(N)= {\rm span}_{\C}\{\zeta^{kn+1}, k=0,\dots, \nu_x-1\},
			$$
			which is dual to the space $\mathcal{D}^{\circ}_x(f)$
			we find in \cite{1999math......2158E};
			\item For $N_x< N \leq 2N_x+1$ we find \eqref{linearcomp} for $i< N_x$.
			Let us fix $N_x\leq i<N$. We find
			$$
			\rho^i (t_i+ (i-N_x)t_{i-N_x})=\rho(t_i+(N_x+1)t_{i-N_x})
			$$
			Thus, we must have $t_i=0$ except for $i$ of the form $i=kn+1$.
			In the latter case, the term $t_j$, for $j=(k-\nu_x)n+ 1$, is annihilated.
			Hence, $\mathbbm{e}^0(N)=\nu_x$; \\
			Observe that for $i=2N_x$ the equation
			to solve is slightly different. However, it yields always $t_{2N_x}=0$;
			\item Let $N=2N_x+2$. For $i=2N_x+1$, we find
			$$
			\rho ( t_{2N_x+1}+ (N_x+1)t_{N_x+1})=
			\rho ( t_{2N_x+1}+ (N_x+1)t_{N_x+1}),
			$$
			so the term $t_{2N_x+1}$ adds one dimension to the kernel.
		\end{enumerate}
	\end{enumerate}
\end{proof}

\begin{definition}
	Let $\langle x\rangle$ be a cycle of $f$ which is not superattracting.
	We denote with $\Delta_{\bul}(\langle x\rangle)$ the ``dynamical'' divisor
	which is uniquely determind by the following conditions:
	\begin{enumerate}[$\bul$]
		\item $\Delta_0=\Delta_1$;
		\item $|\Delta_0|=\langle x\rangle$;
		\item $\Delta_0$ is sharp rigid for $f$.
	\end{enumerate}
\end{definition}

The following is an immediate consequence of the above Lemma.
\begin{cor}\label{localHomEps}
	$$
	{\rm \mathbb{H}om}(\Omega_{\bul}, \mathcal{O}_{\Delta_{\bul}(\langle x\rangle)})
	=\gamma_{\langle x\rangle}
	$$
\end{cor}
It is left to analyze the contribution to $ \mathbbm{e}^0$
of the superattracting cycles. In Epstein's theory they give no
contribution to the space of invariant divergences, while in our
discussion they play a fundamental role.
In fact, in order to apply Theorem \ref{Epsteinvanishing},
the ``dynamical'' divisor $\Delta_{\bul}$
realizing \eqref{simplicialF-S}
must be choosen among the rigid divisors.
Hence, $\Delta_0$
contains with multiplicity the divisor $\Gamma_f+ {\bf S_f}$,
so we have to take account of all the ramification,
not only the periodic part.

For $p \in \PP \setminus C_f$ we denote with
$\mathcal{O}^{-}_{(n)}(p)= \bigcup_{k=1}^n (f^k)^{-1}(p)$ its $n$-th iterated preimage.
We set
$$
\Delta_p^{(n)}=\sum_{x \in  \mathcal{O}^{-}_{(n)}(p)} \deg_x(f)[x],
$$
and
$$
\Gamma_p^{(n)}\coloneqq \Gamma_f \wedge \Delta_p^{(n)}.
$$
\begin{lemma}\label{locallemma2}
	Let $p \in \PP\setminus C_f$ and $n$ a positive integer.
	Consider the ``dynamical'' divisor $\Delta^{(n)}_{\bul, p}$
	defined by setting $\Delta^{(n)}_{0,p}= [p]+ \Delta_p^{(n)}$ and
	$\Delta^{(n)}_{1,p}= \Delta_p^{(n)}$. We have
	$$
	\dim_{\C}\mathbb{H}{\rm om}(\Omega_{\bul}, \mathcal{O}_{\Delta^{(n)}_{\bul,p}})=\max\{ \deg(\Gamma_p^{(n)}),1\}
	$$
\end{lemma}
\begin{proof}
	We proceed by induction. To simplify notation, we set
	$$
	\mathbb{E}^0(n)={\rm \mathbb{H}om}(\Omega_{\bul}, \mathcal{O}_{\Delta_{\bul,p}^{(n)}})
	$$
	and we write $\mathbbm{e}^0(n)$ for its dimension.
	For each $n\geq1$ we have two cases:
	\begin{enumerate}[a)]
		\item $\Gamma_p^{(n)}= \emptyset$;
		\item $\Gamma_p^{(n)}\neq \emptyset$.
	\end{enumerate}
	Let us choose a local coordinate $\zeta_0$ at $p$.
	Let $n=1$ and assume we are in case ${\rm a})$.
	If $\displaystyle f^{*}[p]=\sum_{j=1}^D [y_j]$,
	we can choose local coordinates $\zeta_j$, $1\leq j \leq D$
	at $y_j$, such that the map $f$ in these coordinates is the identity:
	$\zeta_j \mapsto \zeta_0$.
	We have
	$$
	\begin{aligned}
		&{\rm Hom}(\Omega, \mathcal{O}_p) \simeq \C\left[\frac{\partial}{\partial \zeta_0}\right],  \\
		& {\rm Hom}(\Omega, \mathcal{O}_{y_j}) \simeq \C\left[\frac{\partial}{\partial \zeta_j}\right],\; \forall j \in[1,D]
	\end{aligned}
	$$
	Note that the map
	$$
	d_1^{0,0}: T_p \times \prod_{j=1}^{D}T_{y_j} \to \prod_{j=1}^{D}T_{y_j},
	$$
	written in these coordinates, reads as follows
	\begin{equation}\label{computingscheme}
		\begin{tikzcd}[row sep=1pc, column sep=1pc]
			& \Big( t^0[f^*\frac{\partial}{\partial \zeta_0}],
			\dots
			,t^0[f^*\frac{\partial}{\partial \zeta_0}] \Big)
			\\
			\Big(t^0[\frac{\partial}{\partial \zeta_0}],
			t^1[\frac{\partial}{\partial \zeta_1}],
			\dots
			,t^D[\frac{\partial}{\partial \zeta_D}]\Big)
			\arrow[ru, mapsto, bend left, shift right={0.2em}, end anchor={[xshift=1ex]west}, "f^{\star}"]
			\arrow[rd, mapsto, bend right,  shift right={0.2em}, end anchor={[xshift=1ex]west}, "df"] \\
			&
			\Big( t^1[f^*\frac{\partial}{\partial \zeta_0}],
			\dots,
			t^D[f^*\frac{\partial}{\partial \zeta_0}] \Big)
		\end{tikzcd}
	\end{equation}
	Thus, the space $\mathbb{E}^0(1)$
	is identified with the 1-dimensional
	subspace of $\C^{D+1}$ given by $t^0=t^1=\dots=t^D$.
	An easy sub-induction shows that if we stay in case a)
	for $m$ stages, then the same computation applies,
	\textit{i.e.} for each $1\leq k \leq m-1$, the map
	has the form \eqref{computingscheme} with $p$ replaced by
	any $z \in (f^{k})^{-1}(p)$.
	The result is the same: the space
	$\mathbb{E}^0(m)$ is identified with the ``diagonal"
	1-dimensional subspace.\\
	Assume now we enter case b) at the stage $m\geq 1$.\\
	Choose an element $z \in (f^{m-1})^{-1}(p)$ and observe that the restriction
	of $d_1^{0,0}$ to the subspace
	$$
	{\rm Hom}(\Omega, \mathcal{O}_{z+ \Delta_z^{(1)}})
	$$
	has the following representation: let $\zeta$ be the coordinate at $z$ choosen
	inductively as above. If $f^{*}[z]= e_1[x_1]+ \dots + e_r[x_r]+ y_{r+1}, \dots, y_{r+r'}$,
	we have coordinates $\zeta_{r+1}, \dots, \zeta_{r+r'}$ at $y_j$ as before.
	We choose coordinates $\xi_1, \dots, \xi_r$ at $x_1, \dots, x_r$ such that
	the map $f$ in these coordinates is $\xi_j \mapsto \xi_j^{e_j}$.\\
	The representation of $f^{\star}$ is the same as before, while $df$
	is schematized as follows: for $t_i^j \in \C$ we have
	\begin{equation}\label{computingschemeR}
		\begin{tikzcd}[row sep=3.5em, column sep=0.5em]
			\Big(t^0[\frac{\partial}{\partial \zeta}], &[-14pt]
			( \sum_{i=0}^{e_1-1}t_i^1\xi_1^i)[\frac{\partial}{\partial \xi_1}],
			\dots
			,t^{r+r'}[\frac{\partial}{\partial \zeta_{r+r'}}]\Big)
			\arrow[d, mapsto, "df"]\\
			&
			\Big((t_0^1e_1\xi_1^{e_1-1})[f^*\frac{\partial}{\partial \zeta}],
			\dots,
			t^{r+r'}[f^*\frac{\partial}{\partial \zeta}] \Big)
		\end{tikzcd}
	\end{equation}
	and the kernel of (the restriction of) $d_1^{0,0}$ is determined by
	$$
	\begin{cases}
		t^0=0, \\
		t_0^j=0, & \mbox{for } j=1,\dots, r,  \\
		t^{r+j}=0, & \mbox{for } j=1,\dots, r'.
	\end{cases}
	$$
	There are $r+r'+1$ independent equations in
	$r+r'+1+\deg(\Gamma_p^{(m)}\wedge f^*[z])$ variables.
	Note that the first equation implies the vanishing of all
	the germs of vector fields at the points in $\Delta_0^{(m-1)}$.
	Consequently,
	the count of the dimension of $\mathbb{E}^0(m)$ is
	$$
	\sum_{z \in (f^{m-1})^{-1}(p)}
	\deg(\Gamma_p^{(m)}\wedge f^*[z])= \deg(\Gamma_p^{(m)}).
	$$
	A sub-induction starting at $m$ now proves the Lemma for $k \geq m$. 
	For each $z \in (f^k)^{-1}(p)$
	we can choose coordinates at the points in $f^{-1}(z)$
	and reproduce exactly the same computation as above, 
	with the only difference that now the ``$z$" slot
	may be $e$-dimensional, with $e\geq 1$.
	However, this doesn't change the computation, and again only the germs
	in the subspace $ \mathfrak{m}_z/\mathfrak{m}^{\deg_z(f)}_z$ are annihilated by $d_1^{0,0}$.
	In particular, there is no new condition on the germs of vector fields at $\Delta_p^{(k)}$
	to be added. Let $\overline{f^*[z]}\coloneqq f^*[z] \wedge (\Delta_p^{(k+1)}-\Delta_p^{(k)})$.
	The count of the dimension at the stage $k+1$ is
	$$
	\deg(\Gamma_p^{(k)})+ \sum_{z \in  (f^k)^{-1}(p) }
	\deg(\Gamma_p^{(k+1)}\wedge \overline{f^*[z]})=  \deg(\Gamma_p^{(k+1)}).
	$$
\end{proof}
A (slight) variation of the argument used in the proof above shows the following result. 
Let us consider the entire forward orbit of the critical points $C_f$,
which we denote by $\mathcal{P}^0_f\coloneqq C_f \cup \mathcal{P}_f$,
and consider the associated formal sum
$$
\Lambda^{+}_f\coloneqq \sum_{x \in \mathcal{P}^0_f} \deg_x(f)[x].
$$
Clearly, $\Lambda^{+}_f$ is a divisor on $\PP$ if and only if $f$ is \textit{post-critically finite}. In 
\cite{MR1251582} we find a nice characterization of these maps, 
which have only repelling cycles and this is the 
content of the  Fatou-Shishikura inequality in this case, cf. \ref{F-S}.
Since $\Lambda^{+}_f$ is a forward invariant set, 
we would like to consider the associated forward invariant divisor on $\QQ{f}$. 
Although this is not always possible, for any rational map $f$,
 we can define the truncation of $\Lambda^{+}_f$
 as follows: let $N>0$ large enough such that $S_{f^N}$ contains 
 all the critical orbit relations and all the cycles in the post-critical set 
 and define the truncation $\Lambda^N_f$ of $\Lambda^+_f$ as
 \begin{equation}\label{criticaldivisortruncated}
 	\Lambda^N_f=\sum_{x \in C_f \cup S_{f^N}}\deg_x(f)[x],
 \end{equation}
and by $\Lambda^N_{\bul}$ the associated ``dynamical'' divisor, namely
\begin{equation}\label{dynamicalcriticaldiv}
	\Lambda^N_{0}\coloneqq	\Lambda^{N+1}_f, \quad \Lambda^N_{1}\coloneqq \Lambda^N_f.
\end{equation}
\begin{cor}\label{computationExtcriticaldivisor}
	We have 
	$$
	\dim_{\C}\mathbb{E}{\rm xt^1}(\Omega_{\bul}, \mathcal{O}_{\Lambda^N_{\bul}})= 
	deg(\Gamma_f)-\delta_f.
	$$
\end{cor}
\begin{proof}
Note that by construction we have 
$$
\deg(\Lambda^N_0)- \deg(\Lambda^N_1)=\delta_f.
$$
Moreover, since $\Lambda^N_{\bul}$ is a rigid dynamical divisor,
it is sufficient, cf. \ref{lemmaRR}, \ref{simplicialF-S}, to prove that
$$
\dim_{\C}\mathbb{H}{\rm om}(\Omega_{\bul}, \mathcal{O}_{\Lambda^N_{\bul}})= \deg(\Gamma_f).
$$
The proof is analogous to the proof of \ref{locallemma2}:
for each $p \in f^{N+1}(C_f)$ the computation works the same as it worked for
$\mathbb{H}{\rm om}(\Omega_{\bul}, \mathcal{O}_{\Delta^{(N)}_{\bul,p}})$,
 with the only difference given by the fact that
 we are now selecting only a subset of the counter-image of each $p$. 
 However, this does not change the count of the dimension, 
 since the total dimension of the domain of $d_1^{0,0}$ is reduced.
\end{proof}
As a consequence of the above computation we deduce the following.
\begin{cor}
	Let $\langle x\rangle$ be a superattracting cycle,
	$\Delta_0(\langle x\rangle)$ the divisor
	$$
	\Delta_0(\langle x\rangle)=\sum_{p\in \langle x\rangle}\deg_p(f)[p],
	$$
	and $\Delta_{\bul}(\langle x\rangle)$ the correspondent
	forward invariant ``dynamical'' divisor.
	Then,
	$$
	{\rm \mathbb{H}om}(\Omega_{\bul}, \mathcal{O}_{\Delta_{\bul}(\langle x\rangle)})=
	\bigoplus_{w \in C_f\cap \langle x\rangle} 
\left( 	\mathfrak{m}_{w}/ \mathfrak{m}^{\deg_w(f)}_{w} \right)
	$$
	and its dimension is $\deg(\Gamma_f|_{\langle x\rangle})$.
\end{cor}

\begin{proof}[{\bf Proof of Claim \ref{claim1}}]
	Let $N>1$ be large enough as in \eqref{dynamicalcriticaldiv} and consider a finite collection of cycles $\mathcal{C}=\{C_i\}_i$ of $f$
	such that $\mathcal{C}$ contains no superattracting cycles. 
	Let us set the following notation 
	(which is very similar, yet slightly different from \ref{localHomEps}):
	\begin{enumerate}[$\bul$]
		\item If $C_i$ is disjoint
		from the post-critical set, we let $\Delta[C_i]$ be the sharp
		rigid divisor on $\PP$ supported on $C_i$;
		\item If $C_i$ is contained in the post-critical set, 
		we let $\Delta[C_i]+ [C_i]$ be the sharp rigid divisor 
		on $\PP$ supported on $C_i$.
	\end{enumerate}
In each case, let $\Delta_{\bul}[C_i]$
the correspondent divisor on $\QQ{f}$ supported on $\Delta[C_i]$.
	The good choice for the ``dynamical'' divisor $\Delta_{\bul}$
	in \ref{claim1} is the following
	\begin{equation}\label{sharprigiddivisorEf}
		\Delta_{\bul}\coloneqq\Lambda^N_{\bul} + \sum_{i}\Delta_{\bul}(C_i).
	\end{equation}
	Observe that with this choice of $\Delta_{\bul}$,
	$$
	\deg(\Delta_0)-\deg(\Delta_1)=\delta_f.
	$$
	Moreover $\Gamma_f \preceq \Delta_1$ and
	their difference is, by construction, a sharp rigid divisor.
Note that the change of notation from \ref{localHomEps} 
has been necessary to put the right multiplicity on the cycles in the post-critical set.
	Consequently, Theorem \ref{Epsteinvanishing} implies
	$$
	{\rm \mathbb{E}xt}^2(\Omega_{\bul}, \mathcal{O}(-\Delta_{\bul}))=0.
	$$
	Finally, applying property \ref{splitproperty},
	we can put together the latter results, cf. \ref{localHomEps}, \ref{computationExtcriticaldivisor} in order to conclude that
	$$
	\dim_{\C} {\rm \mathbb{E}xt^1}(\Omega_{\bul}, \mathcal{O}_{\Delta_{\bul}})=
	2D-2+ \gamma_{A}- \delta_f,
	$$
	where $A$ is the support of $\Delta_0$.
\end{proof}
\newpage
\phantom{h}
\thispagestyle{empty}
\newpage
\chapter*{Appendix A}
	\addtocontents{toc}{\def\string\@dotsep{100}}
\addcontentsline{toc}{chapter}{\large{\textbf{Appendix A}}}
\section*{Extensions in abelian categories}
Let $\mathscr{C}$ be an abelian category.
\begin{definition}
	Let $A,B $ two objects of $\mathscr{C}$ and $i>0$. The set 
	of $i$-extensions of $A$ by $B$ is denoted by $\mathcal{E}xt^i(A,B)$,
	and consists of exact sequences of the form
		\begin{equation*}
		\begin{tikzcd}[row sep=1.6pc, column sep=1.4pc]
			\xi: & 0 \arrow[r] &  B \arrow[r, "e^i"] &
			E^i \arrow[r,  "e^{i-1}"] & \dots \arrow[r, "e^1" ]&
			E^1 \arrow[r,  "e^0"]& A \arrow[r] &  0.
		\end{tikzcd}
	\end{equation*}
	 Let us consider the following equivalence relation on $\mathcal{E}xt^i(A,B)$:
	 $$
	 \xi \sim \xi'  \Leftrightarrow  \; \exists \, \mbox{ a chain map } \eta: \xi \to \xi' 
	 \mbox{ agreeing with } id_A \mbox{ and } id_B \mbox{ on the sides.}
	 $$
	 It can be shown that each resulting map $E^i \to (E')^i$ is an isomorphism
	 (e.g. by the Five Lemma). Therefore, the relation $\sim$ defined above is an
	 equivalence relation. Let 
	 $$
	 {\rm Ext}^i(A,B)\coloneqq \mathcal{E}xt^i(A,B)/\sim.
	 $$
\end{definition}

\begin{fact}
	Let $\mathscr{C}$ be an abelian category with enough injectives.
	Then, for $i>0$, we have a natural isomorphism
	$$
	{\rm Ext}^i(A,B) \overset{\sim}{\longrightarrow} R^i {\rm Hom}(A,B).
	$$
\end{fact}
\begin{proof}[Sketch of the Proof]
	The proof proceed by induction on $i$. Let $i=1$
	and consider an injective object $I$ such that $B \hookrightarrow I$. 
	Let $Q$ denote the quotient $I/B$. Applying ${\rm Hom}(A,-)$
	to the resulting short exact sequence we get an exact sequence
		\begin{equation*}
		\begin{tikzcd}[row sep=1.6pc, column sep=1.4pc]
			0 \arrow[r] & {\rm Hom}(A,B) \arrow[r] & 
			{\rm Hom}(A, I) \arrow[r] & {\rm Hom}(A, Q) \arrow[r] &
			{\rm Ext^1}(A,B) \arrow[r] & 0.
		\end{tikzcd}
	\end{equation*}
Therefore, we have 
$$
	{\rm Ext^1}(A,B) \cong {\rm Hom}(A, Q)/ {\rm Im}\left({\rm Hom}(A, I)\right),
$$
which is the desired isomorphism, since there is an isomorphism
${\rm Hom}(A, Q)\overset{\sim}{\longrightarrow} \mathcal{E}xt^1(A,B)$, 
obtained by pull-back,
\begin{equation*}
	\begin{tikzcd}[row sep=1.6pc, column sep=1.4pc]
		0 \arrow[r] & B \arrow[r] & 
		I \arrow[r] &  Q \arrow[r] \arrow[dl, phantom, ""{name=U, below , draw=black}]{} & 0 \\
		0 \arrow[r] & B \arrow[r] \arrow[u, equal] & 
	 E \arrow[r] \arrow[u] &  A \arrow[r] \arrow[u, "f"] & 0 
	\end{tikzcd}
\end{equation*}
Finally, note that the class of trivial extension, \textit{i.e.} the split one, 
is isomorphic to the image of $	{\rm Hom}(A, I)$. In fact, a map
$A \to Q$ lifts to $I$ if and only if lifts to $E$, and hence gives a section.
Let now $i>1$ and recall that from 
the long exact sequence discussed above we get an isomorphism
$$
	{\rm Ext^i}(A,Q) \overset{\sim}{\longrightarrow} 	{\rm Ext^{i+1}}(A,B).
$$
By induction, the term on the left is a $i$-extension, so we have
	\begin{equation*}
	\begin{tikzcd}[row sep=1.6pc, column sep=1.4pc]
  0 \arrow[r] &  Q \arrow[r, "e^i"] &
		E^i \arrow[r,  "e^{i-1}"] & \dots \arrow[r, "e^1" ]&
		E^1 \arrow[r,  "e^0"]& A \arrow[r] &  0.
	\end{tikzcd}
\end{equation*}
It is easy to check that the isomorphism above
sends this  $i$-extension to the $i+1$-extension
	\begin{equation*}
	\begin{tikzcd}[row sep=1.6pc, column sep=1.4pc]
		 0 \arrow[r] &   B \arrow[r, "e^{i+1}"]& I  \arrow[r, "e^i"] &
		E^i \arrow[r,  "e^{i-1}"] & \dots \arrow[r, "e^1" ]&
		E^1 \arrow[r,  "e^0"]& A \arrow[r] &  0
	\end{tikzcd}
\end{equation*}
obtained by composition.
\end{proof}
\chapter*{Appendix B}
\addtocontents{toc}{\def\string\@dotsep{100}}
\addcontentsline{toc}{chapter}{\large{\textbf{Appendix B}}}
\section*{Simplicial sheaves}
\begin{definition}[{\bf Classifying simplicial space of $\Sigma$}]
	We denote by $B_{\Sigma}$ the category with one object and arrows given 
	by $\Sigma$, \textit{i.e}
	$$
	ob(B_{\Sigma})=pt, \quad {\rm Hom}_{B_{\Sigma}}(pt,pt)=\Sigma.
	$$
	Note that this category is well defined since $\Sigma$ contains the identity.
	The nerve, cf.{\rm \cite{nlab:nerve}}, of this category is a simplicial 
	set denoted by $B_{\bul,\Sigma}$, which in degree $n$ is
	the n-fold fibered product of $B_{1,\Sigma}\coloneqq ar(B_{\Sigma}) $ 
	over $B_{0,\Sigma}\coloneqq ob(B_{\Sigma})$ with respect to the source and target functors.
\end{definition}
Let us describe explicitly the simplicial set $B_{\bul,\Sigma}$:
\begin{equation*}
	\begin{tikzcd}[column sep=2.5pc,row sep=2pc]
		B_{0,\Sigma}=pt \arrow[r]
		&\arrow[l, shift left=2]
		\arrow[l, shift right=2] B_{1,\Sigma}=pt \times \Sigma \arrow[r,shift left=2] \arrow[r,shift right=2]
		& \arrow[l] \arrow[l, shift left=4]
		\arrow[l, shift right=4] B_{2,\Sigma}=pt \times \Sigma \times \Sigma \cdots
	\end{tikzcd}
\end{equation*}
\begin{itemize}
	\item For $n > 0$ and $0\leq j \leq n$ the face map
	$d_{n,j}:B_{n,\Sigma} \to B_{n-1,\Sigma}$ ``forgets" the $j$th entry,
	\textit{i.e.} for any $(pt,\underline{\sigma}) \in B_{n,\Sigma}$,
	where $\underline{\sigma}=(\sigma_1,\dots,\sigma_{n})$ we have
	\begin{equation*}\label{maps}
		\begin{array}{ll}
			d_{n,0}(pt,\underline{\sigma})=(pt,\sigma_2,\dots,\sigma_n),&  \\
			d_{n,n}(pt,\underline{\sigma})=(pt,\sigma_1,\dots,\sigma_{n-1}), &  \\
			d_{n,j}(pt,\underline{\sigma})=
			(pt, \sigma_1,\dots, \sigma_{j-1},\sigma_{j+1}\sigma_j,\sigma_{j+2},\dots, \sigma_n),& j=1,\dots,n-1;
		\end{array}
	\end{equation*}
	\item For $n\geq 0$ and $0\leq j \leq n$ the
	degeneracy map $s_{n,j}:B_{n,\Sigma} \to B_{n+1,\Sigma}$
	adds a new edge by means of the identity,
	\textit{i.e.}
	\begin{equation*}\label{maps1}
		s_{n,j}(pt,\underline{\sigma})=(pt,\sigma_1,\dots,id_{\Sigma},\sigma_j,\dots,\sigma_n)
	\end{equation*}
\end{itemize}
Note that we can endow $B_{\Sigma}$ with the discrete topology,
hence $B_{\bul,\Sigma}$ is naturally a simplicial object 
with values in the category of topological spaces.
Let us recall the following definition, cf. \cite{MR0498552}, of the category
$\textrm{SimpSh}_{\mathcal{C}}(X_{\bul})$.
Let $X_{\bullet}=(X_n)_{n\in\N}$ be \textit{any} simplicial object
in the category of topological spaces and let
$\mathcal{C}$ be a small category. \\

\begin{definition} A simplicial sheaf on $X_{\bullet}$
	with values in $\mathcal{C}$
	consists of:
	\begin{enumerate}
		\item A collection $\{\mathcal{F}_n\}_{n\in \N}$,
		where $\mathcal{F}_n$ is a sheaf on $X_n$
		with values in $\mathcal{C}$;
		\item For each simplicial map $g:X_n \to X_m$,
		a morphism of sheaves on $X_n$
		\begin{equation}\label{structmap}
			\varphi_g: g^*\mathcal{F}_m \to \mathcal{F}_n,
		\end{equation}
		to which we will refer as ``structural morphisms'', or
		``structure maps"
		of the simplicial sheaf $\mathcal{F}_{\bul}$,
		satisfying the following composition property:\\
		for any $g:X_n\to X_m$, $h:X_m\to X_l$, we have
		\begin{equation}\label{comp}
			\varphi_g\circ g^*\varphi_h=\varphi_{hg}.
		\end{equation}
	\end{enumerate}
\end{definition}
\begin{definition}
	A morphism $\theta_{\bul}: \mathcal{F}_{\bul} \to \mathcal{G}_{\bul}$
	between two simplicial sheaves on $X_{\bul}$ consists of a collection
	$\theta_n:\mathcal{F}_n\to \mathcal{G}_n$ of morphisms of sheaves on $X_n$
	such that, for each simplicial map $g:X_n \to X_m$,
	the following diagram commutes
	\begin{equation}\label{simpmap}
		\begin{tikzcd}[row sep=2.6pc, column sep=2.6pc]
			g^*\mathcal{F}_m \arrow[r, "\varphi_g"] \arrow[d,"g^*\theta_m"'] &
			\mathcal{F}_n \arrow[d, "\theta_n"] \\
			g^*\mathcal{G}_m \arrow[r, "\gamma_g"'] & \mathcal{G}_n
		\end{tikzcd}
	\end{equation}
	where $\varphi_g$ and $\gamma_g$ are the structural morphisms,
	\eqref{structmap}, of, respectively,  $\mathcal{F}_{\bul}$ and $\mathcal{G}_{\bul}$. 
	The set of morphisms $\theta_{\bul}: \mathcal{F}_{\bul} \to \mathcal{G}_{\bul}$, denoted by 
	$\mathbb{H}{\rm om}(\mathcal{F}_{\bul}, \mathcal{G}_{\bul})$
	could have been equivalently defined as the equalizer of
	\begin{equation}\label{Homgeneralsimplicial}
		\begin{tikzcd}[row sep=2.6pc, column sep=1.6pc]
			\prod\limits_{n\in \N}{\rm Hom}(\mathcal{F}_n,\mathcal{G}_n)
			\arrow[r, shift left] \arrow[r, shift right]
			&  \prod\limits_{g:X_n\to X_m} {\rm Hom}(g^*\mathcal{F}_m,\mathcal{G}_n)
		\end{tikzcd}
	\end{equation}
\end{definition}
where the two arrows assign to $(\theta_n)_n$
the two maps obtained in \eqref{simpmap}, \textit{i.e.}
$\theta_n\circ \varphi_g$ and $\gamma_g\circ g^*\theta_m$.
Let us illustrate the
properties of the T\`opos $Sh([pt/\G])$ through the following
remark due to Deligne, cf. \cite[6.1.2,b)]{MR0498552}.
\begin{example*}
	Let $G$ be a group and take $B_{\bul, G}$ to be the nerve
	of the classifying space $B_G=[pt/G]$ of $G$,
	(the category whose objects are $ob(B_G)=\{pt\}$
	and whose arrows are ${\rm Hom}_{B_G}(pt,pt)=G$).
	The simplicial space $B_{\bul, G}$ is known
	as the classifying simplicial space of $G$.
	There is an equivalence between
	the category of $G$-modules 
	and the subcategory of the category of
	simplicial sheaves on $B_{\bul, G}$, cf. Appendix, 
	consisting of a sequence of sheaves $\mathcal{F}_{\bul}\coloneqq (\mathcal{F}_n)_{n\in \N}$ 
	satisfying the following property
	\begin{equation}\label{D_note}
		\varphi_g:	g^*\mathcal{F}_m \overset{\sim}{\longrightarrow}
		\mathcal{F}_n, \quad \mbox{\rm $\forall g:B_{n, G} \to B_{m, G}$}.
	\end{equation}
	We denote by $Sh(B_{G})$ the above category.
	The above mentioned equivalence
	sends  $\mathcal{F}_{\bul}\mapsto \mathcal{F}_0$. 
	Let us give a brief description of this equivalence.
	Let $s,t$ be the face maps in degree 0 and 1,
	\begin{equation*}
		\begin{tikzcd}
			pt & pt\times G \arrow[l, shift right, "s"'] \arrow[l, shift left, "t"]
		\end{tikzcd}
	\end{equation*}
	defined as follows: $s(pt \times \gamma)=pt^{\gamma}$, $t(pt \times \gamma)=pt$.
	From \eqref{D_note} we get an isomorphism
	\begin{equation}\label{action}
		s^*\mathcal{F}_0 \overset{\sim}{\longrightarrow} t^*\mathcal{F}_0
	\end{equation}
	\textit{i.e.} an isomorphism $s_g:\mathcal{F}_0 \to \mathcal{F}_0$ for each $g\in G$,
	while the composition property $\varphi_g\circ g^*\varphi_h=\varphi_{hg}$ assures these maps
	give raise to an action of $G$ on $\mathcal{F}_0$, hence $\mathcal{F}_0$ is a $G$-module.
	The other direction works roughly as follows: 
	we start with an abelian
	group $\mathcal{F}_0$ together with an action of $G$
	(which is equivalent to giving \eqref{action})
	and construct inductively $\mathcal{F}_n$
	by pulling back $\mathcal{F}_0$ along any simplicial map $g:B_{n, G} \to B_{n-1, G}$. 
	Using the relations among the simplicial maps
	one can show
	that, modulo isomorphism, we get a unique simplicial sheaf on $B_{\bul\, G}$.
\end{example*}
The development of our theory has been suggested 
by the above remark, once one accepts that the nature of 
condition \eqref{D_note} is too restrictive 
when we consider the action of a semigroup, and hence it has to be dropped.

\bibliographystyle{Gamsalpha}
\bibliography{bibbase}

\providecommand{\bysame}{\leavevmode\hbox to3em{\hrulefill}\thinspace}
\providecommand{\MR}{\relax\ifhmode\unskip\space\fi MR }
\providecommand{\MRhref}[2]{%
  \href{http://www.ams.org/mathscinet-getitem?mr=#1}{#2}
}
\providecommand{\href}[2]{#2}
\begin{thebibliography}{{nLa}22}

\bibitem[BE02]{article}
Xavier Buff and Adam~L. Epstein, \emph{{``A parabolic Pommerenke-Levin-Yoccoz
  inequality"}}, Fundamenta Mathematicae \textbf{172} (2002).

\bibitem[BT82]{BottTu}
Raoul Bott and Loring~W. Tu, \emph{Differential forms in algebraic topology},
  Graduate Texts in Mathematics, vol.~82, Springer-Verlag, New York-Berlin,
  1982. \MR{\htmladdnormallink{658304
  (83i:57016)}{http://www.ams.org/mathscinet/search/publdoc.html?arg3=&co4=AND&co5=AND&co6=AND&co7=AND&dr=all&pg4=AUCN&pg5=AUCN&pg6=PC&pg7=ALLF&pg8=ET&review_format=html&s4=bott&s5=tu&s6=&s7=&s8=All&vfpref=html&yearRangeFirst=&yearRangeSecond=&yrop=eq&r=3&mx-pid=658304}}

\bibitem[Del74]{MR0498552}
Pierre Deligne, \emph{{``Th\'eorie de {H}odge. {III}"}}, Inst. Hautes \'Etudes
  Sci. Publ. Math. (1974), no.~44, 5--77. \MR{0498552}

\bibitem[DH93]{MR1251582}
Adrien Douady and John~H. Hubbard, \emph{{``A proof of {T}hurston's topological
  characterization of rational functions"}}, Acta Math. \textbf{171} (1993),
  no.~2, 263--297. \MR{1251582}

\bibitem[EM47]{eilenberg-maclane}
Samuel Eilenberg and Saunders MacLane, \emph{Cohomology theory in abstract
  groups. {I}}, Ann. of Math. (2) \textbf{48} (1947), 51--78. \MR{19092}

\bibitem[{Eps}99]{1999math......2158E}
Adam~L. {Epstein}, \emph{{``Infinitesimal Thurston Rigidity and the
  Fatou-Shishikura Inequality"}}, ArXiv:9902158 (1999).

\bibitem[Eps09]{transversality}
Adam Epstein, \emph{{``Transversality in holomorphic dynamics"}}, Manuscript,
  2009.

\bibitem[GH78]{griffithsharris}
Phillip Griffiths and Joseph Harris, \emph{Principles of algebraic geometry},
  Pure and Applied Mathematics, Wiley-Interscience [John Wiley \& Sons], New
  York, 1978. \MR{507725}

\bibitem[Gir71]{giraud}
Jean Giraud, \emph{Cohomologie non ab\'elienne}, Springer-Verlag, Berlin-New
  York, 1971, Die Grundlehren der mathematischen Wissenschaften, Band 179.
  \MR{\htmladdnormallink{0344253 (49
  \#8992)}{http://www.ams.org/mathscinet/search/publdoc.html?arg3=&co4=AND&co5=AND&co6=AND&co7=AND&dr=all&pg4=AUCN&pg5=TI&pg6=PC&pg7=ALLF&pg8=ET&review_format=html&s4=giraud&s5=cohomologie&s6=&s7=&s8=All&vfpref=html&yearRangeFirst=&yearRangeSecond=&yrop=eq&r=1&mx-pid=344253}}

\bibitem[\htmladdnormallink{SGA1}{http://arxiv.org/pdf/math/0206203v2}]{sga1}
Alexander Grothendieck, \emph{Rev\^etements \'etales et groupe fondamental
  ({SGA} 1)}, Lecture Notes in Mathematics, Vol. 224, Springer-Verlag, Berlin,
  1971, S{\'e}minaire de g{\'e}om{\'e}trie alg{\'e}brique du Bois Marie
  1960--61, Augment{\'e} de deux expos{\'e}s de M. Raynaud. \MR{MR0354561}

\bibitem[\htmladdnormallink{SGA-IV}{http://library.msri.org/books/sga/djvu/SGA\%204-3.tif.djvu}]{SGA4}
\emph{Th\'eorie des topos et cohomologie \'etale des sch\'emas. {T}ome 3},
  Lecture Notes in Mathematics, Vol. 305, Springer-Verlag, Berlin-New York,
  1973, S{\'e}minaire de G{\'e}om{\'e}trie Alg{\'e}brique du Bois-Marie
  1963--1964 (SGA 4), Dirig{\'e} par M. Artin, A. Grothendieck et J. L.
  Verdier. Avec la collaboration de P. Deligne et B. Saint-Donat.
  \MR{\htmladdnormallink{0354654 (50
  \#7132)}{http://www.ams.org/mathscinet/search/publdoc.html?arg3=&co4=AND&co5=AND&co6=AND&co7=AND&dr=all&pg4=AUCN&pg5=TI&pg6=PC&pg7=ALLF&pg8=ET&review_format=html&s4=&s5=topos\%20et\%20co\%2A&s6=&s7=&s8=All&vfpref=html&yearRangeFirst=&yearRangeSecond=&yrop=eq&r=1&mx-pid=354654}}

\bibitem[Joh02]{johnstone2002sketches}
P.T. Johnstone, \emph{Sketches of an elephant: A topos theory compendium, vol.
  1}, @Oxford logic guides, Clarendon Press, 2002.

\bibitem[LMB00]{L-MB}
G{\'e}rard Laumon and Laurent Moret-Bailly, \emph{Champs alg\'ebriques},
  Ergebnisse der Mathematik und ihrer Grenzgebiete. 3. Folge. A Series of
  Modern Surveys in Mathematics [Results in Mathematics and Related Areas. 3rd
  Series. A Series of Modern Surveys in Mathematics], vol.~39, Springer-Verlag,
  Berlin, 2000. \MR{\htmladdnormallink{1771927
  (2001f:14006)}{http://www.ams.org/mathscinet/search/publdoc.html?arg3=&co4=AND&co5=AND&co6=AND&co7=AND&dr=all&pg4=AUCN&pg5=TI&pg6=PC&pg7=ALLF&pg8=ET&review_format=html&s4=laumon&s5=ch\%2A&s6=&s7=&s8=All&vfpref=html&yearRangeFirst=&yearRangeSecond=&yrop=eq&r=3&mx-pid=1771927}}

\bibitem[McQ15]{et}
Michael McQuillan, \emph{Elementary topology of champs}, ArXiv:1507.00797,
  2015.

\bibitem[Mil80]{milne}
James~S. Milne, \emph{\'{E}tale cohomology}, Princeton Mathematical Series,
  vol.~33, Princeton University Press, Princeton, N.J., 1980.
  \MR{\htmladdnormallink{559531
  (81j:14002)}{http://www.ams.org/mathscinet/search/publdoc.html?arg3=&co4=AND&co5=AND&co6=AND&co7=AND&dr=all&pg4=AUCN&pg5=TI&pg6=PC&pg7=ALLF&pg8=ET&review_format=html&s4=milne&s5=cohomology&s6=&s7=&s8=All&vfpref=html&yearRangeFirst=&yearRangeSecond=&yrop=eq&r=5&mx-pid=559531}}

\bibitem[Mil06]{Milnor}
John Milnor, \emph{Dynamics in one complex variable}, third ed., Annals of
  Mathematics Studies, vol. 160, Princeton University Press, Princeton, NJ,
  2006. \MR{2193309}

\bibitem[{nLa}22]{nlab:nerve}
{nLab authors}, \emph{nerve}, \url{http://ncatlab.org/nlab/show/nerve}, June
  2022, \href{http://ncatlab.org/nlab/revision/nerve/73}{Revision 73}.

\bibitem[\htmladdnormallink{TSP}{https://stacks.math.columbia.edu/tag/01DL}]{enoughinjective}
The~Stacks project, \emph{{``Abelian Sheaves on a site"}}.

\bibitem[\htmladdnormallink{TSP}{https://ncatlab.org/nlab/show/finite+limit}]{finitelimits}
\bysame, \emph{{``Finite Limits"}}.

\bibitem[\htmladdnormallink{TSP}{https://stacks.math.columbia.edu/tag/03AG}]{torsor}
\bysame, \emph{{``First cohomology and torsors"}}.

\bibitem[\htmladdnormallink{TSP}{https://stacks.math.columbia.edu/tag/03PZ}]{inverseimage}
\bysame, \emph{{``Inverse Image"}}.

\bibitem[\htmladdnormallink{TSP}{hhttps://stacks.math.columbia.edu/tag/03N4}]{etalesite}
\bysame, \emph{{``The étale topology"}}.

\bibitem[Shi87]{ASENS_1987_4_20_1_1_0}
Mitsuhiro Shishikura, \emph{{``On the quasiconformal surgery of rational
  functions"}}, Annales scientifiques de l'\'Ecole Normale Sup\'erieure
  \textbf{Ser. 4, 20} (1987), no.~1.

\bibitem[SML92]{sheavesmaclane}
Ieke Moerdijk~(auth.) Saunders Mac~Lane, \emph{Sheaves in geometry and logic: A
  first introduction to topos theory}, 1 ed., Universitext, Springer, 1992.

\end{thebibliography}
\newpage
\phantom{h}
\thispagestyle{empty}
\newpage
\section*{Ringraziamenti}
\setstretch{1.2}
\thispagestyle{empty}
Si chiude un percorso importante per me e sento che è doveroso ringraziare 
chi è stato dietro le quinte di questo lavoro.

I ringraziamenti matematici al mio relatore, Michael, 
sono già alla fine dell'introduzione, 
e perciò non mi 
dilungherò ancora su questi. Lo ringrazio ancora per avermi concesso l'opportunità di studiare il lavoro di Grothendieck. 
Ha azzeccato sempre i miei gusti in ambito matematico,
quando ancora io non sapevo cosa mi piacesse studiare.
Umanamente, in aggiunta, gli devo un sentito ringraziamento. 
A lui devo riconoscere soprattutto 
la capacità di avermi spesso fatto sentire, almeno nei primi anni, all'altezza di poter accettare ogni sua sfida. La sua ironia è diventata parte integrante di me in questi anni, e penso che mi mancheranno le sue battute.

Se penso a chi mi ha sostenuto sempre e comunque, con amore, 
il primo pensiero va ovviamente ai miei genitori, Loredana e Maurizio. 
Loro hanno capito da subito l'emozione, per me così importante, che provavo a far parte della comunità matematica, e mi hanno incoraggiato in ogni fase di questa scelta post-universitaria a dare tutto me stesso per dimostrare quanto valessi. 

Ringrazio i miei genitori anche per tutte le volte che mi hanno ricordato
che ero un privilegiato che poteva alzarsi tardi la mattina e comunque portare la pagnotta a casa. 
Farò sempre tesoro di tutti i loro insegnamenti, e porterò sempre nel cuore il 
loro calore e affetto, che mi dimostrano ogni 
giorno, ringraziandoli per la loro capacità di accettarmi così come 
sono, senza cercare di cambiarmi. 
Grazie anche a mia sorella Noemi perché sa trasmettermi la tranquillità di cui ho bisogno, 
la sicurezza di potercela fare in ogni caso perché, in fondo, non c'è nulla da perdere.

\thispagestyle{empty}
La persona che certamente ha contribuito più di tutti alla realizzazione di questo lavoro
è Francesco, che ha subito i miei momenti peggiori e ha raccolto i pezzi di un essere umano 
sfiancato e inerme nei confronti della vita. Anche grazie al suo affetto e alla sua pazienza c'è stata una lieta conclusione 
di questo percorso, senza di quello forse avrei gettato la spugna. 

Uno speciale ringraziamento lo devo ai miei colleghi e amici di università, senza i 
quali passare tre anni (anzi quattro ormai) 
in questa università sarebbe stato certamente molto diverso. 
Tutti loro hanno reso  questi momenti estremamente piacevoli e stimolanti, 
li hanno riempiti di risate e spensieratezza. Grazie a Mattia, 
Daniele e Emanuele, Claudio, Alessio e Andrea, i più stretti, 
ma anche agli altri ragazzi della facoltà di matematica, in particolare Leonardo, Giorgio e Edoardo.
Ultimi, ma non per importanza, i vecchi colleghi che non mi hanno abbandonato anche 
dopo aver preso la loro strada dopo l'università, in particolare grazie a Lorenzo, Fabio, Martina e Cristiano, e infine Gianluca. 

Tra i colleghi più anziani dell'università il grazie più caloroso va sicuramente a Martina! Grazie a lei
la mia vita è stata stravolta positivamente. Non solo mi ha dato l'occasione di 
capire quanto è importante essere indipendente, mettendomi serenamente a disposizione  la sua casa,
ma anche per aver sempre creduto in me matematicamente. 
\`E soprattutto grazie al suo modo di fare inclusivo che ho partecipato con piacere 
agli eventi di dipartimento. 
Grazie anche a Filippo, Antonio, Paolo e Beppe per la loro simpatia, e a 
tutti gli altri professori che ho incontrato negli anni e che mi hanno 
permesso di arricchire il mio bagaglio culturale. Il più importante di questi è stato certamente Roberto, maestro indiscusso. 
\`E stato lui a tenere la mia prima lezione universitaria, segnando il cammino che mi ha fatto innamorare perdutamente della matematica. 
Il ringraziamento ad Adam dato nell'introduzione non è stato sufficiente a mio avviso. 
Per mantenere un minimo di serietà non ho potuto fare cenno alle serate indimenticabili 
passate in sua compagnia  a parlare di matematica, 
a mangiare noodles e a bere birra. 

\thispagestyle{empty}
I momenti felici di cazzeggio trascorsi con i miei amici di sempre, passati a spiegargli quanto io non sia un alieno
ma solo uno a cui piace un po' troppo pensare, meritano più di un ringraziamento. 
Senza la loro leggerezza sarei già impazzito ben prima di iniziare il dottorato. 
Grazie a Lollo, Silvio, Francesco, Matteo e a tutti gli altri matti di Ciampino, grazie di cuore
a Diego per le infinite serate passate a confidarsi e grazie a mia cugina Ilaria senza la quale la mia famiglia non riuscirebbe a tenersi ancora insieme. Un grazie caloroso va al mio veterano studente Valerio.
La stima che lui mi dimostra ogni volta riesce a riempirmi di orgoglio più di 
chiunque altro. Grazie anche a Sonia per i suoi preziosi consigli e per aver contribuito sostanzialmente alla mia crescita personale.

Per finire, un pensiero va a tutti i parenti stretti, dagli zii più cari, Claudio e Anna, 
ai vecchi e saggi nonni, a nonna Fioretta che se avessi insistito sarebbe venuta, 
con un po' di fatica, ad assistere alla mia discussione,  
e agli altri che ormai non ci sono più. In particolare a nonno Guglielmo. 
\`E da sempre la sua tenacia e la sua intelligenza  che cerco di fare mie, la sua forza che cerco di imitare e la sua 
allegria che cerco di ricordare. 

Grazie, grazie, grazie.
\bigbreak 
\smallbreak
Roma, 28 Ottobre 2022 \hfill Jacopo

\end{document}